\documentclass[10pt]{amsart}
\usepackage[margin=0.8in]{geometry}
\usepackage{amsthm,amsmath,amsfonts,amssymb,euscript,hyperref,color,slashed,mathrsfs}
\hypersetup{hypertexnames=false,hidelinks,bookmarksnumbered=true}
\usepackage{graphicx}
\usepackage{tikz}
\usetikzlibrary{3d,calc,patterns}
\usepackage{mathtools}

\newcommand{\Xh}{\widehat{X}}
\newcommand{\Ntop}{N_{_{\rm top}}}
\newcommand{\ut}{\tilde{u}}
\newcommand{\urr}{\mathring{u}}
\newcommand{\Wr}{\mathring{W}}
\newcommand{\Ncomm}{N_{\mathrm{comm}}}
\newcommand{\Xr}{\mathring{X}}
\newcommand{\Tr}{\mathring{T}}
\newcommand{\Ur}{\mathring{U}}
\newcommand{\Rr}{\mathring{R}}
\newcommand{\Zr}{\mathring{Z}}
\newcommand{\Vr}{\mathring{V}}
\newcommand{\kappar}{\mathring{\kappa}}
\newcommand{\wb}{\underline{w}}
\newcommand{\Cb}{\underline{C}}

\newtheorem{theorem}{Theorem}[section]
\newtheorem{lemma}[theorem]{Lemma}
\newtheorem{proposition}[theorem]{Proposition}
\newtheorem{corollary}[theorem]{Corollary}

\newtheorem{remark}[theorem]{Remark}

\numberwithin{equation}{section}
\begin{document}
\title { Multi-dimensional centered rarefaction waves of gas dynamics with vacuum state I}

\author{Jin JIA and Tao LUO}

\address{School of Mathematics, Hunan University\\ Changsha, China}
\email{jiajin2023@hnu.edu.cn}

\address{Department of Mathematics, City University of Hong Kong\\ Hong Kong, China}
\email{taoluo@cityu.edu.hk}

\begin{abstract}
We study whether two centered rarefaction waves enclosing a vacuum region
persist under small two-dimensional perturbations of planar Riemann data.
For the isentropic Euler equations with $p(\rho)=k\rho^\gamma$ and
$\gamma\in(1,1.3]\cup\{\frac75,\frac53\}$, we establish local-in-time
structural stability without an irrotationality assumption. The resulting
conservative entropy solution is continuous for positive time and has five
regions separated by four characteristic surfaces. In the opened
coordinates $(x_1/t,x_2)$, it satisfies uniform spatial $W^{2,\infty}$ bounds
throughout the central region, including both vacuum interfaces. For
$\gamma=7/5,5/3$, we also obtain spacetime $W^{2,\infty}$ regularity in
$(t,x_1/t,x_2)$.

The proof couples characteristic geometry directly to the symmetric
first-order Euler system. Two overlapping characteristic charts separate
matching with the exterior flow from estimates uniform up to vacuum.
Unweighted spatial commutators reveal a first-order null structure, and
subtracting explicit approximate solutions produces damping that cancels
the leading antidamping terms in the energy identity. The initial-slice
ODEs identify fractional powers and logarithmic terms at zero sound speed;
the logarithmic coefficients couple vorticity to tangential variation of
the data. These formulas determine the finite approximation order and
explain the obstruction to higher spatial regularity, even when the
sound-speed powers are integers. Finally, estimates uniform in both the
positive starting time and the vacuum cutoff yield the solution by a
single diagonal limit. The same estimates identify the requirements for
extending the range of exponents using finite evolving profiles.
\end{abstract}

\maketitle

\tableofcontents

\section{Introduction}

\subsection{The problem and its motivation}
Two gas states moving apart sufficiently fast generate a vacuum region
bounded by two centered rarefaction waves. The structural stability
problem asks whether this configuration persists when the two initial
states are perturbed in both spatial variables. We allow nonzero
vorticity and seek both the flow inside the fans and their moving
characteristic boundaries, with estimates that remain uniform up to
vacuum.

We consider the two-dimensional isentropic Euler equations
\begin{equation}\label{eq: Euler in rho v}
\begin{cases}
\partial_t(\rho)+\nabla\cdot(\rho v)=0,\\
\partial_t(\rho v)+\nabla\cdot\bigl(p\,\mathrm{Id}+\rho v\otimes v\bigr)=0,
\end{cases}
\end{equation}
on $\mathbb R\times\mathbb T$, where $\rho\geq0$, $v=(v^1,v^2)$,
$\mathrm{Id}$ is the $2\times2$ identity matrix, and
$p(\rho)=k\rho^\gamma$ with $k>0$ and $1<\gamma<3$. The sound speed is
$c=\sqrt{k\gamma}\,\rho^{(\gamma-1)/2}$.
We write the conservative variables as $(\rho,m)$, with $m=\rho v$
in the gas and $(\rho,m)=(0,0)$ in vacuum. Throughout the paper,
\[
U=\left(\frac{2c}{\gamma-1},v^1,v^2\right)^T
\]
denotes the normalized unknown in the gas regions.
For constant left and right states
$\mathbf U_j=(\frac{2\mathbf c_j}{\gamma-1},\mathbf v_j^1,
\mathbf v_j^2)^T$, $j=l,r$, with $\mathbf c_j>0$,
the planar Riemann solution has the pattern
$R$--\textbf{vacuum state}--$R$ precisely when
\begin{equation}\label{eq: strict vacuum gap}
\mathbf v_l^1+\frac{2\mathbf c_l}{\gamma-1}
<\mathbf v_r^1-\frac{2\mathbf c_r}{\gamma-1}.
\end{equation}
Thus vacuum can form immediately from an initial jump between two
strictly positive gas states; see \cite{CourantFriedrichs,Dafermos}.
The strict gap in \eqref{eq: strict vacuum gap} is the reference
separation that we preserve under perturbation.

Two singular limits govern the problem. As $t\rightarrow 0^{+}$, the fan
collapses to the initial interface and physical spatial derivatives may
grow like $t^{-1}$. We call $(t,\urr,x_2)$, where $\urr=x_1/t$ for
$t>0$, the \emph{opened coordinates}. The map $x_1=t\urr$ collapses
the opened initial fan onto $\{t=0,x_1=0\}$ in physical coordinates.
By keeping $\urr$ as a coordinate at $t=0$, we can describe the initial profile of
the fan. At the inner end of each fan, $c\rightarrow 0^{+}$ instead
causes the acoustic characteristic speeds to coalesce, see \eqref{eq: eigenvalues and eigenvectors} and the
characteristic energy flux to degenerate, see \eqref{eq: linear energy estimate in t x_1 x_2}. The construction must control
both limits while matching the fan to the exterior flow.

Theorem~\ref{theorem: stability of RVR} establishes this construction for
$\gamma\in(1,1.3]\cup\{7/5,5/3\}$. It gives a bounded entropy
solution $(\rho,m)$ with two exterior gas regions, two rarefaction
fans, and an open vacuum region. In the opened coordinates
$(t,\urr,x_2)$, the normalized gas state $U$ satisfies quantitative
profile comparisons and uniform spatial $W^{2,\infty}$ bounds
in each fan up to vacuum. For $\gamma=7/5,5/3$, it also satisfies
uniform spacetime $W^{2,\infty}$ bounds in each fan.
Section~\ref{subsec: proof main theorem} completes the construction of the solution
for positive time $t\in [\delta,t^*]$ and removes the initial-time and vacuum
cutoffs together. In the normalization $t^*=1$, the positive starting time and the
vacuum cutoff are coupled by $0<\delta\leq d_0\alpha$, where
$0<\alpha<1/100$ is the vacuum cutoff; see
Proposition~\ref{prop: positive time construction}. The constant $d_0$
is fixed independently of $\delta$ and $\alpha$. Its purpose is to
ensure that the sound speed on the initial slice $t=\delta$ remains
comparable to the initial sound-speed profile on the part of the fan
retained by the vacuum cutoff. More precisely, if $c_0$ denotes this
profile and $\mathbf c_*$ is the fixed reference sound-speed scale
introduced in Section~\ref{subsec: opened fan}, then the retained
region satisfies $c_0\geq\alpha\mathbf c_*$. The construction gives
the uniform comparison $|c_\delta-c_0|\leq C\varepsilon\delta$,
where $C$ is independent of the cutoffs. Once the smallness threshold
$\varepsilon_0$ is fixed, $d_0$ is chosen so that
$C\varepsilon_0d_0\leq\mathbf c_*/4$. Thus
$\delta\leq d_0\alpha$ guarantees
$3c_0/4\leq c_\delta\leq5c_0/4$ on the retained region. The
comparison and the positive-time construction are proved in
Proposition~\ref{prop: positive time construction}.

\subsection{Preparations for the main theorem}\label{subsec: main theorem preparation}
Fix $k>0$, a time $t^*>0$, and
$\gamma\in(1,1.3]\cup\{7/5,5/3\}$. The sufficiently large finite integer
$N_0$ is chosen as in Section~\ref{subsubsec: regularity orders}. Fix constant states
$\mathbf U_l,\mathbf U_r$ with positive densities satisfying
\eqref{eq: strict vacuum gap}.

\smallskip\noindent\textbf{Initial data.}
We perturb the two states independently:
\begin{equation}\label{eq: data for R-V-R}
U(0,x_1,x_2)=
\begin{cases}
U_l(x_1,x_2),&x_1<0,\\
U_r(x_1,x_2),&x_1>0,
\end{cases}
\qquad x_2\in\mathbb T,
\end{equation}
where $U_j=(\frac{2c_j}{\gamma-1},v_j^1,v_j^2)^T$ and $c_j>0$.
We assume that $U_l$ is smooth on $(-\infty,0]\times\mathbb T$ up to
$x_1=0$ from the left and $U_r$ is smooth on $[0,\infty)\times\mathbb T$
up to $x_1=0$ from the right. The two one-sided data are chosen
independently. We require smallness in the following finite-order
Sobolev norm:
\begin{equation}\label{eq: small data for R-V-R}
\|U_l-\mathbf U_l\|_{H^{N_0}((-\infty,0)\times\mathbb T)}
+\|U_r-\mathbf U_r\|_{H^{N_0}((0,\infty)\times\mathbb T)}
\leq\varepsilon.
\end{equation}
The smallness condition is imposed at order $N_0$. The induced
conservative initial data are given on each side by
$\rho_j=(\frac{c_j}{\sqrt{k\gamma}})^{\frac{2}{\gamma-1}}$
and $m_j=\rho_jv_j$.

\begin{remark}
For $j=l,r$, we use $U_j$ to denote both the initial data $U_j(x_1,x_2)$
(on $x_1\leq0$ for $j=l$ and on $x_1\geq0$ for $j=r$) and the
corresponding classical exterior flow $U_j(t,x_1,x_2)$. The latter is
uniquely determined on the corresponding exterior region by the initial
data on $(-\infty,0]\times\mathbb T$ for $j=l$ and on
$[0,\infty)\times\mathbb T$ for $j=r$. In particular, $U_j(0,x_2)$ denotes the
one-sided spatial trace of the initial data at $x_1=0$, taken from the
corresponding side. The same convention applies to $\rho_j,c_j,v_j^i$
($i=1,2$).
\end{remark}

\smallskip\noindent\textbf{Boundary graphs and regions.}
For the boundaries to be constructed, use the notation
\[
F=(F_1,F_2,F_3,F_4)
  =(\overline f_0,\overline f_{\mathrm{vac}},f_{\mathrm{vac}},f_0),
\]
as graph functions with domain $[0,t^*]\times\mathbb T$ and denote their planar speeds by
\[
s=\left(\mathbf v_l^1-\mathbf c_l,\
\mathbf v_l^1+\frac{2\mathbf c_l}{\gamma-1},\
\mathbf v_r^1-\frac{2\mathbf c_r}{\gamma-1},\
\mathbf v_r^1+\mathbf c_r\right).
\]

The surfaces $x_1=F_a(t,x_2)$, $a=1,2,3,4$, are denoted by
$\Cb_0,\Cb_{\textbf{vacuum}},C_{\textbf{vacuum}},C_0$, respectively.
The outer fronts $\overline f_0,f_0$ issue from $x_1=0$ and are
characteristic for the left and right classical exterior flows. Their
graph equations are detailed in subsection~\ref{subsubsec: boundary graphs theorem}; in particular,
\[
\partial_t(\overline f_0)(0,x_2)=(v_l^1-c_l)(0,x_2),\qquad
\partial_t(f_0)(0,x_2)=(v_r^1+c_r)(0,x_2).
\]
For sufficiently small $\varepsilon$, each map
$y\mapsto y+tv_j^2(0,y)$ with $y\in\mathbb T$ has derivative
$1+t\partial_y(v_j^2)(0,y)\geq1/2$ for $0\leq t\leq t^*$.
Let $y_j(t,x_2)$ be their inverses. The vacuum graphs are
\begin{equation}\label{eq: vacuum graphs main theorem}
\begin{aligned}
\overline f_{\mathrm{vac}}(t,x_2)
&=t\left(v_l^1+\frac{2c_l}{\gamma-1}\right)(0,y_l(t,x_2)),\\
f_{\mathrm{vac}}(t,x_2)
&=t\left(v_r^1-\frac{2c_r}{\gamma-1}\right)(0,y_r(t,x_2)).
\end{aligned}
\end{equation}
For ordered graphs $F_1<F_2<F_3<F_4$ at $t>0$, we describe the two
exterior regions using $(x_1,x_2)$ and the two fans and the vacuum
region using $(\urr,x_2)$. Their time slices are
\[
\begin{aligned}
\mathcal D_l(t)&=\{(x_1,x_2):x_1<F_1(t,x_2)\},&
\mathcal D_r(t)&=\{(x_1,x_2):F_4(t,x_2)<x_1\},\\
\Sigma_t^l&=\{(\urr,x_2):F_1/t<\urr<F_2/t\},&
\Sigma_t^r&=\{(\urr,x_2):F_3/t<\urr<F_4/t\},\\
\Sigma_t^{\mathrm{vac}}&=\{(\urr,x_2):F_2/t\leq\urr\leq F_3/t\}.
\end{aligned}
\]
Here $x_2\in\mathbb T$ in every set and each $F_a$ is evaluated at
$(t,x_2)$. From left to right, the five regions are the left exterior,
left fan, vacuum, right fan, and right exterior. The two vacuum
interfaces are included in $\Sigma_t^{\mathrm{vac}}$; its interior is
the vacuum region. For $j=l,r$, set
\[
\mathcal W_j=\{(t,\urr,x_2):0<t<t^*,\ (\urr,x_2)\in\Sigma_t^j\}.
\]

\smallskip\noindent\textbf{Reference profiles.}
Using the initial traces $U_j(0,x_2)$, define the two affine profiles by
\begin{equation}\label{eq: leading profiles main theorem}
\begin{aligned}
P_{r,0}(\urr,x_2)=\Vr_0(\urr,x_2)
&=U_r(0,x_2)
+\frac{2\bigl(\urr-(v_r^1+c_r)(0,x_2)\bigr)}{\gamma+1}(1,1,0)^T,\\
P_{l,0}(\urr,x_2)=\underline{\Vr}_0(\urr,x_2)
&=U_l(0,x_2)
+\frac{2\bigl(\urr-(v_l^1-c_l)(0,x_2)\bigr)}{\gamma+1}(-1,1,0)^T.
\end{aligned}
\end{equation}
The right profile agrees with $U_r(0,x_2)$ at
$\urr=(v_r^1+c_r)(0,x_2)$ and has zero sound speed at
$\urr=(v_r^1-2c_r/(\gamma-1))(0,x_2)$.
The left profile agrees with $U_l(0,x_2)$ at
$\urr=(v_l^1-c_l)(0,x_2)$ and has zero sound speed at
$\urr=(v_l^1+2c_l/(\gamma-1))(0,x_2)$.
These four graphs describe the boundaries of the two initial profiles
in the opened coordinates at $t=0$. The positive-time vacuum
boundaries use the transported labels $y_j$ in
\eqref{eq: vacuum graphs main theorem}.
The affine formulas are defined for all $\urr$ and are evaluated on
the actual moving fan slices. For $\gamma=7/5,5/3$, also set
$P_{r,1}=\Vr_1$, given by \eqref{eq: exact form of Vr_1}, and
$P_{l,1}=\underline{\Vr}_1$, obtained by the reflection in
Section~\ref{subsubsec: left profile theorem}. At these two values of
$\gamma$, the first profiles are polynomials in the initial affine
sound speed, with coefficients depending on $x_2$. Thus they are
also defined on the actual moving fan slices, even if the initial
affine sound speed there is slightly negative.

\smallskip\noindent\textbf{Coordinates and norms.}
Write
\[
U^\sharp(t,\urr,x_2)=U(t,t\urr,x_2).
\]
The state $U$, and hence $U^\sharp$, is defined in the gas regions,
with gas-side traces where these exist.
All $W^{2,\infty}$ norms on $\mathcal D_j(t)$ ($j=l,r$) are spatial
norms in $(x_1,x_2)$; those on $\Sigma_t^j$ ($j=l,r$) are spatial
norms in $(\urr,x_2)$.
For example,
\[
\|h\|_{W^{2,\infty}(\Sigma_t^j)}
=\sum_{a+b\leq2}\operatorname*{ess\,sup}_{(\urr,x_2)\in\Sigma_t^j}
|\partial_{\urr}^a\partial_{x_2}^b(h)|.
\]
Norms on $\mathcal W_j$, $j=l,r$, include all weak derivatives in
$(t,\urr,x_2)$ of total order at most two, with $\partial_t$ taken
at fixed $(\urr,x_2)$. For the boundary vector, use
$\|F-ts\|_{W^{2,\infty}(\mathbb T)}
=\sum_{a=1}^4\|F_a-ts_a\|_{W^{2,\infty}(\mathbb T)}$.
The constants $C$ may increase between estimates and depend only on
$N_0,k,\gamma,\mathbf U_l,\mathbf U_r,t^*$, not on $t$ or $\varepsilon$.

\smallskip\noindent\textbf{Entropy and initial trace.}
By an entropy solution $(\rho,m)$ we mean a distributional solution
of \eqref{eq: Euler in rho v} satisfying the mechanical energy
inequality. For $\rho>0$, set
\[
\eta(\rho,m)=\frac{|m|^2}{2\rho}+\frac{k}{\gamma-1}\rho^\gamma,
\qquad q_j(\rho,m)=(\eta(\rho,m)+p)v^j.
\]
At $(\rho,m)=(0,0)$, the Euler fluxes, $\eta$ and $q_j$ are defined
to be zero. We require
$\partial_t(\eta)+\sum_{j=1}^2\partial_{x_j}(q_j)\leq0$ in distributions,
and the induced conservative initial data are attained in
$L^1_{\rm loc}(\mathbb R\times\mathbb T)$ as $t\rightarrow0$.
The weak formulations include the initial terms and use compactly
supported tests. Section~\ref{subsubsec: entropy gluing} verifies
these properties for the solution constructed below.

\subsection{Main theorem}
\begin{samepage}
\begin{theorem}[Structural stability of two rarefaction waves with vacuum]
\label{theorem: stability of RVR}
For the fixed parameters and initial data of
Section~\ref{subsec: main theorem preparation}, there is
$\varepsilon_0>0$, depending only on
$N_0,k,\gamma,\mathbf U_l,\mathbf U_r,t^*$, such that for
$0\leq\varepsilon\leq\varepsilon_0$ the Euler equations
\eqref{eq: Euler in rho v} admit an $L^\infty$ entropy solution
$(\rho,m)$ on $[0,t^*]\times\mathbb R\times\mathbb T$ with the
conservative initial data induced by \eqref{eq: data for R-V-R}.
Its gas regions and normalized state $U$ have the following properties
for $0<t\leq t^*$.

\smallskip\noindent\textbf{(i) Characteristic boundaries.}
There are four characteristic boundaries with graph functions
$F_a\in C^2([0,t^*]\times\mathbb T)$, vanishing at $t=0$ and bounding
the five regions of Section~\ref{subsec: main theorem preparation}.
The vacuum boundaries are exactly
\eqref{eq: vacuum graphs main theorem}. On either vacuum boundary
$x_1=f(t,x_2)$ the characteristic condition means
$\partial_t(f)-v^1+v^2\partial_{x_2}(f)=0$, with $v$ the gas-side trace.
Moreover,
\[
F_1<F_2<F_3<F_4,\qquad
\|F(t,\cdot)-ts\|_{W^{2,\infty}(\mathbb T)}\leq C\varepsilon t.
\]

\smallskip\noindent\textbf{(ii) Matching and vacuum.}
The density is positive in the four gas regions, and
$(\rho,m)=(0,0)$ in vacuum. The normalized state $U$ agrees with the
classical exterior flows and matches them continuously across
$\Cb_0,C_0$.

\smallskip\noindent\textbf{(iii) Spatial stability.}
In the exterior regions,
\[
\sum_{j=l,r}\|U(t,\cdot)-\mathbf U_j\|_{W^{2,\infty}(\mathcal D_j(t))}
\leq C\varepsilon.
\]
In each fan ($j=l,r$),
\[
\|U^\sharp(t)-P_{j,0}\|_{W^{2,\infty}(\Sigma_t^j)}
\leq C\varepsilon t.
\]

\smallskip\noindent\textbf{(iv) Additional regularity.}
If $\gamma\in\{7/5,5/3\}$, then for $j=l,r$,
\[
\|U^\sharp(t)-P_{j,0}-tP_{j,1}\|_{W^{2,\infty}(\Sigma_t^j)}
\leq C\varepsilon t^2,
\]
and the spacetime bounds in the two fans are
\[
\sum_{j=l,r}\|U^\sharp\|_{W^{2,\infty}(\mathcal W_j)}\leq C.
\]
\end{theorem}
\end{samepage}

The spacetime conclusion concerns $U$ in each gas fan in the opened
coordinates $(t,\urr,x_2)$, uniformly up to its vacuum boundary.
The proof of the time-derivative bounds uses the Euler equations in
addition to the spatial remainder estimate; see
Section~\ref{subsubsec: time regularity theorem}.

The construction is carried out for the right family, associated with
$v^1+c$. Applying the symmetry $(x_1,v^1)\mapsto(-x_1,-v^1)$ to the
left data gives the other family independently. The strict gap keeps
the two vacuum boundaries apart. Their conservative and flux traces
vanish, allowing the two fans to be joined by the zero state; see
Section~\ref{subsec: proof main theorem}
and~subsection~\ref{subsubsec: entropy gluing}.

\begin{figure}[htbp]
  \centering
  \resizebox{0.70\linewidth}{!}{%
  \begin{tikzpicture}[x={(1cm,0.06cm)},y={(-0.38cm,0.82cm)},
                      z={(0cm,2.85cm)},>=stealth,line join=round]
    \definecolor{fanblue}{RGB}{42,102,173}
    \definecolor{fanred}{RGB}{194,69,50}
    \definecolor{vacuumgold}{RGB}{207,155,40}

    \draw[black!75] (-4.9,0,0)--(4.9,0,0)--(4.9,1,0)
      --(-4.9,1,0)--cycle;
    \fill[vacuumgold!18,opacity=0.65] (0,1,0)--(-1.75,1,1)--(1.75,1,1)--cycle;
    \fill[vacuumgold!22,opacity=0.7] (-1.75,0,1)
      .. controls (-1.98,0.33,1) and (-1.98,0.67,1) .. (-1.75,1,1)
      --(1.75,1,1)
      .. controls (1.52,0.67,1) and (1.52,0.33,1) .. (1.75,0,1)
      --cycle;
    \fill[vacuumgold!15,opacity=0.55] (0,0,0)--(-1.75,0,1)--(1.75,0,1)--cycle;

    \foreach \q/\bend/\col/\sty in {
        -4.2/0.23/fanblue/solid,
        -3.45/-0.23/fanred/dashed,
        -2.7/0.23/fanblue/dashed,
        -1.75/-0.23/fanred/solid,
         1.75/-0.23/fanred/solid,
         2.5/0.23/fanblue/dashed,
         3.25/-0.23/fanred/dashed,
         4.2/0.23/fanblue/solid}{
      \filldraw[fill=\col!18,fill opacity=0.45,draw=\col,\sty,
                line width=0.65pt]
        (0,0,0)--(\q,0,1)
        .. controls ({\q+\bend},0.33,1) and ({\q+\bend},0.67,1)
        .. (\q,1,1)--(0,1,0)--cycle;
      \foreach \s in {0.25,0.5,0.75}{
        \draw[\col,opacity=0.38,line width=0.3pt]
          (0,\s,0)--({\q+3*\bend*\s*(1-\s)},\s,1);
        \draw[\col,opacity=0.38,line width=0.3pt]
          ({\s*\q},0,\s)
          .. controls ({\s*(\q+\bend)},0.33,\s)
                    and ({\s*(\q+\bend)},0.67,\s)
          .. ({\s*\q},1,\s);
      }
    }
    \draw[black!80] (0,0,0)--(0,1,0);
    \node[fanblue,above left] at (-4.2,1,1) {$\Cb_0$};
    \node[fanred,above] at (-1.75,1,1) {$\Cb_{\text{vacuum}}$};
    \node[fanred,above] at (1.75,1,1) {$C_{\text{vacuum}}$};
    \node[fanblue,above right] at (4.2,1,1) {$C_0$};
    \node[vacuumgold!80!black] at (0,0.65,0.71) {\itshape vacuum state};
    \draw[->,line width=0.8pt] (-2.25,-0.65,0)--(3.55,-0.65,0)
      node[right] {$x_1$};
  \end{tikzpicture}}
  \caption{Rarefaction fronts in coordinates $(t, x_1, x_2)$.}
  \label{fig:physical-rarefaction}
\end{figure}

\begin{figure}[htbp]
  \centering
  \resizebox{0.70\linewidth}{!}{%
  \begin{tikzpicture}[x={(1cm,0.15cm)},y={(-0.95cm,0.85cm)},
                      z={(0cm,2.95cm)},>=stealth,line join=round]
    \definecolor{fanblue}{RGB}{42,102,173}
    \definecolor{fanred}{RGB}{194,69,50}
    \definecolor{vacuumgold}{RGB}{207,155,40}

    \draw[black!75] (-4.9,0,0)--(4.9,0,0)--(4.9,1,0)
      --(-4.9,1,0)--cycle;
    \fill[vacuumgold!18,opacity=0.65] (-1.75,1,0)--(1.75,1,0)
      --(1.75,1,1)--(-1.75,1,1)--cycle;
    \fill[vacuumgold!22,opacity=0.7] (-1.75,0,1)
      .. controls (-1.98,0.33,1) and (-1.98,0.67,1) .. (-1.75,1,1)
      --(1.75,1,1)
      .. controls (1.52,0.67,1) and (1.52,0.33,1) .. (1.75,0,1)
      --cycle;
    \fill[vacuumgold!15,opacity=0.55] (-1.75,0,0)--(1.75,0,0)
      --(1.75,0,1)--(-1.75,0,1)--cycle;

    \foreach \q/\bend/\col/\sty in {
        -4.2/0.23/fanblue/solid,
        -3.45/-0.23/fanred/dashed,
        -2.7/0.23/fanblue/dashed,
        -1.75/-0.23/fanred/solid,
         1.75/-0.23/fanred/solid,
         2.5/0.23/fanblue/dashed,
         3.25/-0.23/fanred/dashed,
         4.2/0.23/fanblue/solid}{
      \filldraw[fill=\col!18,fill opacity=0.45,draw=\col,\sty,
                line width=0.65pt]
        (\q,0,0)--(\q,0,1)
        .. controls ({\q+\bend},0.33,1) and ({\q+\bend},0.67,1)
        .. (\q,1,1)--(\q,1,0)
        .. controls ({\q+\bend},0.67,0) and ({\q+\bend},0.33,0)
        .. (\q,0,0)--cycle;
      \foreach \s in {0.25,0.5,0.75}{
        \draw[\col,opacity=0.38,line width=0.3pt]
          ({\q+3*\bend*\s*(1-\s)},\s,0)
          --({\q+3*\bend*\s*(1-\s)},\s,1);
        \draw[\col,opacity=0.38,line width=0.3pt]
          (\q,0,\s)
          .. controls ({\q+\bend},0.33,\s) and ({\q+\bend},0.67,\s)
          .. (\q,1,\s);
      }
    }
    \node[fanblue,above left] at (-4.2,1,1) {$\Cb_0$};
    \node[fanred,above] at (-1.75,1,1) {$\Cb_{\text{vacuum}}$};
    \node[fanred,above] at (1.75,1,1) {$C_{\text{vacuum}}$};
    \node[fanblue,above right] at (4.2,1,1) {$C_0$};
    \node[vacuumgold!80!black] at (0,0.50,0.53) {\itshape vacuum state};
    \draw[->,line width=0.8pt] (-3.5,-0.65,0)--(3.55,-0.65,0)
      node[right] {$\urr$};
  \end{tikzpicture}}
  \caption{Rarefaction fronts in the opened coordinates $(t,\urr,x_2)$, where $\urr=x_1/t$ for $t>0$.}
  \label{fig:blowup}
\end{figure} 

\begin{remark}
The range of $\gamma$ reflects both the regularity of the approximate
solutions at vacuum and the energy estimates; see
Remarks~\ref{rem: remark 4.1} and~\ref{rem: remark 4.2}, and
Section~\ref{sec: energy estimates}. Fractional powers and terms
$c^\alpha(\ln c)^\beta$ can obstruct higher spatial regularity.
The present range allows the proof to use only zeroth- or first-order
profiles. The second paper will treat $\gamma\in(1,5/3]$ using finite
evolving profiles that incorporate these singular terms; see
Section~\ref{subsubsec: role in series}.
Figures~\ref{fig:physical-rarefaction} and~\ref{fig:blowup} show the
fans in the physical coordinates $(t,x_1,x_2)$ and in the opened
coordinates $(t,\urr,x_2)$, respectively. Blue and red indicate the charts attached
to $C_0$ and $C_{\textbf{vacuum}}$, respectively; their overlap is
described in Section~\ref{subsec: two charts}.
\end{remark}

\subsection{Previous works}
We recall the results most closely related to the present problem. We
begin with the one-dimensional Riemann problem, then discuss multidimensional
wave patterns and vacuum, and finally review geometric work on singularity
formation.

\smallskip\noindent\textbf{The Riemann problem and one-dimensional theory.}
Riemann's work on gas dynamics \cite{Riemann} and Lax's theory of
hyperbolic systems \cite{Lax1957} identify the elementary waves and
their admissibility conditions. For planar isentropic Euler, the
possible configurations are built from rarefaction waves, shocks,
and vortex sheets; sufficiently separated states instead produce
the rarefaction--vacuum--rarefaction pattern considered here.
Glimm \cite{Glimm1965} constructed global entropy solutions for
small-BV data by controlling wave interactions. Decay and long-time
behavior were studied by Glimm and Lax \cite{GlimmLax}, DiPerna
\cite{Diperna3}, and Liu \cite{liutp2}; Li and Yu \cite{LiTT} treated
the perturbed Riemann problem. General accounts are given in
\cite{Dafermos,Liu2021book,Smoller}.

\smallskip\noindent\textbf{Multidimensional wave patterns.}
The one-dimensional BV estimates generally fail in several dimensions
\cite{Rauch}. A construction of a prescribed pattern must therefore
control the moving interfaces together with the solution. Under the
relevant stability and compatibility assumptions, local constructions
were obtained for shock fronts by Majda \cite{MajdaShock2,MajdaShock3},
for rarefaction waves by Alinhac
\cite{AlinhacWaveRare0,AlinhacWaveRare1,AlinhacWaveRare2}, and for
two-dimensional compressible vortex sheets in the stable supersonic
regime by Coulombel and Secchi
\cite{Coulombel-Secchi1,Coulombel-Secchi2}. Luo and Yu
\cite{Luo-YuRare1,Luo-YuRare2} developed a geometric construction of
irrotational rarefaction waves away from vacuum. We discuss the scope
of these rarefaction results in
Section~\ref{subsec: away from vacuum review}.

Configurations with several waves were studied by Li \cite{Li1991}
and Chen and Li \cite{ChenLi}; see also Qu and Xiang \cite{QuXiang}
for steady corner flows. In our earlier work \cite{JiaLuo}, we studied
mixed Riemann configurations containing centered rarefaction waves,
shocks, and vortex sheets. The present problem is different because
the two fans adjoin an open vacuum region, so the construction must
remain uniform as the sound speed vanishes.

\smallskip\noindent\textbf{Rarefaction waves and vacuum.}
Chen and Chen \cite{Chen-Chen} established stability and uniqueness
of planar rarefaction waves with vacuum in their class of bounded
entropy solutions, using relative entropy estimates. The structural
problem considered here also requires constructing the perturbed fans
and determining their characteristic and vacuum boundaries. Other
settings involving vacuum include rarefaction-wave interactions
\cite{ZhangtongXiaoling} and supersonic nozzle flows approaching
vacuum at spatial infinity \cite{GangXuYinHuicheng}. Here the vacuum
is created immediately from an initial jump between two positive
gas states. The theorem gives the resulting multidimensional wave
pattern, including its boundary locations and regularity, for small
perturbations that may have nonzero vorticity.

\smallskip\noindent\textbf{Geometric analysis of singularities.}
Characteristic geometry has also been central to shock formation:
see Christodoulou \cite{ChristodoulouShockFormation} for relativistic
flow, Christodoulou and Miao \cite{ChristodoulouMiao} for classical
Euler, and Luk and Speck \cite{LukSpeck2D} for two-dimensional flow
with nonzero vorticity. Buckmaster, Shkoller and Vicol used modulated
self-similar variables to construct shocks with vorticity in
two-dimensional flow with an azimuthal structure
\cite{BuckmasterShkollerVicol2D}, and to prove stable point-shock
formation in three-dimensional isentropic Euler without symmetry
assumptions \cite{BuckmasterShkollerVicol3D}. Klainerman, Wang, Yang
and Yu \cite{klainerman2026inevitableshockformation3d} proved inevitable
shock formation for smooth, sufficiently small, compactly supported
perturbations of a positive constant state in three-dimensional
irrotational flow.

For continuation beyond singularity formation, Yin \cite{Yin}
treated shock formation and subsequent development for spherically
symmetric three-dimensional Euler flow; Christodoulou
\cite{ChristodoulouShockDevelopment} studied the restricted shock
development problem; and Shkoller and Vicol \cite{Shkoller-Vicol}
analyzed the geometry of maximal smooth development beyond the
first singular time. The self-similar imploding profiles of Merle,
Rapha\"el, Rodnianski and Szeftel \cite{MRRS-ImplosionI} describe a
different singularity mechanism. Here a prescribed initial discontinuity produces an expanding
rarefaction pattern. We must control both the collapse of the fan at
$t=0$ and the degeneration as $c$ tends to zero.

\subsection{Centered rarefaction waves away from vacuum and up to vacuum}
The literature and the present construction address two different
degeneracies. The first is the collapse of a centered fan at the initial
time; the second is the vanishing sound speed at a vacuum interface.
We review the available constructions before describing the estimates
needed to pass from one setting to the other.

\subsubsection{Constructions away from vacuum}\label{subsec: away from vacuum review}
A centered rarefaction wave resolves an initial jump by a continuous fan
whose width is proportional to $t$. Its derivatives in the original spatial
coordinates can therefore be of order $t^{-1}$, even when its profile is
smooth as a function of $x_1/t$. In several dimensions one must construct
both the solution inside this fan and the characteristic surfaces that
bound it. The works discussed below address this initial singularity in
regions where the density is bounded away from zero.

Alinhac constructed local rarefaction waves for isentropic Euler and, more
generally, for symmetrizable quasilinear hyperbolic systems, under
compatibility and regularity assumptions on the prescribed data
\cite{AlinhacWaveRare0,AlinhacWaveRare1}. In the Euler example the wave
belongs to an extreme genuinely nonlinear characteristic family. His
construction opens the fan into a fixed spatial domain, builds a sufficiently
accurate approximate solution at $t=0$, and uses weighted anisotropic
Sobolev estimates and Nash--Moser iteration. The compatibility conditions
relate the data on the two sides to the required order. Uniqueness among
sufficiently regular rarefaction waves is established in
\cite{AlinhacWaveRare2}; see also the statements and explanation in
\cite[Sections 3--6]{AlinhacWaveRareSurvey}.

Luo and Yu developed a geometric construction for two-dimensional
isentropic Euler with $1<\gamma<3$ for irrotational flow away from
vacuum \cite{Luo-YuRare1,Luo-YuRare2}. Their first paper proves a priori
energy estimates for the acoustical wave equations, using characteristic
coordinates, two null frames, and additional vanishing of derivatives of
the characteristic speed. The estimates close at the stated energy order
without derivative loss within that argument. They are uniform as the
initial time $\delta$ tends to zero, with the sound speed bounded away
from zero; see \cite[Section 3.3]{Luo-YuRare1}.
Their second paper constructs the required data, obtains a rarefaction
wave attached to a prescribed exterior flow, and treats a Riemann
configuration consisting of two rarefaction waves and an intervening gas
region. The Riemann data are assumed distributionally irrotational, and
uniqueness is proved in the entropy class defined there; see
\cite[Definition 1.1, Theorems 1--3 and Proposition 2.12]{Luo-YuRare2}.
The domain in their single-wave Theorem 2 is separated from the reference
vacuum boundary by a fixed positive cutoff.

In our earlier work \cite{JiaLuo}, we extended the geometric construction
to the two-dimensional isentropic Euler equations with nonzero vorticity,
still away from vacuum. We formulated the equations as a coupled
wave--transport system for the Riemann invariants and the specific
vorticity, and derived simultaneous energy estimates for the acoustic
and vorticity variables without loss of derivatives. The construction
also determines the characteristic foliation of each rarefaction region,
its inner characteristic boundary, and the boundary jets needed to match
the fan to the adjacent flow. We then used this rarefaction theory to
construct the shock--rarefaction and rarefaction--vortex-sheet--rarefaction
configurations. In these applications, the remaining interaction region
is reduced to a suitable Cauchy problem by using the characteristic
structure of the rarefaction boundaries and propagating the required
normal derivatives from the initial singular set; the shock or vortex
sheet is treated by the corresponding single-front theory, with the
necessary Rankine--Hugoniot and higher-order compatibility conditions.
See \cite[Theorems~2.5, 2.7--2.9 and Sections~7--10]{JiaLuo} for the
construction and its applications.

These results provide constructions at the initial singularity while the
gas stays away from vacuum. Theorem~\ref{theorem: stability of RVR}
concerns the additional limiting boundary at which each rarefaction wave
reaches $c=0$ and the two waves enclose an open vacuum region. The data
here need not be irrotational, and the estimates must remain uniform as
the cutoff keeping the gas away from vacuum is removed.

\subsubsection{The additional difficulties at vacuum}\label{subsec: vacuum difficulties}
The limits $t\rightarrow0$ and $c\rightarrow0$ cause different difficulties.
The first collapses the width of the fan and is resolved geometrically by
the coordinate $\urr=x_1/t$. The second merges the characteristic speeds
and limits the regularity of the initial profiles. Resolving the first
limit does not by itself give estimates uniform in the second.

For example, the three characteristic speeds in the $x_1$ direction are
\[
v^1-c,\qquad v^1,\qquad v^1+c.
\]
Their gaps vanish with $c$. Likewise the acoustical metric
\[
g=-c^2dt^2+\sum_{i=1}^2(dx_i-v^i dt)^2
\]
becomes degenerate at $c=0$. Consequently, estimates that use inverse
characteristic gaps or inverse powers of $c$ require additional
information near vacuum. A simple commutation identity illustrates the
issue. If $L(\psi)=cF$ in the gas region, then
\[
c^{-1}L(\hat X(\psi))
=\hat X(F)+\frac{\hat X(c)}{c}F+c^{-1}[L,\hat X](\psi).
\]
Here $L$ is tangent to the outgoing characteristics and $\hat X$ is their
unit spatial tangent, as defined in Section~\ref{subsec: geo con}.
A bound for $F=c^{-1}L(\psi)$ alone does not control this commuted expression
uniformly as $c\to0$. The terms must be estimated together with the
geometry and the vanishing furnished by the equations.

A second issue concerns approximation at the initial slice. Set
$\beta=(\gamma+1)/(2(\gamma-1))$.
The importance of $\gamma$ in an analysis reaching vacuum is already
visible in DiPerna's work on convergence of artificial viscosity
\cite{diperna2}. His compensated-compactness analysis treats the
discrete exponents $\gamma=(2n+3)/(2n+1)$, $n=2,3,\ldots$, for which
the relevant entropy kernel is polynomial. Ding, Chen and Luo
\cite{DingChenLuo1989} subsequently treated the full interval
$1<\gamma\leq5/3$ by compensated compactness: they proved convergence
of the fractional-step Lax--Friedrichs and Godunov schemes for an
inhomogeneous isentropic gas system and obtained global solutions
for large data allowing vacuum. For the present profiles, the
dependence on $\gamma$ is explicit in
\[
\rho=\left(\frac{c}{\sqrt{k\gamma}}\right)^{\frac{2}{\gamma-1}},
\qquad
\begin{array}{c|cc}
\gamma&\frac{2}{\gamma-1}&\beta\\[4pt]\hline
\frac53&3&2\\[4pt]
\frac75&5&3
\end{array}
\]
Thus at $\gamma=5/3,7/5$ the density is a cubic or quintic function
of $c$, respectively, and the factors $c^\beta$, $c^{2\beta}$, and
$c^{\beta-1}$ occurring in the initial profiles have integer powers.
This removes the fractional-power obstruction to differentiating
these factors at $c=0$. The coefficient formulas and the remaining
logarithmic terms determine the regularity of each profile. The historical progression
above concerns compensated compactness; in the present approximation
and energy argument, the dependence on $\gamma$ enters through the
regularity of the profiles at vacuum.

The explicit formulas in
Section~\ref{sec: high order approximate solution} contain fractional
powers of $c$ and, at the next order, logarithmic terms. The
vorticity-dependent resonance producing the logarithms is explained in
Section~\ref{subsubsec: profiles and logarithms}. Spatial derivatives
of the profiles need not belong to the same Sobolev space at
successive orders. For instance, at $\gamma=1.3$ one has $\beta=23/6$:
the fourth derivative of a $c^\beta$ term is proportional to $c^{-1/6}$,
whereas that of a $c^{\beta-1}$ term is proportional to $c^{-7/6}$.
On the initial fan, $\partial_{\urr}(c)=(\gamma-1)/(\gamma+1)$, and hence
\[
\int_0^\eta |c^{-1/6}|^2\,dc<\infty,\qquad
\int_0^\eta |c^{-7/6}|^2\,dc=\infty\qquad(\eta>0).
\]
These powers occur in $\Vr_1$ and $\Vr_2$, respectively. When the
latter coefficient is nonzero, $\Vr_2$ is not in $H^4$ on the fixed
initial fan, so it cannot be the second coefficient of a Taylor
expansion in that Banach space. This concerns strong Taylor
regularity on the initial fan. See
Remarks~\ref{rem: remark 4.1} and~\ref{rem: remark 4.2}.

To reach vacuum, we need estimates that do not depend on a positive
lower bound for $c$, together with initial profiles having the spatial
regularity required by those estimates. We establish both for the range
of $\gamma$ in Theorem~\ref{theorem: stability of RVR}.

\subsection{Ideas of the proof}
We first construct the characteristic foliation, then introduce two
overlapping charts and determine the initial profiles. We next use the
cancellation in the commuted equations to obtain estimates uniform in
both cutoffs. The following subsections explain these steps; their proofs
are given in Sections~\ref{sec: set up}--\ref{sec: closing B infinity}.

We use these estimates to construct the gas state $U$ for
$\gamma\in(1,1.3]\cup\{\frac75,\frac53\}$. We then recover the
density and momentum in the gas and set them equal to zero in vacuum.
Section~\ref{subsubsec: entropy gluing} verifies that the resulting pair
is an entropy solution.

\subsubsection{Characteristic geometry for the first-order system}
The normalized Euler equations have symmetric coefficient matrices
$A_1(U),A_2(U)$ whose entries remain regular at $c=0$; see
\eqref{eq: euler equation in c v}. We construct the outgoing
characteristics from the explicitly given eigenvalue
\[
\lambda_3(U,\xi)=v\cdot\xi+c|\xi|,\qquad \xi\ne0.
\]
With phase $-u$, the Hamiltonian $H=\tau+\lambda_3(U,\xi)$ acts on
$(\tau,\xi)=(-\partial_t(u),-\nabla_xu)$. Set
\[
\hat T=\frac{\nabla_xu}{|\nabla_xu|},\qquad
\kappa=|\nabla_xu|^{-1},\qquad T=\kappa\hat T.
\]
Thus $T(u)=1$; $\hat T$ describes the normal direction and $\kappa$
measures the spacing between neighboring characteristic surfaces.
With our orientation $\hat T^1\approx-1$, their generators are
$L=\partial_t+(v-c\hat T)\cdot\nabla_x$.
Hamilton's equations give transport equations for $\hat T$ and $\kappa$
directly from this characteristic speed; see
\eqref{eq: transport eq for geometry} and
\eqref{eq: final form of transport equation}.

For $c>0$, these characteristics agree with the outgoing null geometry
of the acoustical metric used by Christodoulou and Miao
\cite{ChristodoulouMiao} and Luo and Yu \cite[Section 2]{Luo-YuRare1}.
Indeed, for the metric $g$ defined
above and a covector $\tau\,dt+\xi\cdot dx$,
\[
\begin{aligned}
g^{-1}(\tau\,dt+\xi\cdot dx,\tau\,dt+\xi\cdot dx)
&=|\xi|^2-\frac{(\tau+v\cdot\xi)^2}{c^2}\\
&=-\frac{(\tau+v\cdot\xi-c|\xi|)
                 (\tau+\lambda_3(U,\xi))}{c^2}.
\end{aligned}
\]
On $H=0$ with $\xi\ne0$, the other factor is $-2c|\xi|\ne0$.
The two Hamiltonian vector fields are therefore proportional there:
they give the same rays up to reparametrization. With the same initial
phase, the null hypersurfaces of this branch are precisely our
characteristic hypersurfaces, and normalization by $L(t)=1$ gives the same
$L$. The spatially normal field normalized by $T(u)=1$ is also the same;
the inverse foliation density used in the acoustical formulation is
$(-g^{-1}(dt,du))^{-1}=c\kappa$. These identities hold for the
rotational flows considered here. The acoustic cone accounts for the two acoustic branches;
the additional material characteristic $\tau+v\cdot\xi=0$ is not part of it.

We do not use the wave equations associated with the acoustical metric.
Instead, Hamilton's equations for $\lambda_3$ give the transport equations
for $\hat T$ and $\kappa$, while the symmetric first-order system gives
the energy identities for $U$. This keeps the vorticity terms in the
original Euler system and avoids a second-order formulation. The same
characteristic calculation applies locally to a smooth simple real
eigenvalue of the principal symbol of a first-order hyperbolic system;
its characteristic hypersurface need not be a branch of a Lorentzian
quadratic cone.

At $c=0$, the acoustical metric degenerates and the acoustic eigenvalues
coalesce, although the explicit formula for $\lambda_3$ remains smooth
for $\xi\ne0$. We still have to control both the initial collapse
$\kappa\sim t$ and the factor $c$ in the characteristic flux near
vacuum. The energy identities
\eqref{eq: linear energy estimate in t x_1 x_2} and
\eqref{eq: linear energy estimate in t ur x_2} give the estimates
needed for these two limits.

\subsubsection{Two overlapping characteristic charts}
The two boundaries of the fan impose different requirements on the
characteristic coordinates. The chart $(t,u,\vartheta)$ is attached to
the exterior flow by $C_0=\{u=0\}$. On the opened initial slice,
\[
u=-\urr+(v_r^1+c_r)(0,x_2),\qquad
c(0,u,\vartheta)=c_r(0,\vartheta)-\frac{\gamma-1}{\gamma+1}u.
\]
Thus the initial vacuum boundary is generally not a constant level
of $u$. We use a second chart $(t,\tilde u,\vartheta)$, initialized by
\[
\tilde u=-\urr+\left(v_r^1-\frac{2c_r}{\gamma-1}\right)(0,x_2),
\qquad
c(0,\tilde u,\vartheta)=-\frac{\gamma-1}{\gamma+1}\tilde u.
\]
The estimate $|L(c)|\leq C\varepsilon c$ then keeps $c$ comparable to
$-\tilde u$. Consequently, a cutoff $\tilde u^*<0$ can approach zero
while the estimates remain uniform; see
\eqref{eq: c almost constant on C_u}.

We use the two charts in the following order. We first obtain the
higher-regularity estimates in $\mathcal W_1$, where $c$ is bounded
away from zero. These estimates give the boundary data on
$C_{\tilde u_*}\subset\mathcal W_1$ for the four-derivative estimates
in $\mathcal W_{2,\alpha}$. We then remove the cutoff $\alpha$.
The two foliations and their transported angular coordinates are
constructed separately. On their overlap, the common coordinates
$(t,\urr,x_2)$ allow us to compare the solutions and the fixed commuted
derivatives. Section~\ref{subsec: two charts} gives the details of the
construction.

\subsubsection{Initial-slice profiles, vorticity, and logarithmic terms}\label{subsubsec: profiles and logarithms}
The leading profile $\Vr_0$ is determined by the initial trace of the
right state and the simple-wave relation $v^1+c=\urr$; its formula and
the reflected left profile are given in
\eqref{eq: leading profiles main theorem}. We then solve the initial-slice
ODEs for the formal time derivatives $\Vr_1$ and $\Vr_2$ and determine
their spatial regularity up to $c=0$. On a fixed-cutoff solution with the
corresponding Taylor regularity in the opened coordinates, these profiles
coincide with $\partial_t U|_{t=0}$ and $\partial_t^2 U|_{t=0}$.
The near-vacuum construction starts at positive time and uses the explicit
profiles and their residuals. We derive the component formulas in
Section~\ref{sec: high order approximate solution}. The logarithmic
term records an interaction involving vorticity, which we now explain.

A related use of finite-order expansions near a singular boundary appears
in Christodoulou's treatment of the shock development problem. There,
the past boundary of the initial characteristic hypersurface is singular.
The derived data determine truncated power-series approximations in a
coordinate $\tau$ that vanishes at this boundary. Their residuals have
prescribed orders of vanishing in $\tau$, and the approximations are
used to regularize the initial data on $\tau=\tau_0>0$ before passing
to the limit $\tau_0\to0$; see
\cite[Prologue, pp.~32--36, and Chapter~9, especially Section~9.1]
{ChristodoulouShockDevelopment}. Our construction follows a related principle. The initial-slice ODEs
determine the formal time derivatives, and their dependence on the
initial sound speed gives the fractional powers and logarithmic terms
at vacuum. We retain only the finite number of profile terms needed to
obtain the residual bounds in the energy estimates.

In the following initial-slice formulas, $c=c(0,\urr,x_2)$ and all
exterior factors are evaluated at $(t,x_1)=(0,0)$ after differentiation,
as explained in Section~\ref{subsec: notation conventions}. Recall
\[
\beta=\frac{\gamma+1}{2(\gamma-1)},\qquad
w_r=\frac{c_r}{\gamma-1}-\frac{v_r^1}{2},\qquad
\omega_r=\partial_{x_1}(v_r^2)-\partial_{x_2}(v_r^1).
\]
At the initial boundary, $\hat T=(-1,0)$ and hence
$L=\partial_t+(v_r^1+c_r)\partial_{x_1}+v_r^2\partial_{x_2}$.
The transverse momentum equation gives
\[
\begin{aligned}
L(v_r^2)
&=c_r\partial_{x_1}(v_r^2)-\frac{2c_r}{\gamma-1}\partial_{x_2}(c_r),\\
L(v_r^2)+2c_r\partial_{x_2}(w_r)&=c_r\omega_r.
\end{aligned}
\]
Consequently, the third component in \eqref{eq: exact form of Vr_1}
can be written as
\[
\Vr_1^3=-2c\,\partial_{x_2}(w_r)
+\left(\frac{c}{c_r}\right)^{2\beta}c_r\omega_r
-v_r^2\partial_{x_2}(v_r^2).
\]
When substituted into the quadratic source of
\eqref{eq: ODE of Vr_2}, this vorticity-dependent term interacts with
$\partial_{x_2}(w_r)$. Its contribution has the same power
$c^{2\beta}$ as the homogeneous solution of the scalar ODE for
$\Vr_2^1-\Vr_2^2$. Integration therefore produces a logarithm:
at fixed $x_2$, the elementary identity is
\[
\left(c\partial_c-2\beta\right)
\left[\left(\frac{c}{c_r}\right)^{2\beta}
\ln\left(\frac{c}{c_r}\right)\right]
=\left(\frac{c}{c_r}\right)^{2\beta}.
\]
More precisely, the logarithmic contribution displayed in
\eqref{eq: exact logarithmic term of Vr_2} is
\[
\frac{4(\gamma+1)}{\gamma-1}\,
c_r\partial_{x_2}(w_r)\,\omega_r
\left(\frac{c}{c_r}\right)^{2\beta}
\ln\left(\frac{c}{c_r}\right).
\]
Its coefficient is nonzero exactly where
$\partial_{x_2}(w_r)\,\omega_r\ne0$. This contribution vanishes for
irrotational initial data and for data with no tangential variation of
$w_r$. Thus the logarithm records an interaction between vorticity
and tangential variation of the initial data.

The logarithmic profile itself tends to zero at vacuum, but its higher
spatial derivatives can be singular. This remains true when $\beta$
is an integer: the logarithmic factors are $c^4\ln(c/c_r)$ for
$\gamma=5/3$ and $c^6\ln(c/c_r)$ for $\gamma=7/5$, up to factors
depending smoothly on $x_2$. For example,
\[
\partial_c^4\left(c^4\ln\frac{c}{c_r}\right)
=24\ln\frac{c}{c_r}+50,\qquad
\partial_c^5\left(c^4\ln\frac{c}{c_r}\right)=\frac{24}{c}.
\]
Since $\partial_{\urr}(c)=(\gamma-1)/(\gamma+1)$ on the initial fan,
the fourth transverse derivative is square integrable up to $c=0$,
whereas the fifth is not when its coefficient is nonzero.
Remark~\ref{rem: remark 4.2} gives the corresponding profile bounds.
The construction allows this vorticity-dependent behavior and uses
only the finite regularity needed by the energy argument.

This initial-slice analysis determines which approximation can be
used in that argument:
\[
\begin{array}{c|c|c}
\text{range of }\gamma&\text{approximate solution}&\text{remainder}\\ \hline
(1,1.3]&\Ur_0=\Vr_0&V_0=U-\Ur_0\\[2pt]
\{7/5,5/3\}&\Ur_1=\Vr_0+t\Vr_1&V_1=U-\Ur_1.
\end{array}
\]
On a fixed positive-cutoff domain where a solution has the
corresponding Taylor regularity, the initial remainder limits are
$\Vr_1$ for $V_0$ and $\frac{1}{2}\Vr_2$ for $V_1$.
In the respective cases the formal time derivatives $\Vr_1$ and
$\Vr_2$ have four spatial derivatives in $L^2$ uniformly as the cutoff
is removed.
For $\gamma=7/5,5/3$, the integer values $\beta=3,2$ make $\Vr_1$
polynomial in the initial sound speed and suitable for the approximate
solution on the moving fan. The near-vacuum construction prescribes
the approximate solution at $t=\delta>0$, joined to the known
nonvacuum solution on a fixed overlap; its initial remainder estimate
is checked directly.
Propositions~\ref{prop: initial energy bounds on Sigma_0} and
\ref{prop: Res bounds of Ur_k} give the profile coefficient and residual
bounds. The matching along $C_0$ is treated in
Section~\ref{sec: normal jets}, and the regularity of the original data
is chosen in Section~\ref{subsubsec: regularity orders}.

\subsubsection{Null structure and damping--antidamping cancellation}
We commute with the fixed vector fields
\[
\Tr=-t\partial_{x_1}=-\partial_{\urr},\qquad
\Xr=\partial_{x_2},\qquad V_{k;a,b}=\Tr^a\Xr^b(U-\Ur_k).
\]
Thus the controlled derivatives carry no vanishing sound-speed weight.
We compute the leading commutator before estimating its coefficients;
this is where the cancellation appears. In physical coordinates,
$[\Tr,t\partial_t]=-\Tr$, whereas the explicit profile satisfies
$\Tr(\Vr_0)=-\frac{2}{\gamma+1}(1,1,0)^T$.
Since $A_1$ is linear in $U$, the terms produced by the time commutator
and one transversal derivative of the leading profile combine as
\[
\begin{aligned}
aV_{k;a,b}+aA_1(\Tr(\Vr_0))V_{k;a,b}
&=\frac{a(\gamma-1)}{\gamma+1}
\begin{pmatrix}
1&-1&0\\[-2pt]
-1&1&0\\[-2pt]
0&0&1
\end{pmatrix}V_{k;a,b}.
\end{aligned}
\]
This is the leading linear commutator contribution in
\eqref{eq: eq of V_k a b}; the terms containing $U-\Ur_k$ or
$t\Vr_1$ are kept separately there. Pairing the displayed expression
with $2V_{k;a,b}$ gives
\[
\frac{2a(\gamma-1)}{\gamma+1}
\left\{\big(V_{k;a,b}^{1}-V_{k;a,b}^{2}\big)^2
+\big(V_{k;a,b}^{3}\big)^2\right\}.
\]
In particular, this contribution contains no square of the outgoing
component $(V_{k;a,b}^{1}+V_{k;a,b}^{2})/\sqrt{2}$. Its matrix annihilates
$r_3(U,(1,0))=(1,1,0)^T/\sqrt{2}$, the same characteristic direction
as the non-small derivative $\Tr(\Vr_0)$. We use this cancellation as the null structure of the commuted system;
see \eqref{eq: null structure in commutators}.

This cancellation is analogous to the classical null condition.
Classical wave null forms vanish on proportional null covectors
\cite{KlainermanNull1986}. Here the cancellation occurs in the Euler
components of the commuted unknown: the displayed bilinear form vanishes
whenever either argument lies in the outgoing eigendirection. The acoustic
family remains genuinely nonlinear, so its self-interaction still produces
the rarefaction.

A second cancellation occurs in the energy identity. For $V=V_{k;a,b}$,
the leading normal-derivative bulk term from integration by parts is
\[
-\big(V,A_1(\Tr(\Vr_0))V\big)
=(V^1)^2+(V^2)^2
-\frac{\gamma-1}{\gamma+1}(V^1-V^2)^2
+\frac{2}{\gamma+1}(V^3)^2.
\]
Thus, modulo the null forms in the last two terms, integration by parts
produces the \emph{antidamping} $(V^1)^2+(V^2)^2$. On the other hand,
subtracting $\Ur_k$ supplies the leading source $A_1(V)\Tr(\Vr_0)$,
whose contribution to the energy identity is
\[
\begin{aligned}
2\big(V,A_1(V)\Tr(\Vr_0)\big)
={}&-(V^1)^2-(V^2)^2
-\frac{2(\gamma-1)}{\gamma+1}(V^1)^2\\
&-\frac{4}{\gamma+1}(V^2)^2+(V^1-V^2)^2.
\end{aligned}
\]
The first two negative squares cancel the antidamping exactly, leaving
additional damping in the first two components, modulo null forms;
see \eqref{eq: damping mechanism} and
\eqref{eq: anti damping by integration by parts}. Including the
transversal commutator contribution computed above, the resulting
leading bulk density is
\[
\begin{aligned}
&-\frac{2(\gamma-1)}{\gamma+1}(V^1)^2
-\frac{4}{\gamma+1}(V^2)^2\\
&\qquad+
\left(\frac{2a(\gamma-1)}{\gamma+1}+\frac{2}{\gamma+1}\right)
\left\{(V^1-V^2)^2+(V^3)^2\right\}.
\end{aligned}
\]
This is the quadratic form in
\eqref{eq: simple higher order energy estimate}. On the outgoing
eigenline $V\in\operatorname{span}\{(1,1,0)^T\}$ it equals $-|V|^2$:
the profile subtraction more than compensates for the integration-by-parts
antidamping, while the leading commutator contributes zero.

The remaining null forms have a direct geometric meaning. For any
vector $V=(V^1,V^2,V^3)^T$, the characteristic flux density in
\eqref{eq: linear energy estimate in t x_1 x_2} is
\[
c\left\{(V^1+\hat T^1V^2+\hat T^2V^3)^2
+(\hat X^1V^2+\hat X^2V^3)^2\right\}.
\]
At the planar frame $\hat T=(-1,0)$ and $\hat X=(0,1)$, this becomes
$c\{(V^1-V^2)^2+(V^3)^2\}$: precisely the same quadratic form as
in the commutator, multiplied by $c$. The geometric estimates control
the change of frame. This explains why the flux can be used to control
the leading commutator terms away from vacuum; see
Remark~\ref{re: how to use null structure}. As $c\rightarrow0$, however,
this flux loses its weight while the commutator coefficient does not.
The algebraic cancellation persists, but the vacuum argument must also
use the damping and the full energy quadratic form.

Null structures in acoustical wave formulations of Euler have been
studied by Luk and Speck \cite{LukSpeckNull} and are also used in the
rarefaction estimates of Luo and Yu \cite{Luo-YuRare1}. The feature
used here is the explicit cancellation in the transversally commuted
\emph{first-order} system, coupled to the damping of the profile
remainder and to estimates uniform up to vacuum.

\subsubsection{Energy closure and positive-time construction}
For the vacuum estimate, we keep the full quadratic form left by the
cancellation above, including both the remaining damping and the null
terms. We estimate its largest eigenvalue in the energy inequality
\eqref{eq: simple higher order energy estimate}.

This gives the quantitative condition
$\lambda_{\max}<2k+3$ for commuted derivatives of order at most four:
$\lambda_{\max}<3$ when $k=0$ and $1<\gamma\leq1.3$, and
$\lambda_{\max}<5$ when $k=1$ and $\gamma\in\{7/5,5/3\}$.
The exponent $2k+3$ comes from the expected remainder size $t^{k+1}$
and the factor $\kappa\sim t$ in the energy.
The singular Gronwall estimate of Lemma~\ref{lem: Gronwall} then closes
$(\mathbf B_2)$ uniformly in the vacuum cutoff; see
\eqref{eq: close of B_2 up to vacuum}. Thus the range in the theorem
reflects both profile regularity and the coefficient in the energy
estimate.

The resulting spatial bounds also provide the vanishing needed to
control the characteristic geometry. Indeed,
$\Vr_0^2+\frac{\gamma-1}{2}\Vr_0^1=\urr$, so on the right fan
\[
v^1+c-\urr=(U-\Vr_0)^2+\frac{\gamma-1}{2}(U-\Vr_0)^1.
\]
Consequently, for the right fan slice $\Sigma_t^r$ defined in
Section~\ref{subsec: main theorem preparation},
\[
\sum_{a+b\leq2}
\big\|\partial_{\urr}^a\partial_{x_2}^b
(v^1+c-\urr)(t,\cdot)\big\|_{L^\infty(\Sigma_t^r)}
\leq C\varepsilon t.
\]
This expresses in the present coordinates the additional vanishing of
the characteristic speed used by Luo and Yu \cite{Luo-YuRare1}. Here it
follows from the estimate for the explicit profile remainder. Together
with the corresponding bounds for $w$ and $\psi_2$, it yields improved
transport estimates for $\hat T$ and $\kappa$, and closes
$(\mathbf B_\infty)$; see
\eqref{eq: improved estimates for Zr(psi)} and
Section~\ref{sec: closing B infinity}. The spatial Sobolev inequalities
on the evolving slices, proved in
Section~\ref{sec: Preparations for the energy estimate}, connect these
pointwise bounds to the four-derivative energy estimate. Reflection
gives the left fan, and the limiting conservative traces allow the two
fans to be joined to the vacuum state.

\subsubsection{The role of this paper in the series}\label{subsubsec: role in series}
This paper proves existence and stability in the range where the
zeroth- or first-order approximate solution suffices. The proof uses
the characteristic geometry, the two overlapping charts, the null
structure of the commuted equations, and the damping--antidamping
cancellation. Together with the energy estimates, these give
Proposition~\ref{prop: positive time construction} and hence
Theorem~\ref{theorem: stability of RVR}.

The second paper will extend the range to $\gamma\in(1,5/3]$. To
treat the remaining exponents, we will construct finite evolving
profiles that contain the singular spatial terms, match the boundary
data, and leave a residual with the decay and spatial regularity required
by the energy estimates. The initial-slice calculations here show why
these profiles are needed; the estimates proved here specify the bounds
their remainders must satisfy. In particular, the chart
$(t,\tilde u,\vartheta)$ places the vacuum boundary at $\tilde u=0$
and makes $c$ comparable to $-\tilde u$; see
Section~\ref{subsec: two charts}. We can therefore analyze the
singular factors and their derivatives at a fixed boundary, while the
overlapping chart is used to match the solution to the exterior flow.

\subsection{Organization of the paper}
Section~\ref{sec: set up} introduces the symmetric Euler system,
constructs the characteristic geometry, opens the initial fan, and
defines the two overlapping charts. It then states the approximation
estimates and bootstrap assumptions and gives the proof outline.
Section~\ref{subsec: proof main theorem} completes the positive-time
construction, proves Proposition~\ref{prop: positive time construction},
and establishes the main theorem. Section~\ref{subsec: notation conventions}
collects the coordinate, unknown, measure, and regularity conventions
used throughout the paper.

Sections~\ref{sec: high order approximate solution} and
\ref{sec: normal jets} determine the initial profiles and establish
boundary matching. Section 3 derives the profiles and estimates their
residuals. Section 4 proves tangential matching and controls the normal
derivatives along $C_0$.

Sections~\ref{sec: Preparations for the energy estimate}--\ref{sec: closing B infinity}
close the bootstrap. Section 5 controls the characteristic frame and
proves the Sobolev inequalities on the evolving slices. Section 6
establishes the energy estimates, with separate arguments away from
and up to vacuum. Section 7 uses these estimates to close the pointwise
bootstrap. Their uniformity in the vacuum cutoff allows us to pass to
the common limit in Section~\ref{subsec: common limit}.

\section{Characteristic charts and energy identities}\label{sec: set up}
We construct the right rarefaction wave first; the left wave follows
by reflection. Sections~\ref{subsec: geo con}--\ref{subsec: opened geometry}
derive the characteristic equations, open the initial fan, and establish
the energy identities in the opened coordinates. Section~\ref{subsec: two charts}
introduces the two overlapping charts and explains which boundary each
chart is adapted to. Section~\ref{subsec: bootstrap framework} states
the approximation estimates and bootstrap assumptions, and
Section~\ref{subsec: proof outline main theorem} gives the proof outline.
The full construction is completed in Section~\ref{subsec: proof main theorem}.
Section~\ref{subsec: notation conventions} collects the notation and
conventions used throughout the paper.

\subsection{The symmetric Euler system and its characteristic geometry}\label{subsec: geo con}

We use the normalized unknown $U$ and exterior states $U_l,U_r$
introduced above. When a coordinate system is fixed, the same
letter denotes the expression of a field in those coordinates.
The letter $V$ below denotes a test field or a remainder, as
stated locally.

For $\rho>0$, the mechanical energy is a strictly convex
entropy in the conservative variables, and its Hessian gives a
symmetrizer; this is an instance of the convex-extension principle
of Friedrichs and Lax \cite{FriedrichsLax1971}. In the present isentropic
problem, the sound-speed normalization gives the following explicit
symmetric form of \eqref{eq: Euler in rho v}:
\begin{equation}\label{eq: euler equation in c v}
\partial_{t}(U)+A_{1}(U)\partial_{x_1}(U)+A_{2}(U)\partial_{x_2}(U)=0,
\end{equation}
where 
\[
U=\begin{pmatrix}
\frac{2c}{\gamma-1} \\ v^{1} \\ v^{2}
\end{pmatrix}, A_{1}(U)=\begin{pmatrix}
v^{1} & c & 0 \\
c & v^{1} & 0 \\
0 & 0 & v^{1}
\end{pmatrix}, A_{2}(U)=\begin{pmatrix}
v^{2} & 0 & c \\
0 & v^{2} & 0 \\
c & 0 & v^{2}
\end{pmatrix}.
\]
Both matrices are symmetric, so the identity matrix is a symmetrizer
for this normalized system. Their entries remain regular at $c=0$.
The change from conservative to normalized variables is invertible
for $\rho>0$; its behavior at vacuum must be treated separately.
For an account of symmetric hyperbolic systems and their energy
estimates, see \cite[Chapters 1--2]{Benzoni-Gavage-Serre2007book}.

For $(\xi_{1}, \xi_{2})\in \mathbb{R}^{2}\setminus\{0\}$, the eigenvalues and associated eigenvectors of $A_{1}\xi_{1}+A_{2}\xi_{2}$ are given by
\begin{equation}\label{eq: eigenvalues and eigenvectors}
\begin{cases}
\lambda_{1}(U, \xi)=v^{1}\xi_{1}+v^{2}\xi_{2}-c|\xi|\\
\lambda_{2}(U, \xi)=v^{1}\xi_{1}+v^{2}\xi_{2},\\
\lambda_{3}(U, \xi)=v^{1}\xi_{1}+v^{2}\xi_{2}+c|\xi|.
\end{cases}
\begin{cases}
r_{1}(U, \xi)=\frac{\sqrt{2}}{2}\big( 1, -\xi_{1}/|\xi|, -\xi_{2}/|\xi| \big)^{T},\\
r_{2}(U, \xi)=\big( 0, -\xi_{2}/|\xi|, \xi_{1}/|\xi| \big)^{T},\\
r_{3}(U, \xi)=\frac{\sqrt{2}}{2}\big( 1, \xi_{1}/|\xi|, \xi_{2}/|\xi| \big)^{T}.
\end{cases}
\end{equation}
The characteristic families have the following types:
\begin{itemize}
\item Genuinely nonlinear: $r_{m}\cdot \nabla_{U}\lambda_{m}\not=0, m=1, 3$. 
\item Linearly degenerate: $r_{2}\cdot \nabla_{U}\lambda_{2}=0$.
\end{itemize}
We focus on the genuinely nonlinear eigenvalues, which are also extremal in the sense $\lambda_{1}\leq \lambda_{2}\leq \lambda_{3}$ where the equality holds if and only if $c=0$. Therefore, vacuum corresponds to the degeneration of strict hyperbolicity of Euler equations.

\subsubsection{Characteristic hypersurfaces and the geometric frame}
We first record the characteristic construction for a smooth
background in physical coordinates. Its version on the opened initial
fan is given in Section~\ref{subsec: opened geometry}.
Let $C_u$ be a level hypersurface of a phase $u$, let $\Sigma_t$ be a
physical time slice, and set $S_{t,u}=\Sigma_t\cap C_u$.
The physical unit normal to $S_{t,u}$ is
$\hat T=\nabla_xu/|\nabla_xu|$. The phase satisfies\footnote{We use the convention that $u$ increases along the negative direction of $x_{1}$ axis.}
\begin{equation}\label{eq: Ekinol equation}
\partial_{t}(u)+\lambda_{3}(U, -\hat{T})\big(-\hat{T}(u)\big)=0\quad\Longleftrightarrow\quad \partial_{t}(u)-\lambda_{3}(U, -\nabla u)=0.
\end{equation}
This is the eikonal equation. Prescribing an initial foliation
$S_{0,u}$ on $\Sigma_0$ and solving \eqref{eq: Ekinol equation} constructs
the spacetime foliation $C_u$, with $S_{0,u}=C_u\cap\Sigma_0$, for as long
as the phase remains regular and its spatial gradient does not vanish. 

For a smooth background and a smooth initial phase with nonvanishing
spatial gradient, the eikonal equation is solved locally by the
Hamiltonian method of characteristics; see
\cite[Chapter II]{CourantHilbertII}. The construction remains valid
while the projection of the Hamiltonian flow to spacetime is
invertible. The bicharacteristic equations are
\begin{equation}\notag
\begin{cases}
\frac{\partial t}{\partial s}=\frac{\partial H}{\partial \tau},\\
\frac{\partial x_{i}}{\partial s}=\frac{\partial H}{\partial \xi_{i}},\\
\frac{\partial \tau}{\partial s}=-\frac{\partial H}{\partial t},\\
\frac{\partial \xi_{i}}{\partial s}=-\frac{\partial H}{\partial x_{i}}.
\end{cases}
\end{equation}
where 
\[
H(t, x_{1}, x_{2}, \tau, \xi_{1}, \xi_{2}):=\tau+\lambda_{3}\big( U(t, x_{1}, x_{2}), \xi_{1}, \xi_{2} \big).
\]
The Hamiltonian acts on the phase $-u$, so $(\tau,\xi)=(-\partial_t(u),-\nabla u)$. Its initial data lie in $H=0$ and are given by
\[
\begin{cases}
t(0)=0,\\
x_{i}(0)=y_{i},\\
\tau(0)=-\lambda_{3}\big(U(0,y_{1},y_{2}),\xi(0)\big),\\
\xi_{i}(0)=-\partial_{x_i}(u)(0,y_{1},y_{2}).
\end{cases} \ \ (y_{1}, y_{2})\in \mathbb{R}\times \mathbb{T}.
\]
Since $\partial_\tau(H)=1$, the parameter $s$ agrees with $t$.
The projections of the Hamiltonian trajectories issuing from each
initial curve $S_{0,u}$ sweep out its characteristic hypersurface
$C_u$. Thus the construction evolves both the hypersurfaces and their
normal covectors.

For the present argument, the useful output of Hamilton's equations
is a closed set of transport identities for the normal direction and
the spacing of the characteristic hypersurfaces. We record these
identities before introducing the characteristic coordinates.

We use $\hat{T}^{i}:=-\frac{\xi_{i}}{|\xi|}$ to denote the unit normal vector of the embedding $S_{t, u}\subset \Sigma_{t}$, $\hat{X}:=(\hat{T}^{2}, -\hat{T}^{1})$ to denote the unit tangent vector of the embedding $S_{t, u}\subset \Sigma_{t}$ and $\kappa:=\frac{1}{|\xi|}$ to denote the inverse density of  the embedding $S_{t, u}\subset \Sigma_{t}$. Set $T:=\kappa\hat{T}$, so that $T(u)=1$. We also use $L$ to denote the projection of generators\footnote{As tangent vector fields on $T^{*}\big( [0, t^{*}]\times \mathbb{R}\times \mathbb{T}$\big).} of bi-characteristic onto tangent vector fields on $[0, t^{*}]\times \mathbb{R}\times \mathbb{T}$. Then we have
\begin{equation}\notag
\begin{cases}
L=\frac{\partial}{\partial t}+\frac{\partial \lambda_{3}}{\partial \xi_{j}}(U, -\hat{T})\frac{\partial}{\partial x_{j}},\\
L(\hat{T}^{i})=\hat{X}(U)\cdot \nabla_{U}\lambda_{3}(U, -\hat{T})\hat{X}^{i},\\
L(\kappa)=-T(U)\cdot \nabla_{U}\lambda_{3}(U, -\hat{T}).
\end{cases}
\end{equation}
In particular, for isentropic Euler equation, $\lambda_{3}(U, \xi)=v^{1}\xi_{1}+v^{2}\xi_{2}+c|\xi|$, we have
\begin{equation}\label{eq: transport eq for geometry}
\begin{cases}
L=\partial_{t}+(v^{1}-c\hat{T}^{1})\partial_{x_1}+(v^{2}-c\hat{T}^{2})\partial_{x_2},\\
L(\hat{T}^{i})=\big( \hat{X}(c)+\hat{T}^{j}\hat{X}(\psi_{j}) \big)\hat{X}^{i},\\
L(\kappa)=-T(c)-\hat{T}^{j}T(\psi_{j})
\end{cases}
\end{equation}
where we used the notation $\psi_{1}=-v^{1}, \psi_{2}=-v^{2}$. We also record commutators for future use,
\begin{equation}\label{eq: commutators t x_1 x_1}
\begin{cases}
[T, \hat{X}]=-\kappa\theta\cdot \hat{X},\ \ \theta:=-\hat{X}(\hat{X}^{2})\hat{X}^{1}+\hat{X}(\hat{X}^{1})\hat{X}^{2},\\
[L, \hat{X}]=-\chi\cdot \hat{X},\ \ \chi:=-\hat{X}^{j}\hat{X}(\psi_{j})+c\hat{X}(\hat{X}^{2})\hat{X}^{1}-c\hat{X}(\hat{X}^{1})\hat{X}^{2},\\
[L, T]=-(\zeta+\eta)\cdot \hat{X},\ \ (\zeta+\eta):=-\kappa\hat{T}^{j}\hat{X}(\psi_{j})-\hat{X}^{j}T(\psi_{j})-\kappa\hat{X}(c)+c\hat{X}(\kappa).
\end{cases}
\end{equation}

\begin{remark}
In the geometric formulas below, $\zeta+\eta$ denotes a single coefficient, defined by the expression in \eqref{eq: commutators t x_1 x_1}; we do not assign separate meanings to the geometric symbols $\zeta$ and $\eta$. The notation is chosen to facilitate comparison with the acoustical-geometric formulation of Luo and Yu \cite[Section~2.2]{Luo-YuRare1}, where the torsion 1-forms $\zeta$ and $\eta$ are introduced separately and their sum appears in the commutator formula for $[L,T]$. In our setting, only the combined coefficient is needed in the commutator and the subsequent transport identities.
\end{remark}

To parametrize $S_{t, u}=\Sigma_{t}\cap C_{u}$, we also need the third parameter $\vartheta$ defined by
\begin{equation}\label{eq: angular coordinate in t x_1 x_2}
\begin{cases}
L(\vartheta)=0,\\
\vartheta=\vartheta_{0, u}, \ \textbf{on $S_{0, u}$}.
\end{cases}
\end{equation}
Then we have a new coordinate system $(t, u, \vartheta)$. Setting $\slashed{g}$ and $\Xi$ such that
\begin{equation}\label{eq: expression of T and Xh}
\begin{cases}
\frac{\partial}{\partial \vartheta}=\sqrt{\slashed{g}}\cdot \hat{X},\\
\frac{\partial}{\partial u}=T+\Xi\sqrt{\slashed{g}}\cdot \hat{X}.
\end{cases}
\end{equation}
and $\slashed{g}$ and $\Xi$ satisfy the transport equation along $L$ as follows,
\begin{equation}\label{eq: tansport for g and Xi}
\begin{cases}
L(\slashed{g})=2\chi\cdot \slashed{g},\\
L(\Xi)=\frac{1}{\sqrt{\slashed{g}}}(\zeta+\eta).
\end{cases}
\end{equation}

 \subsubsection{The symmetric energy and characteristic flux} Characteristic coordinates describe the integration region by fixed intervals in the characteristic labels. 
 
Consider the linear system
\[
\partial_{t}(V)+A_{1}(U)\partial_{x_1}(V)+A_{2}(U)\partial_{x_2}(V)=F.
\]
Taking the Euclidean inner product with $2V$, we obtain
\[
\begin{aligned}
&\partial_t\big(V,V\big)+\partial_{x_1}\big(V,A_1(U)V\big)
 +\partial_{x_2}\big(V,A_2(U)V\big)\\
&\quad=2(V,F)+\big(V,\frac{\partial A_1(U)}{\partial U^m}V\big)
                  \partial_{x_1}(U^m)\\
&\qquad+\big(V,\frac{\partial A_2(U)}{\partial U^m}V\big)
                  \partial_{x_2}(U^m).
\end{aligned}
\]
Integrating this identity over the region 
\[
\mathcal{D}(t_{1}, t_{2}; u_{1}, u_{2}):=\{(t, u, \vartheta)\in [t_{1}, t_{2}]\times [u_{1}, u_{2}]\times \mathbb{T}\}
\]
and applying the divergence theorem gives

\begin{equation}\label{eq: linear energy estimate in t x_1 x_2}
\begin{split}
&\mathcal{E}(V)(t_{2}; u_{1}, u_{2})+\mathcal{F}(V)(t_{1}, t_{2}; u_{2})=\mathcal{E}(V)(t_{1}; u_{1},  u_{2})+\mathcal{F}(V)(t_{1}, t_{2}; u_{1})\\
&+\int_{t_{1}}^{t_{2}}\int_{u_{1}}^{u_{2}}\int_{\mathbb{T}}\Big( (V, \frac{\partial A_{1}}{\partial U^{m}}(U)V)\partial_{x_1}(U^{m})+(V, \frac{\partial A_{2}}{\partial U^{m}}(U)V)\Xr(U^{m})\Big)\kappa \sqrt{\slashed{g}} dt' du d\vartheta\\
&+\int_{t_{1}}^{t_{2}}\int_{u_{1}}^{u_{2}}\int_{\mathbb{T}} 2(V, F) \kappa\sqrt{\slashed{g}} dt' du d\vartheta.
\end{split}
\end{equation}
where 
\begin{equation}\notag
\begin{cases}
\mathcal{E}(V)(t; u_{1}, u):=\int_{u_{1}}^{u} \int_{\mathbb{T}} (V, V)(t, u', \vartheta) \kappa \sqrt{\slashed{g}} du' d\vartheta,\\
\mathcal{F}(V)(t_{1}, t; u):=\int_{t_{1}}^{t} \int_{\mathbb{T}} \Big(2c(V, \alpha_{1})^{2}+c(V, \alpha_{2})^{2}\Big)(t', u, \vartheta) \sqrt{\slashed{g}} dt' d\vartheta.
\end{cases}
\end{equation}
Here $\psi^{(\hat T)}:=\psi_1\hat T^1+\psi_2\hat T^2=-v\cdot\hat T$.
The following eigenpairs use the physical symbol of
$A_1(U)(-\hat T^1)+A_2(U)(-\hat T^2)$; thus
$\lambda_j=\lambda_j(U,-\hat T)$ and $\alpha_j=r_j(U,-\hat T)$:
\begin{equation}\notag
\begin{cases}
\lambda_{1}=\psi^{(\hat{T})}-c,\ \ \alpha_{1}=\frac{\sqrt{2}}{2}(1, \hat{T}^{1}, \hat{T}^{2})^T,\\
\lambda_{2}=\psi^{(\hat{T})},\ \ \alpha_{2}=(0, \hat{X}^{1}, \hat{X}^{2})^T,\\
\lambda_{3}=\psi^{(\hat{T})}+c,\ \ \alpha_{3}=\frac{\sqrt{2}}{2}(1, -\hat{T}^{1}, -\hat{T}^{2})^T.
\end{cases}
\end{equation}

\subsection{Opening the fan and determining the initial profiles}\label{subsec: opened fan}

We first explain the normalization used for each one-sided
construction. Let $c_*>0$ and $v_*\in\mathbb R^2$ be the sound speed
and velocity of its original reference state, and let $T_*>0$ be the
original final time. Set
\[
\tau=\frac{t}{T_*},\qquad
\xi=\frac{x-v_*t}{c_*T_*},\qquad
c'=\frac{c}{c_*},\qquad v'=\frac{v-v_*}{c_*}.
\]
The symmetric system \eqref{eq: euler equation in c v} has the same
form in these variables, with reference sound speed $1$, reference
velocity $0$, and final time $1$. At $t=0$ the dividing line remains
$\xi_1=0$; the transverse period is rescaled by the fixed factor
$(c_*T_*)^{-1}$, and we continue to denote the resulting torus by
$\mathbb T$. These fixed affine changes of the coordinates and
normalized state alter the Sobolev smallness bound only by a fixed
multiplicative constant.

Dropping primes and renaming $(\tau,\xi)$ as $(t,x)$, we may therefore
work with $\mathbf c_*=t^*=1$ and initial data satisfying
\begin{equation}\notag
\|U_r(0,\cdot,\cdot)-(2\mathbf c_*/(\gamma-1),0,0)^T
\|_{H^{N_0}([0,\infty)\times\mathbb T)}\leq\varepsilon.
\end{equation}
We keep $\mathbf c_*$ and $t^*$ in formulas to display dependence
on fixed parameters. As in Theorem~\ref{theorem: stability of RVR},
$N_0$ is sufficiently large and $\varepsilon$ is sufficiently small;
the lower regularity index $N$ is fixed in
Section~\ref{subsec: notation conventions}.
Undoing the Galilean transformation and the change of units recovers
the original gas state and time interval. The density and momentum
are then recovered using the original pressure constant $k$.

Let $\mathcal{D}_{\textbf{right}}$ be the domain of dependence of the
initial data $U_{r}(0, x_{1}, x_{2})$ on $\{t=0, x_{1}\geq 0\}$.
We also write $U_r$ for the exterior solution on this domain.

We construct $U_{\textbf{rare}}$, abbreviated below to $U$, on the
complete rarefaction region $\mathcal{W}_{\textbf{complete}}$. Its
outer boundary is the interface $C_0$ with $\mathcal{D}_{\textbf{right}}$,
and the matching condition is
\begin{equation}\label{eq: continu along C0}
U_{\textbf{rare}}=U_{r} \ \ \textbf{on $C_{0}$}.
\end{equation}
where $C_{0}$ denotes both the left boundary of $\mathcal{D}_{\textbf{right}}$ and right boundary of $\mathcal{W}_{\textbf{complete}}$.

We denote by $C_{\textbf{vacuum}}$ the limiting level set
$\{c_{\textbf{rare}}=0\}$, which separates the complete rarefaction
region $\mathcal{W}_{\textbf{complete}}$ from the vacuum region.

\begin{remark}
The entropy solution across vacuum is described in the conservative
variables $(\rho,m)$. We first construct the gas flow
where $c>0$, using the normalized variables
$U=(2c/(\gamma-1),v^1,v^2)^T$ and estimates independent of a positive
lower bound for $c$. The conservative variables then have a well-defined
zero trace at the vacuum boundary. 
\end{remark}

The boundary $C_{0}$ is determined completely by the initial data $U_{r}$ on $x_{1}\geq 0$. Moreover, given 
\[
C_{*}:=\{(t,x_1,x_2):0\leq t\leq t^*,\ x_2\in\mathbb T,\ x_1=t\mathbf c_*\}.
\]
For any $t\in [0, t^{*}]$, we use $\Sigma_{t}$ to denote the time slice, then $\Sigma_{t}\cap C_{0}$ is contained in $C\varepsilon\cdot t$ neighborhood of $\Sigma_{t}\cap C_{*}$.

We will show that the vacuum boundary $C_{\textbf{vacuum}}$ is given by ruled surfaces
\[
\begin{cases}
t=s,\\
x_{1}=s\cdot (v_{r}^{1}-\frac{2}{\gamma-1}c_{r})(0, y_{2}),\\
x_{2}=y_{2}+s\cdot v_{r}^{2}(0, y_{2}).
\end{cases} 
\]
parametrized by $(s, y_{2})\in [0, t^{*}]\times \mathbb{T}$. 

The complete rarefaction wave region $\mathcal{W}_{\textbf{complete}}$ is strictly contained in the region $\Wr$,
\[
\Wr:=\{(t,x_1,x_2):0\leq t\leq t^*,\ x_2\in\mathbb T,\ -\frac{4t}{\gamma-1}\mathbf c_*\leq x_1\leq2t\mathbf c_*\}.
\]
We therefore parametrize $\Wr$ using the opened coordinates
$(t,\urr,x_2)$:
\[
\begin{cases}
t=t,\\
x_{1}=t\cdot \mathring{u},\\
x_{2}=x_{2}
\end{cases}
\]
where $(t,\urr,x_2)$ ranges over
$[0,t^*]\times[-4\mathbf c_*/(\gamma-1),2\mathbf c_*]\times\mathbb T$.
For $t>0$ this is a change of coordinates. At $t=0$ it replaces the
initial set $\mathbf S_*=\{t=0,x_1=0\}$ by an interval of fan labels.
We use the same symbol for a field and its expression in the new
coordinates. The physical time derivative and the derivative at fixed
$(\urr,x_2)$ are related by
\[
\left.\partial_t\right|_{x_1,x_2}
=\left.\partial_t\right|_{\urr,x_2}-\frac{\urr}{t}\partial_{\urr}.
\]
In equations expressed in the opened coordinates $(t,\urr,x_2)$,
$\partial_t$ keeps $(\urr,x_2)$ fixed. Throughout the paper, the term
``opened coordinates'' refers to $(t,\urr,x_2)$. The characteristic
charts are $(t,u,\vartheta)$ and $(t,\tilde u,\vartheta)$. In the physical-coordinate description of $C_0$
below, $\partial_t$ in $L$ instead keeps $(x_1,x_2)$ fixed;
$\partial_t(f)(t,x_2)$ keeps $x_2$ fixed.
As in the constructions away from vacuum reviewed in
Section~\ref{subsec: away from vacuum review}, we work with solutions
regular on the opened fan. The additional task is to obtain estimates
independent of the cutoff separating the gas from vacuum; the passage to vacuum is
described in Section~\ref{subsec: proof main theorem}.

\subsubsection{The exterior boundary in physical coordinates}\label{subsubsec: C_0 in t ur x_2}

We use $S_{t, 0}$ to denote 
\[
S_{t, 0}:=\Sigma_{t}\cap C_{0}\subset \Sigma_{t}.
\]
and $\hat{T}$ to denote unit normal vector of the embedding $S_{t, 0}\subset \Sigma_{t}$ with $\hat{T}^{1}\approx -1$. Then $C_{0}$ is ruled by integral curves of $L=\partial_{t}+(v^{1}-c\hat{T}^{1})\partial_{x_1}+(v^{2}-c\hat{T}^{2})\partial_{x_2}$ emanating from $\mathbf{S}_{*}=\{t=0, x_{1}=0\}$. Note that $C_{0}$ can be parametrized by $x_{1}=f(t, x_{2})$ in coordinate $(t, x_{1}, x_{2})$ with some regular $f$ with $f(0, x_{2})=0$. Here $f=f_0$ in the notation of Theorem~\ref{theorem: stability of RVR}. Moreover, we have 
\[
L\big( x_{1}-f(t, x_{2}) \big)=0\implies (v_{r}^{1}-c_{r}\hat{T}^{1})=\frac{\partial f}{\partial t}(t, x_{2})+\frac{\partial f}{\partial x_2}(t, x_{2})\cdot (v_{r}^{2}-c_{r}\hat{T}^{2}).
\]
Taking $t=0$, we have $\frac{\partial f}{\partial t}(0, x_{2})=(v_{r}^{1}+c_{r})(0, x_{2})$.

\subsubsection{The exterior boundary in opened coordinates}
We still use $\Sigma_{t}$ to denote the time slice in coordinate $(t, \mathring{u}, x_{2})$. Note that $\mathring{u}=\frac{f(t, x_{2})}{t}$ is also a regular function and $C_{0}$ can be viewed as a hypersurface (still denoted by $C_{0}$) in coordinate $(t, \mathring{u}, x_{2})$. Moreover, the intersection of $C_{0}$ and the initial slice $\Sigma_{0}:=\{t=0\}$ is given by
\[
\{ (0, \mathring{u}, x_{2}): \mathring{u}=(v_{r}^{1}+c_{r})(0, x_{2}) \}.
\]
and we have $U\equiv U_{r}$\footnote{We write $U_{\textbf{rare}}$ as $U$ for simplicity.} along $C_{0}$ in view of \eqref{eq: continu along C0}. In particular, we have
\begin{equation}\label{eq: initial matching}
U\big(0, (v_{r}^{1}+c_{r})(0, x_2), x_{2} \big)=U_{r}(0, x_{2}).
\end{equation}

\subsubsection{The Euler equations and the leading profile in opened coordinates}

We write the Euler equations \eqref{eq: euler equation in c v} in coordinate $(t, \mathring{u}, x_{2})$, 
\begin{equation}\notag
\frac{\partial U}{\partial t}+\frac{1}{t}\big( A_{1}(U)-\mathring{u} \big)\frac{\partial U}{\partial \mathring{u}}+A_{2}(U)\partial_{x_2}(U)=0.
\end{equation}
To determine the leading initial profile, multiply by $t$ and
equate the coefficient of $t^0$. This gives 
\[
\Big( A_{1}(U)-\mathring{u} \Big)\frac{\partial U}{\partial \mathring{u}}=0.
\]
At points where $\partial_{\mathring u}(U)\ne0$, the preceding
equation makes $\mathring u$ an eigenvalue of $A_1(U)$. The matching condition
\eqref{eq: initial matching} gives 
\[
U\big(0, (v_{r}^{1}+c_{r})(0, x_{2}), x_{2}\big)=U_{r}(0, x_{2}), x_{2}\in \mathbb{T}. 
\]
We construct the right-going rarefaction by choosing the third
characteristic family, so
\[
\mathring{u}=\lambda_{3}(U)=v^{1}+c.
\]
On the other hand, $\frac{\partial U}{\partial \mathring{u}}$ must be parallel to the third eigenvector $r_{3}:=\frac{\sqrt{2}}{2}(1, 1, 0)^{T}$, which implies that
\[
\frac{\partial}{\partial \mathring{u}}(\frac{2c}{\gamma-1}-v^{1})=0, \frac{\partial}{\partial \mathring{u}}(v^{2})=0.
\]
Hence the solution $U=(\frac{2c}{\gamma-1}, v^{1}, v^{2})^T$ has the following form on the initial slice $\Sigma_{0}$:
\begin{equation}\label{eq: data on t=0}
\begin{cases}
\frac{2}{\gamma-1}c(0, \mathring{u}, x_{2})=\frac{2}{\gamma-1}\cdot \Big( c_{r}(0, x_{2})+\frac{\gamma-1}{\gamma+1}\big(\mathring{u}-(v_{r}^{1}+c_{r})(0, x_{2})\big)\Big),\\
v^{1}(0, \mathring{u}, x_{2})=v_{r}^{1}(0, x_{2})+\frac{2}{\gamma+1}\big(\mathring{u}-(v_{r}^{1}+c_{r})(0, x_{2})\big),\\
v^{2}(0, \mathring{u}, x_{2})=v_{r}^{2}(0, x_{2}).
\end{cases}
\end{equation}
We have now determined the data for $U$ on both $\Sigma_{0}$ and $C_{0}$. Corollary~\ref{cor: compatibility on S_0,0} establishes their compatibility on $\Sigma_{0}\cap C_{0}$ to the required order. Section~\ref{sec: normal jets} then derives the boundary jets from transport equations along $C_{0}$.

In view of \eqref{eq: data on t=0}, the intersection of vacuum boundary $C_{\textbf{vacuum}}$ and initial slice $\Sigma_{0}$ can also be determined,
\[
S_{0, \textbf{vacuum}}:=C_{\textbf{vacuum}}\cap \Sigma_{0}:=\{(0, \mathring{u}, x_{2}), c_{r}(0, x_{2})+\frac{\gamma-1}{\gamma+1}\big(\mathring{u}-(v_{r}^{1}+c_{r})(0, x_{2})\big)=0\}.
\] 

By differentiating the multiplied opened Euler equations in $t$
and setting $t=0$, one obtains equations for successive formal
initial time derivatives. We use only the first two such equations.
For the fixed nonvacuum solution they give its actual time
derivatives at $t=0$; in the near-vacuum analysis the same equations
define the formal time derivatives on the initial fan. Thus we write 
\begin{equation}\notag
\begin{cases}
\Vr_{0}(\mathring{u}, x_{2}):=U(0, \mathring{u}, x_{2}),\\
\Vr_{1}(\mathring{u}, x_{2}):=\partial_{t}(U)(0, \mathring{u}, x_{2}),\\
\Vr_{2}(\mathring{u}, x_{2}):=\partial_{t}^{2}(U)(0, \mathring{u}, x_{2}).
\end{cases}
\end{equation}
which
motivates the following approximate solutions:
\begin{equation}\notag
\begin{cases}
\Ur_{0}(t, \mathring{u}, x_{2}):=\Vr_{0}(\mathring{u}, x_{2}),\\
\Ur_{1}(t, \mathring{u}, x_{2}):=\Vr_{0}(\mathring{u}, x_{2})+t\cdot \Vr_{1}(\mathring{u}, x_{2}).
\end{cases}
\end{equation}
Section~\ref{sec: high order approximate solution} derives the formulas and bounds for $\Vr_{1}, \Vr_{2}$. Propositions~\ref{prop: initial energy bounds on Sigma_0}, \ref{prop: Res bounds of Ur_k}, and~\ref{prop: matching of U and Ur_k along C_0} give the initial coefficient, residual, and boundary estimates used in the bootstrap. Remarks~\ref{rem: remark 4.1} and~\ref{rem: remark 4.2} explain the restrictions on profile regularity near vacuum.

\subsection{Characteristic geometry and energy in the opened fan}\label{subsec: opened geometry}
The identities used later are the geometric transport system
\eqref{eq: formulas to control the geometry} and the energy identity
\eqref{eq: linear energy estimate in t ur x_2}. We derive them in
the opened coordinates to keep track of the factors of $t$ at the
initial slice. The choice of characteristic foliation remains general
until Section~\ref{subsec: two charts}. 

\subsubsection{Transport equations in the opened coordinates}
We first allow a general initial phase $u=-\urr+h(x_2)$.
The formulas in this subsection will be used for both choices of $h$
in Section~\ref{subsec: two charts}; the geometric fields are
reconstructed for each choice.
In the opened coordinates $(t, \mathring{u}, x_{2})$, the Euler equations take the form
\[
\frac{\partial U}{\partial t}+\frac{1}{t}\big( A_{1}(U)-\mathring{u} \big)\frac{\partial U}{\partial \mathring{u}}+A_{2}(U)\partial_{x_2}(U)=0.
\]
The maximal eigenvalue $\lambda_{3}(U, \xi_{1}, \xi_{2})$ associated with the covector $(\xi_{1}, \xi_{2})$ is
\[
\begin{split}
\lambda_{3}(U, \xi_{1}, \xi_{2})&=\frac{1}{t}(v^{1}-\mathring{u})\xi_{1}+v^{2}\xi_{2}+c\sqrt{\frac{\xi_{1}^{2}}{t^{2}}+\xi_{2}^{2}}\\
&=\frac{v^{1}+c-\mathring{u}}{t}\xi_{1}+v^{2}\xi_{2}
  +\frac{ct\xi_{2}^{2}}{\sqrt{\xi_{1}^{2}+t^{2}\xi_{2}^{2}}+\xi_{1}},\qquad \xi_{1}>0.
\end{split}
\]
For $t>0$ the two expressions agree. On a fixed positive-cutoff
domain, assume the constructed solution has the required time
regularity in the opened coordinates. The identity
$(v^1+c)|_{t=0}=\mathring u$ then gives the extension
\[
\frac{v^1(t,\mathring u,x_2)+c(t,\mathring u,x_2)-\mathring u}{t}
=\int_0^1\partial_t(v^1+c)(st,\mathring u,x_2)\,ds.
\]
Thus $C^{m+1}$ regularity of the fields gives $C^m$ regularity of
this quotient, and the rationalized eigenvalue is regular for
$\xi_1$ bounded away from zero. This regular extension is
available in the fixed region supplied by the nonvacuum theorem.
Near vacuum we instead solve the characteristic equations from
$t=\delta>0$ and pass to the limit in
Proposition~\ref{prop: positive time construction}.
For the following initial-slice description, let $u$ have the form
\[
u=-\mathring{u}+h(x_{2}),\ |\frac{d^{k}h}{dx_{2}^{k}}|\leq C\varepsilon,\ 1\leq k\leq \Ntop.
\]
whose level sets\footnote{The level set of $u$ is almost parallel to the level set of $\mathring{u}$.} give a foliation of the initial slice $\Sigma_{0}$. We can construct characteristic surfaces associated with initial foliation of $u$ by solving Eikonal equation,
\[
\partial_{t}(u)-\lambda_{3}(U, -\partial_{\mathring{u}}(u), -\partial_{x_2}(u))=0.
\]
Here $\partial_t$ is taken at fixed $(\mathring{u},x_2)$, and
$(\tau,\xi_1,\xi_2)=(-\partial_t(u),-\partial_{\mathring{u}}(u),-\partial_{x_2}(u))$
are the components of $d(-u)$ in this chart. With $H=\tau+\lambda_3(U,\xi_1,\xi_2)$,
where $\lambda_3$ is the symbol displayed above, the Hamilton--Jacobi approach gives the ODE system
\begin{equation}\label{eq: singular HJ}
\begin{cases}
\frac{dt}{ds}=\frac{\partial H}{\partial \tau}=1,\\
\frac{d\mathring{u}}{ds}=\frac{\partial H}{\partial \xi_{1}}=\frac{1}{t}\big( v^{1}-\mathring{u}+c\cdot\frac{\xi_{1}}{\sqrt{(\xi_{1})^{2}+t^{2}(\xi_{2})^{2}}} \big),\\
\frac{d x_{2}}{ds}=\frac{\partial H}{\partial \xi_{2}}=v^{2}+c\cdot \frac{t\xi_{2}}{\sqrt{(\xi_{1})^{2}+t^{2}(\xi_{2})^{2}}},\\
\frac{d\tau}{ds}=-\frac{\partial H}{\partial t},\\
\frac{d\xi_{1}}{ds}=-\frac{\partial H}{\partial \mathring{u}}=-\frac{1}{t}(\frac{\partial v^{1}}{\partial \mathring{u}}-1)\cdot \xi_{1}-\frac{\partial v^{2}}{\partial \mathring{u}}\cdot \xi_{2}-\frac{\partial c}{\partial \mathring{u}}\cdot \frac{\sqrt{(\xi_{1})^{2}+t^{2}(\xi_{2})^{2}}}{t},\\
\frac{d\xi_{2}}{ds}=-\frac{\partial H}{\partial x_{2}}=-\frac{1}{t}\frac{\partial v^{1}}{\partial x_2}\cdot \xi_{1}-\frac{\partial v^{2}}{\partial x_2}\cdot \xi_{2}-\frac{\partial c}{\partial x_2}\cdot \frac{\sqrt{(\xi_{1})^{2}+t^{2}(\xi_{2})^{2}}}{t}.
\end{cases}
\end{equation}
with initial data 
\begin{equation}\notag
\begin{cases}
t(0)=0, \mathring{u}(0)=y_{1}, x_{2}(0)=y_{2},\\
\xi_{1}(0)=1, \xi_{2}(0)=-\frac{dh}{dx_{2}}(y_{2}),\\
\tau(0)=-\partial_t(v^{1}+c)(0,y_{1},y_{2})
         +v^{2}(0,y_{1},y_{2})\frac{dh}{dx_{2}}(y_{2}).
\end{cases} 
\end{equation}
When the regular extension above is available, the value of
$\tau(0)$ is $-\lambda_3|_{t=0}$, and the ordinary Hamilton--Jacobi
construction gives local existence of $u$ from this initial slice.
For the positive-time approximations the same system starts at
$t=\delta$, with $\tau(\delta)=-\lambda_3|_{t=\delta}$; all its
coefficients are then regular. The projected generator is
\begin{equation}\notag
L=\frac{\partial}{\partial t}+\frac{1}{t}\big( v^{1}-\mathring{u}+c\cdot\frac{\xi_{1}}{\sqrt{\xi_{1}^{2}+t^{2}\xi_{2}^{2}}} \big)\frac{\partial}{\partial \mathring{u}}+\big( v^{2}+c\cdot \frac{t\xi_{2}}{\sqrt{\xi_{1}^{2}+t^{2}\xi_{2}^{2}}} \big)\frac{\partial}{\partial x_{2}}
\end{equation}

For $t>0$, the physical spatial gradient is $\nabla_xu=(-\xi_1/t,-\xi_2)$. Thus the definitions $\hat{T}=\nabla_xu/|\nabla_xu|$, $\kappa=|\nabla_xu|^{-1}$ and $T=\kappa\hat{T}$ take the following forms. We study the transport equation of geometric quantities $\hat{T}^{1}, \hat{T}^{2}$ and $\kappa$\footnote{One could instead use transport equations for $\hat{\tilde{T}}^{i}:=-\frac{\xi_{i}}{|\xi|}$ and $\tilde{\kappa}:=\frac{1}{|\xi|}$ once $U$ is known. For the a priori estimates, it is more convenient to use $\hat{T}^{1}, \hat{T}^{2}, \kappa$ rather than $\hat{\tilde{T}}^{1}, \hat{\tilde{T}}^{2}, \tilde{\kappa}$.},
\begin{equation}\notag
\begin{cases}
(\hat{T}^{1}, \hat{T}^{2})=(-\hat{X}^{2}, \hat{X}^{1}):=\big(-\frac{\xi_{1}}{\sqrt{(\xi_1)^{2}+t^{2}(\xi_2)^{2}}}, -\frac{t\xi_{2}}{\sqrt{(\xi_1)^{2}+t^{2}(\xi_2)^{2}}}\big),\\
\kappa:=\frac{t}{\sqrt{(\xi_1)^{2}+t^{2}(\xi_2)^{2}}},\\
T:=\frac{\kappa\hat{T}^{1}}{t}\cdot \frac{\partial}{\partial \mathring{u}}+\kappa \hat{T}^{2}\cdot \frac{\partial}{\partial x_2},\\
\hat{X}:=\frac{\hat{X}^{1}}{t}\cdot \frac{\partial}{\partial \mathring{u}}+\hat{X}^{2}\frac{\partial}{\partial x_2}.
\end{cases}
\end{equation}
Using \eqref{eq: singular HJ}, we obtain
\begin{equation}\label{eq: final form of transport equation}
\begin{cases}
L=\frac{\partial}{\partial t}+\frac{1}{t}\big( v^{1}-\mathring{u}-c\hat{T}^{1} \big)\frac{\partial}{\partial \mathring{u}}+\big( v^{2}-c\hat{T}^{2} \big)\frac{\partial}{\partial x_2},\\
L(\hat{T}^{i})=\big( \hat{X}(c)+\hat{T}^{j}\hat{X}(\psi_{j}) \big)\hat{X}^{i},\\
L(\kappa)=-T(c)-\hat{T}^{j}T(\psi_{j}).
\end{cases}
\end{equation}
with initial data
\begin{equation}\notag
\begin{cases}
\hat{T}^{1}(0, \mathring{u}, x_2)=-1, \hat{T}^{2}(0, \mathring{u}, x_2)=0,\\
\kappa(0, \mathring{u}, x_2)=0.
\end{cases}
\end{equation}
Equation~\eqref{eq: final form of transport equation} is the opened-coordinate form of \eqref{eq: transport eq for geometry}. 

The corresponding commutator identities are
\begin{equation}\notag
\begin{cases}
[T, \hat{X}]=-\kappa\theta\cdot \hat{X},\ \ \theta:=-\hat{X}(\hat{X}^{2})\hat{X}^{1}+\hat{X}(\hat{X}^{1})\hat{X}^{2},\\
[L, \hat{X}]=-\chi\cdot \hat{X},\ \ \chi:=-\hat{X}^{j}\hat{X}(\psi_{j})+c\hat{X}(\hat{X}^{2})\hat{X}^{1}-c\hat{X}(\hat{X}^{1})\hat{X}^{2},\\
[L, T]=-(\zeta+\eta)\cdot \hat{X},\ \ (\zeta+\eta):=-\kappa\hat{T}^{j}\hat{X}(\psi_{j})-\hat{X}^{j}T(\psi_{j})-\kappa\hat{X}(c)+c\hat{X}(\kappa).
\end{cases}
\end{equation}
This is the opened-coordinate form of \eqref{eq: commutators t x_1 x_1}.

To parametrize $S_{t, u}=\Sigma_{t}\cap C_{u}$, we also need the third parameter $\vartheta$ defined by
\begin{equation}\notag
\begin{cases}
L(\vartheta)=0,\\
\vartheta(0, \mathring{u}, x_{2})=x_{2}.
\end{cases}
\end{equation}
This is the opened-coordinate form of \eqref{eq: angular coordinate in t x_1 x_2}.

Then we have a new coordinate system $(t, u, \vartheta)$. Setting $\slashed{g}$ and $\Xi$ such that
\begin{equation}\notag
\begin{cases}
\frac{\partial}{\partial \vartheta}=\sqrt{\slashed{g}}\cdot \hat{X},\\
\frac{\partial}{\partial u}=T+\Xi\sqrt{\slashed{g}}\cdot \hat{X}.
\end{cases}
\end{equation}
This is the opened-coordinate form of \eqref{eq: expression of T and Xh}.

Then $\slashed{g}$ and $\Xi$ satisfy the transport equation along $L$ as follows,
\begin{equation}\notag
\begin{cases}
L(\slashed{g})=2\chi\cdot \slashed{g},\\
\slashed{g}(0, \mathring{u}, x_2)=1,\\
L(\Xi)=\frac{1}{\sqrt{\slashed{g}}}(\zeta+\eta),\\
\Xi(0, \mathring{u}, x_{2})=0.
\end{cases}
\end{equation}
This is the opened-coordinate form of \eqref{eq: tansport for g and Xi}. 

To summarize, we have
\begin{equation}\label{eq: formulas to control the geometry}
\begin{cases}
L=\frac{\partial}{\partial t}+\frac{1}{t}\big( v^{1}-\mathring{u}-c\hat{T}^{1} \big)\frac{\partial}{\partial \mathring{u}}+\big( v^{2}-c\hat{T}^{2} \big)\frac{\partial}{\partial x_2},\\
L(\hat{T}^{i})=\big( \hat{X}(c)+\hat{T}^{j}\hat{X}(\psi_{j}) \big)\hat{X}^{i},\\
\hat{T}^{1}(0, \mathring{u}, x_2)=-1, \hat{T}^{2}(0, \mathring{u}, x_2)=0,\\
L(\kappa)=-T(c)-\hat{T}^{j}T(\psi_{j}),\\
\kappa(0, \mathring{u}, x_{2})=0,\\
[T, \hat{X}]=-\kappa\theta\cdot \hat{X},\ \ \theta:=-\hat{X}(\hat{X}^{2})\hat{X}^{1}+\hat{X}(\hat{X}^{1})\hat{X}^{2},\\
[L, \hat{X}]=-\chi\cdot \hat{X},\ \ \chi:=-\hat{X}^{j}\hat{X}(\psi_{j})+c\hat{X}(\hat{X}^{2})\hat{X}^{1}-c\hat{X}(\hat{X}^{1})\hat{X}^{2},\\
[L, T]=-(\zeta+\eta)\cdot \hat{X},\ \ (\zeta+\eta):=-\kappa\hat{T}^{j}\hat{X}(\psi_{j})-\hat{X}^{j}T(\psi_{j})-\kappa\hat{X}(c)+c\hat{X}(\kappa),\\
L(\vartheta)=0,\\
\vartheta(0, \mathring{u}, x_{2})=x_{2},\\
\frac{\partial}{\partial \vartheta}=\sqrt{\slashed{g}}\cdot \hat{X},\\
\frac{\partial}{\partial u}=T+\Xi\sqrt{\slashed{g}}\cdot \hat{X},\\
L(\slashed{g})=2\chi\cdot \slashed{g},\\
\slashed{g}(0, \mathring{u}, x_2)=1,\\
L(\Xi)=\frac{1}{\sqrt{\slashed{g}}}(\zeta+\eta),\\
\Xi(0, \mathring{u}, x_{2})=0.
\end{cases}
\end{equation}

We also record the form of Euler equations in frame $\{L, \hat{T}, \hat{X}\}$ in terms of both $(c, v^{1}, v^{2})$ and Riemann invariants $(\wb, w, \psi_{2})$,
\begin{equation}\notag
\begin{cases}
L (\frac{2}{\gamma-1}c) &= -c \widehat{T}(\frac{2}{\gamma-1}c)+c \widehat{T}(\psi_i)\widehat{T}^i +c\Xh(\psi_i)\Xh^i,\\
L (\psi_1) &= -c \widehat{T}(\psi_1)+\frac{2}{\gamma-1}c \widehat{T}(c)\widehat{T}^1+\frac{2}{\gamma-1}c\Xh(c)\Xh^1,\\
L (\psi_2) &= -c \widehat{T}(\psi_2)+\frac{2}{\gamma-1}c \widehat{T}(c)\widehat{T}^2+\frac{2}{\gamma-1}c\Xh(c)\Xh^2.
\end{cases}
\end{equation}
and
\begin{equation}\label{Euler equations:form 2}
\begin{cases}
L (\wb) &= -c \widehat{T}(\wb)(\widehat{T}^1+1)+\frac{1}{2}c \widehat{T}(\psi_2)\widehat{T}^2 +\frac{1}{2}c \Xh(\psi_2)\Xh^2-c\Xh(\wb)\Xh^1,\\
L (w) &= 	c \widehat{T}(w)(\widehat{T}^1-1)	+\frac{1}{2}c \widehat{T}(\psi_2)\widehat{T}^2 +c\Xh(w)\Xh^1+\frac{1}{2}c\Xh(\psi_2)\Xh^2,\\
L (\psi_2) &= -c \widehat{T}(\psi_2)+c\widehat{T}( w+\wb)\widehat{T}^2+c\Xh( w+\wb)
\Xh^2.
\end{cases}
\end{equation}
where we have used the conventions
\[
\begin{cases}
\wb=\frac{c}{\gamma-1}+\frac{v^{1}}{2}, w=\frac{c}{\gamma-1}-\frac{v^{1}}{2},\\
\psi_{1}=-v^{1}, \psi_{2}=-v^{2}.
\end{cases}
\]

\subsubsection{The energy identity in the opened fan}

Take the Euclidean inner product with $2tV$\footnote{The single factor of $t$ gives the opened-coordinate form of \eqref{eq: linear energy estimate in t x_1 x_2}.} on both sides of
\[
\frac{\partial V}{\partial t}+\frac{1}{t}\big( A_{1}(U)-\mathring{u} \big)\frac{\partial V}{\partial \mathring{u}}+A_{2}(U)\partial_{x_2}(V)=F.
\]

This gives
\[
\begin{split}
&\frac{\partial}{\partial t} \big(t\cdot (V, V)  \big)+\frac{\partial}{\partial \mathring{u}}\Big( V, \big(A_{1}(U)-\mathring{u}\big)V \Big)+t\cdot \frac{\partial}{\partial x_2}\big( V, A_{2}(U)V \big)=\\
&2t\cdot (V, F)+(V, V)+\big( V, \frac{\partial A_{1}(U)}{\partial U^{m}}V \big)\frac{\partial U^{m}}{\partial \mathring{u}}-(V, V)+t\cdot \big( V, \frac{\partial A_{2}(U)}{\partial U^{m}}V \big)\frac{\partial U^{m}}{\partial x_{2}}.
\end{split}
\]

Integrating by parts on the region
\[
\mathcal{D}(t; u_{1}, u_{2}):=\{(t', u, \vartheta)\in [0, t]\times [u_{1}, u_{2}]\times \mathbb{T}\},
\]
we have
\begin{equation}\label{eq: linear energy estimate in t ur x_2}
\begin{split}
&\mathcal{E}(V)(t; u_{1}, u_{2})+\mathcal{F}(V)(0, t; u_{2})=\mathcal{E}(V)(0; u_{1}, u_{2})+\mathcal{F}(V)(0, t; u_{1})\\
&+\int_{0}^{t}\int_{u_{1}}^{u_{2}}\int_{\mathbb{T}}\Big( (V, \frac{\partial A_{1}}{\partial U^{m}}(U)V)\frac{\partial U^{m}}{\partial \mathring{u}}+t'(V, \frac{\partial A_{2}}{\partial U^{m}}(U)V)\frac{\partial U^{m}}{\partial x_2}\Big)\frac{\kappa}{t'} \sqrt{\slashed{g}} dt' du d\vartheta\\
&+\int_{0}^{t}\int_{u_{1}}^{u_{2}}\int_{\mathbb{T}} 2t'\cdot (V, F) \frac{\kappa}{t'}\sqrt{\slashed{g}} dt' du d\vartheta.
\end{split}
\end{equation}
where 
\begin{equation}\notag
\begin{cases}
\mathcal{E}(V)(t; u_{1}, u):=\int_{u_{1}}^{u} \int_{\mathbb{T}} t\cdot (V, V)(t, u', \vartheta) \frac{\kappa}{t} \sqrt{\slashed{g}} du' d\vartheta,\\
\mathcal{F}(V)(0, t; u):=\int_{0}^{t} \int_{\mathbb{T}} \Big(2c(V, \alpha_{1})^{2}+c(V, \alpha_{2})^{2}\Big)(t', u, \vartheta) \sqrt{\slashed{g}} dt' d\vartheta.
\end{cases}
\end{equation}
Here $\psi^{(\hat T)}:=\psi_1\hat T^1+\psi_2\hat T^2=-v\cdot\hat T$.
The following eigenpairs use the physical symbol of
$A_1(U)(-\hat T^1)+A_2(U)(-\hat T^2)$; thus
$\lambda_j=\lambda_j(U,-\hat T)$ and $\alpha_j=r_j(U,-\hat T)$:
\begin{equation}\notag
\begin{cases}
\lambda_{1}=\psi^{(\hat{T})}-c, \alpha_{1}=\frac{\sqrt{2}}{2}(1, \hat{T}^{1}, \hat{T}^{2})^T,\\
\lambda_{2}=\psi^{(\hat{T})}, \alpha_{2}=(0, \hat{X}^{1}, \hat{X}^{2})^T,\\
\lambda_{3}=\psi^{(\hat{T})}+c, \alpha_{3}=\frac{\sqrt{2}}{2}(1, -\hat{T}^{1}, -\hat{T}^{2})^T.
\end{cases}
\end{equation}

\begin{remark}
On the one hand, in {\bf away from vacuum case}, $c\approx \mathbf{c}_{*}$, since
\[
(V, \alpha_{1})^{2}=\frac{1}{2}\big(V, \begin{pmatrix}
1&\hat{T}^{1}&\hat{T}^{2}\\
\hat{T}^{1}&(\hat{T}^{1})^{2}&\hat{T}^{1}\hat{T}^{2}\\
\hat{T}^{2}&\hat{T}^{1}\hat{T}^{2}&(\hat{T}^{2})^{2}
\end{pmatrix}V\big),  (V, \alpha_{2})^{2}=\big(V, \begin{pmatrix}
0&0&0\\
0&(\hat{X}^{1})^{2}&\hat{X}^{1}\hat{X}^{2}\\
0&\hat{X}^{1}\hat{X}^{2}&(\hat{X}^{2})^{2}
\end{pmatrix}V\big)
\]
and $\hat{T}^{1}+1=1-\hat{X}^{2}=O(\varepsilon^{2}t^{2}), \hat{T}^{2}=\hat{X}^{1}=O(\varepsilon t)$, the fluxes can be exploited to control dangerous linear terms 
\[
\big(V, \begin{pmatrix}
1&-1&0\\
-1&1&0\\
0&0&0
\end{pmatrix}V\big)
\]
arising from commuting $\Tr=-\frac{\partial}{\partial \mathring{u}}$ with the Euler equations. Their null structure allows the highest-order terms to be controlled away from vacuum; see Section~\ref{sec: energy estimates}.

Near vacuum, the flux weight degenerates with $c$, so this flux argument alone does not control the leading terms produced by commuting $\Tr=-\frac{\partial}{\partial \mathring{u}}$ with the Euler equations. Section~\ref{sec: energy estimates} instead uses the full damping and null quadratic form.
\end{remark}

\subsection{Two overlapping characteristic charts}\label{subsec: two charts}
The same characteristic identities are used with two different initial
foliations. The first makes the exterior boundary $C_0$ a coordinate
surface and supports the higher-regularity estimates away from vacuum.
The second makes the initial sound speed a function of its normal
coordinate alone and supports estimates uniform as the vacuum cutoff
is removed. Their coordinate ranges are
\[
\begin{aligned}
\mathcal W_1:\quad &(t,u,\vartheta)\in
 [0,t^*]\times[0,u^*]\times\mathbb T,\\
\mathcal W_{2,\alpha}:\quad &(t,\tilde u,\vartheta)\in
 [0,t^*]\times[\tilde u_*,\tilde u^*]\times\mathbb T,
 \qquad 0<\alpha<1/100.
\end{aligned}
\]
For fixed $\alpha>0$, their union stops short of vacuum.
The union over the cutoffs covers the gas part of the full fan;
the vacuum boundary is recovered as a limiting boundary.
The overlap supplies the data needed to pass from the first estimate
to the second.

The symbols $L,T,\hat T,\hat X,\kappa,\slashed g,\Xi,\vartheta$ are
used for the fields of the chart under discussion. In particular,
$L(\vartheta)=0$ is solved separately in each chart. The two sets of
fields are compared on the overlap through their maps to the common
coordinates $(t,\urr,x_2)$; see
Section~\ref{subsec: notation conventions}.

\subsubsection{The chart attached to the exterior flow}
 The intersection of $C_{0}$ and $\Sigma_{0}$ in coordinate $(t, \mathring{u}, x_{2})$ is given by
\[
\{ (0, \mathring{u}, x_{2}): \mathring{u}=(v_{r}^{1}+c_{r})(0, x_{2}) \}.
\]
which motivates us to define $(u, \vartheta)$ as follows,
\begin{equation}\label{eq: first characteristic coordinate}
\begin{cases}
L(u)=0,\\
u(0, \mathring{u}, x_{2})=-\big(\mathring{u}-(v_{r}^{1}+c_{r})(0, x_{2})\big),\\
L(\vartheta)=0,\\
\vartheta(0, \mathring{u}, x_{2})=x_{2}.
\end{cases}
\end{equation}
The minus sign makes $u$ increase from $C_0$ into the fan.
The initial level $u=0$ generates $C_0$, so $C_0=\{u=0\}$.
The transport and energy identities
\eqref{eq: formulas to control the geometry} and
\eqref{eq: linear energy estimate in t ur x_2} apply with this phase.

\begin{center}
\begin{tikzpicture}[scale=1.1, >=stealth]
\draw[black!80, line width=0.8pt] (0,0) 
    -- (7.2,0) 
    -- (7.2,2) 
    -- (0,2) 
    -- cycle;
\draw[->, black!80, line width=0.6pt] (0,-0.4) -- (7.5,-0.4);
\node at (7.7,-0.4) {$\mathring{u}$};
\draw[->, black!80, line width=0.6pt] (-0.4,0) -- (-0.4,2.3);
\node at (-0.4,2.5) {$x_2$};
\foreach \x in {0.3,1.1,1.9,2.7,3.5,4.3,5.1,5.9,6.7}{
    \draw[dashed, gray!70, line width=0.55pt] (\x,0) -- (\x,2);
}
\foreach \xb in {2.52,3.12,3.72,4.32,4.92,5.52}{
    \draw[dashed, blue!70, line width=0.55pt] plot[smooth] coordinates {
        (\xb,0) (\xb-0.09,0.5) (\xb-0.12,1) (\xb-0.09,1.5) (\xb,2)
    };
}
\draw[blue!80, line width=0.7pt] plot[smooth] coordinates {
    (5.88,0) (5.79,0.5) (5.76,1) (5.79,1.5) (5.88,2)
};
\node[blue!80, above] at (5.88,2.15) {\small $C_0$};

\node at (3.6,2.5) {\small $t=0\quad (\Sigma_0)$};
\end{tikzpicture}
\end{center}

On the initial slice $\Sigma_{0}$, the solution $U$ takes the form
\[
U(0, u, \vartheta)=\begin{pmatrix}
\frac{2}{\gamma-1}\cdot \big( c_{r}(0, \vartheta)-\frac{\gamma-1}{\gamma+1}u \big)\\
v^{1}_{r}(0, \vartheta)-\frac{2}{\gamma+1}u\\
v^{2}_{r}(0, \vartheta)
\end{pmatrix}
\]
in coordinate $(u, \vartheta)$. The intersection $C_{\textbf{vacuum}}\cap \Sigma_{0}$ is given by
\[
\{(0, u, \vartheta), u=\frac{\gamma+1}{\gamma-1}c_{r}(0, \vartheta)\}.
\]
which means $u$ does not align with the $C_{\textbf{vacuum}}$. Therefore, we restrict the range of $u$ in $[0, u^{*}]$,  
\[
u^{*}:=\frac{1}{2}\frac{\gamma+1}{\gamma-1}\mathbf{c}_{*},
\]
Then we will show that the sound speed $c$ is strictly bounded from below and above,
\[
\frac{1}{4} \mathbf{c}_{*}\leq c(t, u, \vartheta)\leq 2\mathbf{c}_{*}. 
\]
in the region $\mathcal{W}_{1}$ defined as follows,
\begin{equation}\notag
\mathcal{W}_{1}:=\{(t, u, \vartheta)\in[0, t^{*}]\times [0, u^{*}]\times \mathbb{T}\}.
\end{equation}
The region $\mathcal{W}_{1}$ is foliated by the characteristic hypersurfaces $C_{u}$:
\[
\mathcal{W}_{1}=\cup_{u\in [0, u^{*}]}C_{u}.
\]
\begin{remark}
The choice of $(t, u, \vartheta)$ for \textbf{away from vacuum case} essentially coincides with Yu-Luo's last slice argument where they have constructed $(u_{\delta}, \vartheta_{\delta})$ and obtained $\delta$ independent estimates as $\delta\rightarrow 0$.
\end{remark}

\subsubsection{The chart adapted to the vacuum limit}

The intersection $C_{\textbf{vacuum}}\cap \Sigma_{0}$ is given by
\[
\{(0, u, \vartheta), u=\frac{\gamma+1}{\gamma-1}c_{r}(0, \vartheta)\}.
\]
We therefore initialize the second phase by
$\tilde u=-\urr+(v_r^1-2c_r/(\gamma-1))(0,x_2)$.
The resulting system is
\begin{equation}\label{eq: second characteristic coordinate}
\begin{cases}
L(\tilde{u})=0,\\
\tilde{u}(0, \mathring{u}, x_{2})=-\frac{\gamma+1}{\gamma-1}c_{r}(0, x_{2})-\big( \mathring{u}-(v_{r}^{1}+c_{r})(0, x_{2}) \big)\\
L(\vartheta)=0,\\
\vartheta(0, \mathring{u}, x_{2})=x_{2}.
\end{cases}
\end{equation}
On $\Sigma_0$ this choice gives the exact relation
$c(0,\tilde u,\vartheta)=-(\gamma-1)\tilde u/(\gamma+1)$.
Here and in \eqref{eq: second characteristic coordinate}, $L$ is
constructed from $\tilde u$ by the Hamiltonian equations.
We restrict the range of $\tilde u$ to $[\tilde u_*,\tilde u^*]$, 
\[
\begin{cases}
\tilde{u}^{*}=-\alpha\cdot \frac{\gamma+1}{\gamma-1}\mathbf{c}_{*},\\
\tilde{u}_{*}=-\frac{3}{4}\cdot \frac{\gamma+1}{\gamma-1}\mathbf{c}_{*}.
\end{cases} \ \ \alpha\in (0, \frac{1}{100})
\]

On the one hand, on the initial slice $\Sigma_{0}$, the curve $\tilde{u}=\tilde{u}_{*}$ is given by
\[
u=\frac{1}{4}\frac{\gamma+1}{\gamma-1}\mathbf{c}_{*}+\frac{\gamma+1}{\gamma-1}\big( c_{r}(0, x_{2})-\mathbf{c}_{*} \big)\in [0, u^{*}]
\]
Thus the domains $0\leq u\leq u^*$ and $\tilde u_*\leq\tilde u\leq\tilde u^*$
have an open overlap on $\Sigma_0$.
The geometric estimates preserve this overlap for $0\leq t\leq t^*$,
and $C_{\tilde u_*}$ lies in the interior of $\mathcal W_1$. On the other hand, on the initial slice $\Sigma_{0}$, the sound speed takes value $\alpha\cdot \mathbf{c}_{*}$ along the curve $\tilde{u}=\tilde{u}^{*}$. Moreover, we will show that\footnote{The sound speed is comparable to its initial value along $C_{\tilde{u}}$, as proved in Section~\ref{sec: Preparations for the energy estimate}.}
\[
-\frac{1}{2}\frac{\gamma-1}{\gamma+1}\tilde{u} \leq c(t, \tilde{u}, \vartheta) \leq -2\frac{\gamma-1}{\gamma+1}\tilde{u}
\]
in the region $\mathcal{W}_{2, \alpha}$ defined as follows,
\[
\mathcal{W}_{2, \alpha}:=\{(t, \tilde{u}, \vartheta), (t, \tilde{u}, \vartheta)\in [0, t^{*}]\times [\tilde{u}_{*}, \tilde{u}^{*}]\times \mathbb{T}\}.
\]
We remark that \eqref{eq: formulas to control the geometry} and \eqref{eq: linear energy estimate in t ur x_2} apply to this particular choice of $\tilde{u}$. Moreover, the limit of level set $C_{\tilde{u}^{*}}$ as $\alpha\rightarrow 0$,
\[
\lim_{\alpha\rightarrow 0} C_{\tilde{u}^{*}}
\]
aligns with the vacuum boundary $C_{\textbf{vacuum}}$. 

\begin{center}
\begin{tikzpicture}[scale=1.1, >=stealth]
\draw[black!80, line width=0.8pt] (0,0) 
    -- (7.2,0) 
    -- (7.2,2) 
    -- (0,2) 
    -- cycle;
\draw[->, black!80, line width=0.6pt] (0,-0.4) -- (7.5,-0.4);
\node at (7.7,-0.4) {$\mathring{u}$};
\draw[->, black!80, line width=0.6pt] (-0.4,0) -- (-0.4,2.3);
\node at (-0.4,2.5) {$x_2$};
\foreach \x in {0.3,1.1,1.9,2.7,3.5,4.3,5.1,5.9,6.7}{
    \draw[dashed, gray!70, line width=0.55pt] (\x,0) -- (\x,2);
}
\foreach \xb in {1.44,2.04,2.64,3.24,3.96}{
    \draw[dashed, red!70, line width=0.55pt] plot[smooth] coordinates {
        (\xb,0) (\xb+0.07,0.5) (\xb+0.09,1) (\xb+0.07,1.5) (\xb,2)
    };
}
\draw[red!80, line width=0.7pt] plot[smooth] coordinates {
    (0.84,0) (0.91,0.5) (0.93,1) (0.91,1.5) (0.84,2)
};
\node[red!80, above] at (0.84,2.15) {\small $C_{\textbf{vacuum}}$};
\foreach \xb in {2.52,3.12,3.72,4.32,4.92,5.52}{
    \draw[dashed, blue!70, line width=0.55pt] plot[smooth] coordinates {
        (\xb,0) (\xb-0.09,0.5) (\xb-0.12,1) (\xb-0.09,1.5) (\xb,2)
    };
}
\draw[blue!80, line width=0.7pt] plot[smooth] coordinates {
    (5.88,0) (5.79,0.5) (5.76,1) (5.79,1.5) (5.88,2)
};
\node[blue!80, above] at (5.88,2.15) {\small $C_0$};

\node at (3.6,2.5) {\small $t=0\quad (\Sigma_0)$};
\end{tikzpicture}
\end{center}

\subsubsection{How the overlap enters the construction}
The proof first estimates $U-\Ur_k$ in $\mathcal W_1$.
Since $C_{\tilde u_*}$ is an interior hypersurface of that region,
these estimates give its traces and the incoming flux needed for
$\mathcal W_{2,\alpha}$. The data are compared through
$(t,\urr,x_2)$, where the commutators $\Tr,\Xr$ are fixed.
The second estimate is uniform in $\alpha$. The construction in
Proposition~\ref{prop: positive time construction} initializes the
near-vacuum problem at $t=\delta>0$, keeps an open overlap on which
the solution agrees with the known outer solution, and removes both cutoffs by one
diagonal limit. Thus the two charts organize exterior matching,
the transfer of boundary estimates, and the construction at vacuum.

\subsection{Approximation estimates and bootstrap assumptions}\label{subsec: bootstrap framework}
We use fixed spatial commutators in the opened coordinates $(t,\urr,x_2)$,
so that the same differentiated remainder is available in both charts:
\begin{equation}\notag
\Xr:=\frac{\partial}{\partial x_2},\qquad
\Tr=-\frac{\partial}{\partial\mathring u}=-t\partial_{x_1}.
\end{equation}
We also set $\kappar=t$ and $\hat{\Tr}=-\partial_{x_1}$, so
$\Tr=\kappar\hat{\Tr}$ for $t>0$.
The initial coefficients and boundary estimates used in the bootstrap
are listed next. Sections~\ref{sec: high order approximate solution}
and~\ref{sec: normal jets} derive their formulas and bounds. The
Taylor-limit assertions below are used for the already constructed
solution away from vacuum. Near vacuum, the construction uses the
explicit profiles as positive-time Cauchy data and the residual
bounds. The initial terms are verified directly in
Proposition~\ref{prop: positive time construction}.
\begin{remark}
In $\mathcal{W}_{2, \alpha}$, an alternative to $\Xr$ is the vector field
\[
\Xr_{\textbf{new}}:=\frac{\partial}{\partial x_{2}}-\big( \frac{2}{\gamma-1}\partial_{x_2}(c_{r})-\partial_{x_2}(v_{r}^{1}) \big)(0, x_{2})\frac{\partial}{\partial \mathring{u}}
\]
which is tangent to the level sets of $c(0, \mathring{u}, x_{2})$ on the initial slice $\Sigma_{0}$:
\[
\Xr_{\textbf{new}}\Big( c_{r}(0, x_{2})+\frac{\gamma-1}{\gamma+1}\big( \mathring{u}-(v_{r}^{1}+c_{r})(0, x_{2}) \big) \Big)=0
\]
This choice is suggested by the singular factors $c^{\alpha}\cdot \ln(c)^{\beta}$ in the profile formulas; see Remarks~\ref{rem: remark 4.1} and~\ref{rem: remark 4.2}. For the construction in this paper, we use the fixed commutators
$\Xr=\frac{\partial}{\partial x_{2}}$ and
$\Tr=-\frac{\partial}{\partial\mathring{u}}$.
Together with the approximate solutions used below, they suffice
to close the estimates for
$\gamma\in(1,1.3]\cup\{\frac{7}{5},\frac{5}{3}\}$.
\end{remark}

\subsubsection{Choice of the regularity orders}\label{subsubsec: regularity orders}
We distinguish smoothness of the initial data, the finite order of
the norm in \eqref{eq: small data for R-V-R}, and the regularity that
we actually estimate in the rarefaction region. Smooth initial data
allow us to perform the differentiations below. All quantitative
bounds used in the proof depend only on a finite number of their
derivatives. We now choose these orders.

Choose an integer $N_{\rm away}\geq8$ and use $\Ntop=N_{\rm away}$ in
the region away from vacuum. For the a priori estimate in that region
alone, $N_{\rm away}\geq6$ suffices. We use $\Ntop=4$ in the region
near vacuum. We first describe the regularity needed in the region
away from vacuum. Write $n=N_{\rm away}$ in the following discussion.
\begin{itemize}
\item The affine profile uses the trace of the initial state. The first
and second formal time derivatives use the initial state and its
physical derivatives through orders one and two, respectively;
see Section~\ref{sec: high order approximate solution}. Thus
$C^{n+3}$ bounds for the initial data give $C^{n+1}$ bounds for the
profiles on any fixed interval separated from $c=0$. The extra spatial
derivative suffices for the residuals differentiated through order $n$.
The same initial regularity, in particular the bounds for the
initial derivatives through order six, supplies the coefficient
bounds for the four-derivative profile estimates near vacuum; there, the powers of
$c$ and $\ln c$, not the regularity of the exterior data, determine
which derivatives are integrable.
\item In Section~\ref{sec: normal jets}, derivatives of the exterior
trace are taken in $(t,x_2)$ and then in $(t,\vartheta)$. Restriction of a $C^m$ field to a $C^m$ graph preserves this order.
In obtaining pointwise bounds from exterior $H^q$ estimates in two
space dimensions we use $H^q\hookrightarrow C^{q-2}$ for integer
$q\geq2$, and time derivatives of total order at most $q-2$ follow
from the symmetric system. Forming $f/t$ from the characteristic
graph with $f(0)=0$ costs one time derivative. The normal-derivative recursion can use
two tangential derivatives per normal derivative. We allow derivatives
of order $2(n+1)+8$ of the exterior fields and the geometry, which
also covers the extra derivative in $Z^\alpha T(U)$ and the two time
derivatives used for Taylor estimates. We use the fixed nonvacuum solution constructed in our earlier work
\cite{JiaLuo}, with state and normalized characteristic geometry of
class $C^{J_4}$ up to the opened initial slice and with quantitative
$O(\varepsilon)$ comparison to the corresponding planar fields, where
$J_4\geq2n+12$. This supplies the boundary jets,
changes of coordinates, and Taylor limits in this region; in
particular, the $C^{n+2}$ opened-coordinate bound suffices for the
Taylor limits differentiated through order $n$.
\item On the fixed overlap, the estimates through order $n$ and
Sobolev embedding give the fixed-coordinate pointwise remainder
estimates through order $n-2\geq6$. Only the estimates through
order four are needed for the incoming flux and for the initial
joining of the two regions in
Proposition~\ref{prop: positive time construction}.
The near-vacuum energy estimate uses four spatial derivatives, the
frame estimates at most three, and the metric and shift estimates
at most two. The corresponding initial geometry uses derivatives
of the initial trace through order four; see
\eqref{eq: positive time initial geometry} and
\eqref{eq: positive time initial weights}. Higher regularity at a
fixed positive pair $(\delta,\alpha)$ is used only to justify
commutation and continuation, not in a bound uniform in the pair.
\item In the limit, the four spatial derivatives yield the
$W^{2,\infty}$ spatial bounds. For the two special exponents the
opened Euler equations also give the time derivatives through total
order two. The same data order covers the return to the original
coordinates.
\end{itemize}

Here $J_4$ denotes the regularity required in the opened coordinates.
To obtain it from \cite[Theorem~2.5 and Section~7.4.1]{JiaLuo},
denote the regularity indices of that work by $\widehat N_i$ and choose
\begin{equation}\tag{R}\label{eq: global regularity orders}
J_4\geq2N_{\rm away}+12,\qquad
\widehat N_4\geq J_4+3,\qquad N_0\geq\widehat N_0.
\end{equation}
First fix $\widehat N_4$; then choose
$\widehat N_3,\widehat N_2,\widehat N_1,\widehat N_0$ successively
large enough for the derivative losses described in
\cite[Section~2.7]{JiaLuo}. These losses are either a fixed number
of derivatives or a fixed factor in the order of differentiation.
Only the one-sided nonvacuum construction is used.

The extra three derivatives in \eqref{eq: global regularity orders}
allow for the change from characteristic to opened coordinates.
Indeed, at fixed $(u,\vartheta)$,
\[
\urr(t,u,\vartheta)=\frac{x_1(t,u,\vartheta)}t
=\int_0^1 L(x_1)(st,u,\vartheta)\,ds.
\]
Thus division by $t$ loses at most one derivative. At $t=0$,
$\partial_u(\urr)=-1$ and $x_2=\vartheta$; the spatial Jacobian
stays uniformly invertible on the fixed nonvacuum region.
The normalized geometry $\hat T,\kappa/t,\slashed g,\Xi$ is
recovered from first spatial derivatives of $(\urr,x_2)$ by the
Jacobian formulas, costing at most one further derivative.
The state, map and normalized geometry therefore have the
$C^{J_4}$ regularity used above. Increase $N_0$ if needed for
the exterior $C^{2n+12}$ bounds. All orders are fixed before
choosing $\varepsilon$; the uniform near-vacuum order remains four.

\subsubsection{Initial coefficients, residuals, and boundary matching}
We next record the bounds for the approximate solutions $\Ur_{k}, k=0, 1$, under the fixed commutators $\Zr\in \{\Tr, \Xr\}$.
\begin{proposition}\label{prop: initial energy bounds on Sigma_0}
For a solution on a fixed positive-vacuum-cutoff domain with a
second-order Taylor expansion at $t=0$ in the indicated spatial norms,
the coefficients and limits are as follows. The Taylor limits are taken
at a fixed cutoff, and the coefficient bounds below are uniform in the cutoff.
In both \textbf{region away from vacuum} and \textbf{region near vacuum} (independent of $\alpha\in (0, \frac{1}{100})$), for $\gamma\in (1, 3)$, we have
\begin{equation}\notag
\begin{cases}
\lim_{t\rightarrow 0} \frac{(U-\Ur_{0})(t, \mathring{u}, x_{2})}{t}=\Vr_{1}(\mathring{u}, x_2),\\
\lim_{t\rightarrow 0} \frac{(U-\Ur_{1})(t, \mathring{u}, x_{2})}{t^{2}}=\frac{1}{2}\Vr_{2}(\mathring{u}, x_2).
\end{cases}
\end{equation}
Moreover, we have quantitative information higher order on Sobolev norms.
\begin{itemize}
\item In \textbf{region $\mathcal{W}_{1}$ away from vacuum}, for any $\gamma\in (1, 3), k=0, 1$ and $a+b\leq N$, 
\begin{equation}\notag
\lim_{t\rightarrow 0}\frac{\|\Tr^{a}\Xr^{b}(U-\Ur_{k})\|_{L^{2}(\Sigma_{t}^{0, u^{*}})}}{t^{k+1}}=\frac{1}{k+1}\|\Tr^{a}\Xr^{b}(\Vr_{k+1})\|_{L^{2}(\Sigma_{0}^{0, u^{*}})}\leq C\varepsilon.
\end{equation}
\item In \textbf{region $\mathcal{W}_{2, \alpha}$ near vacuum} (independent of $\alpha\in (0, \frac{1}{100})$), for $\gamma\in (1, 1.3]$ and $a+b\leq 4$,
\begin{equation}\notag
\lim_{t\rightarrow 0}\frac{\|\Tr^{a}\Xr^{b}(U-\Ur_{0})\|_{L^{2}(\Sigma_{t}^{\ut_{*}, \ut^{*}})}}{t^{1}}=\|\Tr^{a}\Xr^{b}(\Vr_{1})\|_{L^{2}(\Sigma_{0}^{\ut_{*}, \ut^{*}})}\leq C\varepsilon.
\end{equation}
\item In \textbf{region $\mathcal{W}_{2, \alpha}$ near vacuum} (independent of $\alpha\in (0, \frac{1}{100})$), for $\gamma\in \{\frac{7}{5}, \frac{5}{3}\}$ and $a+b\leq 4$,
\begin{equation}\notag
\lim_{t\rightarrow 0}\frac{\|\Tr^{a}\Xr^{b}(U-\Ur_{1})\|_{L^{2}(\Sigma_{t}^{\ut_{*}, \ut^{*}})}}{t^{2}}=\frac{1}{2}\|\Tr^{a}\Xr^{b}(\Vr_{2})\|_{L^{2}(\Sigma_{0}^{\ut_{*}, \ut^{*}})}\leq C\varepsilon.
\end{equation}
\end{itemize}
\end{proposition}
This proposition will be proved in section \ref{sec: high order approximate solution}. Remarks~\ref{rem: remark 4.1} and~\ref{rem: remark 4.2} explain the
choice $\Ntop=4$ near vacuum. They exhibit obstructions to higher
profile regularity for particular exponents and data.

As a byproduct of the proof of the proposition \ref{prop: initial energy bounds on Sigma_0} in section \ref{sec: high order approximate solution}, we have
\begin{proposition}\label{prop: Res bounds of Ur_k}
In both \textbf{region $\mathcal{W}_{1}$ away from vacuum} and \textbf{region $\mathcal{W}_{2, \alpha}$ near vacuum} (independent of $\alpha\in (0, \frac{1}{100})$), for $a+b\leq N$, we have
\begin{itemize}
\item For $k=0, \gamma\in (1, 3)$ and $a+b\leq N$,                  
\begin{equation}\notag
|\Tr^{a}\Xr^{b}\Big( t\partial_{t}(\Ur_{0})+\big(A_{1}(\Ur_{0})-\mathring{u}\big)\partial_{\mathring{u}}(\Ur_{0})+tA_{2}(\Ur_{0})\partial_{x_2}(\Ur_{0}) \Big)|\leq C\varepsilon t.
\end{equation}
\item For $k=1$ and $a+b\leq N$, for every $\gamma\in(1,3)$
in $\mathcal W_1$, and for $\gamma\in\{7/5,5/3\}$ in
$\mathcal W_{2,\alpha}$, 
\begin{equation}\notag
|\Tr^{a}\Xr^{b}\Big( t\partial_{t}(\Ur_{1})+\big(A_{1}(\Ur_{1})-\mathring{u}\big)\partial_{\mathring{u}}(\Ur_{1})+tA_{2}(\Ur_{1})\partial_{x_2}(\Ur_{1}) \Big)|\leq C\varepsilon t^{2}.
\end{equation}
\end{itemize}
\end{proposition}
which means $\Ur_{0}, \Ur_{1}$ serves as good approximations in certain cases. We also refer to Remarks~\ref{rem: remark 4.1} and~\ref{rem: remark 4.2} on the discussion on the different treatment for different $\gamma$. 

Setting $\wb:=\frac{c}{\gamma-1}+\frac{1}{2}v^{1}, w=\frac{c}{\gamma-1}-\frac{1}{2}v^{1}, \psi_{2}:=-v^{2}$. As a byproduct of the proof of the proposition \ref{prop: initial energy bounds on Sigma_0} in section \ref{sec: high order approximate solution}, we also have
\begin{proposition}\label{prop: formulas for c^-1 L(psi)}
In both charts, the affine profile and the first formal time derivative
determine the initial values of
$\wb,w,\psi_2,c^{-1}L(\wb),c^{-1}L(w),c^{-1}L(\psi_2)$.
For these initial quantities, with $Z\in\{T,\hat X\}$,
$1\leq|\beta|\leq N$ and $\psi\in\{\wb,w,\psi_2\}$, one has
on $\Sigma_0^{0,u^*}\cup\Sigma_0^{\ut_*,\ut^*}$:
\begin{equation}\notag
\begin{cases}
T(\wb)+\frac{2}{\gamma+1}=0, T(w)=0, T(\psi_{2})=0, \\
TZ^{\beta}(\psi)=0,\\
|(\hat{X})^{m}(\psi)|\leq C\varepsilon,\quad 1\leq m\leq N, \\
|c^{-1}L(\psi)|\leq C\varepsilon.
\end{cases}
\end{equation}
\end{proposition}
We refer to \eqref{eq: formula for L(psi)} for detailed formula of $c^{-1}\cdot L(\wb), c^{-1}\cdot L(w), c^{-1}\cdot L(\psi_{2})$.

At last, we need higher order matching of $U$ and $\Ur_{k}$ along $C_{0}$.
\begin{proposition}\label{prop: matching of U and Ur_k along C_0}
For the nonvacuum solution in the first characteristic chart with
the regularity stated in Section~\ref{subsubsec: regularity orders},
for $k=0,1$, $\gamma\in(1,3)$ and all $a+b\leq\Ntop$, we have
\begin{equation}\notag
|\Tr^{a}\Xr^{b}(U-\Ur_{k})|(t, \mathring{u}, x_{2})\leq C\varepsilon t^{k+1}, \ \ \textbf{along $C_{0}$}.
\end{equation}
\end{proposition}
Since $C_{0}$ lies away from vacuum, the same argument treats $k=0, 1$ for every $\gamma\in (1, 3)$. The proof is given in Section~\ref{sec: normal jets}.

\subsubsection{The bootstrap away from vacuum} 

For $k=0, 1$, we define 
\[
V_{k; a, b}:=\Tr^{a}\Xr^{b}(U-\Ur_{k}).
\]  
for any fixed $\gamma\in (1, 3)$. We use $\Sigma_{t}^{0, u^{*}}, (t, u)\in[0, t^{*}]\times [0, u^{*}]$ to denote the time slice in $\mathcal{W}_{1}$.

\begin{itemize}
\item In view of propositions \ref{prop: initial energy bounds on Sigma_0}, \ref{prop: matching of U and Ur_k along C_0}, we can assume that there exists $M>0$ such that for all $(t, u)\in [0, t^{*}]\times [0, u^{*}]$, the following inequalities hold:
\begin{equation}\notag{(\mathbf{B}_{2})}
\quad\mathcal{E}(V_{k; a, b})(t;0, u^{*})+\mathcal{F}(V_{k; a,b})(0, t; u)\leq M\varepsilon^{2}t^{2k+3},\quad \forall a+b\leq \Ntop.
\end{equation}

\item In view of proposition \ref{prop: formulas for c^-1 L(psi)}, we can assume that there exists $M>0$ such that for all $t\in [0, t^{*}]$ and $Z\in \{T, \hat{X}\}$, the following inequalities hold:
\begin{equation}\notag{(\mathbf{B}_{\infty})}
\begin{cases}
\|c^{-1}L(\underline{w})\|_{L^{\infty}(\Sigma_{t}^{0, u^{*}})}+\|c^{-1}L(w)\|_{L^{\infty}(\Sigma_{t}^{0, u^{*}})}+\|c^{-1}L(\psi_{2})\|_{L^{\infty}(\Sigma_{t}^{0, u^{*}})}\leq M\varepsilon,\\
\|\hat{X}Z^{\beta}(\psi)\|_{L^{2}(\Sigma_{t}^{0, u^{*}})}\leq M\varepsilon, \psi\in \{\wb, w, \psi_{2}\}, 2\leq |\beta|\leq 3,\\
\|\hat{X}Z^{\beta}(\psi)\|_{L^{\infty}(\Sigma_{t}^{0, u^{*}})}\leq M\varepsilon, \psi\in \{\wb, w, \psi_{2}\}, |\beta|\leq 1,\\
\|T(w)\|_{L^{\infty}(\Sigma_{t}^{0, u^{*}})}+\|T(\psi_{2})\|_{L^{\infty}(\Sigma_{t}^{0, u^{*}})}+\|T(\underline{w})+\frac{2}{\gamma+1}\|_{L^{\infty}(\Sigma_{t}^{0, u^{*}})}\leq M\varepsilon t,\\
\|TZ^{\beta}(\psi)\|_{L^{2}(\Sigma_{t}^{0, u^{*}})}\leq M\varepsilon t, 2\leq |\beta|\leq 3,\\
\|TZ^{\beta}(\psi)\|_{L^{\infty}(\Sigma_{t}^{0, u^{*}})}\leq M\varepsilon t, |\beta|=1.
\end{cases}
\end{equation}

\item In view of \eqref{eq: formulas to control the geometry}, we can assume that for all $t\in [0, t^{*}]$ the following inequalities hold:
\begin{equation}\notag{(\mathbf{B}_{\text{diffeomorphism}})}
\begin{cases}
|\hat{T}^{1}+1|\leq \frac{1}{100},\\
|\frac{\partial u}{\partial \mathring{u}}+1|\leq \frac{1}{100}, |\frac{\partial u}{\partial x_{2}}|\leq \frac{1}{100},\\
|\frac{\partial \vartheta}{\partial \mathring{u}}|\leq \frac{1}{100}, |\frac{\partial \vartheta}{\partial x_{2}}-1|\leq \frac{1}{100}.
\end{cases}
\end{equation}
\end{itemize}
The assumption $(\mathbf B_{\text{diffeomorphism}})$ controls the
spatial coordinate map at each fixed time:
\[
(u,\vartheta)\longmapsto(\mathring u,x_2),\qquad
(u,\vartheta)\in [0, u^{*}]\times\mathbb T.
\]
It is a $W^{1,\infty}$ diffeomorphism onto its image with uniform
bounds for its spatial Jacobian and inverse. Time derivatives of the
map are given separately by its generator $L$.

We now state the a priori estimate in the first chart. All the
bootstrap inequalities are imposed on the whole label interval for
$0<t\leq T$, where $T\leq t^*$. The a priori estimate assumes the initial coefficient, residual,
and boundary estimates separately.
\begin{theorem}\label{theorem: a prior estimate in W_1}
Fix $k\in\{0,1\}$, $\gamma\in(1,3)$ and $\Ntop\geq6$. Choose the data
order as in Section~\ref{subsubsec: regularity orders}, with
$N_{\rm away}=\Ntop$.
Suppose that a positive-density solution is given in the first
characteristic chart on $0<t\leq T\leq t^*$. Assume that the state,
the map to the opened coordinates, and the normalized geometry
$\hat T,\kappa/t,\slashed g,\Xi$ have class $C^{\Ntop+2}$ on the
closed parameter domain including $t=0$. The state has the
initial profile \eqref{eq: data on t=0}; the geometry is initialized by
\eqref{eq: formulas to control the geometry} and
\eqref{eq: first characteristic coordinate}, with $\kappa/t\to1$
on the initial slice. Assume the coefficient and Taylor bounds of
Propositions~\ref{prop: initial energy bounds on Sigma_0} and
\ref{prop: formulas for c^-1 L(psi)}, the residual bounds of
Proposition~\ref{prop: Res bounds of Ur_k}, and the incoming
boundary remainder bounds of
Proposition~\ref{prop: matching of U and Ur_k along C_0}.
In particular, the initial energy divided by $t^{2k+3}$ is bounded
as $t\rightarrow0$ and the incoming flux is bounded by
$C\varepsilon^2t^{2k+3}$ through order $\Ntop$.

There are constants $M_0$ and, for each fixed $M\geq M_0$,
$\varepsilon_0>0$ such that, if $0<\varepsilon\leq\varepsilon_0$
and $(\mathbf B_2)$, $(\mathbf B_\infty)$ and
$(\mathbf B_{\rm diffeomorphism})$ hold with the constants displayed
above, then the same inequalities hold throughout $0<t\leq T$
with $M$ replaced by $M/2$ and $1/100$ replaced by $1/200$.
The choices are independent of $T$; they depend only on the fixed
reference parameters, regularity orders and the constants in the
bounds assumed in the theorem (and, for $\varepsilon_0$, on $M$).
\end{theorem}
The proof of Theorem~\ref{theorem: a prior estimate in W_1} is given in Sections~\ref{sec: Preparations for the energy estimate}--\ref{sec: closing B infinity}.  

As a byproduct of theorem \ref{theorem: a prior estimate in W_1} and its proof, we have
\begin{corollary}
Under the assumptions of Theorem~\ref{theorem: a prior estimate in W_1},
for the fixed $k$ in that theorem and any $0<t\leq T$ and
$a+b\leq\Ntop-2$, we have
\begin{equation}\label{eq: L infty bound in W_1}
\|V_{k; a, b}\|_{L^{\infty}(\Sigma_{t}^{0, u^{*}})}\leq C\varepsilon t^{k+1}.
\end{equation}
\end{corollary}
\begin{proof}
The energy estimate in Theorem~\ref{theorem: a prior estimate in W_1}, together with $\kappa/t\approx1$, gives the slice $L^2$ bounds for $V_k$ and all its ordered spatial derivatives through order $N_{\rm top}$, with size $C\varepsilon t^{k+1}$. The first-chart coordinate bounds make the corresponding two-dimensional Sobolev inequality uniform on its fixed label cylinder. Applying $H^2\hookrightarrow L^\infty$ to $\Tr^a\Xr^b V_k$ and using two further spatial derivatives gives \eqref{eq: L infty bound in W_1} for $a+b\leq N_{\rm top}-2$. This is the Sobolev step used at the end of Subsection~6.3. This proves the corollary.
\end{proof}

\subsubsection{The bootstrap uniform up to vacuum} 
For $k=0, 1$, we still define 
\[
V_{k; a,b}:=\Tr^{a}\Xr^{b}(U-\Ur_{k}).
\]
However, we stress that we need to treat different $\gamma$ in separate ways, to be more precise,
\begin{equation}\label{eq: different treatment}
\begin{cases}
k=0, \ \ \gamma\in (1, 1.3],\\
k=1, \ \ \gamma\in \{\frac{7}{5}, \frac{5}{3}\}.
\end{cases}
\end{equation}   
We use $\Sigma_{t}^{\tilde{u}_{*}, \tilde{u}}, (t, \tilde u)\in[0, t^{*}]\times [\tilde{u}_{*}, \ut^{*}]$ to denote the time slice in $\mathcal{W}_{2, \alpha}$.
For the construction, impose the following assumptions first on
$[\delta,t^*]$, with every flux integrated from $\delta$.
Their initialization follows from \eqref{eq: positive time data}
and \eqref{eq: positive time initial geometry}, and the incoming
flux comes from the fixed overlap with the outer solution.
The formulas below use zero-lower-time notation for the a priori
estimates. The construction applies their finite-$\delta$ versions
to regular approximations and passes the bulk and first-order
flow bounds to the limit; see
Proposition~\ref{prop: positive time construction}.

\begin{itemize}

\item With $M>0$ fixed sufficiently large, assume for the near-vacuum remainder:
\begin{equation}\notag{(\mathbf{B}_{2})}
\quad\mathcal{E}(V_{k; a, b})(t; \ut_{*}, \ut^{*})+\mathcal{F}(V_{k; a,b})(0, t; \tilde u)\leq M\varepsilon^{2}t^{2k+3},\quad \forall a+b\leq 4.
\end{equation}
In view of \eqref{eq: different treatment}, $(\mathbf{B}_{2})$ should be understood as
\begin{equation}\notag{(\mathbf{B}_{2})}
\begin{cases}
\mathcal{E}(V_{0; a, b})(t; \ut_{*}, \ut^{*})+\mathcal{F}(V_{0; a,b})(0, t; \ut)\leq M\varepsilon^{2}t^{3},\ \gamma\in (1, 1.3],\\
\mathcal{E}(V_{1; a, b})(t; \ut_{*}, \ut^{*})+\mathcal{F}(V_{1; a,b})(0, t; \tilde{u})\leq M\varepsilon^{2}t^{5}, \ \gamma\in \{\frac{7}{5}, \frac{5}{3}\}.
\end{cases}
\end{equation}

\item For $Z\in\{T,\hat X\}$, assume the following pointwise and differentiated bounds:
\begin{equation}\notag{(\mathbf{B}_{\infty})}
\begin{cases}
\|c^{-1}L(\underline{w})\|_{L^{\infty}(\Sigma_{t}^{\ut_{*}, \tilde{u}^{*}})}+\|c^{-1}L(w)\|_{L^{\infty}(\Sigma_{t}^{\ut_{*}, \tilde{u}^{*}})}+\|c^{-1}L(\psi_{2})\|_{L^{\infty}(\Sigma_{t}^{\ut_{*}, \tilde{u}^{*}})}\leq M\varepsilon,\\
\|\hat{X}Z^{\beta}(\psi)\|_{L^{2}(\Sigma_{t}^{\ut_{*}, \tilde{u}^{*}})}\leq M\varepsilon, \psi\in \{\wb, w, \psi_{2}\}, 2\leq |\beta|\leq 3,\\
\|\hat{X}Z^{\beta}(\psi)\|_{L^{\infty}(\Sigma_{t}^{\ut_{*}, \tilde{u}^{*}})}\leq M\varepsilon, \psi\in \{\wb, w, \psi_{2}\}, |\beta|\leq 1,\\
\|T(w)\|_{L^{\infty}(\Sigma_{t}^{\ut_{*}, \tilde{u}^{*}})}+\|T(\psi_{2})\|_{L^{\infty}(\Sigma_{t}^{\ut_{*}, \tilde{u}^{*}})}+\|T(\underline{w})+\frac{2}{\gamma+1}\|_{L^{\infty}(\Sigma_{t}^{\ut_{*}, \tilde{u}^{*}})}\leq M\varepsilon t,\\
\|TZ^{\beta}(\psi)\|_{L^{2}(\Sigma_{t}^{\ut_{*}, \tilde{u}^{*}})}\leq M\varepsilon t, 2\leq |\beta|\leq 3,\\
\|TZ^{\beta}(\psi)\|_{L^{\infty}(\Sigma_{t}^{\ut_{*}, \tilde{u}^{*}})}\leq M\varepsilon t, |\beta|=1.
\end{cases}
\end{equation}

\item In view of \eqref{eq: formulas to control the geometry}, we can assume that the following inequalities hold on the stated time interval:
\begin{equation}{(\mathbf{B}_{\text{diffeomorphism}})}\label{eq: B diff}
\begin{cases}
|\hat{T}^{1}+1|\leq \frac{1}{100},\\
|\frac{\partial \tilde u}{\partial \mathring{u}}+1|\leq \frac{1}{100}, |\frac{\partial \tilde u}{\partial x_{2}}|\leq \frac{1}{100},\\
|\frac{\partial \vartheta}{\partial \mathring{u}}|\leq \frac{1}{100}, |\frac{\partial \vartheta}{\partial x_{2}}-1|\leq \frac{1}{100}.
\end{cases}
\end{equation}
\end{itemize}
The assumption $(\mathbf B_{\text{diffeomorphism}})$ controls the
spatial coordinate map at each fixed time:
\[
(\ut,\vartheta)\longmapsto(\mathring u,x_2),\qquad
(\ut,\vartheta)\in [\ut_{*}, \ut^{*}]\times\mathbb T.
\]
It is a $W^{1,\infty}$ diffeomorphism onto its image with uniform
bounds for its spatial Jacobian and inverse. Time derivatives of the
map are given separately by its generator $L$.

For a positive lower time $\delta$, we use the same energy
$\mathcal E$ and define $\mathcal F(\delta,t;\tilde u)$ by changing
the lower time of the flux integral to $\delta$.
\begin{theorem}\label{theorem: a prior estimate in W_2}
Fix the pair $(k,\gamma)$ in \eqref{eq: different treatment},
$\Ntop=4$, $0<\alpha<1/100$, and $0\leq\delta<T\leq1$ in the
normalization of Section~\ref{subsec: opened fan}.
Suppose a positive-density solution is given in the second
characteristic chart for $\delta\leq t\leq T$ (for $t>0$ if
$\delta=0$). At each fixed positive cutoff, assume that the state,
map to the opened coordinates and normalized geometry
$\hat T,\kappa/t,\slashed g,\Xi$ are of class $C^6$ up to the
lower-time slice, using their extensions to $t=0$ when $\delta=0$.
Their full $C^6$ norms may depend on the cutoffs.
The explicit profile coefficient bounds of
Proposition~\ref{prop: initial energy bounds on Sigma_0} in the
corresponding case and the residual bounds of
Proposition~\ref{prop: Res bounds of Ur_k} are assumed through order
four; we use the Taylor limits only when $\delta=0$.
We prescribe the data on the lower-time slice and the incoming
characteristic as follows.
\begin{itemize}
\item If $\delta=0$, the initial state and geometry are those of
\eqref{eq: data on t=0}, \eqref{eq: formulas to control the geometry}
and \eqref{eq: second characteristic coordinate}, with $\kappa/t\to1$
and the corresponding Taylor and initial geometric bounds. In
particular the initial energy divided by $t^{2k+3}$ is bounded at zero.
\item If $\delta>0$, the initial remainder satisfies
\eqref{eq: positive time data bounds}, the chart has the initial
geometry \eqref{eq: positive time initial geometry} with the
weighted bounds \eqref{eq: positive time initial weights}, and the
initial sound speed lies between $3c_0/4$ and $5c_0/4$, where
$c_0=-(\gamma-1)\tilde u/(\gamma+1)$ on the initial labels.
In particular, the initial energy is bounded by
$C\varepsilon^2\delta^{2k+3}$.
\item For both lower-time choices, the given incoming flux satisfies
\[
\mathcal F(V_{k;a,b})(\delta,t;\tilde u_*)
\leq C\varepsilon^2t^{2k+3},\qquad a+b\leq4.
\]
In the construction it is supplied by the known nonvacuum part of
the fan on the fixed overlap, before the present estimate is applied.
\end{itemize}
There are constants $M_0$ and, for each fixed $M\geq M_0$,
$\varepsilon_0>0$ such that, for $0<\varepsilon\leq\varepsilon_0$,
if $(\mathbf B_2)$, $(\mathbf B_\infty)$ and
$(\mathbf B_{\rm diffeomorphism})$ hold on the stated interval
(with fluxes integrated from $\delta$), then the same inequalities
hold with $M$ replaced by $M/2$ and $1/100$ by $1/200$.
The constants depend only on the fixed parameters and the constants
in the assumed bounds (and, for $\varepsilon_0$, on $M$); in particular
they are independent of $\alpha$, $\delta$ and $T$.
\end{theorem}
Sections~\ref{sec: Preparations for the energy estimate}--\ref{sec: closing B infinity}
prove the estimates with both lower-time choices, including the initial
geometric, energy and pointwise terms when $\delta>0$.
In Proposition~\ref{prop: positive time construction}, the restriction
$\delta\leq d_0\alpha$ verifies the initial sound-speed comparison;
it is not an additional restriction in the a priori estimates.

\paragraph{Construction up to vacuum.}
The next proposition records the construction up to vacuum used in
the proof of the main theorem. We use our earlier construction
\cite{JiaLuo} only for the solution away from vacuum; no smallness
threshold depending on the vacuum cutoff is used. Work in the normalization
$t^*=1$ of Section~\ref{subsec: opened fan}, using the opened coordinates
$(t,\urr,x_2)$. Recall that the initial sound-speed profile is
\[
c_0(\urr,x_2)=\frac{\gamma-1}{\gamma+1}
\left[\urr-\left(v_r^1-\frac{2c_r}{\gamma-1}\right)(0,x_2)\right].
\]
Here $c_0$ is held fixed as $t$ varies. Let $U^a$ denote the right rarefaction solution supplied by the
nonvacuum result discussed in Section~\ref{subsec: away from vacuum review},
restricted, when necessary, to a slightly smaller fixed region.
Its domain contains $3\mathbf c_*/5\leq c_0\leq19\mathbf c_*/20$
for $0\leq t\leq1$ and $x_2\in\mathbb T$, after reducing the fixed
$\varepsilon_0$.
Constants in the uniform estimates below may depend on this fixed
region, but not on the cutoffs. Local Cauchy lifespans may depend
on both cutoffs.

\begin{proposition}[Construction up to vacuum]\label{prop: positive time construction}
Fix the data, regularity orders and outer solution $U^a$ above, and
choose $(k,\gamma)$ as in \eqref{eq: different treatment}.
There are $d_0,\varepsilon_0,C>0$, independent of $\delta,\alpha$,
such that, whenever
\[
0\leq\varepsilon\leq\varepsilon_0,\qquad
0<\alpha<\frac1{100},\qquad 0<\delta\leq d_0\alpha,
\]
the following assertions hold.

\smallskip\noindent\textup{(i)} The data
\eqref{eq: positive time data} and \eqref{eq: positive time initial geometry}
give a positive-density $C^6$ solution $U^{\delta,\alpha}$ for
$\delta\leq t\leq1$ and
$\tilde u_*\leq\tilde u\leq\tilde u^*$, $\vartheta\in\mathbb T$.
It agrees with $U^a$ on a fixed open overlap near $\tilde u=\tilde u_*$
and satisfies the three bootstrap bounds of
Theorem~\ref{theorem: a prior estimate in W_2} with a fixed constant,
with fluxes integrated from $\delta$. In particular, on each
near-vacuum slice $\Sigma_t$,
\begin{equation}\label{eq: positive time remainder}
\|U^{\delta,\alpha}(t)-\Ur_k(t)\|_{H^4(\Sigma_t)}
+\|U^{\delta,\alpha}(t)-\Ur_k(t)\|_{W^{2,\infty}(\Sigma_t)}
\leq C\varepsilon t^{k+1}.
\end{equation}
These norms use the fixed derivatives $\Tr,\Xr$ and the slice measure.

\smallskip\noindent\textup{(ii)} A single sequence with
$(\delta,\alpha)\to(0,0)$ gives a limiting solution $U$ of class $C^2$
in the positive-time gas interior, agreeing with $U^a$ on the overlap.
It satisfies \eqref{eq: positive time remainder}, with $U$ in place
of $U^{\delta,\alpha}$, and the spatial
geometric and field bounds of Sections
\ref{sec: Preparations for the energy estimate}--\ref{sec: closing B infinity}
on every positive-time slice, with higher derivatives understood weakly.
After the identification \eqref{eq: fixed strip identification},
the maps to $(\urr,x_2)$ and the pulled-back states converge in
$C([\tau,1];C^1(Q))$ for every $\tau>0$, where
$Q=[\tilde u_*,0]\times\mathbb T$.
The limiting spatial diffeomorphism $Y(t)$ extends in $C^1(Q)$ to
\[
Y(0,q,\vartheta)=
\left(\left(v_r^1-\frac{2c_r}{\gamma-1}\right)(0,\vartheta)-q,
\vartheta\right),
\]
with $\|Y(t)-Y(0)\|_{C^1(Q)}\leq C\varepsilon t$.
Moreover, $|L(c)|+|L(v)|\leq C\varepsilon c$ in the gas, and
\eqref{eq: limit sound speed through vacuum} identifies $q=0$ as its
vacuum boundary.
\end{proposition}

\subsection{\texorpdfstring{Proof outline for Theorem~\ref{theorem: stability of RVR}}{Proof outline for the main theorem}}\label{subsec: proof outline main theorem}
The proof of Theorem~\ref{theorem: stability of RVR} is completed in
Section~\ref{subsec: proof main theorem}. Sections~3--7 establish the
initial profiles, boundary matching, geometric estimates, energy
estimates, and pointwise estimates. We summarize how these results enter
the construction.

\begin{itemize}
\item \textbf{(A) Initial and exterior data.}
Propositions~\ref{prop: initial energy bounds on Sigma_0}
and~\ref{prop: matching of U and Ur_k along C_0} provide the initial
profiles and the boundary data for the exterior construction. The
near-vacuum construction uses the positive-time data and initial geometry
in Proposition~\ref{prop: positive time construction}.

\begin{center}
\begin{tikzpicture}[scale=1.1, >=stealth]
\draw[black!80, line width=0.8pt] (0,0) 
    -- (7.2,0) 
    -- (7.8,-0.3) 
    -- (0.6,-0.3) 
    -- cycle;
\node at (3.9, -0.7) {$\Sigma_0$};
\foreach \x in {0.3,1.1,1.9,2.7,3.5,4.3,5.1,5.9,6.7}{
    \fill[gray!15, opacity=0.45] 
        (\x,0) -- (\x+0.6,-0.3) -- (\x+0.6,2.1) -- (\x,2.4) -- cycle;
    \draw[dashed, gray!70, line width=0.55pt] 
        (\x,0) -- (\x+0.6,-0.3) -- (\x+0.6,2.1) -- (\x,2.4) -- cycle;
}
\fill[blue!15, opacity=0.5]
    plot[smooth] coordinates {
        (5.88,0) (5.90,-0.03) (5.92,-0.06) (5.94,-0.075) (5.96,-0.09) (6.00,-0.12) (6.06,-0.15) (6.12,-0.18) (6.20,-0.21) (6.24,-0.225) (6.28,-0.24) (6.38,-0.27) (6.48,-0.3)
    }
    -- (6.24,2.1)
    plot[smooth] coordinates {
        (6.24,2.1) (6.14,2.13) (6.04,2.16) (6.00,2.175) (5.96,2.19) (5.88,2.22) (5.82,2.25) (5.76,2.28) (5.72,2.31) (5.70,2.325) (5.68,2.34) (5.66,2.37) (5.64,2.4)
    }
    -- (5.88,0);
\draw[blue!80, line width=0.7pt] plot[smooth] coordinates {(5.88,0) (5.90,-0.03) (5.92,-0.06) (5.94,-0.075) (5.96,-0.09) (6.00,-0.12) (6.06,-0.15) (6.12,-0.18) (6.20,-0.21) (6.24,-0.225) (6.28,-0.24) (6.38,-0.27) (6.48,-0.3)};
\draw[blue!80, line width=0.7pt] (6.48,-0.3) -- (6.24,2.1);
\draw[blue!80, line width=0.7pt] plot[smooth] coordinates {(6.24,2.1) (6.14,2.13) (6.04,2.16) (6.00,2.175) (5.96,2.19) (5.88,2.22) (5.82,2.25) (5.76,2.28) (5.72,2.31) (5.70,2.325) (5.68,2.34) (5.66,2.37) (5.64,2.4)};
\draw[blue!80, line width=0.7pt] (5.64,2.4) -- (5.88,0);
\node[blue!80, above] at (6.0,2.5) {\small $C_0$};
\draw[->, black!80, line width=0.6pt] (3.3,-1.5) -- (3.3,-1.0);
\node at (3.3,-0.85) {$t$};
\draw[->, black!80, line width=0.6pt] (3.3,-1.5) -- (4.3,-1.5);
\node at (4.5,-1.5) {$\mathring{u}$};
\draw[->, black!80, line width=0.6pt] (3.3,-1.5) -- (3.8,-1.8);
\node at (4.0,-1.9) {$x_2$};
\end{tikzpicture}
\end{center}

\item \textbf{(B) Construction away from vacuum.}
The classical exterior flows are obtained from the symmetric-hyperbolic
Cauchy theory. Our earlier nonvacuum rarefaction result
\cite{JiaLuo} supplies a fixed region away from vacuum, together with
the boundary jets needed for matching. Theorem~\ref{theorem: a prior
estimate in W_1} gives the corresponding estimates there.

\begin{center}
\begin{tikzpicture}[scale=1.1, >=stealth]
\draw[black!80, line width=0.8pt] (0,0) 
    -- (7.2,0) 
    -- (7.8,-0.3) 
    -- (0.6,-0.3) 
    -- cycle;
\node at (3.9, -0.7) {$\Sigma_0$};
\foreach \x in {0.3,1.1,1.9,2.7,3.5,4.3,5.1,5.9,6.7}{
    \fill[gray!15, opacity=0.45] 
        (\x,0) -- (\x+0.6,-0.3) -- (\x+0.6,2.1) -- (\x,2.4) -- cycle;
    \draw[dashed, gray!70, line width=0.55pt] 
        (\x,0) -- (\x+0.6,-0.3) -- (\x+0.6,2.1) -- (\x,2.4) -- cycle;
}
\foreach \xb in {2.52,3.12,3.72,4.32,4.92,5.52}{
\fill[blue!12, opacity=0.45]
    plot[smooth] coordinates {
        (\xb,0) (\xb+0.02,-0.03) (\xb+0.04,-0.06) (\xb+0.06,-0.075) (\xb+0.08,-0.09) (\xb+0.12,-0.12) (\xb+0.18,-0.15) (\xb+0.24,-0.18) (\xb+0.32,-0.21) (\xb+0.36,-0.225) (\xb+0.40,-0.24) (\xb+0.50,-0.27) (\xb+0.6,-0.3)
    }
    -- (\xb+0.36,2.1)
    plot[smooth] coordinates {
        (\xb+0.36,2.1) (\xb+0.26,2.13) (\xb+0.16,2.16) (\xb+0.12,2.175) (\xb+0.08,2.19) (\xb+0.0,2.22) (\xb-0.06,2.25) (\xb-0.12,2.28) (\xb-0.16,2.31) (\xb-0.18,2.325) (\xb-0.20,2.34) (\xb-0.22,2.37) (\xb-0.24,2.4)
    }
    -- (\xb,0);
\draw[dashed, blue!70, line width=0.55pt] plot[smooth] coordinates {
    (\xb,0) (\xb+0.02,-0.03) (\xb+0.04,-0.06) (\xb+0.06,-0.075) (\xb+0.08,-0.09) (\xb+0.12,-0.12) (\xb+0.18,-0.15) (\xb+0.24,-0.18) (\xb+0.32,-0.21) (\xb+0.36,-0.225) (\xb+0.40,-0.24) (\xb+0.50,-0.27) (\xb+0.6,-0.3)
};
\draw[dashed, blue!70, line width=0.55pt] (\xb+0.6,-0.3) -- (\xb+0.36,2.1);
\draw[dashed, blue!70, line width=0.55pt] plot[smooth] coordinates {
    (\xb+0.36,2.1) (\xb+0.26,2.13) (\xb+0.16,2.16) (\xb+0.12,2.175) (\xb+0.08,2.19) (\xb+0.0,2.22) (\xb-0.06,2.25) (\xb-0.12,2.28) (\xb-0.16,2.31) (\xb-0.18,2.325) (\xb-0.20,2.34) (\xb-0.22,2.37) (\xb-0.24,2.4)
};
\draw[dashed, blue!70, line width=0.55pt] (\xb-0.24,2.4) -- (\xb,0);
}
\fill[blue!15, opacity=0.5]
    plot[smooth] coordinates {
        (5.88,0) (5.90,-0.03) (5.92,-0.06) (5.94,-0.075) (5.96,-0.09) (6.00,-0.12) (6.06,-0.15) (6.12,-0.18) (6.20,-0.21) (6.24,-0.225) (6.28,-0.24) (6.38,-0.27) (6.48,-0.3)
    }
    -- (6.24,2.1)
    plot[smooth] coordinates {
        (6.24,2.1) (6.14,2.13) (6.04,2.16) (6.00,2.175) (5.96,2.19) (5.88,2.22) (5.82,2.25) (5.76,2.28) (5.72,2.31) (5.70,2.325) (5.68,2.34) (5.66,2.37) (5.64,2.4)
    }
    -- (5.88,0);
\draw[blue!80, line width=0.7pt] plot[smooth] coordinates {(5.88,0) (5.90,-0.03) (5.92,-0.06) (5.94,-0.075) (5.96,-0.09) (6.00,-0.12) (6.06,-0.15) (6.12,-0.18) (6.20,-0.21) (6.24,-0.225) (6.28,-0.24) (6.38,-0.27) (6.48,-0.3)};
\draw[blue!80, line width=0.7pt] (6.48,-0.3) -- (6.24,2.1);
\draw[blue!80, line width=0.7pt] plot[smooth] coordinates {(6.24,2.1) (6.14,2.13) (6.04,2.16) (6.00,2.175) (5.96,2.19) (5.88,2.22) (5.82,2.25) (5.76,2.28) (5.72,2.31) (5.70,2.325) (5.68,2.34) (5.66,2.37) (5.64,2.4)};
\draw[blue!80, line width=0.7pt] (5.64,2.4) -- (5.88,0);
\node[blue!80, above] at (6.0,2.5) {\small $C_0$};
\draw[->, black!80, line width=0.6pt] (3.3,-1.5) -- (3.3,-1.0);
\node at (3.3,-0.85) {$t$};
\draw[->, black!80, line width=0.6pt] (3.3,-1.5) -- (4.3,-1.5);
\node at (4.5,-1.5) {$\mathring{u}$};
\draw[->, black!80, line width=0.6pt] (3.3,-1.5) -- (3.8,-1.8);
\node at (4.0,-1.9) {$x_2$};
\end{tikzpicture}
\end{center}

\item \textbf{(C) Transfer through the overlap.}
The fixed overlap lies in both the known nonvacuum region and the
near-vacuum characteristic chart. The known solution supplies the
incoming flux, and Theorem~\ref{theorem: a prior estimate in W_2}
controls the positive-time construction in Proposition~\ref{prop: positive time construction}.

\begin{center}
\begin{tikzpicture}[scale=1.1, >=stealth]
\draw[black!80, line width=0.8pt] (0,0) 
    -- (7.2,0) 
    -- (7.8,-0.3) 
    -- (0.6,-0.3) 
    -- cycle;
\node at (3.9, -0.7) {$\Sigma_0$};
\foreach \x in {0.3,1.1,1.9,2.7,3.5,4.3,5.1,5.9,6.7}{
    \fill[gray!15, opacity=0.45] 
        (\x,0) -- (\x+0.6,-0.3) -- (\x+0.6,2.1) -- (\x,2.4) -- cycle;
    \draw[dashed, gray!70, line width=0.55pt] 
        (\x,0) -- (\x+0.6,-0.3) -- (\x+0.6,2.1) -- (\x,2.4) -- cycle;
}
\foreach \xb in {1.44,2.04,2.64,3.24,3.96}{
\fill[red!12, opacity=0.45]
    plot[smooth] coordinates {
        (\xb,0) (\xb+0.09,-0.03) (\xb+0.18,-0.06) (\xb+0.22,-0.075) (\xb+0.26,-0.09) (\xb+0.33,-0.12) (\xb+0.39,-0.15) (\xb+0.45,-0.18) (\xb+0.50,-0.21) (\xb+0.52,-0.225) (\xb+0.54,-0.24) (\xb+0.57,-0.27) (\xb+0.6,-0.3)
    }
    -- (\xb+0.89,2.1)
    plot[smooth] coordinates {
        (\xb+0.89,2.1) (\xb+0.86,2.13) (\xb+0.83,2.16) (\xb+0.81,2.175) (\xb+0.79,2.19) (\xb+0.74,2.22) (\xb+0.68,2.25) (\xb+0.62,2.28) (\xb+0.55,2.31) (\xb+0.51,2.325) (\xb+0.47,2.34) (\xb+0.38,2.37) (\xb+0.29,2.4)
    }
    -- (\xb,0);
\draw[dashed, red!70, line width=0.55pt] plot[smooth] coordinates {
    (\xb,0) (\xb+0.09,-0.03) (\xb+0.18,-0.06) (\xb+0.22,-0.075) (\xb+0.26,-0.09) (\xb+0.33,-0.12) (\xb+0.39,-0.15) (\xb+0.45,-0.18) (\xb+0.50,-0.21) (\xb+0.52,-0.225) (\xb+0.54,-0.24) (\xb+0.57,-0.27) (\xb+0.6,-0.3)
};
\draw[dashed, red!70, line width=0.55pt] (\xb+0.6,-0.3) -- (\xb+0.89,2.1);
\draw[dashed, red!70, line width=0.55pt] plot[smooth] coordinates {
    (\xb+0.89,2.1) (\xb+0.86,2.13) (\xb+0.83,2.16) (\xb+0.81,2.175) (\xb+0.79,2.19) (\xb+0.74,2.22) (\xb+0.68,2.25) (\xb+0.62,2.28) (\xb+0.55,2.31) (\xb+0.51,2.325) (\xb+0.47,2.34) (\xb+0.38,2.37) (\xb+0.29,2.4)
};
\draw[dashed, red!70, line width=0.55pt] (\xb+0.29,2.4) -- (\xb,0);
}
\fill[red!15, opacity=0.5]
    plot[smooth] coordinates {
        (0.84,0) (0.93,-0.03) (1.02,-0.06) (1.06,-0.075) (1.10,-0.09) (1.17,-0.12) (1.23,-0.15) (1.29,-0.18) (1.34,-0.21) (1.36,-0.225) (1.38,-0.24) (1.41,-0.27) (1.44,-0.3)
    }
    -- (1.73,2.1)
    plot[smooth] coordinates {
        (1.73,2.1) (1.70,2.13) (1.67,2.16) (1.65,2.175) (1.63,2.19) (1.58,2.22) (1.52,2.25) (1.46,2.28) (1.39,2.31) (1.35,2.325) (1.31,2.34) (1.22,2.37) (1.13,2.4)
    }
    -- (0.84,0);
\draw[dashed, red!80, line width=0.7pt] plot[smooth] coordinates {(0.84,0) (0.93,-0.03) (1.02,-0.06) (1.06,-0.075) (1.10,-0.09) (1.17,-0.12) (1.23,-0.15) (1.29,-0.18) (1.34,-0.21) (1.36,-0.225) (1.38,-0.24) (1.41,-0.27) (1.44,-0.3)};
\draw[dashed, red!80, line width=0.7pt] (1.44,-0.3) -- (1.73,2.1);
\draw[dashed, red!80, line width=0.7pt] plot[smooth] coordinates {(1.73,2.1) (1.70,2.13) (1.67,2.16) (1.65,2.175) (1.63,2.19) (1.58,2.22) (1.52,2.25) (1.46,2.28) (1.39,2.31) (1.35,2.325) (1.31,2.34) (1.22,2.37) (1.13,2.4)};
\draw[dashed, red!80, line width=0.7pt] (1.13,2.4) -- (0.84,0);
\node[red!80, above] at (1.25,2.5) {\small $C_{\ut^*}$};
\foreach \xb in {2.52,3.12,3.72,4.32,4.92,5.52}{
\fill[blue!12, opacity=0.45]
    plot[smooth] coordinates {
        (\xb,0) (\xb+0.02,-0.03) (\xb+0.04,-0.06) (\xb+0.06,-0.075) (\xb+0.08,-0.09) (\xb+0.12,-0.12) (\xb+0.18,-0.15) (\xb+0.24,-0.18) (\xb+0.32,-0.21) (\xb+0.36,-0.225) (\xb+0.40,-0.24) (\xb+0.50,-0.27) (\xb+0.6,-0.3)
    }
    -- (\xb+0.36,2.1)
    plot[smooth] coordinates {
        (\xb+0.36,2.1) (\xb+0.26,2.13) (\xb+0.16,2.16) (\xb+0.12,2.175) (\xb+0.08,2.19) (\xb+0.0,2.22) (\xb-0.06,2.25) (\xb-0.12,2.28) (\xb-0.16,2.31) (\xb-0.18,2.325) (\xb-0.20,2.34) (\xb-0.22,2.37) (\xb-0.24,2.4)
    }
    -- (\xb,0);
\draw[dashed, blue!70, line width=0.55pt] plot[smooth] coordinates {
    (\xb,0) (\xb+0.02,-0.03) (\xb+0.04,-0.06) (\xb+0.06,-0.075) (\xb+0.08,-0.09) (\xb+0.12,-0.12) (\xb+0.18,-0.15) (\xb+0.24,-0.18) (\xb+0.32,-0.21) (\xb+0.36,-0.225) (\xb+0.40,-0.24) (\xb+0.50,-0.27) (\xb+0.6,-0.3)
};
\draw[dashed, blue!70, line width=0.55pt] (\xb+0.6,-0.3) -- (\xb+0.36,2.1);
\draw[dashed, blue!70, line width=0.55pt] plot[smooth] coordinates {
    (\xb+0.36,2.1) (\xb+0.26,2.13) (\xb+0.16,2.16) (\xb+0.12,2.175) (\xb+0.08,2.19) (\xb+0.0,2.22) (\xb-0.06,2.25) (\xb-0.12,2.28) (\xb-0.16,2.31) (\xb-0.18,2.325) (\xb-0.20,2.34) (\xb-0.22,2.37) (\xb-0.24,2.4)
};
\draw[dashed, blue!70, line width=0.55pt] (\xb-0.24,2.4) -- (\xb,0);
}
\fill[blue!15, opacity=0.5]
    plot[smooth] coordinates {
        (5.88,0) (5.90,-0.03) (5.92,-0.06) (5.94,-0.075) (5.96,-0.09) (6.00,-0.12) (6.06,-0.15) (6.12,-0.18) (6.20,-0.21) (6.24,-0.225) (6.28,-0.24) (6.38,-0.27) (6.48,-0.3)
    }
    -- (6.24,2.1)
    plot[smooth] coordinates {
        (6.24,2.1) (6.14,2.13) (6.04,2.16) (6.00,2.175) (5.96,2.19) (5.88,2.22) (5.82,2.25) (5.76,2.28) (5.72,2.31) (5.70,2.325) (5.68,2.34) (5.66,2.37) (5.64,2.4)
    }
    -- (5.88,0);
\draw[blue!80, line width=0.7pt] plot[smooth] coordinates {(5.88,0) (5.90,-0.03) (5.92,-0.06) (5.94,-0.075) (5.96,-0.09) (6.00,-0.12) (6.06,-0.15) (6.12,-0.18) (6.20,-0.21) (6.24,-0.225) (6.28,-0.24) (6.38,-0.27) (6.48,-0.3)};
\draw[blue!80, line width=0.7pt] (6.48,-0.3) -- (6.24,2.1);
\draw[blue!80, line width=0.7pt] plot[smooth] coordinates {(6.24,2.1) (6.14,2.13) (6.04,2.16) (6.00,2.175) (5.96,2.19) (5.88,2.22) (5.82,2.25) (5.76,2.28) (5.72,2.31) (5.70,2.325) (5.68,2.34) (5.66,2.37) (5.64,2.4)};
\draw[blue!80, line width=0.7pt] (5.64,2.4) -- (5.88,0);
\node[blue!80, above] at (6.0,2.5) {\small $C_0$};
\draw[->, black!80, line width=0.6pt] (3.3,-1.5) -- (3.3,-1.0);
\node at (3.3,-0.85) {$t$};
\draw[->, black!80, line width=0.6pt] (3.3,-1.5) -- (4.3,-1.5);
\node at (4.5,-1.5) {$\mathring{u}$};
\draw[->, black!80, line width=0.6pt] (3.3,-1.5) -- (3.8,-1.8);
\node at (4.0,-1.9) {$x_2$};
\end{tikzpicture}
\end{center}

\item \textbf{(D) Removal of both cutoffs.}
Choose $\alpha_n\rightarrow0$ and $\delta_n\rightarrow0$ with
$0<\delta_n\leq d_0\alpha_n$. Proposition~\ref{prop: positive time construction}
gives one common limit on the entire open fan. Section~\ref{subsec: common limit}
identifies its geometry and passes the estimates of Sections~5--7 to
every positive-time slice. The remaining parts of Section~\ref{subsec: proof main theorem}
verify the boundary graphs, the reflected left fan, the entropy solution,
and the additional spacetime regularity.

\begin{center}
\begin{tikzpicture}[scale=1.1, >=stealth]
\draw[black!80, line width=0.8pt] (0,0) 
    -- (7.2,0) 
    -- (7.8,-0.3) 
    -- (0.6,-0.3) 
    -- cycle;
\node at (3.9, -0.7) {$\Sigma_0$};
\foreach \x in {0.3,1.1,1.9,2.7,3.5,4.3,5.1,5.9,6.7}{
    \fill[gray!15, opacity=0.45] 
        (\x,0) -- (\x+0.6,-0.3) -- (\x+0.6,2.1) -- (\x,2.4) -- cycle;
    \draw[dashed, gray!70, line width=0.55pt] 
        (\x,0) -- (\x+0.6,-0.3) -- (\x+0.6,2.1) -- (\x,2.4) -- cycle;
}
\foreach \xb in {1.44,2.04,2.64,3.24,3.96}{
\fill[red!12, opacity=0.45]
    plot[smooth] coordinates {
        (\xb,0) (\xb+0.09,-0.03) (\xb+0.18,-0.06) (\xb+0.22,-0.075) (\xb+0.26,-0.09) (\xb+0.33,-0.12) (\xb+0.39,-0.15) (\xb+0.45,-0.18) (\xb+0.50,-0.21) (\xb+0.52,-0.225) (\xb+0.54,-0.24) (\xb+0.57,-0.27) (\xb+0.6,-0.3)
    }
    -- (\xb+0.89,2.1)
    plot[smooth] coordinates {
        (\xb+0.89,2.1) (\xb+0.86,2.13) (\xb+0.83,2.16) (\xb+0.81,2.175) (\xb+0.79,2.19) (\xb+0.74,2.22) (\xb+0.68,2.25) (\xb+0.62,2.28) (\xb+0.55,2.31) (\xb+0.51,2.325) (\xb+0.47,2.34) (\xb+0.38,2.37) (\xb+0.29,2.4)
    }
    -- (\xb,0);
\draw[dashed, red!70, line width=0.55pt] plot[smooth] coordinates {
    (\xb,0) (\xb+0.09,-0.03) (\xb+0.18,-0.06) (\xb+0.22,-0.075) (\xb+0.26,-0.09) (\xb+0.33,-0.12) (\xb+0.39,-0.15) (\xb+0.45,-0.18) (\xb+0.50,-0.21) (\xb+0.52,-0.225) (\xb+0.54,-0.24) (\xb+0.57,-0.27) (\xb+0.6,-0.3)
};
\draw[dashed, red!70, line width=0.55pt] (\xb+0.6,-0.3) -- (\xb+0.89,2.1);
\draw[dashed, red!70, line width=0.55pt] plot[smooth] coordinates {
    (\xb+0.89,2.1) (\xb+0.86,2.13) (\xb+0.83,2.16) (\xb+0.81,2.175) (\xb+0.79,2.19) (\xb+0.74,2.22) (\xb+0.68,2.25) (\xb+0.62,2.28) (\xb+0.55,2.31) (\xb+0.51,2.325) (\xb+0.47,2.34) (\xb+0.38,2.37) (\xb+0.29,2.4)
};
\draw[dashed, red!70, line width=0.55pt] (\xb+0.29,2.4) -- (\xb,0);
}
\fill[red!15, opacity=0.5]
    plot[smooth] coordinates {
        (0.84,0) (0.93,-0.03) (1.02,-0.06) (1.06,-0.075) (1.10,-0.09) (1.17,-0.12) (1.23,-0.15) (1.29,-0.18) (1.34,-0.21) (1.36,-0.225) (1.38,-0.24) (1.41,-0.27) (1.44,-0.3)
    }
    -- (1.73,2.1)
    plot[smooth] coordinates {
        (1.73,2.1) (1.70,2.13) (1.67,2.16) (1.65,2.175) (1.63,2.19) (1.58,2.22) (1.52,2.25) (1.46,2.28) (1.39,2.31) (1.35,2.325) (1.31,2.34) (1.22,2.37) (1.13,2.4)
    }
    -- (0.84,0);
\draw[red!80, line width=0.7pt] plot[smooth] coordinates {(0.84,0) (0.93,-0.03) (1.02,-0.06) (1.06,-0.075) (1.10,-0.09) (1.17,-0.12) (1.23,-0.15) (1.29,-0.18) (1.34,-0.21) (1.36,-0.225) (1.38,-0.24) (1.41,-0.27) (1.44,-0.3)};
\draw[red!80, line width=0.7pt] (1.44,-0.3) -- (1.73,2.1);
\draw[red!80, line width=0.7pt] plot[smooth] coordinates {(1.73,2.1) (1.70,2.13) (1.67,2.16) (1.65,2.175) (1.63,2.19) (1.58,2.22) (1.52,2.25) (1.46,2.28) (1.39,2.31) (1.35,2.325) (1.31,2.34) (1.22,2.37) (1.13,2.4)};
\draw[red!80, line width=0.7pt] (1.13,2.4) -- (0.84,0);
\node[red!80, above] at (1.25,2.5) {\small $C_{\textbf{vacuum}}$};
\foreach \xb in {2.52,3.12,3.72,4.32,4.92,5.52}{
\fill[blue!12, opacity=0.45]
    plot[smooth] coordinates {
        (\xb,0) (\xb+0.02,-0.03) (\xb+0.04,-0.06) (\xb+0.06,-0.075) (\xb+0.08,-0.09) (\xb+0.12,-0.12) (\xb+0.18,-0.15) (\xb+0.24,-0.18) (\xb+0.32,-0.21) (\xb+0.36,-0.225) (\xb+0.40,-0.24) (\xb+0.50,-0.27) (\xb+0.6,-0.3)
    }
    -- (\xb+0.36,2.1)
    plot[smooth] coordinates {
        (\xb+0.36,2.1) (\xb+0.26,2.13) (\xb+0.16,2.16) (\xb+0.12,2.175) (\xb+0.08,2.19) (\xb+0.0,2.22) (\xb-0.06,2.25) (\xb-0.12,2.28) (\xb-0.16,2.31) (\xb-0.18,2.325) (\xb-0.20,2.34) (\xb-0.22,2.37) (\xb-0.24,2.4)
    }
    -- (\xb,0);
\draw[dashed, blue!70, line width=0.55pt] plot[smooth] coordinates {
    (\xb,0) (\xb+0.02,-0.03) (\xb+0.04,-0.06) (\xb+0.06,-0.075) (\xb+0.08,-0.09) (\xb+0.12,-0.12) (\xb+0.18,-0.15) (\xb+0.24,-0.18) (\xb+0.32,-0.21) (\xb+0.36,-0.225) (\xb+0.40,-0.24) (\xb+0.50,-0.27) (\xb+0.6,-0.3)
};
\draw[dashed, blue!70, line width=0.55pt] (\xb+0.6,-0.3) -- (\xb+0.36,2.1);
\draw[dashed, blue!70, line width=0.55pt] plot[smooth] coordinates {
    (\xb+0.36,2.1) (\xb+0.26,2.13) (\xb+0.16,2.16) (\xb+0.12,2.175) (\xb+0.08,2.19) (\xb+0.0,2.22) (\xb-0.06,2.25) (\xb-0.12,2.28) (\xb-0.16,2.31) (\xb-0.18,2.325) (\xb-0.20,2.34) (\xb-0.22,2.37) (\xb-0.24,2.4)
};
\draw[dashed, blue!70, line width=0.55pt] (\xb-0.24,2.4) -- (\xb,0);
}
\fill[blue!15, opacity=0.5]
    plot[smooth] coordinates {
        (5.88,0) (5.90,-0.03) (5.92,-0.06) (5.94,-0.075) (5.96,-0.09) (6.00,-0.12) (6.06,-0.15) (6.12,-0.18) (6.20,-0.21) (6.24,-0.225) (6.28,-0.24) (6.38,-0.27) (6.48,-0.3)
    }
    -- (6.24,2.1)
    plot[smooth] coordinates {
        (6.24,2.1) (6.14,2.13) (6.04,2.16) (6.00,2.175) (5.96,2.19) (5.88,2.22) (5.82,2.25) (5.76,2.28) (5.72,2.31) (5.70,2.325) (5.68,2.34) (5.66,2.37) (5.64,2.4)
    }
    -- (5.88,0);
\draw[blue!80, line width=0.7pt] plot[smooth] coordinates {(5.88,0) (5.90,-0.03) (5.92,-0.06) (5.94,-0.075) (5.96,-0.09) (6.00,-0.12) (6.06,-0.15) (6.12,-0.18) (6.20,-0.21) (6.24,-0.225) (6.28,-0.24) (6.38,-0.27) (6.48,-0.3)};
\draw[blue!80, line width=0.7pt] (6.48,-0.3) -- (6.24,2.1);
\draw[blue!80, line width=0.7pt] plot[smooth] coordinates {(6.24,2.1) (6.14,2.13) (6.04,2.16) (6.00,2.175) (5.96,2.19) (5.88,2.22) (5.82,2.25) (5.76,2.28) (5.72,2.31) (5.70,2.325) (5.68,2.34) (5.66,2.37) (5.64,2.4)};
\draw[blue!80, line width=0.7pt] (5.64,2.4) -- (5.88,0);
\node[blue!80, above] at (6.0,2.5) {\small $C_0$};
\draw[->, black!80, line width=0.6pt] (3.3,-1.5) -- (3.3,-1.0);
\node at (3.3,-0.85) {$t$};
\draw[->, black!80, line width=0.6pt] (3.3,-1.5) -- (4.3,-1.5);
\node at (4.5,-1.5) {$\mathring{u}$};
\draw[->, black!80, line width=0.6pt] (3.3,-1.5) -- (3.8,-1.8);
\node at (4.0,-1.9) {$x_2$};
\end{tikzpicture}
\end{center}
\end{itemize}
The double-cutoff construction itself is described in
Section~\ref{subsec: positive time construction}; its fixed-$x_2$ sections
are shown there. The proof outline above records the order of the
arguments without repeating their estimates.

\subsection{Notation and conventions}\label{subsec: notation conventions}
We collect the conventions used throughout the proof. In particular,
the four coordinate systems describe different parametrizations, and
the two characteristic charts have separately constructed geometric
fields. The same letters are reused only with the chart understood.
The rules below apply throughout the paper, including schematic estimates.
A local specialization (for example, an initial-slice value or a choice
of characteristic chart) is stated where it is used.

\subsubsection{Unknowns and initial traces}
Throughout the paper,
$U=(\frac{2c}{\gamma-1},v^1,v^2)^T$ is the normalized unknown in
the gas regions. The same convention applies to the exterior
solutions $U_r,U_l$. The conservative variables are written
$(\rho,m)$, with $m=\rho v$ in the gas and $(\rho,m)=(0,0)$ in vacuum.
Bold letters denote constant normalized reference states; in the
right-hand problem after normalization,
$\mathbf U_*=(\frac{2\mathbf c_*}{\gamma-1},0,0)^T$.
The sound speed is $c=\sqrt{p'(\rho)}=\sqrt{k\gamma}\,\rho^{(\gamma-1)/2}$.

The profile notation in the main theorem is
$P_{r,m}=\Vr_m$ and $P_{l,m}=\underline{\Vr}_m$, with $m=0,1$;
the first-order profiles in its conclusions are used only for
$\gamma=7/5,5/3$. These profiles are distinct from the generic
polynomial notation $P(M)$ used in estimates.
The superscript $\sharp$ means evaluation at $x_1=t\urr$. In particular, $U^\sharp$ is
the normalized gas state in the opened coordinates.
The vectors $F$ and $s$ in its preparation denote the four boundary
graphs and their planar speeds; later occurrences of $F$ as a source
term are defined locally.

We use the singular phrase \emph{the vacuum state} for the zero
conservative state $(\rho,m)=(0,0)$.
A \emph{vacuum region} is a spacetime region occupied by this state,
and a \emph{vacuum boundary} is an interface separating it from gas.
The two vacuum boundaries in the theorem adjoin the same conservative
state. Each gas-side velocity trace is determined by its corresponding
fan. The phrase \emph{away from vacuum}
means that the density, equivalently the sound speed for fixed
$k,\gamma$, has a strictly positive lower bound on the region under
discussion. Within a fan, $\rho>0$ and $\rho$ tends to zero at its
vacuum boundary.

The two-argument expression $U_r(0,x_2)$ denotes the one-sided spatial
trace of the initial data at $x_1=0$, equivalently
$U_r(t=0,x_1=0,x_2)$. The same convention applies to $c_r,v_r^i,w_r$
and to the left data. Derivatives are taken before this trace:
for example,
\[
(\partial_{x_1}(U_r))(0,x_2)
:=(\partial_{x_1}(U_r))(t=0,x_1=0,x_2).
\]
Likewise, $L(w_r)(0,x_2)$ means the derivative of the exterior solution
along the generator of $C_0$, followed by evaluation at its initial
boundary. On $C_0$, arguments $(t,\vartheta)$ mean restriction to $u=0$.
By contrast, $U(0,\urr,x_2)$ is the trace on the
opened initial fan, and $U(0,u,\vartheta)$ or
$U(0,\tilde u,\vartheta)$ is that trace in a characteristic chart.

All state vectors are columns. We write $(V,W)$ or $V\cdot W$ for the
Euclidean inner product and $|V|$ for its norm. Spatial component
indices $i,j$ range over $\{1,2\}$, whereas state component indices
$m$ range over $\{1,2,3\}$. Repeated component indices are summed when
they occur in contractions. A superscript on $v^i$ or $V^m$ is a
component index. Eigenvalue and eigenvector labels range over
$\{1,2,3\}$.
The symbols $a,b,m,n$ in derivative counts are nonnegative integers
within the displayed ranges. Summation indices in explicitly written
sums are dummy indices, and an empty sum is zero. We use
\[
\psi_i=-v^i,\qquad
\psi^{(\hat T)}=\psi_1\hat T^1+\psi_2\hat T^2=-v\cdot\hat T,\qquad
\wb=\frac{c}{\gamma-1}+\frac{v^1}{2},\quad
w=\frac{c}{\gamma-1}-\frac{v^1}{2}.
\]
Thus $c=\frac{\gamma-1}{2}(\wb+w)$ and $v^1=\wb-w$.
The underline in $\wb$ is part of the name of a Riemann invariant. Underlines on the
left profiles and boundaries, and bars on the left boundary graphs,
have the meanings given in
Section~\ref{subsec: main theorem preparation}.
In particular, $f=f_0$ in Sections~\ref{sec: set up}--\ref{sec: normal jets}.
The scalar graph function $g(t,x_2)=f(t,x_2)/t$ used there is distinct
from the acoustical metric $g$ in the introduction and from $\slashed g$.
The symbol $f$ in a functional inequality denotes an arbitrary test
function.

\subsubsection{The four coordinate systems and partial derivatives}
The following table fixes the role of each coordinate system.
\begin{center}
\renewcommand{\arraystretch}{1.18}
\begin{tabular}{p{0.23\textwidth}p{0.69\textwidth}}
\hline
Coordinates & Meaning and use\\ \hline
$(t,x_1,x_2)$ &
Physical spacetime coordinates; $x_2\in\mathbb T$ is periodic.\\
$(t,\urr,x_2)$ &
The opened coordinates, with $\urr=x_1/t$ for $t>0$; adjoining $t=0$
opens the initial rarefaction fan.\\
$(t,u,\vartheta)$ &
The characteristic chart on $\mathcal W_1$ with $0\leq u\leq u^*$
and $C_0=\{u=0\}$.\\
$(t,\tilde u,\vartheta)$ &
The characteristic chart on $\mathcal W_{2,\alpha}$ with
$\tilde u_*\leq\tilde u\leq\tilde u^*<0$; the limit
$\tilde u\rightarrow0$ approaches the vacuum boundary.\\ \hline
\end{tabular}
\end{center}
The period of $\mathbb T$ is fixed throughout; constants may depend on
this fixed period.
We use the same letter for a function and its pullback under a coordinate
change. The chart is clear from its displayed arguments, the domain of its
norm, or the surrounding equation. At $t=0$, the map $x_1=t\urr$
sends the whole opened slice $\Sigma_0$ to
$\mathbf S_*=\{t=0,x_1=0\}$. Expressions containing $x_1/t$ on
the opened initial slice mean their extension as the coordinate $\urr$.
The definitions using $\nabla_xu$ or $1/t$ are first made for $t>0$;
their initial values are the limits in the opened coordinates.
Near vacuum, initial time derivatives of order two or higher are
understood formally. They agree with the corresponding strong traces
when these traces exist.
Unless a subscript says otherwise, $\nabla$ in physical Euler or
eikonal equations is the spatial gradient in $(x_1,x_2)$.
$L^i=L(x_i)=v^i-c\hat T^i$ denotes a spatial component of $L$;
$L^m(f)$ denotes repeated differentiation when it acts on a function.

A partial derivative keeps the other coordinates of its chart fixed.
For $t>0$, the exact conversion rules are
\[
\begin{aligned}
\left.\partial_t\right|_{x_1,x_2}
&=\left.\partial_t\right|_{\urr,x_2}
  -\frac{\urr}{t}\partial_{\urr},&
\partial_{x_1}&=\frac1t\partial_{\urr},&
\left.\partial_{x_2}\right|_{t,x_1}
&=\left.\partial_{x_2}\right|_{t,\urr}.
\end{aligned}
\]
In either characteristic chart,
$\left.\partial_t\right|_{u,\vartheta}=L$ or
$\left.\partial_t\right|_{\tilde u,\vartheta}=L$, respectively.
These are different meanings of $\partial_t$; for example, the initial
time derivatives $\Vr_j=\partial_t^j(U)|_{t=0}$ use
$\left.\partial_t\right|_{\urr,x_2}$, whereas
\eqref{eq: eq of V_k} and \eqref{eq: rough eq of V_k} use the physical
time derivative. When $L$ acts on a function of time alone, such as a
slice norm, it means the ordinary derivative $d/dt$.

The covector components in a Hamiltonian are also chart dependent.
In physical coordinates they are
$(\tau,\xi)=(-\partial_t(u),-\nabla_xu)$; in the opened coordinates they are
$(-\partial_t(u),-\partial_{\urr}(u),-\partial_{x_2}(u))$, with the corresponding
time derivative understood. The symbol called $\lambda_3$ in
Section~\ref{subsec: opened geometry} is the transformed symbol
displayed there. In particular, its $\xi_1$ corresponds to $d\urr$.
In the initial simple-wave formulas, $\lambda_3(U)$ abbreviates
$\lambda_3(U,(1,0))=v^1+c$.

The two characteristic charts are constructed separately from the
initial phases in \eqref{eq: first characteristic coordinate} and
\eqref{eq: second characteristic coordinate}. In each chart,
$\vartheta=x_2$ initially and $L(\vartheta)=0$ for that chart's $L$.
Thus the symbols
\[
\vartheta,\quad L,\quad T,\quad\hat T,\quad\hat X,\quad
\kappa,\quad\slashed g,\quad\Xi,\quad\theta,\quad\chi,\quad\zeta+\eta
\]
refer to the selected chart.
The relation
\[
\tilde u=u-\frac{\gamma+1}{\gamma-1}c_r(0,x_2)
\]
holds on $\Sigma_0$. For $t>0$, the transition map is obtained by
composing the two maps
through $(t,\urr,x_2)$. This is also how a solution or a fixed-coordinate
derivative is compared on the overlap.
In the common initial-phase calculation of
Section~\ref{sec: high order approximate solution}, the local symbol
$u$ is explicitly allowed to stand for either phase. Its initial
form is $-\urr+h(x_2)$; the two displayed choices of $h$ determine which
chart is meant.

\subsubsection{Geometric fields, commutators, and profiles}
For the selected phase $u$ (with $\tilde u$ substituted in the second
chart), all unit vectors and spatial components are measured in the
physical Euclidean coordinates:
\[
\hat T=\frac{\nabla_xu}{|\nabla_xu|},\qquad
\hat X=(\hat T^2,-\hat T^1),\qquad
\kappa=|\nabla_xu|^{-1},\qquad T=\kappa\hat T.
\]
Our orientation is $\hat T^1\approx-1$ and $\hat X^2\approx1$.
The normalizations are $T(u)=1$, $\hat X(u)=0$, $L(u)=0$, and $L(t)=1$.
In the second chart they read $T(\tilde u)=1$, $\hat X(\tilde u)=0$,
$L(\tilde u)=0$, and $L(t)=1$.
The coefficients $\slashed g$ and $\Xi$ are defined by
\[
\partial_\vartheta=\sqrt{\slashed g}\,\hat X,\qquad
\partial_u=T+\Xi\partial_\vartheta
\]
in the first chart, with $\partial_{\tilde u}$ replacing $\partial_u$
in the second. The angular coordinate $\vartheta$ is distinct from the
commutator coefficient $\theta$ in
\eqref{eq: formulas to control the geometry}. Only the combination
$\zeta+\eta$ is needed in these equations.

The fixed commutators and the reference spacing are
\[
\Tr=-\partial_{\urr}=-t\partial_{x_1},\qquad
\Xr=\partial_{x_2},\qquad
\hat{\Tr}=-\partial_{x_1},\qquad \kappar=t.
\]
The first equality for $\Tr$ is the definition on the opened fan; its
physical-coordinate expression is used for $t>0$.
In particular, $\Tr$ and $T$ are different vector fields.
We write $\Zr\in\{\Tr,\Xr\}$ for a fixed commutator and
$Z\in\{T,\hat X\}$ for a geometric spatial derivative.
The fixed fields commute. An expression $Z^\beta$ denotes an ordered
composition of $|\beta|$ geometric derivatives; their order matters,
because $[T,\hat X]$ need not vanish. Our commutator convention is
$[A,B](f)=A(B(f))-B(A(f))$. A sum over $|\beta|\leq m$
includes the corresponding words of length at most $m$.
For a single field, $(\hat X)^m(f)$ means $m$ repeated derivatives,
whereas $\hat X^i$ is a physical component. For the commuting fields,
$\Zr^\alpha$ denotes $\Tr^{\alpha_1}\Xr^{\alpha_2}$, with
$|\alpha|=\alpha_1+\alpha_2$.
The notation $Z^{\leq m}(f)$ denotes the family of words of length at
most $m$, and $\#(T,Z^\alpha)$ counts occurrences of $T$ in a word.
Norms of such families mean the sum of the norms of their members;
in schematic inequalities, an unindexed $Z(f)$ or $\Zr(f)$ likewise
denotes the family of the indicated derivatives. The tangential
family $\slashed{\mathfrak d}\in\{L,\hat X\}$ in
Section~\ref{sec: normal jets} uses the same ordered-word convention.

Parenthesized superscripts on a vector field denote coefficients in
the indicated frame. Explicitly,
\[
\begin{aligned}
T&=T^{(\Tr)}\Tr+T^{(\Xr)}\Xr,&
T^{(\Tr)}&=-\frac{\kappa}{t}\hat T^1,&
T^{(\Xr)}&=\kappa\hat T^2,\\
\hat X&=\hat X^{(\Tr)}\Tr+\hat X^{(\Xr)}\Xr,&
\hat X^{(\Tr)}&=-\frac{\hat T^2}{t},&
\hat X^{(\Xr)}&=-\hat T^1.
\end{aligned}
\]
These coefficients refer to the selected characteristic chart.

The initial profiles and the approximations have different roles:
\[
\Vr_j=\left.\partial_t^j(U)\right|_{t=0}\ \text{formally},\qquad
\Ur_0=\Vr_0,\qquad \Ur_1=\Vr_0+t\Vr_1,\qquad
V_k=U-\Ur_k,\qquad V_{k;a,b}=\Tr^a\Xr^b(V_k).
\]
The profiles $\Vr_j$ represent formal time derivatives, so the
quadratic Taylor term is $t^2\Vr_2/2$.
The residual $\Rr_k$ in Lemma~\ref{lem: bounds on Res} is the
residual after multiplying the Euler equations by $t$; the physical
residual is $t^{-1}\Rr_k$.
The ring in $\mathring c,\mathring v^i,\mathring w$ denotes the
corresponding components or invariants of $\Ur_0$.
In the initial-profile calculations, an unadorned $c$ abbreviates
$c(0,\urr,x_2)=\mathring c(\urr,x_2)$. Exterior factors $c_r,v_r^i,w_r$ in those formulas are evaluated
at their initial spatial traces unless other arguments are displayed.
The scalar vorticity has orientation
$\omega=\partial_{x_1}(v^2)-\partial_{x_2}(v^1)$.
The matrices $A_1,A_2$ are linear in their vector arguments. Their
use on $V_k$ or $\Vr_j$ means the same linear maps.

\subsubsection{Domains, measures, and Sobolev norms}
For the first characteristic chart we use
\[
\begin{aligned}
C_u&=\{u=\text{constant}\},&
S_{t,u}&=\Sigma_t\cap C_u,\\
\Sigma_t^{u_1,u_2}
&=\{(t,u,\vartheta):u_1\leq u\leq u_2,\ \vartheta\in\mathbb T\},&
C_u^t&=\{(t',u,\vartheta):0\leq t'\leq t,\ \vartheta\in\mathbb T\}.
\end{aligned}
\]
The analogous notation with tildes refers to the second chart.
$\mathcal D(t_1,t_2;u_1,u_2)$ is the corresponding spacetime slab;
$\mathcal D(t;u_1,u_2)=\mathcal D(0,t;u_1,u_2)$.
The abbreviations $\mathcal D(t,u)$ and $\mathcal D(t,\tilde u)$ have
lower characteristic endpoints $0$ and $\tilde u_*$, respectively.
Correspondingly,
\[
\begin{aligned}
\mathcal E(V)(t,u)&=\mathcal E(V)(t;0,u),&
\mathcal F(V)(t,u)&=\mathcal F(V)(0,t;u),\\
\mathcal E(V)(t,\tilde u)&=\mathcal E(V)(t;\tilde u_*,\tilde u),&
\mathcal F(V)(t,\tilde u)&=\mathcal F(V)(0,t;\tilde u).
\end{aligned}
\]
We keep the argument $V$ when using these abbreviations.
The notation $\mathcal D_r=\mathcal D_{\textbf{right}}$ denotes the
right exterior domain. The slices $\Sigma_t^r,\Sigma_t^l$ in
Theorem~\ref{theorem: stability of RVR} instead use $(\urr,x_2)$.

For $1\leq p<\infty$, the characteristic-slice norms used in the
bootstrap estimates are
\[
\begin{aligned}
\|f\|_{L^p(\Sigma_t^{u_1,u_2})}^p
&=\int_{u_1}^{u_2}\int_{\mathbb T}
 |f(t,u,\vartheta)|^p\sqrt{\slashed g(t,u,\vartheta)}\,d\vartheta\,du,\\
\|f\|_{L^p(\Sigma_t^{\tilde u_1,\tilde u_2})}^p
&=\int_{\tilde u_1}^{\tilde u_2}\int_{\mathbb T}
 |f(t,\tilde u,\vartheta)|^p\sqrt{\slashed g(t,\tilde u,\vartheta)}
 \,d\vartheta\,d\tilde u.
\end{aligned}
\]
$L^\infty$ is the essential supremum over the displayed domain.
Using absolute Jacobians, the physical and opened area elements are
\[
dx_1\,dx_2=\kappa\sqrt{\slashed g}\,du\,d\vartheta,\qquad
d\urr\,dx_2=\frac{\kappa}{t}\sqrt{\slashed g}\,du\,d\vartheta,
\]
with $d\tilde u$ replacing $du$ in the second chart.
These formulas hold for $t>0$; the opened Jacobian is extended to
$t=0$ through the controlled ratio $\kappa/t$.
Consequently,
\[
\mathcal E(V)(t;u_1,u_2)
=\int_{u_1}^{u_2}\int_{\mathbb T}|V|^2\kappa\sqrt{\slashed g}
\,d\vartheta\,du
\approx t\|V\|_{L^2(\Sigma_t^{u_1,u_2})}^2.
\]
Here $\approx$ means two inequalities with positive fixed constants,
under the established bounds for $\kappa/t$. Flux integrals on $C_u^t$
use $\sqrt{\slashed g}\,dt'\,d\vartheta$ and contain the explicit
sound-speed factors in \eqref{eq: linear energy estimate in t ur x_2}.
Unwritten measures in the spacetime error integrals of
Section~\ref{sec: energy estimates} mean
$(\kappa/t)\sqrt{\slashed g}\,dt'\,du\,d\vartheta$,
evaluated at time $t'$, with the analogous convention in the second chart.
All fields in an integrand, including its density and geometric
coefficients, are evaluated at the integration variables. Thus a
factor $t'$ inside $\mathcal D(t,u)$ is the running time, whereas the
argument $t$ of $\mathcal D$ is its upper endpoint.

$H^m=W^{m,2}$, and $W^{m,p}$ on a slice uses spatial derivatives in
the indicated coordinates only. Sobolev norms in characteristic coordinates
are compared with the displayed $T,\hat X$ or $\Tr,\Xr$ derivatives
using the estimates in Section~\ref{sec: Preparations for the energy estimate}.
On coordinate rectangles or on $\mathbb T$, the measure is the usual
Lebesgue measure. On $\Sigma_t^r,\Sigma_t^l$ the derivatives and measure
are in $(\urr,x_2)$, as stated in
Section~\ref{subsec: main theorem preparation}. We explicitly state when a Sobolev norm is taken in spacetime;
such norms include time derivatives. In particular,
the time-slice $W^{2,\infty}$ conclusion is different from the
additional spacetime conclusion for $\gamma=7/5,5/3$.

\subsubsection{Regularity indices, constants, and other repeated symbols}
The data order $N_0$ is fixed once using
Section~\ref{subsubsec: regularity orders}. In particular
$N_{\rm away}$ is the fixed bootstrap order away from vacuum and
$J_4$ is the required opened-coordinate regularity, while
$\widehat N_i$ are the regularity indices of the cited nonvacuum result.
The letter $N$ denotes a sufficiently large finite regularity order
available at the stage of the argument under discussion. It may
decrease after taking traces, differentiating coefficients, dividing
a vanishing function by $t$, or estimating normal derivatives.
After such a loss, we may again call the remaining order $N$.
Within a single assertion, its repeated occurrences have the same
value.
For example, the precise elementary statement is
\[
f\in C^{N+1}([0,1]),\quad f(0)=0
\quad\Longrightarrow\quad
\frac{f(t)}t=\int_0^1 f'(st)\,ds\in C^N([0,1]).
\]
The choice \eqref{eq: global regularity orders} ensures the higher
regularity orders needed in each of these steps. In particular the
normal-jet estimates are used through order $N_{\rm away}$ (with
one extra derivative of the state for $Z^\alpha T(U)$).

By contrast, $\Ntop$ is a fixed order in each bootstrap argument:
$\Ntop=N_{\rm away}$ away from vacuum (at least six for the a priori
estimate alone and at least eight for the main theorem), and
$\Ntop=4$ near vacuum.
The commutation order $\Ncomm$ in Section~\ref{sec: energy estimates}
equals $\Ntop$ away from vacuum and $4$ near vacuum.

The letter $C$ denotes a positive constant whose value may change
between estimates. It may depend on fixed parameters, including
$\gamma,\mathbf c_*,t^*$ and the required regularity orders, and on
the fixed reference states and pressure normalization when these are
not normalized. Unless dependence is stated explicitly, it is
independent of $t$, $\varepsilon$, the lower time $\delta$, and the
vacuum cutoff $\alpha\in(0,1/100)$.
Dependence on the bootstrap constant $M$ is displayed as $P(M)$ or
explicit powers of $M$; $P$ is a generic polynomial with nonnegative
coefficients and may change between estimates. A bound $P(M)\varepsilon$
may be absorbed into a fixed constant after imposing the stated
smallness condition, but it is not an $O(\varepsilon)$ bound
independent of $M$. The choice of $M$ is
made before the final choice of sufficiently small $\varepsilon$.
The notation $A\lesssim B$ means $A\leq CB$, and $O(\varepsilon t^j)$
has the corresponding pointwise meaning unless a norm is given.
For nonnegative quantities, $A\approx B$ means $C^{-1}B\leq A\leq CB$.
For a vector or a signed component, $\approx$ describes closeness to
the stated reference value, quantified by the displayed estimates.
In the profile-regularity discussion, a displayed power after $\approx$
records a possible singular term; a blow-up assertion requires that
its coefficient be nonzero.
The generic constant $C$ is distinct from a characteristic hypersurface
such as $C_u$.

Some letters have separate uses that are explained locally. The scalar cutoff
$\alpha$ is distinct from a derivative multi-index $\alpha$ and from
the eigenvectors $\alpha_1,\alpha_2,\alpha_3$ in the flux decomposition;
there $\alpha_j=r_j(U,-\hat T)$.
The scalar $\beta=(\gamma+1)/(2(\gamma-1))$ in the profile formulas is
distinct from a derivative multi-index $\beta$. In a schematic
expression $c^\alpha(\ln c)^\beta$, before these exponents are
fixed, $\alpha,\beta$ simply denote powers. Dummy time variables
such as $s,t',\tau$ in integrals are unrelated to the Hamiltonian
covector component $\tau$.
The pressure constant $k$ in $p=k\rho^\gamma$ is fixed and positive;
the index $k\in\{0,1\}$ on $\Ur_k,V_k,\Rr_k$ denotes the approximation
order. A scalar subtracted from a matrix means that scalar times the
identity matrix; for example, $A_1(U)-\urr=A_1(U)-\urr I$.

Finally, $E(V)$ and $F(V)$ in Section~\ref{sec: energy estimates}
abbreviate $\mathcal E(V)$ and $\mathcal F(V)$.
An expression such as $F(V,W)$ instead denotes the generic multilinear
combination defined locally, and a standalone right-hand side $F$
in a linear equation is its forcing term.
The coefficient symbols $\Wr_{O(\varepsilon)}$ and $\Wr_{O(1)}$
denote families with the differentiated bounds stated in that section.
Their $O(\varepsilon)$ or $O(1)$ subscripts include the stated
derivative bounds and are stronger than an undifferentiated size bound.
The notation $V_{k;\leq a,\leq b}$ denotes the finite family
$\{V_{k;a',b'}:0\leq a'\leq a,\ 0\leq b'\leq b\}$.
Terms involving $V_{k;a-1,b+1}$ are omitted when $a=0$.

An equality described as schematic, including an expression involving
\textbf{knowns} or \textbf{good knowns} in
Section~\ref{sec: normal jets}, records a finite sum of products of
the indicated factors with the locally stated coefficient bounds.
Different occurrences may represent different such sums. In an estimate,
$F$ may be scalar-valued after contraction or vector-valued before
contraction; its arity and the surrounding inner product determine the
meaning. The auxiliary $F(t)$ in Lemma~\ref{lem: Gronwall} is defined
in that lemma and is unrelated to the flux.
The matrix $M(\hat T,U)$ in the normal-jet calculation is a locally
defined coefficient matrix, distinct from the scalar bootstrap
constant $M$.

\section{Initial profiles and approximate solutions}\label{sec: high order approximate solution}
We determine the first two formal time derivatives of $U$ on the opened
initial slice. We use them to construct the approximate solutions and
compute the residuals in the equations for the remainders. The same
formulas show which spatial derivatives remain integrable as $c$ tends
to zero and where fractional powers and logarithmic terms appear.

Subsections~3.1--3.3 prove Propositions~\ref{prop: initial energy bounds on Sigma_0}
and~\ref{prop: formulas for c^-1 L(psi)} by deriving the two profiles
and their initial geometric and characteristic values. Subsection~3.4
computes the residuals of the approximate solutions and proves
Proposition~\ref{prop: Res bounds of Ur_k}.

Throughout this section, time derivatives of $U$ are taken at fixed
$(\urr,x_2)$, while derivatives of the exterior solution $U_r$ in
$(t,x_1,x_2)$ are taken before restriction to the boundary.
In initial-profile formulas we abbreviate
$c(0,\urr,x_2)$ by $c$, and exterior factors without displayed
arguments are evaluated at $(t,x_1)=(0,0)$.

In the opened coordinates $(t, \mathring{u}, x_{2})$, the initial boundary $\Sigma_{0}\cap C_{0}$ is
\[
\mathring{u}=(v_{r}^{1}+c_{r})(0, x_{2}),
\]
and the solution $U$ on $\Sigma_{0}$ takes the form
\[
U(0, \mathring{u}, x_{2})=\Vr_{0}(\mathring{u}, x_{2})=\begin{pmatrix}
\frac{2}{\gamma-1}c(0, \mathring{u}, x_{2})\\
v^{1}(0, \mathring{u}, x_{2})\\
v^{2}(0, \mathring{u}, x_{2})
\end{pmatrix}=\begin{pmatrix}
\frac{2}{\gamma-1}\cdot \Big( c_{r}(0, x_{2})+\frac{\gamma-1}{\gamma+1}\big(\mathring{u}-(v_{r}^{1}+c_{r})(0, x_{2})\big)\Big)\\
v_{r}^{1}(0, x_{2})+\frac{2}{\gamma+1}\big(\mathring{u}-(v_{r}^{1}+c_{r})(0, x_{2})\big)\\
v_{r}^{2}(0, x_{2})
\end{pmatrix}.
\]
We determine $\Vr_1$ and $\Vr_2$ by substituting the formal expansion
$U=\Vr_0+t\Vr_1+\tfrac12t^2\Vr_2+\cdots$ into the Euler equations and
the exterior matching condition, then equating coefficients. In these
calculations, $\partial_t(U)(0)=\Vr_1$ and $\partial_t^2(U)(0)=\Vr_2$
denote formal time derivatives. They agree with the actual time
derivatives whenever the fixed-cutoff solution has the regularity stated in
Proposition~\ref{prop: initial energy bounds on Sigma_0}. The existence
construction in Proposition~\ref{prop: positive time construction}
uses these profiles to prescribe data at positive time and estimates
the resulting solutions before taking either cutoff to zero.

Write the exterior boundary as $C_{0}=\{x_{1}=f(t, x_{2})\}$ in physical coordinates $(t, x_{1}, x_{2})$. Its initial values satisfy
\[
f(0, x_{2})=0.
\]
We also have $L^{k}\big( x_{1}-f(t, x_{2})\big)=0$ for $k=1, 2, 3$ since the generator of $C_{0}$ is
\[
L=\frac{\partial}{\partial t}+(v_{r}^{1}-c_{r}\hat{T}^{1})\frac{\partial}{\partial x_{1}}+(v_{r}^{2}-c_{r}\hat{T}^{2})\frac{\partial}{\partial x_{2}}.
\]
which implies that
\begin{equation}\notag
\begin{cases}
\frac{\partial f}{\partial t}(0, x_{2})=(v_{r}^{1}+c_{r})(0, x_{2}),\\
\frac{\partial^{2} f}{\partial t^{2}}(0, x_{2})=\big(L(v_{r}^{1})+L(c_{r})-2v_{r}^{2}\frac{\partial^{2} f}{\partial x_{2}\partial t}\big)(0, x_{2}).
\end{cases}
\end{equation}
and
\begin{equation}\notag
\begin{split}
\frac{\partial^{3} f}{\partial t^{3}}(0, x_{2})&=\Big(L^{2}(v_{r}^{1})+L^{2}(c_{r})-c_{r}(\partial_{x_2}(v_{r}^{1})+\partial_{x_2}(c_{r}))^{2}\\
&-3v_{r}^{2}\frac{\partial^{3} f}{\partial x_{2}\partial t^{2}}-3(v_{r}^{2})^{2}\frac{\partial^{3} f}{\partial x_{2}^{2}\partial t}-3\frac{\partial^{2} f}{\partial x_{2} \partial t}\big( L(v_{r}^{2})-c_{r}\partial_{x_2}(c_{r})-c_{r}\partial_{x_2}(v_{r}^{1}) \big)\Big)(0, x_{2}).
\end{split}
\end{equation}
We remark that $\frac{\partial^{2} f}{\partial t^{2}}(0, x_{2})$ and $\frac{\partial^{3} f}{\partial t^{3}}(0, x_{2})$ depend on $\partial^{a}_{x_1}\partial^{b}_{x_2}(U_{r})(0, x_{2}), a+b\leq 2$ and have bounds
\begin{equation}\notag
\|\frac{\partial^{2} f}{\partial t^{2}}(0, x_{2})\|_{C^{N}(\mathbb{T})}+\|\frac{\partial^{3} f}{\partial t^{3}}(0, x_{2})\|_{C^{N}(\mathbb{T})}\leq C\varepsilon.
\end{equation}

In coordinate $(t, \mathring{u}, x_{2})$, $C_{0}$ is parametrized by 
\[
\mathring{u}:=g(t, x_{2})=\frac{f(t, x_{2})}{t}=\int_{0}^{1} \frac{\partial f}{\partial t}(ts, x_{2})\,ds.
\]
and we have
\begin{equation}\notag
\begin{cases}
g(0, x_{2})=\frac{\partial f}{\partial t}(0, x_{2}),\\
\frac{\partial g}{\partial t}(0, x_{2})=\frac{1}{2}\frac{\partial^{2} f}{\partial t^{2}}(0, x_{2}),\\
\frac{\partial^{2} g}{\partial t^{2}}(0, x_{2})=\frac{1}{3}\frac{\partial^{3} f}{\partial t^{3}}(0, x_{2}).
\end{cases}
\end{equation}
We remark that $\frac{\partial g}{\partial t}(0, x_{2})$ and $\frac{\partial^{2} g}{\partial t^{2}}(0, x_{2})$ depend on $\partial^{a}_{x_1}\partial^{b}_{x_2}(U_{r})(0, x_{2}), a+b\leq 2$ and have bounds
\begin{equation}\notag
\|\frac{\partial g}{\partial t}(0, x_{2})\|_{C^{N}(\mathbb{T})}+\|\frac{\partial^{2} g}{\partial t^{2}}(0, x_{2})\|_{C^{N}(\mathbb{T})}\leq C\varepsilon.
\end{equation}

\subsection{The first initial profile}
Since $U=U_{r}$ along $C_{0}$, we have
\begin{equation}\label{eq: coincides on S_00}
U(t, g(t, x_{2}), x_{2})=U_{r}(t, f(t, x_{2}), x_{2}).
\end{equation}
Differentiating \eqref{eq: coincides on S_00} once in $t$ and setting $t=0$ gives
\[
\Vr_{1}\big((v_{r}^{1}+c_{r})(0, x_{2}), x_{2}\big)=-\frac{\partial g}{\partial t}(0, x_{2})\cdot\begin{pmatrix}
\frac{2}{\gamma+1}\\ \frac{2}{\gamma+1} \\ 0
\end{pmatrix}+\big(\partial_{t}(U_{r})+\partial_{x_1}(U_{r})\cdot \frac{\partial f}{\partial t}\big)(0, x_{2})
\]
Thus the boundary values of $\Vr_{1}$ on $\Sigma_{0}\cap C_{0}$ are
\begin{equation}\label{eq: initial data of Vr_1}
\begin{cases}
(\Vr_{1}^{1}+\Vr_{1}^{2})\big((v_{r}^{1}+c_{r})(0, x_{2}), x_{2}\big)=\frac{3-\gamma}{\gamma+1}\cdot L(w_{r})-\frac{2(3-\gamma)}{\gamma+1}v_{r}^{2}\cdot \partial_{x_2}(w_{r})-\frac{1}{2}c_{r}\cdot \partial_{x_2}(v_{r}^{2}),\\
(\Vr_{1}^{1}-\Vr_{1}^{2})\big((v_{r}^{1}+c_{r})(0, x_{2}), x_{2}\big)=2L(w_{r})-2v_{r}^{2}\cdot \partial_{x_2}(w_{r}),\\
\Vr_{1}^{3}\big((v_{r}^{1}+c_{r})(0, x_{2}), x_{2}\big)=L(v_{r}^{2})-v_{r}^{2}\cdot \partial_{x_2}(v_{r}^{2}).
\end{cases}
\end{equation}
Differentiate once in $t$ and set $t=0$ in
\[
t\cdot \frac{\partial U}{\partial t}+\big( A_{1}(U)-\mathring{u} \big)\frac{\partial U}{\partial \mathring{u}}+t\cdot A_{2}(U)\frac{\partial U}{\partial x_2}=0.
\]
we have
\[
\frac{\partial U}{\partial t}+A_{1}(\frac{\partial U}{\partial t})\frac{\partial U}{\partial \mathring{u}}+(A_{1}(U)-\mathring{u})\frac{\partial }{\partial \mathring{u}}\big( \frac{\partial U}{\partial t} \big)+A_{2}(U)\frac{\partial U}{\partial x_2}=0.
\]
This is an ODE system for $\Vr_{1}$ with coefficients determined by $\Vr_{0}$. In components, it reads
\begin{equation}\label{eq: ODE of Vr_1}
\begin{pmatrix}
-c & c & 0 \\
c  & -c & 0\\
0 & 0 & -c
\end{pmatrix}\frac{\partial \Vr_{1}}{\partial \mathring{u}}+\Vr_{1}+\begin{pmatrix}
\frac{\gamma-1}{\gamma+1} & \frac{2}{\gamma+1} & 0 \\
\frac{\gamma-1}{\gamma+1} & \frac{2}{\gamma+1} & 0 \\
0 & 0 & 0 \\
\end{pmatrix}\Vr_{1}+A_{2}(\Vr_{0})\frac{\partial \Vr_{0}}{\partial x_2}=0
\end{equation}
For each $x_{2}\in \mathbb{T}$, solving \eqref{eq: ODE of Vr_1} with the boundary data \eqref{eq: initial data of Vr_1} gives
\begin{equation}\label{eq: exact form of Vr_1}
\begin{cases}
(\Vr_{1}^{1}+\Vr_{1}^{2})(\mathring{u}, x_{2})=\frac{1}{2}\frac{3-\gamma}{\gamma+1}\cdot (\Vr_{1}^{1}-\Vr_{1}^{2})(\mathring{u}, x_{2})-\frac{1}{2}\Big( c\cdot \partial_{x_2}(v_{r}^{2})+\frac{2(3-\gamma)}{\gamma+1}v_{r}^{2}\cdot \partial_{x_2}(w_{r}) \Big),\\
(\Vr_{1}^{1}-\Vr_{1}^{2})(\mathring{u}, x_{2})=-\frac{\gamma+1}{3-\gamma}\Big( c-\big(\frac{c}{c_{r}}\big)^{\frac{1}{2}\frac{\gamma+1}{\gamma-1}}c_{r} \Big)\cdot \partial_{x_2}(v_{r}^{2})+2\big( \frac{c}{c_{r}} \big)^{\frac{1}{2}\frac{\gamma+1}{\gamma-1}}\cdot L(w_{r})-2v_{r}^{2}\cdot \partial_{x_2}(w_{r}),\\
\Vr_{1}^{3}(\mathring{u}, x_{2})=-2c\cdot \partial_{x_2}(w_{r})+\big(\frac{c}{c_r}\big)^{\frac{\gamma+1}{\gamma-1}}\cdot \Big(2c_{r}\cdot \partial_{x_2}(w_{r})+L(v_{r}^{2})\Big)-v_{r}^{2}\cdot \partial_{x_2}(v_{r}^{2}).
\end{cases}
\end{equation}

\begin{remark}\label{rem: remark 4.1}
The formulas for $\Vr_{1}(\mathring{u}, x_{2})$ have the following consequences for compatibility and regularity.
\begin{itemize}
\item The ODE \eqref{eq: ODE of Vr_1} suffers algebraic degeneration in direction $(1, 1, 0)^{T}$, which means that the initial data for $\Vr_{1}^{1}+\Vr_{1}^{2}$ are uniquely determined once the initial data of $\Vr_{1}^{1}-\Vr_{1}^{2}$ and $\Vr_{1}^{3}$ are determined. This compatibility is required by the characteristic nature of $C_{0}$ and can also be checked directly.
\item The analytic degeneracy occurs as $c\rightarrow 0$. The leading profile $\Vr_0$ is affine in $\mathring u$, with
coefficients depending on $x_2$. The first profile instead contains
factors such as $c^{\frac{1}{2}\frac{\gamma+1}{\gamma-1}}$\footnote{The exponent $\frac12(\gamma+1)/(\gamma-1)$ decreases as $\gamma$ increases.} whose higher Sobolev norms may diverge near vacuum. When $\gamma$ approaches 3 from below, $\frac{1}{2}\frac{\gamma+1}{\gamma-1}$ approaches $1$ from above. Therefore, we have $\frac{\partial^{2} \Vr_{1}}{\partial \mathring{u}^{2}}\approx c^{-1+\alpha}, \alpha<<1$ near \textbf{vacuum}, which, when its coefficient is nonzero and $\alpha<1/2$, makes the second spatial derivative fail to be square integrable near vacuum.
\item In this paper, we focus on the case $\gamma\in (1, 1.3]\cup \{\frac{7}{5}, \frac{5}{3}\}$. When $\gamma\in (1, 1.3]$, $\frac{1}{2}\frac{\gamma+1}{\gamma-1}\geq \frac{23}{6}$. Taking $\gamma=1.3$ as an example, when the coefficient of $c^{23/6}$ is nonzero, we have 
\[
\frac{\partial^{5} \Vr_{1}}{\partial \mathring{u}^{5}}\approx c^{-\frac{7}{6}}, \ \ \textbf{near vacuum},
\]
while
\[
|\frac{\partial^{a}}{\partial \mathring{u}^{a}}\frac{\partial^{b}}{\partial x_{2}^{b}}(\Vr_{1})| \leq C\varepsilon\cdot c^{-\frac{1}{6}}, \ \ a+b\leq 4.
\]
Therefore, the fourth order Sobolev norm of $\Vr_{1}$ is bounded, while its fifth order Sobolev norm generally blows up when that coefficient is nonzero, which explains the choice $\Ntop=4$ in the near-vacuum argument for $\gamma\in (1, 1.3]$. When $\gamma\in \{\frac{7}{5}, \frac{5}{3}\}$, the exponents satisfy 
\[
\begin{cases}
\frac{1}{2}\frac{\gamma+1}{\gamma-1}=3, \gamma=\frac{7}{5},\\
\frac{1}{2}\frac{\gamma+1}{\gamma-1}=2, \gamma=\frac{5}{3},\\
\end{cases}
\]
Therefore, $\Vr_{1}$ is polynomial in $c$ and has the tangential regularity allowed by the initial data. The vacuum restriction on higher regularity instead comes from $\Vr_{2}$ for these two exponents. 
\item In \textbf{away from vacuum} case, $\Vr_{1}$ is regular up to the available data order for all $\gamma\in (1, 3)$, and we have
\begin{equation}\notag
|\frac{\partial^{a}}{\partial \mathring{u}^{a}}\frac{\partial^{b}}{\partial x_{2}^{b}}(\Vr_{1})| \leq C\varepsilon, \ \ a+b\leq N.
\end{equation}
\end{itemize}
\end{remark}

\subsection{The second initial profile}
Differentiating \eqref{eq: coincides on S_00} twice in $t$ and setting $t=0$ determines $\Vr_{2}\big( (v_{r}^{1}+c_{r})(0, x_{2}), x_{2} \big)$ in terms of $\partial^{a}_{x_1}\partial^{b}_{x_2}(U_{r})(0, x_{2}), a+b\leq 2$. We use the following bounds:
\begin{equation}\notag
\|\Vr_{2}\big( (v_{r}^{1}+c_{r})(0, x_{2}), x_{2} \big)\|_{C^{N}(\mathbb{T})}\leq C\varepsilon.
\end{equation}
Differentiate twice in $t$ and set $t=0$ in
\[
t\cdot \frac{\partial U}{\partial t}+\big( A_{1}(U)-\mathring{u} \big)\frac{\partial U}{\partial \mathring{u}}+t\cdot A_{2}(U)\frac{\partial U}{\partial x_2}=0.
\]
we have
\[
\begin{split}
&2\frac{\partial^{2}U}{\partial t^{2}}+A_{1}(\frac{\partial^{2} U}{\partial t^{2}})\frac{\partial U}{\partial \mathring{u}}+(A_{1}(U)-\mathring{u})\frac{\partial }{\partial \mathring{u}}\big( \frac{\partial^{2} U}{\partial t^{2}} \big)\\
&+\Big(A_{1}(\frac{\partial U}{\partial t})\frac{\partial }{\partial \mathring{u}}(\frac{\partial U}{\partial t})+A_{1}(\frac{\partial U}{\partial t})\frac{\partial }{\partial \mathring{u}}(\frac{\partial U}{\partial t})+2A_{2}(\frac{\partial U}{\partial t})\frac{\partial U}{\partial x_2}+2A_{2}(U)\frac{\partial}{\partial x_2}(\frac{\partial U}{\partial t})\Big)=0.
\end{split}
\]
This gives an ODE system for $\Vr_{2}$, with coefficients and forcing determined by $\Vr_{0}, \Vr_{1}$:
\begin{equation}\label{eq: ODE of Vr_2}
\begin{pmatrix}
-c & c & 0 \\
c  & -c & 0\\
0 & 0 & -c
\end{pmatrix}\frac{\partial \Vr_{2}}{\partial \mathring{u}}+2\Vr_{2}+\begin{pmatrix}
\frac{\gamma-1}{\gamma+1} & \frac{2}{\gamma+1} & 0 \\
\frac{\gamma-1}{\gamma+1} & \frac{2}{\gamma+1} & 0 \\
0 & 0 & 0 \\
\end{pmatrix}\Vr_{2}+\Big( 2A_{1}(\Vr_{1})\frac{\partial \Vr_{1}}{\partial \mathring{u}}+2A_{2}(\Vr_{1})\frac{\partial \Vr_{0}}{\partial x_{2}}+2A_{2}(\Vr_{0})\frac{\partial \Vr_{1}}{\partial x_{2}}\Big)=0.
\end{equation}
Set $\beta:=\frac{1}{2}\frac{\gamma+1}{\gamma-1}$. In the three
component formulas below each occurrence of $f_\alpha$ refers to
the coefficient family of that formula. The subscripts refer to the displayed
powers before specialization of $\gamma$. If two powers coincide
for a particular $\gamma$ (for example $\beta=2$ when $\gamma=5/3$),
their coefficients are initially separate and are then added to
form the coefficient of that power. The bounds also hold for the
resulting coefficients. The displayed logarithmic terms are kept
separate. For each $x_{2}\in \mathbb{T}$, solving \eqref{eq: ODE of Vr_2} with boundary value $\Vr_{2}\big( (v_{r}^{1}+c_{r})(0, x_{2}), x_{2} \big)$ yields
\begin{equation}\notag
\begin{split}
\Vr_{2}^{1}+\Vr_{2}^{2}&=c^{2\beta+1}\cdot f_{2\beta+1}(x_2)+c^{2\beta}\cdot f_{2\beta}(x_2)+c^{2\beta-1}\cdot f_{2\beta-1}(x_2)+c^{\beta}\cdot f_{\beta}(x_2)\\
&+\frac{4(3-\gamma)}{3(\gamma-1)}\big(\frac{c}{c_r}\big)^{2\beta}\ln\big(\frac{c}{c_r}\big)\partial_{x_2}(w_r)\big(L(v_r^2)+2c_r\partial_{x_2}(w_r)\big)\\
&+\underbrace{\big(\frac{c}{c_{r}}\big)^{\beta-1} \cdot \Big( -2v_{r}^{2}\partial_{x_2}(w_{r})\big( \frac{3-\gamma}{\gamma+1}c_{r}^{-1}L(w_{r})+\frac{1}{2}\partial_{x_2}(v_{r}^{2}) \big) \Big)}_{\textbf{II}}\\
&+c^{2}\cdot f_{2}(x_2)+c\cdot f_{1}(x_{2})+f_{0}(x_2).
\end{split}
\end{equation}
where $f_{\alpha}(x_{2}), \alpha\in \{0, 1, 2, \beta, 2\beta-1, 2\beta, 2\beta+1\}$ depend on $\partial^{a}_{x_1}\partial^{b}_{x_2}(U_{r})(0, x_{2}), a+b\leq 2$, and satisfy
\[
\|f_{\alpha}(x_2)\|_{C^{N}(\mathbb{T})}\leq C\varepsilon.
\]
\begin{equation}\label{eq: exact logarithmic term of Vr_2}
\begin{split}
\Vr_{2}^{1}-\Vr_{2}^{2}&=c^{2\beta+1}\cdot f_{2\beta+1}(x_2)\\
&+\underbrace{\big(\frac{c}{c_{r}}\big)^{2\beta}\ln\big( \frac{c}{c_{r}} \big)\Big(\frac{4(\gamma+1)}{\gamma-1}\partial_{x_2}(w_{r})\big( L(v_{r}^{2})+2c_{r}\partial_{x_2}(w_{r}) \big)\Big)}_{\textbf{I}}\\
&+c^{2\beta}\cdot f_{2\beta}(x_2)+c^{2\beta-1}\cdot f_{2\beta-1}(x_2)+c^{\beta}\cdot f_{\beta}(x_2)\\
&+\underbrace{\big(\frac{c}{c_{r}}\big)^{\beta-1}\cdot \Big( -2v_{r}^{2}\partial_{x_2}(w_{r})\big( 2c_{r}^{-1}L(w_{r})+\frac{\gamma+1}{3-\gamma}\partial_{x_2}(v_{r}^{2}) \big) \Big)}_{\textbf{II}}\\
&+c^{2}\cdot f_{2}(x_2)+c\cdot f_{1}(x_{2})+f_{0}(x_2).
\end{split}
\end{equation}
where $f_{\alpha}(x_{2}), \alpha\in \{0, 1, 2, \beta, 2\beta-1, 2\beta, 2\beta+1\}$ depend on $\partial^{a}_{x_1}\partial^{b}_{x_2}(U_{r})(0, x_{2}), a+b\leq 2$, and satisfy
\[
\|f_{\alpha}(x_2)\|_{C^{N}(\mathbb{T})}\leq C\varepsilon.
\]
\begin{equation}\notag
\begin{split}
\Vr_{2}^{3}&=c^{4\beta}\cdot f_{4\beta}(x_{2})+c^{3\beta-1}\cdot f_{3\beta-1}(x_{2})+c^{2\beta}\cdot f_{2\beta}(x_{2})+c^{2\beta-1}\cdot f_{2\beta-1}(x_{2})\\
&+c^{\beta+1}\cdot f_{\beta+1}(x_{2})+c^{\beta}\cdot f_{\beta}(x_{2})+c^{2}\cdot f_{2}(x_{2})+c\cdot f_{1}(x_2)+f_{0}(x_2).
\end{split}
\end{equation}
where $f_{\alpha}(x_{2}), \alpha\in \{0, 1, 2, \beta, \beta+1, 2\beta-1, 2\beta, 3\beta-1, 4\beta\}$ depend on $\partial^{a}_{x_1}\partial^{b}_{x_2}(U_{r})(0, x_{2}), a+b\leq 2$, and satisfy
\[
\|f_{\alpha}(x_2)\|_{C^{N}(\mathbb{T})}\leq C\varepsilon.
\]

\begin{remark}\label{rem: remark 4.2}
The second profile $\Vr_{2}$ exhibits two distinct restrictions on regularity.
\begin{itemize}
\item As for $\Vr_{1}$, the algebraic degeneracy in \eqref{eq: ODE of Vr_2} is compatible with the characteristic boundary data. Away from vacuum, the formulas for $\Vr_{2}$ have the regularity supplied by those data.
\item The logarithmic term \textbf{I} in
\eqref{eq: exact logarithmic term of Vr_2} comes from resonance in
the equation for $\Vr_2^1-\Vr_2^2$: the contribution of $\Vr_1^3$
to the quadratic source has the same power $c^{2\beta}$ as the
homogeneous solution. Its vorticity dependence follows from
\[
\omega_r(0,x_2)
:=(\partial_{x_1}(v_r^2)-\partial_{x_2}(v_r^1))(0,x_2)
=\big(2\partial_{x_2}(w_r)+c_r^{-1}L(v_r^2)\big)(0,x_2).
\]
Thus its coefficient is proportional to
$c_r\partial_{x_2}(w_r)\,\omega_r$; see
Section~\ref{subsubsec: profiles and logarithms} for the calculation
and the scalar resonance identity. In particular, this term vanishes
in the irrotational case, and nonzero vorticity contributes to it
through the tangential variation of $w_r$.

For $\gamma=5/3$, one has $\beta=2$. When
$\partial_{x_2}(w_r)\,\omega_r$ is nonzero, the logarithmic term gives
a nonzero $c^{-1}$ contribution to
$\partial_{\urr}^5(\Vr_2)$ near vacuum, whereas
\[
\left|\partial_{\urr}^a\partial_{x_2}^b(\Vr_2)\right|
\leq C\varepsilon\bigl(1+|\ln c|\bigr),\qquad a+b\leq4.
\]
The four-derivative $L^2$ bounds therefore remain uniform as the
vacuum cutoff is removed, while the fifth-derivative bound generally
does not. For $\gamma=7/5$, the corresponding logarithmic factor is
$c^6\ln(c/c_r)$ and is compatible with the same four-derivative
bound. We use $\Ntop=4$ in both cases; the fifth-derivative obstruction
just described is specific to $\gamma=5/3$.

\item The term \textbf{II} also limits strong Taylor regularity.
Here $c$ in the profile formulas is the initial affine sound speed
$c_0(\mathring u,x_2)$, rather than the evolving sound speed.
For $\gamma=1.3$, $\beta=23/6$, and on the initial fan
\[
\partial_{\mathring u}^4(\Vr_1)=O(c_0^{-1/6}),\qquad
\partial_{\mathring u}^4(\Vr_2)
=a(x_2)c_0^{-7/6}+\hbox{less singular terms}.
\]
If the coefficient $a$ is nonzero on an interval, then $\Vr_1$
has the asserted four-derivative $L^2$ regularity whereas $\Vr_2$
does not: $\int_0^\eta c_0^{-7/3}\,dc_0=\infty$.
Thus no twice differentiable $H^4$-valued Taylor expansion on the
fixed initial fan can have these coefficients. In particular, the
initial-coefficient bounds used for $k=0$ cannot be upgraded to the
$k=1$ bounds for these data.

The fractional-power calculation is on the fixed initial fan
$c_0\geq0$. In the special-exponent estimates with $k=1$, $\Vr_1$ is
polynomial in $c_0$ and is evaluated on the moving fan slices.
\end{itemize}
\end{remark}

\subsection{Initial values of the solution and its characteristic derivatives}
To treat the initial values of $U$ and $L(U)$ in both $\mathcal{W}_{1}$ and $\mathcal{W}_{2, \alpha}$, write $u=-\mathring{u}+h(x_{2})$ for either characteristic label $u$ or $\tilde{u}$, with
\[
h(x_2)=\begin{cases}
(v_{r}^{1}+c_{r})(0, x_{2}), \ \ \textbf{in $\mathcal{W}_{1}$},\\
(v_{r}^{1}+c_{r})(0, x_{2})-\frac{\gamma+1}{\gamma-1}c_{r}(0, x_{2}),\quad\textbf{in $\mathcal W_{2,\alpha}$}.
\end{cases}
\]
Since $u=-\mathring{u}+h(x_{2})$, $U(0, u, \vartheta)$ can be obtained directly in view of \eqref{eq: data on t=0} as follows,
\begin{equation}\notag
U(0, u, \vartheta)=\begin{pmatrix}
\frac{2}{\gamma-1}c_{r}(0, \vartheta)+\frac{2}{\gamma+1}\big( -u+h(\vartheta)-(v_{r}^{1}+c_{r})(0, \vartheta) \big)\\
v_{r}^{1}(0, \vartheta)+\frac{2}{\gamma+1}\big( -u+h(\vartheta)-(v_{r}^{1}+c_{r})(0, \vartheta) \big)\\
v_{r}^{2}(0, \vartheta)
\end{pmatrix}
\end{equation}
On the initial slice $\Sigma_{0}$, the identities $T=\frac{\partial}{\partial u}, \hat{X}=\frac{\partial}{\partial \vartheta}$ give
\begin{equation}\notag
\begin{cases}
\wb(0, u, \vartheta)=\frac{1}{\gamma-1}c_{r}(0, \vartheta)+\frac{1}{2}v_{r}^{1}(0, \vartheta)+\frac{2}{\gamma+1}\big( -u+h(\vartheta)-(v_{r}^{1}+c_{r})(0, \vartheta) \big),\\
w(0, u, \vartheta)=\frac{1}{\gamma-1}c_{r}(0, \vartheta)-\frac{1}{2}v_{r}^{1}(0, \vartheta),\\
\psi_{2}(0, u, \vartheta)=-v_{r}^{2}(0, \vartheta).
\end{cases}
\end{equation}

To compute $L(U)$, use the expression for $L$ in the opened coordinates $(t, \mathring{u}, x_{2})$:
\[
L=\frac{\partial}{\partial t}+\frac{1}{t}\big( v^{1}-\mathring{u}-c\hat{T}^{1} \big)\frac{\partial}{\partial \mathring{u}}+\big( v^{2}-c\hat{T}^{2} \big)\frac{\partial}{\partial x_2}.
\]
and we have
\[
\begin{split}
\lim_{t\rightarrow 0} \frac{1}{t}\big( v^{1}-\mathring{u}-c\hat{T}^{1} \big)&=\lim_{t\rightarrow 0} \frac{1}{t}\big( v^{1}-\mathring{u}+c \big)\\
&=\partial_t(v^{1}+c)(0,\mathring u,x_2)
\end{split}
\]
and
\[
\lim_{t\rightarrow 0} \big( v^{2}-c\hat{T}^{2} \big)=v_{r}^{2}(0, x_{2}).
\]
Therefore, we have
\[
L(U)(0, \mathring{u}, x_{2})=\Vr_{1}+(\Vr_{1}^{2}+\frac{\gamma-1}{2}\Vr_{1}^{1})\frac{\partial \Vr_{0}}{\partial \mathring{u}}+v_{r}^{2}\frac{\partial \Vr_{0}}{\partial x_2}.
\]
The corresponding formulas for $L(\wb), L(w), L(v^{2})$ are
\begin{equation}\label{eq: formula for L(psi)}
\begin{cases}
L(\wb)(0, \mathring{u}, x_{2})=-\frac{1}{2}c\cdot \partial_{x_2}(v_{r}^{2}),\\
L(w)(0, \mathring{u}, x_{2})=-\frac{\gamma+1}{2(3-\gamma)}\Big( c-c_r\big(\frac{c}{c_{r}}\big)^{\frac{\gamma+1}{2(\gamma-1)}} \Big)\cdot\partial_{x_2}(v_{r}^{2})+\big( \frac{c}{c_{r}} \big)^{\frac{1}{2}\frac{\gamma+1}{\gamma-1}}\cdot L(w_{r}),\\
L(v^{2})(0, \mathring{u}, x_{2})=-2c\cdot \partial_{x_2}(w_{r})+\big(\frac{c}{c_{r}}\big)^{\frac{\gamma+1}{\gamma-1}}\cdot \big( 2c_{r}\partial_{x_2}(w_{r})+L(v_{r}^{2}) \big).
\end{cases}
\end{equation}
Therefore, we can conclude that for $Z\in \{T, \hat{X}\}, |\alpha|\leq N, 1\leq |\beta|\leq N$ and $\psi\in \{\wb, w, \psi_{2}\}$, on $\Sigma_{0}^{0, u^{*}}\cup \Sigma_{0}^{\ut_{*}, \ut^{*}}$,
\begin{equation}\notag
\begin{cases}
T(\wb)+\frac{2}{\gamma+1}=0, T(w)=0, T(\psi_{2})=0, \\
TZ^{\beta}(\psi)=0,\\
|(\hat{X})^{m}(\psi)|\leq C\varepsilon,\quad 1\leq m\leq N, \\
|c^{-1}L(\psi)|\leq C\varepsilon.
\end{cases}
\end{equation}

\subsection{Approximate solutions and residual bounds} 

The regularity of $\Vr_{1}$ and $\Vr_{2}$ determines the following choice of approximate solution:
\begin{equation}\notag
\begin{cases}
\Ur_{0}(t, \mathring{u}, x_{2}):=\Vr_{0}(\mathring{u}, x_{2}),\\
\Ur_{1}(t, \mathring{u}, x_{2}):=\Vr_{0}(\mathring{u}, x_{2})+t\cdot \Vr_{1}(\mathring{u}, x_{2}).
\end{cases}
\end{equation}
By direct calculation, we have
\[
t\partial_{t}(\Ur_{0})+\big(A_{1}(\Ur_{0})-\mathring{u}\big)\partial_{\mathring{u}}(\Ur_{0})+t\cdot A_{2}(\Ur_{0})\partial_{x_2}(\Ur_{0})=t\cdot A_{2}(\Ur_{0})\partial_{x_2}(\Ur_{0}).
\]
for $\Vr_{0}$. We also have
\[
\begin{split}
&t\partial_{t}(\Ur_{1})+\big(A_{1}(\Ur_{1})-\mathring{u}\big)\partial_{\mathring{u}}(\Ur_{1})+t\cdot A_{2}(\Ur_{1})\partial_{x_2}(\Ur_{1})\\
&=t\cdot \Big( \Vr_{1}+A_{1}(\Vr_{1})\partial_{\mathring{u}}(\Vr_{0})+\big( A_{1}(\Vr_{0})-\mathring{u} \big)\partial_{\mathring u}(\Vr_{1})+A_{2}(\Vr_{0})\partial_{x_2}(\Vr_{0}) \Big)\\
&+t^{2}\cdot \Big( A_{1}(\Vr_{1})\partial_{\mathring{u}}(\Vr_{1})+A_{2}(\Vr_{0})\partial_{x_2}(\Vr_{1})+A_{2}(\Vr_{1})\partial_{x_2}(\Vr_{0}) \Big)+t^{3}\cdot A_{2}(\Vr_{1})\partial_{x_2}(\Vr_{1})\\
&=t^{2}\cdot \Big( A_{1}(\Vr_{1})\partial_{\mathring{u}}(\Vr_{1})+A_{2}(\Vr_{0})\partial_{x_2}(\Vr_{1})+A_{2}(\Vr_{1})\partial_{x_2}(\Vr_{0}) \Big)+t^{3}\cdot A_{2}(\Vr_{1})\partial_{x_2}(\Vr_{1}).
\end{split}
\]

In summary, we have 
\begin{lemma}\label{lem: bounds on Res}
\begin{equation}\notag
\begin{cases}
t\partial_{t}(\Ur_{k})+\big(A_{1}(\Ur_{k})-\mathring{u}\big)\partial_{\mathring{u}}(\Ur_{k})+t\cdot A_{2}(\Ur_{k})\partial_{x_2}(\Ur_{k})=\mathring{R}_{k}, \ \ k=0, 1,\\
\mathring{R}_{0}=t\cdot A_{2}(\Ur_{0})\partial_{x_2}(\Ur_{0}), \\
\mathring{R}_{1}=t^{2}\cdot \Big( A_{1}(\Vr_{1})\partial_{\mathring{u}}(\Vr_{1})+A_{2}(\Vr_{0})\partial_{x_2}(\Vr_{1})+A_{2}(\Vr_{1})\partial_{x_2}(\Vr_{0}) \Big)+t^{3}\cdot A_{2}(\Vr_{1})\partial_{x_2}(\Vr_{1}).
\end{cases}
\end{equation}
Moreover, the following bounds on $\mathring{R}_{k}$, for any $a+b\leq N$, 
\begin{itemize}
\item When $k=0$, for any $\gamma\in (1, 3)$, 
\begin{equation}\notag
|\frac{\partial^{a} }{\partial \mathring{u}^{a}}\frac{\partial^{b}}{\partial x_{2}^{b}}(\mathring{R}_{0})|(t, \mathring{u}, x_{2})\leq C\varepsilon t.
\end{equation}
in both \textbf{away from vacuum} case and \textbf{near vacuum} case.
\item When $k=1$, for every $\gamma\in(1,3)$ away from vacuum,
and for $\gamma\in\{7/5,5/3\}$ near vacuum,
\begin{equation}\notag
|\frac{\partial^{a} }{\partial \mathring{u}^{a}}\frac{\partial^{b}}{\partial x_{2}^{b}}(\mathring{R}_{1})|(t, \mathring{u}, x_{2})\leq C\varepsilon t^{2}.
\end{equation}
\end{itemize}
\end{lemma}

The calculations in Subsections~3.1--3.3 prove Propositions~\ref{prop: initial energy bounds on Sigma_0} and~\ref{prop: formulas for c^-1 L(psi)}. The residual identity and its estimates prove Proposition~\ref{prop: Res bounds of Ur_k}. This completes the proofs of the initial-profile and residual results stated in Section~\ref{sec: high order approximate solution}.

\section{Boundary matching and normal-derivative estimates}\label{sec: normal jets}
The energy estimates require bounds for the profile remainder and its
derivatives along $C_0$. Tangential derivatives follow from the matching
with the exterior solution. For normal derivatives, we use the
characteristic equations: two projections are determined algebraically,
while the remaining projection is coupled to the transport equation for
$\kappa$. We first establish tangential matching and initial
compatibility, then derive the normal-derivative bounds and apply them
to the remainder.

These estimates prove Proposition~\ref{prop: matching of U and Ur_k along C_0}.
Subsection~4.1 treats tangential matching, Subsection~4.2 proves
compatibility at $\Sigma_0\cap C_0$, Subsection~4.3 derives the normal
jets along $C_0$, and Subsection~4.4 applies them to the profile remainder.

We use the regularity and constant conventions of
Section~\ref{subsec: notation conventions}. In particular, $N$ may be
reduced after a trace, a division by $t$, or a normal-derivative
estimate; it remains large enough for the fixed bootstrap orders.

\subsection{Tangential matching along the exterior boundary}
We use $\mathbf{U}_{*}$ to denote the constant state $(\frac{2\mathbf{c}_{*}}{\gamma-1}, 0, 0)^{T}$, recall that
\[
\|U_{r}(0,\cdot,\cdot)-\mathbf{U}_{*}\|_{H^{N}([0,\infty)\times\mathbb T)}\leq \varepsilon.
\] 
By classical theory, there exists $C^{N}$ surface $C_{0}: x_{1}=f(t, x_{2}), (t, x_{2})\in [0, t^{*}]\times \mathbb{T}$ emanating from $x_{1}=0$ and solution $U_{r}(t, x_{1}, x_{2})$ defined on the region determined by $U_{r}$ on $x_{1}\geq 0$ completely,
\[
\mathcal{D}_{r}=\{(t, x_{1}, x_{2}), (t, x_{2})\in [0, t^{*}]\times \mathbb{T}, x_{1}\geq f(t, x_{2})\}
\]
Moreover, we have
\begin{equation}\notag
\begin{cases}
|\partial^{a}_{t}\partial^{b}_{x_2}\big(\frac{f(t, x_{2})}{t}-\mathbf{c}_{*}\big)|\leq C\varepsilon,\ \ a+b\leq N,\\
\|U_{r}-\mathbf{U}_{*}\|_{C^{N}(\mathcal{D}_{r})}\leq C\varepsilon.
\end{cases}
\end{equation}
In particular, since $\mathring{u}=\frac{x_1}{t}$, $C_{0}$ is also a $C^{N}$ surface in coordinate $(t, \mathring{u}, x_{2})$ parametrized by 
\[
\mathring{u}:=g(t, x_{2})=\frac{f(t, x_{2})}{t}=\int_{0}^{1} \frac{\partial f}{\partial t}(s\cdot t, x_2)ds.
\]
and we have 
\[
|\partial^{a}_{t}\partial^{b}_{x_2}\big(g-(v_{r}^{1}+c_{r})(0, x_{2})\big)|(t, x_2)\leq C\varepsilon t^{\max\{1-a,0\}},\ \ a+b\leq N.
\]
Since $U=U_{r}$ along $C_{0}$, we have
\[
U\big(t, g(t, x_{2}), x_{2}\big)=U_{r}\big(t, f(t, x_{2}), x_{2}\big).
\]
Using $(t,x_2)$ as coordinates on $C_0$, we have for $a+b\leq N$
\[
\begin{cases}
|\partial_{t}^{a}\partial_{x_2}^{b}(U-\mathbf{U}_{*})|(t, x_{2})\leq C\varepsilon,\\
|\partial_{t}^{a}\partial_{x_2}^{b}(\Ur_{k}-\mathbf{U}_{*})|(t, x_{2})\leq C\varepsilon.
\end{cases}
\]
Moreover, in view of the construction of $\Ur_{k}$, we have
\[
\partial^{j}_{t}(U-\Ur_{k})(0, x_{2})=0, \ 0\leq j\leq k.
\] 
Therefore, we can conclude that for $a+b\leq N$ along $C_{0}$,
\[
|\partial_{t}^{a}\partial_{x_2}^{b}(U-\Ur_{k})|(t, x_{2})\leq C\varepsilon t^{\textbf{max}\{k+1-a, 0\}}.
\]
To express these estimates in the geometric frame, use the characteristic coordinates $(t, u, \vartheta)$ from \eqref{eq: first characteristic coordinate}. Then $C_{0}=\{u=0\}$ is parametrized by $(t, \vartheta)\in [0, t^{*}]\times \mathbb{T}$. The Jacobian of $(t, \vartheta)\mapsto (t, x_{2})$ is close to the identity: for $a+b\leq N$,
\[
|\partial^{a}_{t}\partial^{b}_{\vartheta}\Big(\begin{pmatrix}
1 & 0 \\
(v^{2}-c\hat{T}^{2}) & \sqrt{\slashed{g}}\hat{X}^{2}
\end{pmatrix}-\begin{pmatrix}
1 & 0\\
0 & 1
\end{pmatrix}\Big)|\leq C\varepsilon.
\]
Therefore, we can conclude that for $a+b\leq N$ along $C_{0}$,
\[
|\partial_{t}^{a}\partial_{\vartheta}^{b}(U-\Ur_{k})|(t, x_{2})\leq C\varepsilon t^{\textbf{max}\{k+1-a, 0\}}.
\]
To summarize, we list all the bounds in frame $\slashed{\mathfrak{d}}\in \{L, \hat{X}\}$ and $|\alpha|\leq N, a\leq N$ for later use,
\begin{equation}\notag
\begin{cases}
|\slashed{\mathfrak{d}}^{\alpha}(U-\mathbf{U}_{*})|\leq C\varepsilon,\\
|\slashed{\mathfrak{d}}^{\alpha}(U-\Ur_{k})|\leq C\varepsilon t^{\textbf{max}\{k+1-\#(L, \slashed{\mathfrak{d}}^{\alpha}), 0\}},\\
|(\hat X)^{a}(\hat{T}^{1}+1)|\leq C\varepsilon^{2} t^{2}, |(\hat X)^{a}(\hat{T}^{2})|\leq C\varepsilon t, |(\hat X)^{a}(\frac{\kappa}{t}-1)|\leq C\varepsilon t,\\
|(\hat X)^{a}(\slashed{g}-1)|\leq C\varepsilon t.
\end{cases}
\end{equation}
where $\#(L, \slashed{\mathfrak{d}})$ denotes the number of $L$ appearing in $\slashed{\mathfrak{d}}^{\alpha}$.
\begin{remark}
We remark the bounds on $\hat{T}^{1}, \hat{T}^{2}, \kappa, \slashed{g}$ can be derived exactly the same way as section \ref{sec: Preparations for the energy estimate} via transport equations and initial data \eqref{eq: formulas to control the geometry} for them.
\end{remark}

\subsection{Compatibility at the initial boundary}
This subsection proves Corollary~\ref{cor: compatibility on S_0,0}. We first derive the characteristic constraints for the normal derivatives, then check the pure normal constraints and the mixed opened-coordinate identities at $\Sigma_0\cap C_0$.

\subsubsection{The characteristic form of the Euler equations}
To determine normal derivatives $T(U)$, we can rewrite Euler equations in frame $\{L, T, \hat{X}\}$ as follows,
\begin{equation}\notag
\big( A_{j}\hat{T}^{j}+\lambda_{3} \big)\cdot T(U)=
-\kappa L(U)-c\kappa\cdot \begin{pmatrix}
0&\hat{X}^{1}&\hat{X}^{2}\\
\hat{X}^{1}&0&0\\
\hat{X}^{2}&0&0
\end{pmatrix}\hat{X}(U)
\end{equation}
Here $\lambda_3=\lambda_3(U,-\hat T)$, $r_i=r_i(U,-\hat T)=\alpha_i$, and $e_1=(1,0,0)^T$. Taking projections on $r_{1}, r_{2}, r_{3}$, we have
\begin{equation}\label{eq: formula of (T(U), r_i)}
\begin{cases}
\big( T(U), r_{1} \big)=-\frac{\kappa}{2c}\big( L(U), r_{1} \big)-\frac{\sqrt{2}}{4}\kappa\cdot \big( \hat{X}(U), r_{2} \big),\\
\big( T(U), r_{2} \big)=-\frac{\kappa}{c}\big( L(U), r_{2} \big)-\kappa\cdot \big( \hat{X}(U), e_{1} \big),\\
\big( L(U), r_{3} \big)+\frac{\sqrt{2}}{2}c\cdot \big( \hat{X}(U), r_{2} \big)=0\\
\end{cases}
\end{equation}
The algebraic relations \eqref{eq: formula of (T(U), r_i)}
determine the first two projections of $T(U)$ from tangential data
and $\kappa$. They leave the third projection undetermined.
The differentiated Euler constraint gives the third projection its
transport equation
\eqref{eq: ODE for T(psi) and kappa}, and the initial profile fixes
its initial value. 

For higher normal derivatives, let $Z\in \{T, \hat{X}\}$ and introduce the following schematic families.
An ordered word $Z^{\leq j}$ ranges over all words of length from
zero through $j$, with the word of length zero equal to the identity.
If $j<0$, the family is empty and the corresponding terms are absent.
Products of families denote the products occurring in finite linear
combinations, with smooth bounded coefficients on the fixed nonvacuum
region. Only the displayed orders and the indicated small factors
are recorded. The index $m$ ranges over $1\leq m\leq N$, where the
regularity of the initial data is understood as in
Section~\ref{subsubsec: regularity orders}.
\[
\begin{split}
\textbf{knowns}_{m}:&=\{c^{-1}, \kappa\}\cup \{Z^{\leq m-1}(\hat{T}^{i})\}\cup\{ (\hat X)^{a}T^{b}\big(\frac{\kappa}{t}\big), a+b\leq m-1, b\leq m-2 \}\\
& \cup \{(\hat X)^{\leq a}T^{\leq b}(U), a+b\leq m, b\leq m-1\}\cup \{L^{a}(\hat X)^{b}(U), a+b \leq m\}.
\end{split}
\]
and
\[
\begin{cases}
\textbf{good knowns}_{1, m}:=\textbf{knowns}_{m}\cdot \{L^{a}(\hat X)^{b}(U-\mathbf{U}_{*}), a+b\leq m\},\\
 \textbf{good knowns}_{2, m}:=\begin{cases}
\begin{aligned}[t]\textbf{knowns}_{m}\\ &\quad\cdot \{\kappa, (\hat X)^{a}T^{b}\big(\frac{\kappa}{t}-1\big), a+b\leq m-1, b\leq m-2 \} \\ &\quad\cdot \{L^{a}(\hat X)^{b}(U-\mathbf{U}_{*}), a+b\leq m\},\end{aligned}\\
\textbf{knowns}_{m}\cdot Z^{\leq m-2}T\hat{X}(U),\\
\textbf{knowns}_{m}\cdot Z^{\leq m-2}Z(\hat{T}^{i}).
\end{cases}
\end{cases}
\]
Then we have
\begin{lemma}
For $1\leq m\leq N$ and $i\in \{1, 2\}$,
\begin{equation}\label{eq: higher order formula of (T(U), r_i)}
T^{m-1}\big( T(U), r_{i}\big)=\textbf{good knowns}_{1,m}\cdot T^{\leq m-1}(\kappa)+\textbf{good knowns}_{2,m}.
\end{equation}
\end{lemma}

\begin{proof}
The base case $m=1$ follows from \eqref{eq: formula of (T(U), r_i)}. And we can assume that we have
\[
T^{n-1}\big( T(U), r_{i}\big)=\textbf{good knowns}_{1,n}\cdot T^{\leq n-1}(\kappa)+\textbf{good knowns}_{2,n}.
\]
Here the coefficient matrix in the tangential term is
\[
M(\hat T,U)=-c\begin{pmatrix}
0&\hat X^1&\hat X^2\\
\hat X^1&0&0\\
\hat X^2&0&0
\end{pmatrix};
\]
it is distinct from the scalar bootstrap constant $M$.
Before handling the case $m+1$, we derive a schematic formula for $LZ^{\alpha}(U)$ with $Z^{\alpha}\not= T^{m}$,
\[
\begin{split}
\kappa\cdot LZ^{\alpha}(U)&=\kappa\cdot Z^{\alpha}L(U)+\textbf{good knowns}_{2, m+1}\\
&=\kappa\cdot Z^{\alpha}\big( \frac{1}{\kappa}\cdot \kappa L(U) \big)+\textbf{good knowns}_{2, m+1}\\
&=\kappa\cdot Z^{\alpha}\Big( \frac{1}{\kappa}\cdot \big( -(A_{j}\hat{T}^{j}+\lambda_{3})T(U)+\kappa M(\hat{T}, U)\hat{X}(U) \big) \Big)+\textbf{good knowns}_{2, m+1}\\
&=\textbf{good knowns}_{2, m+1}.
\end{split}
\]
Therefore, we have
\[
\begin{cases}
\kappa L(\textbf{good knowns}_{1, m})= \textbf{good knowns}_{2, m+1},\\
\kappa L(\textbf{good knowns}_{2, m})= \textbf{good knowns}_{2, m+1}.
\end{cases}
\]

To handle the case $m+1$, we apply $\kappa L$ on both sides of
\[
T^{m-1}\big( T(U), r_{i}\big)=\textbf{good knowns}_{1,m}\cdot T^{\leq m-1}(\kappa)+\textbf{good knowns}_{2,m}.
\]
and we have
\[
\begin{split}
\textbf{good knowns}_{2, m+1}&=\kappa LT^{m-1}\big( T(U), r_{i}\big)\\
&=\kappa T^{m-1}\Big( T\big(\frac{1}{\kappa}\cdot \kappa L(U)\big), r_{i} \Big)+\textbf{good knowns}_{2, m+1}\\
&=-T^{m}(\kappa)\cdot \big( L(U)\cdot r_{i} \big)\\
&\quad+\sum_{m_{1}+m_{2}\leq m, m_{2}\leq m-1}\prod_{j, 1\leq b_{j}\leq m-1} \big(\frac{T^{b_{j}}(\kappa)}{\kappa}\big)\cdot \Big( T^{m_1}\big(\kappa L(U)\big), T^{m_2}(r_{i}) \Big)\\
&+\textbf{good knowns}_{2, m+1}\\
&=-T^{m}(\kappa)\cdot \big( L(U)\cdot r_{i} \big)+\big( T^{m}(\kappa L(U)), r_{i} \big)+\textbf{good knowns}_{2, m+1}
\end{split}
\]
It suffices to handle $\Big( T^{m}\big(\kappa L(U)\big), r_{i} \Big)$
\[
\begin{split}
\Big( T^{m}\big(\kappa L(U)\big), r_{i} \Big)&=\Big( T^{m}\big( -(A_{j}\hat{T}^{j}+\lambda_{3})T(U) \big), r_{i} \Big)+\Big( T^{m}\big( \kappa M(\hat{T}, U)\hat{X}(U) \big), r_{i} \Big)\\
&=\Big( T^{m}\big( -\sum_{j=1}^{2}(\lambda_{3}-\lambda_{j})(T(U), r_{j})r_{j} \big), r_{i} \Big)+\textbf{good knowns}_{2, m+1}\\
&=-(\lambda_{3}-\lambda_{i})\cdot \Big( T^{m}\big( T(U), r_{i} \big) \Big)+\textbf{good knowns}_{2, m+1}.
\end{split}
\]
where we have used the induction assumption for cases that $1\leq n\leq m$. Note that
\[
\lambda_{3}-\lambda_{1}=2c, \lambda_{3}-\lambda_{2}=c.
\]
Therefore, we have for $i=1, 2$,
\[
T^{m}\big( T(U), r_{i} \big)=\textbf{good knowns}_{1, m+1}\cdot T^{m}(\kappa)+\textbf{good knowns}_{2, m+1}.
\]
This proves the case $m+1$ and hence the lemma.
\end{proof}

For mixed derivatives, fix $1\leq m\leq n\leq N$ and define
$\textbf{knowns}_{m,n}$ so that $n$ records the total derivative
order and $m$ records the number of normal derivatives.
\[
\begin{split}
\textbf{knowns}_{m, n}:&=\{c^{-1}, \kappa\}\\
&\cup \{Z^{\leq n-1}(\hat{T}^{i}), \#(T, Z^{\leq n-1})\leq m-1\}\\
&\cup\{ Z^{\leq n-1}(\frac{\kappa}{t}), \#(T, Z^{\leq n-1})\leq m-2 \}\\
& \cup \{Z^{\leq n}(U), \#(T, Z^{\leq n})\leq m-1\}\\
&\cup \{\slashed{\mathfrak{d}}^{\leq n}(U)\}.
\end{split}
\]
where we use the convention $Z\in \{T, \hat{X}\}, \slashed{\mathfrak{d}}\in \{L, \hat{X}\}$ and $\#(T, Z^{\alpha})$ counts the occurrences of $T$ in $Z^{\alpha}$. We also introduce the notations $\textbf{good knowns}_{1, m, n}$ and $\textbf{good knowns}_{2, m, n}$.
\[
\begin{cases}
\textbf{good knowns}_{1, m, n}:=\textbf{knowns}_{m, n}\cdot \{\slashed{\mathfrak{d}}^{\leq n}(U-\mathbf{U}_{*})\},\\
 \textbf{good knowns}_{2, m, n}:=\begin{cases}
\begin{aligned}[t]\textbf{knowns}_{m,n}\cdot\Big(&\{\kappa\}\\
&\cup\{Z^{\leq n-1}(\frac{\kappa}{t}-1),\ \#(T,Z^{\leq n-1})\leq m-2\}\Big)\\
&\cdot\{\slashed{\mathfrak{d}}^{\leq n-1}(U-\mathbf U_*)\},\end{aligned}\\
\textbf{knowns}_{m, n}\cdot \{\hat{X}Z^{\leq n-1}(U), 1\leq \#(T, Z^{\leq n-1})\leq m-1\},\\
\textbf{knowns}_{m, n}\cdot \{Z^{\leq n-2}Z(\hat{T}^{i}), \#(T, Z^{\leq n-2}Z)\leq m-1\}.
\end{cases}
\end{cases}
\]
Applying $(\hat X)^{n-m}$ on \eqref{eq: higher order formula of (T(U), r_i)}, we can conclude that 
\begin{equation}\notag
(\hat X)^{n-m}T^{m-1}\big( T(U), r_{i} \big)=(\hat X)^{n-m}T^{m-1}(\kappa)\cdot \textbf{good knowns}_{1, m, n}+\textbf{good knowns}_{2, m, n}.
\end{equation}

We now define what compatibility of the two sets of data means.
Use the opened coordinates $(t,\urr,x_2)$, in which $\Sigma_0=\{t=0\}$
and $C_0=\{\urr=g(t,x_2)\}$, with $g=f/t$ extended to $t=0$ as above.
By compatibility through order $N$ we mean that at each point of
$\Sigma_0\cap C_0$ the formal derivatives through total order $N$
from the two prescribed restrictions agree as derivatives of a
single formal $N$-jet of the normalized gas state. More precisely,
write this jet in the coordinates $(t,\urr-g(t,x_2),x_2)$ so that its
restrictions to the two coordinate faces are the Taylor jets of
the prescribed data there. In the opened coordinates the jet
satisfies, on $\{t=0,\urr=g(0,x_2)\}$,
\[
\partial_t^a\partial_{\urr}^b\partial_{x_2}^d
\big[t\partial_t(U)+(A_1(U)-\urr)\partial_{\urr}(U)
       +t A_2(U)\partial_{x_2}(U)\big]=0,
\qquad a+b+d\leq N-1.
\]
This is the compatibility induced by the opened Euler equations
multiplied by $t$. The normal components not fixed by the algebraic
characteristic constraints are initialized by the affine profile
and transported along $C_0$. We use this compatibility for the
already constructed nonvacuum solution.

\begin{corollary}\label{cor: compatibility on S_0,0}
In the opened coordinates $(t,\urr,x_2)$, the prescribed restrictions
\eqref{eq: data on t=0} on $\Sigma_0$ and
\eqref{eq: continu along C0} on $C_0$ for the nonvacuum solution
of (B) are compatible through order $N$ in the above sense, for
$1\leq N\leq J_4$ in Section~\ref{subsubsec: regularity orders}.
In particular, the two
constrained normal components satisfy
$T^{m-1}(T(U),r_i)=0$ on $\{t=0,\urr=g(0,x_2)\}$ for $1\leq m\leq N$ and $i=1,2$.
\end{corollary}
\begin{proof}
On $\Sigma_{0}$, we have
\[
U(0, u, \vartheta)=\begin{pmatrix}
\frac{2}{\gamma-1}c_{r}(0, \vartheta)-\frac{2}{\gamma+1}u\\
v^{1}_{r}(0, \vartheta)-\frac{2}{\gamma+1}u\\
v_{r}^{2}(0, \vartheta)
\end{pmatrix}
\]
and $r_{1}=\frac{\sqrt{2}}{2}(1, -1, 0)^T, r_{2}=(0, 0, 1)^T$. Hence, for any $m\geq 1$,
\[
\begin{cases}
T^{m-1}\big( T(U), r_{1} \big)=\sqrt{2} (\frac{\partial}{\partial u})^{m} \big(w_{r}(0, x_{2})\big)=0, \\
T^{m-1}\big( T(U), r_{2} \big)=(\frac{\partial}{\partial u})^{m} \big(v^{2}_{r}(0, x_{2})\big)=0.
\end{cases}
\]
On the other hand, \eqref{eq: higher order formula of (T(U), r_i)} implies that for any $m\leq N$,
\[
T^{m-1}\big( T(U), r_{i} \big)=0, \ i=1, 2.
\]
since $\frac{\kappa}{t}(0,u,\vartheta)=1$. The remaining component
is initialized by the affine profile, which gives the normal
derivatives displayed below. These identities verify the pure
normal constraints of the jet on $\{t=0,\urr=g(0,x_2)\}$.

For the mixed time and spatial derivatives in the opened
coordinates, use the nonvacuum solution of (B) with the
$C^{J_4}$ regularity in those coordinates up to $t=0$ supplied by
\cite{JiaLuo}. Its Taylor jet on $\{t=0,\urr=g(0,x_2)\}$ restricts to the two
prescribed traces. The multiplied opened Euler equation holds
for $t>0$; differentiating it through order $N-1$ and taking its
one-sided limit proves precisely the differentiated equations in
the definition. This also identifies the mixed formal derivatives
with one common jet. This proves the corollary.
\end{proof}

In summary, all normal jets of $U$ on $C_{0}\cap \Sigma_{0}$ can be given as follows,
\begin{equation}\notag
T(U)=\begin{pmatrix}
-\frac{2}{\gamma+1}\\
-\frac{2}{\gamma+1}\\
0
\end{pmatrix}, \ \ T^{m}(U)=\begin{pmatrix}
0\\
0\\
0
\end{pmatrix}, \ \ 2\leq m\leq N.
\end{equation}

\subsection{Normal derivatives along the exterior boundary}
This subsection proves the normal-jet estimates used in Proposition~\ref{prop: matching of U and Ur_k along C_0}. We derive a coupled transport system for the remaining normal component
and the foliation density. Commuting this system then gives the higher
normal-derivative estimates.

\subsubsection{Transport of the remaining normal component and the foliation density} Recall Euler equation in frame $\{L, T, \hat{X}\}$,
\[
\big( A_{j}\hat{T}^{j}+\lambda_{3} \big)\cdot T(U)=
-\kappa L(U)-c\kappa\cdot \begin{pmatrix}
0&\hat{X}^{1}&\hat{X}^{2}\\
\hat{X}^{1}&0&0\\
\hat{X}^{2}&0&0
\end{pmatrix}\hat{X}(U)
\]
and we have
\begin{equation}\label{eq: exact form in L T Xh}
\begin{cases}
\big( T(U), r_{1} \big)=-\frac{\kappa}{2c}\big( L(U), r_{1} \big)-\frac{\sqrt{2}}{4}\kappa\cdot \big( \hat{X}(U), r_{2} \big),\\
\big( T(U), r_{2} \big)=-\frac{\kappa}{c}\big( L(U), r_{2} \big)-\kappa\cdot \big( \hat{X}(U), e_{1} \big),\\
\big( L(U), r_{3} \big)+\frac{\sqrt{2}}{2}c\cdot \big( \hat{X}(U), r_{2} \big)=0\\
\end{cases}
\end{equation}
The transport calculation requires the exact coefficients in \eqref{eq: exact form in L T Xh}. Applying $T$ on both sides of $\big( L(U), r_{3} \big)+\frac{\sqrt{2}}{2}c\cdot \big( \hat{X}(U), r_{2} \big)=0$, we have
\begin{equation}\label{eq: qiji1}
\big( TL(U), r_{3} \big)+\big( L(U), T(r_{3}) \big)+\frac{\sqrt{2}}{2}T(c)\cdot \big( \hat{X}(U), r_{2} \big)+\frac{\sqrt{2}}{2}c\cdot \big( T\hat{X}(U), r_{2} \big)+\frac{\sqrt{2}}{2}c\cdot \big( \hat{X}(U), T(r_{2}) \big)=0.
\end{equation}
Applying $\hat{X}$ on both sides of $\big( T(U), r_{2} \big)=-\frac{\kappa}{c}\big( L(U), r_{2} \big)-\kappa\cdot \big( \hat{X}(U), e_{1} \big)$, we have
\begin{equation}\label{eq: qiji2}
\begin{aligned}\big( \hat{X}T(U), r_{2} \big)&=-\big( T(U), \hat{X}(r_{2}) \big)-\hat{X}(\kappa)\cdot c^{-1}\big( L(U), r_{2} \big)\\ &\quad-\kappa \hat{X}\Big( c^{-1}\big( L(U), r_{2} \big) \Big)-\hat{X}(\kappa)\cdot \big( \hat{X}(U), e_{1} \big)-\kappa\cdot \big( (\hat X)^2(U), e_{1} \big).\end{aligned}
\end{equation}
By \eqref{eq: higher order formula of (T(U), r_i)}, the component
$\big(T(U), r_{3}\big)$ must be estimated together with $\kappa$.
A term involving $\hat{X}(\kappa)$ would prevent closure of the desired
transport system. Combining \eqref{eq: qiji1} and \eqref{eq: qiji2}
eliminates $\big( \hat{X}T(U), r_{2} \big)$ and cancels all terms
containing $\hat{X}(\kappa)$. The resulting equations for
$\big( T(U), r_{3}\big)$ and $\kappa$ are
\[
\begin{split}
L\big( T(U), r_{3} \big)&=\big( T(U), L(r_3) \big)-\Big( \kappa\hat{T}^{j}\hat{X}(v^{j})+\big(T(U), r_{2}\big)-\kappa\hat{X}(c) \Big)\cdot\big( \hat{X}(U), r_{3} \big)-\frac{\sqrt{2}}{2}T(c)\cdot \big( \hat{X}(U), r_{2} \big)\\
&-\frac{\sqrt{2}}{2}c\cdot \Big( -\big( T(U), \hat{X}(r_{2}) \big)-\kappa \hat{X}\Big( c^{-1}\big( L(U), r_{2} \big) \Big)-\kappa\cdot \big( (\hat X)^2(U), e_{1} \big) \Big)\\
&+\frac{\sqrt{2}}{2}c\kappa \cdot\theta \big( \hat{X}(U), r_{2} \big),\\
L(\kappa)&=-\frac{\sqrt{2}}{4}(\gamma+1)\cdot \big( T(U), r_{3} \big)-\frac{\sqrt{2}}{4}(\gamma-3)\cdot \big( T(U), r_{1} \big).
\end{split}
\]
Therefore, we can conclude that
\begin{equation}\label{eq: ODE for T(psi) and kappa}
L\left(\begin{pmatrix}
\big( T(U), r_3 \big)\\
\kappa
\end{pmatrix}\right)=\begin{pmatrix}
f_{1} & f_{2} \\
-\frac{\sqrt{2}}{4}(\gamma+1) & f_{3}
\end{pmatrix}\begin{pmatrix}
\big( T(U), r_3 \big)\\
\kappa
\end{pmatrix}
\end{equation}
where $f_{1}, f_{2}, f_{3}$ are intrinsic to $C_{0}$ and we have $(|f_{1}|+|f_{2}|+|f_{3}|)(t, \vartheta)\leq C\varepsilon$. Solving \eqref{eq: ODE for T(psi) and kappa} with initial data 
\[
\begin{cases}
\big( T(U), r_{3}\big)(0, \vartheta)=-\frac{2\sqrt{2}}{\gamma+1}, \\
\kappa(0, \vartheta)=0.
\end{cases}
\]
We can conclude that along $C_{0}$,
\begin{equation}\notag
\begin{cases}
|\big( T(U), r_{3} \big)+\frac{2\sqrt{2}}{\gamma+1}|(t, \vartheta)\leq C\varepsilon t,\\
|\kappa-t|(t, \vartheta)\leq C\varepsilon t^{2},\\
|\big( T(U), r_{i} \big)|(t, \vartheta)\leq C\varepsilon t,\ i=1, 2.
\end{cases}
\end{equation}
in view of \eqref{eq: higher order formula of (T(U), r_i)}. To estimate the mixed derivatives $Z^{\alpha}(U)$ and the geometric derivatives $Z^{\alpha}(\hat{T}^{1})$, $Z^{\alpha}(\hat{T}^{2})$, and $Z^{\alpha}(\kappa)$, treat \eqref{eq: ODE for T(psi) and kappa}, \eqref{eq: higher order formula of (T(U), r_i)}, together with transport equation \eqref{eq: transport eq for geometry} for $\hat{T}^{i}$ as a coupled system and commute with $Z\in \{T, \hat{X}\}$. This yields
\begin{lemma}
for any $|\alpha|\leq N$ and $Z\in \{T, \hat{X}\}$, we have along $C_{0}$,
\begin{equation}\label{eq: first control of normal jets along C0}
\begin{cases}
|Z^{\alpha}\big(T(U)+\frac{2}{\gamma+1}(1, 1, 0)^T\big)|(t, \vartheta)\leq C\varepsilon t,\\
|Z^{\alpha}(\hat{T}^{1}+1)|(t, \vartheta)\leq C\varepsilon^{2} t^{2}, |Z^{\alpha}(\hat{T}^{2})|(t, \vartheta)\leq C\varepsilon t,\\
|Z^{\alpha}(\kappa-t)|(t, \vartheta)\leq C\varepsilon t^{2}.
\end{cases}
\end{equation}
\end{lemma}
\begin{remark}
For the away-from-vacuum solution used in (B) of
Section~\ref{subsec: proof main theorem}, the higher-order bounds
come from our earlier results \cite{JiaLuo}: the initial flux
estimates, the mixed boundary estimates, and the proposition on
regularity up to $t=0$, all of which allow vorticity. These give
the required jets of the Riemann variables and of the normalized
characteristic geometry on $C_0$.
The normalizations agree: $L(t)=1$, $\kappa/t\to1$,
$u|_{C_0}=0$, and $\vartheta|_{t=0}=x_2$.

To express those estimates in the present variables, use
$2c/(\gamma-1)=w+\wb$, $v^1=\wb-w$, and $v^2=-\psi_2$.
For mixed orders, move $T$ to the left using
$[\hat X,T]=\kappa\theta\hat X$. The frame estimates give
$Z^{\leq N-1}(\kappa\theta)=O(\varepsilon t^2)$, while the
remaining differentiated fields are bounded, so the commutator
terms satisfy the required $O(\varepsilon t)$ bound.
This proves \eqref{eq: first control of normal jets along C0}
for that constructed solution, after reducing $N$ to the available
order. The fixed positive sound-speed bound near $C_0$ makes this
estimate independent of the later vacuum cutoff.

The first-order calculation above gives a direct derivation of the
lowest jet. The derivative count allows two tangential derivatives
per normal derivative, as indicated by the term
$\kappa\big((\hat X)^2(U),e_1\big)$.
\end{remark}

The coupled transport system proves the stated normal-jet estimates on $C_0$.

\subsection{Normal derivatives of the profile remainder}
We now complete the proof of Proposition~\ref{prop: matching of U and Ur_k along C_0}. The normal-jet estimates above give the case $k=0$ directly; the remainder transport system gives the case $k=1$.

Since $T(\mathring{u})=\frac{\kappa\hat{T}^{1}}{t}, T(x_2)=\kappa\hat{T}^{2}, \hat{X}(\mathring{u})=\frac{\hat{T}^{2}}{t}, \hat{X}(x_{2})=\hat{X}^{2}$, \eqref{eq: first control of normal jets along C0} implies that along $C_{0}$, $Z\in \{T, \hat{X}\}$ and $|\alpha|\leq N$,
\[
|Z^{\alpha} \big( T(\Ur_{0})+\frac{2}{\gamma+1}(1, 1, 0)^T\big)|(t, \vartheta)\leq C\varepsilon t.
\]
Therefore, we have along $C_{0}$, $Z\in \{T, \hat{X}\}$ and $|\alpha|\leq N$,
\[
|Z^{\alpha}(U-\Ur_{0})|(t, \vartheta)\leq C\varepsilon t.
\]
Note also
\[
\begin{cases}
\Tr=-\frac{t}{\kappa}\hat{T}^{1}\cdot T-t\hat{X}^{1}\cdot \hat{X},\\
\Xr=\frac{\hat{T}^{2}}{\kappa}\cdot T+\hat{X}^{2}\cdot \hat{X}.
\end{cases}
\]
Therefore, we can conclude that for $a+b\leq \Ntop$ along $C_{0}$,
\begin{equation}\notag
|\Tr^{a}\Xr^{b}(U-\Ur_{0})|(t, \vartheta)\leq C\varepsilon t.
\end{equation}
This proves the case $k=0$ of Proposition \ref{prop: matching of U and Ur_k along C_0}.

To control higher order derivatives of $U-\Ur_{k}$ along $C_{0}$, we can project the equation for $V_{k}:=U-\Ur_{k}$ on direction $r_{1}, r_{2}$ as follows,
\begin{equation}\label{eq: U-Ur_1 in L T Xh}
\begin{cases}
\begin{aligned}[t]&\big( \kappa L(V_{k}), r_{1} \big)+2c\cdot \big( T(V_{k}), r_{1} \big)+\frac{\sqrt{2}}{2}c\kappa\cdot \big( \hat{X}(V_{k}), r_{2} \big)\\ &=\frac{\kappa}{t}\Big( -A_{1}(V_{k})\partial_{\mathring{u}}(\Ur_{k})-t\cdot A_{2}(V_{k})\partial_{x_2}(\Ur_{k}), r_{1} \Big)-\frac{\kappa}{t}(\mathring{R}_{k}, r_{1}),\end{aligned}\\
\begin{aligned}[t]&\big(\kappa L(V_{k}), r_{2} \big)+c\cdot \big( T(V_{k}), r_{2} \big)+c\kappa\cdot \big( \hat{X}(V_{k}), e_{1} \big)\\ &=\frac{\kappa}{t}\Big( -A_{1}(V_{k})\partial_{\mathring{u}}(\Ur_{k})-t\cdot A_{2}(V_{k})\partial_{x_2}(\Ur_{k}), r_{2} \Big)-\frac{\kappa}{t}(\mathring{R}_{k}, r_{2}),\end{aligned}\\
\begin{aligned}[t]&\big( L(V_{k}), r_{3} \big)+\frac{\sqrt{2}}{2}c\cdot \big( \hat{X}(V_{k}), r_{2} \big)\\ &=\frac{1}{t}\Big( -A_{1}(V_{k})\partial_{\mathring{u}}(\Ur_{k})-t\cdot A_{2}(V_{k})\partial_{x_2}(\Ur_{k}), r_{3} \Big)-\frac{1}{t}(\mathring{R}_{k}, r_{3}).\end{aligned}
\end{cases}
\end{equation}

The preceding normal-jet estimates are needed only at the finite
orders fixed in Section~\ref{subsubsec: regularity orders}.
For the nonvacuum solution actually used in (B), one can also
track the derivative count directly in its first characteristic
chart. The remainder and, for $k=1$, its first time derivative
vanish at $t=0$ as functions of the spatial labels. The fixed opened
spatial derivatives are combinations of spatial label derivatives
with $C^{J_4-1}$ coefficients. Thus taking these derivatives through
order $N_{\rm away}$ and Taylor expanding at fixed labels uses at
most $N_{\rm away}+2$ derivatives of the state and the corresponding
coordinate coefficients, all within the chosen $J_4$. This is an
alternative way of seeing why the finite loss in the boundary
calculation is covered by the choice of orders.

In view of \eqref{eq: first control of normal jets along C0} and lemma \ref{lem: bounds on Res}, $Z^{\alpha}(U), Z^{\alpha}(\hat{T}^{1}), Z^{\alpha}(\hat{T}^{2}), Z^{\alpha}(\kappa)$ can be viewed as \textbf{knowns}. Schematically, \eqref{eq: U-Ur_1 in L T Xh} implies that for $i=1, 2$,
\[
\big( T(V_{k}), r_{i} \big)=\kappa L(V_{k})+\hat{X}(V_{k})+V_{k}+t^{k+1}\varepsilon.
\]
Applying $\kappa L$ and $\hat{X}$ iteratively, we have
\[
(\hat X)^{n-m}T^{m-1}\big( T(V_{k}), r_{i} \big)=\cdot \sum_{a+b\leq n} \kappa^{a}L^{a}(\hat X)^{b}(V_{k})+t^{k+1}\varepsilon.
\]
Therefore, for $Z\in\{T,\hat X\}$ and $|\alpha|+1\leq N$, we have
\begin{equation}\label{eq: bound for T(V_k) on C_0}
|Z^{\alpha}\big( T(V_{k}), r_{i} \big)|(t, \vartheta)\leq C\varepsilon t^{k+1}, \ i=1, 2.
\end{equation}
along $C_{0}$. As for the remaining term $\big( T(V_{k}), r_{3}\big)$, we can derive transport equation for $\big( T(V_{k}), r_{3} \big)$ exactly the same way as the case for $\big( T(U), r_{3} \big)$ via \eqref{eq: U-Ur_1 in L T Xh}.
\begin{equation}\label{eq: transport equation for TV_1}
L\big( T(V_{k}), r_{3} \big)=\underbrace{-\frac{1}{t}\frac{1}{\gamma+1}(1-\hat{T}^{1})(\frac{\gamma-1}{2}-\hat{T}^{1})\cdot \big( T(V_{k}), r_{3} \big)}_{\textbf{damping terms since $\hat{T}^{1}\approx -1$}}+\sum_{a+b\leq 2, a\leq 1}L^{a}(\hat X)^{b}(V_{k})+t^{k}\varepsilon.
\end{equation}
Solving \eqref{eq: transport equation for TV_1} with initial data $\big( T(V_{k}), r_{3}\big)(0, \vartheta)=0$, we have 
\[
|\big(T(V_{k}), r_{3}\big)|(t, \vartheta)\leq C\varepsilon t^{k+1},
\]
along $C_{0}$. Commuting $(\hat X)^{n-m}T^{m-1}$ with \eqref{eq: transport equation for TV_1}, we have
\[
\begin{split}
L(\hat X)^{n-m}T^{m-1}\big( T(V_{k}), r_{3} \big)&=\underbrace{-\frac{1}{t}\frac{1}{\gamma+1}(1-\hat{T}^{1})(\frac{\gamma-1}{2}-\hat{T}^{1})\cdot (\hat X)^{n-m}T^{m-1}\big( T(V_{k}), r_{3} \big)}_{\textbf{damping terms since $\hat{T}^{1}\approx -1$}}\\
&+\sum_{a+b\leq n+1, a\leq 1, c\leq m-1}L^{a}(\hat X)^{b}T^{c}(V_{k})+t^{k}\varepsilon.
\end{split}
\]
which can be solved iteratively. Combining \eqref{eq: bound for T(V_k) on C_0}, we can conclude that: for any $Z\in \{T, \hat{X}\}, |\alpha|\leq N$, we have
\[
|Z^{\alpha}(U-\Ur_{k})|(t, \vartheta)\leq C\varepsilon t^{k+1}.
\]
along $C_{0}$. Since
\[
\begin{cases}
\Tr=-\frac{t}{\kappa}\hat{T}^{1}\cdot T-t\hat{X}^{1}\cdot \hat{X},\\
\Xr=\frac{\hat{T}^{2}}{\kappa}\cdot T+\hat{X}^{2}\cdot \hat{X}.
\end{cases}
\]
Therefore, we can conclude that: for $a+b\leq \Ntop$,
\begin{equation}\notag
|\Tr^{a}\Xr^{b}(U-\Ur_{1})|(t, \vartheta)\leq C\varepsilon t^{2}.
\end{equation}
This proves the case $k=1$ of Proposition \ref{prop: matching of U and Ur_k along C_0}. Together with the $k=0$ estimate above, this completes the proof of Proposition~\ref{prop: matching of U and Ur_k along C_0}.

\section{Geometric estimates and Sobolev inequalities}\label{sec: Preparations for the energy estimate}
This section establishes the geometric bounds needed in the energy and
pointwise estimates. We first control the sound speed and the coordinate
Jacobian, then estimate geometric derivatives through order three.
These bounds give Sobolev inequalities on the evolving slices, which
we use to convert energy estimates into pointwise bounds. The estimates
near vacuum have constants independent of the cutoff.

This proves the geometric part of Theorems~\ref{theorem: a prior estimate in W_1}
and~\ref{theorem: a prior estimate in W_2}. Sections~6--7 establish
the energy and pointwise parts of these theorems.

\paragraph{Common assumptions and lower time convention.}
All results in this section concern a regular gas solution in the
first or the second characteristic chart of
Section~\ref{subsec: two charts}, with the corresponding geometric
normalization, the three bootstrap inequalities of
Section~\ref{subsec: bootstrap framework}, and the lower-time data
and incoming flux hypotheses of the respective a priori theorem.
For the second chart take the pair $(k,\gamma)$ in
\eqref{eq: different treatment} and the fourth-order bootstrap;
for the first chart use its fixed higher order. The estimates on
positive-time approximations use $0<\alpha<1/100$ and have constants independent of the cutoff-dependent
lower bound for the sound speed. The limiting fields satisfy
the estimates in the norms introduced in
Proposition~\ref{prop: positive time construction} and
Section~\ref{subsec: common limit}.

The formulas below are written with lower time zero.
In the positive-time construction, integrate every transport
equation from $\delta$ and keep the initial values in
\eqref{eq: positive time initial geometry} and
\eqref{eq: positive time initial weights}.
The comparisons below show that their powers of $\delta$ suffice
for the required bounds at $t\geq\delta$. Thus these integrations
apply directly to the positive-time approximations.

We present the estimates in the near-vacuum region; the same arguments
apply away from vacuum. The notation $P(M)$ denotes a generic polynomial
in the bootstrap constant $M$. Fixed factors depending on the final
time are absorbed into its coefficients. In the normalization $t^*=1$,
for example, $t^{2}\leq P(M)t$ for $t\in [0, t^{*}]$.

\paragraph{Initial terms at positive time.}
The geometric initial terms needed in the transport estimates are
\begin{equation}\label{eq: positive time initial weights}
\begin{aligned}
\|Z^j(\hat T^1+1)(\delta)\|&\leq C\varepsilon^2\delta^2,
&\|Z^j(\hat T^2)(\delta)\|&\leq C\varepsilon\delta,\\
\|Z^j(\kappa-t)(\delta)\|&\leq C\varepsilon^2\delta^3,\\
\|Z^j(\slashed g-1)(\delta)\|&\leq C\varepsilon^2\delta^2,
&\|Z^j(\Xi)(\delta)\|&\leq C\varepsilon\delta^2.
\end{aligned}
\end{equation}
For the first three quantities use $j\leq3$, with $L^\infty$ at
$j\leq1$ and $L^2$ at $j=2,3$; for the last two use $j\leq2$,
with $L^\infty$ at $j=0$ and $L^2$ at $j=1,2$.
The same bounds hold for the fixed spatial derivatives needed
in the coordinate conversions. They follow by differentiating
\eqref{eq: positive time initial geometry}; the required derivatives
of the initial phase $\tilde u(\delta,\urr,x_2)$ are controlled by
the fixed data order $N_0$.

For the scalar geometric transport estimates, we use the following
comparison for a nonnegative function $q$:
\[
q(t)\leq q(\delta)+C\varepsilon\int_\delta^t q(s)\,ds
             +C\varepsilon^j\int_\delta^t s^{m-1}\,ds,\qquad
q(\delta)\leq C\varepsilon^j\delta^m,
\]
which gives $q(t)\leq eC\varepsilon^j t^m$ for $m\geq1$,
provided $C\varepsilon\leq1$.
The cases for $(\hat T^2,\hat T^1+1,\kappa-t,\slashed g-1,\Xi)$
have respectively $(j,m)=(1,1),(2,2),(1,2),(1,1),(1,2)$.
The initial terms in \eqref{eq: positive time initial weights}
have these or higher powers of $\delta$.
For the commuted frame equations use the coupled estimate
\eqref{eq: coupled geometric norms}, including the quadratic
second-derivative products in \eqref{eq: quadratic second geometry}.
Its weighted comparison \eqref{eq: weighted geometric comparison}
includes exactly the finite-$\delta$ terms above and gives the
same time powers. The frame matrices and coefficient estimates
in that proof make the independence of both cutoffs explicit. Once the frame is bounded, the
metric and shift equations use the scalar comparison just stated.
For the $L^2$ norms, $L(\sqrt{\slashed g})=\chi\sqrt{\slashed g}$
contributes the regular coefficient $P(M)\varepsilon$.
Thus the frame, metric, shift and coordinate estimates, including
$\kappa/t=1+O(P(M)\varepsilon t)$, hold uniformly.
The lengths of the label intervals $[\tilde u_*,\tilde u^*]$
are bounded above and below by fixed positive constants, independently
of $\alpha$. Their Sobolev constants are therefore uniform as well.

\subsection{Sound-speed comparison along characteristics}

In view of $(\mathbf{B}_{\infty})$, we have 
\[
|\partial_t(c)(t, \tilde{u}, \vartheta)|\leq M\varepsilon |c(t, \tilde{u}, \vartheta)|.
\]
By Gronwall's inequality, for $t\in [0, t^{*}]$
\[
e^{-M\varepsilon t}c(0, \tilde{u}, \vartheta)\leq c(t, \tilde{u}, \vartheta)\leq e^{M\varepsilon t}c(0, \tilde{u}, \vartheta).
\]
Since $c(0, \tilde{u}, \vartheta)=-\frac{\gamma-1}{\gamma+1}\tilde{u}$, the sound speed remains comparable to its initial value along $C_{\tilde{u}}$:
\begin{equation}\label{eq: c almost constant on C_u}
\frac{1}{2}\alpha\cdot \mathbf{c}_{*}\leq-\frac{1}{2}\frac{\gamma-1}{\gamma+1}\tilde{u} \leq c(t, \tilde{u}, \vartheta) \leq -2\frac{\gamma-1}{\gamma+1}\tilde{u}\leq 2\mathbf{c}_{*}.
\end{equation}
For an approximation started at $\delta$, the same calculation gives
\[
e^{-M\varepsilon(t-\delta)}c_\delta
\leq c(t,\tilde u,\vartheta)
\leq e^{M\varepsilon(t-\delta)}c_\delta.
\]
The initial comparison $3c_0/4\leq c_\delta\leq5c_0/4$ therefore
gives the same fixed factors in \eqref{eq: c almost constant on C_u}.
Similarly, $\|v^{1}\|_{L^{\infty}(\Sigma_{t}^{\ut_{*},\ut^{*}})}\leq C$ and $\|v^{2}\|_{L^{\infty}(\Sigma_{t}^{\ut_{*},\ut^{*}})}\leq CM\varepsilon$, provided that $M\varepsilon$ is sufficiently small.

\begin{remark}
The lower bound for sound speed $c$ depends on $\alpha\in (0, \frac{1}{100})$. 
\end{remark}

\subsection{Pointwise geometry and the diffeomorphism bootstrap}

\begin{proposition}
For any $t\in [0, t^{*}]$, we have
\begin{equation}\label{eq: 0-th control of the frame}
\begin{cases}
\|\big(\hat{T}^{1}+1\big)\|_{L^{\infty}(\Sigma_{t}^{\ut_{*}, \tilde{u}^{*}})}\leq P(M)\varepsilon^{2}t^{2},\\
\|\hat{T}^{2}\|_{L^{\infty}(\Sigma_{t}^{\ut_{*}, \tilde{u}^{*}})}\leq P(M)\varepsilon t.\\
\|\kappa-\kappar\|_{L^{\infty}(\Sigma_{t}^{\ut_{*}, \tilde{u}^{*}})}\leq P(M)\varepsilon t^{2}.
\end{cases}.
\end{equation}
\end{proposition}

\begin{proof}
In view of the transport equation for $\hat{T}^{i}$ and $(\mathbf{B}_{\infty})$, we have
\[
|L(\hat{T}^{i})|=|\big(\hat{T}^{j}\hat{X}(\psi_{j})+\hat{X}(c) \big)\hat{X}^{i}|\leq P(M)\varepsilon,
\]
Recall that 
\[
\hat{T}^{1}(0, \tilde{u}, \vartheta)=-1, \hat{T}^{2}(0, \tilde{u}, \vartheta)=0.
\]
Integrating from $0$ to $t$, we have 
\[
\|\hat{T}^{1}+1\|_{L^{\infty}(\Sigma_{t}^{\ut_*,\tilde{u}^{*}})}\leq P(M)\varepsilon t, \|\hat{T}^{2}\|_{L^{\infty}(\Sigma_{t}^{\ut_*,\tilde{u}^{*}})}\leq P(M)\varepsilon t
\]
Then $|\hat{X}^{1}|=|\hat{T}^{2}|\leq P(M)\varepsilon t$ implies that 
\[
|L(\hat{T}^{1})|=|\big(\hat{T}^{j}\hat{X}(\psi_{j})+\hat{X}(c) \big)\hat{X}^{1}|\leq P(M)\varepsilon^{2}t,
\]
Integrating from $0$ to $t$, we have 
\[
\|\hat{T}^{1}+1\|_{L^{\infty}(\Sigma_{t}^{\ut_{*}, \tilde{u}^{*}})}\leq P(M)\varepsilon^{2} t^{2}.
\]
In view of the transport equation for $\kappa$ and $(\mathbf{B}_{\infty})$ on $T(\wb), T(w), T(\psi_{2})$, we have
\[
|L(\kappa-\kappar)|=|-\big(T(c+v^{1})+1\big)-(\hat{T}^{1}+1)T(\psi_{1})-\hat{T}^{2}T(\psi_{2})|\leq P(M)\varepsilon t.
\]
Recall that 
\[
\kappa(0, \tilde{u}, \vartheta)=0.
\]
Integrating from $0$ to $t$, we have 
\[
\|\kappa-\kappar\|_{L^{\infty}(\Sigma_{t}^{\ut_{*}, \tilde{u}^{*}})}\leq P(M)\varepsilon t^{2}.
\]
\end{proof}

\begin{proposition}
For any $t\in [0, t^{*}]$, and $Z\in \{T, \hat{X}\}$, we have
\begin{equation}\label{eq: 1-th control of the frame}
\begin{cases}
\|Z\big(\hat{T}^{1}+1\big)\|_{L^{\infty}(\Sigma_{t}^{\ut_{*}, \tilde{u}^{*}})}\leq P(M)\varepsilon^{2}t^{2},\\
\|Z\big(\hat{T}^{2}\big)\|_{L^{\infty}(\Sigma_{t}^{\ut_{*}, \tilde{u}^{*}})}\leq P(M)\varepsilon t.\\
\|Z\big(\kappa-\kappar\big)\|_{L^{\infty}(\Sigma_{t}^{\ut_{*}, \tilde{u}^{*}})}\leq P(M)\varepsilon t^{2}.
\end{cases}.
\end{equation}
\end{proposition}

\begin{proof}
We will commute $\hat{X}$ before $T$.
\begin{itemize}
\item The case $Z=\hat{X}$. We commute $\hat{X}$ with the transport equation \eqref{eq: formulas to control the geometry} of $\hat{T}^{i}$, 
\[
\begin{split}
L(\hat{X}(\hat{T}^{i}))&=\hat{X}L(\hat{T}^{i})-\chi\cdot \hat{X}(\hat{T}^{i}),\\
&=\hat{X}\Big(\big(\hat{T}^{j}\hat{X}(\psi_{j})+\hat{X}(c)\big)\hat{X}^{i}\Big)-(-\hat{X}^{j}\hat{X}(\psi_{j})+c\hat{X}^{1}\hat{X}(\hat{X}^{2})-c\hat{X}^{2}\hat{X}(\hat{X}^{1}))\cdot\hat{X}(\hat{T}^{i}).
\end{split}
\]
In view of $(\mathbf{B}_{\infty})$ and \eqref{eq: 0-th control of the frame}, we have for $i=1, 2$,
\[
|L(\hat{X}(\hat{T}^{i}))|\leq P(M)\varepsilon \big(|\hat{X}(\hat{T}^{1})|+|\hat{X}(\hat{T}^{2})|\big)+C(|\hat{X}(\hat{T}^{1})|+|\hat{X}(\hat{T}^{2})|\big)^{2}+P(M)\varepsilon.
\]
Recall that
\[
\hat{X}(\hat{T}^{1}+1)(0, \tilde{u}, \vartheta)=0, \hat{X}(\hat{T}^{2})(0, \tilde{u}, \vartheta)=0.
\]
The small initial values and the fixed time interval control the
quadratic term. Integrating from $0$ to $t$ and applying Gronwall's
inequality, we obtain 
\[
\|\hat{X}(\hat{T}^{i})\|_{L^{\infty}(\Sigma_{t}^{\ut_{*}, \tilde{u}^{*}})}\leq P(M)\varepsilon t.
\]
which implies that
\[
|L(\hat{X}(\hat{T}^{1}))|\leq P(M)\varepsilon^{2} t.
\]
Integrating from $0$ to $t$, we have 
\[
\|\hat{X}(\hat{T}^{1})\|_{L^{\infty}(\Sigma_{t}^{\ut_{*}, \tilde{u}^{*}})}\leq P(M)\varepsilon^{2} t^{2}.
\] 

We commute $\hat{X}$ with transport equation \eqref{eq: formulas to control the geometry} of $\kappa$,
\[
\begin{split}
L(\hat{X}(\kappa))&=\hat{X}L(\kappa)-\chi\cdot\hat{X}(\kappa)\\
&=\hat{X}\big(-T(c)-\hat{T}^{j}T(\psi_{j})\big)-(-\hat{X}^{j}\hat{X}(\psi_{j})+c\hat{X}^{1}\hat{X}(\hat{X}^{2})-c\hat{X}^{2}\hat{X}(\hat{X}^{1}))\cdot\hat{X}(\kappa).
\end{split}
\]
In view of \eqref{eq: formulas to control the geometry}, for $\psi\in \{\wb, w, \psi_{2}\}$,
\[
\hat{X}T(\psi)=T\hat{X}(\psi)+\kappa\theta\cdot \hat{X}(\psi)=T\hat{X}(\psi)+\kappa(-\hat{X}(\hat{X}^{2})\hat{X}^{1}+\hat{X}(\hat{X}^{1})\hat{X}^{2})\cdot \hat{X}(\psi),
\]
we have $\|\hat{X}T(\psi)\|_{L^{\infty}(\Sigma_{t}^{\ut_*,\tilde{u}^{*}})}\leq P(M)\varepsilon t$. In view of estimate of $\hat{X}(\hat{T}^{i})$ and $\mathbf{B}_{\infty}$,
\[
|L(\hat{X}(\kappa))|\leq P(M)\varepsilon t+P(M)\varepsilon |\hat{X}(\kappa)|.
\]
Recall that
\[
\hat{X}(\kappa)(0, \tilde{u}, \vartheta)=0.
\]
By Gronwall's inequality, we have 
\[
\|\hat{X}(\kappa-\kappar)\|_{L^{\infty}(\Sigma_{t}^{\ut_{*}, \tilde{u}^{*}})}\leq P(M)\varepsilon t^{2}.
\]

\item The case $Z=T$. We commute $T$ with the transport equation \eqref{eq: formulas to control the geometry} of $\hat{T}^{i}$, 
\[
\begin{split}
L(T(\hat{T}^{i}))&=TL(\hat{T}^{i})-(\zeta+\eta)\cdot\hat{X}(\hat{T}^{i})\\
&=T\Big(\big(\hat{T}^{j}\hat{X}(\psi_{j})+\hat{X}(c)\big)\hat{X}^{i}\Big)-\big(-\kappa \hat{T}^{j}\hat{X}(\psi_{j})-\hat{X}^{j}T(\psi_{j})-\kappa\hat{X}(c)+c\hat{X}(\kappa)\big)\hat{X}(\hat{T}^{i})
\end{split}
\]
In view of $(\mathbf{B}_{\infty})$ and \eqref{eq: 0-th control of the frame}, the term $\hat{X}^{i}T(\psi_{i})$ can be bounded by
\[
|\hat{X}^{1}T(\psi_{1})+\hat{X}^{2}T(\psi_{2})|\leq P(M)\varepsilon t.
\]
and we have
\[
|L(T(\hat{T}^{i}))|\leq P(M)\varepsilon \big(|T(\hat{T}^{1})|+|T(\hat{T}^{2})|\big)+P(M)\varepsilon t.
\]
Recall that
\[
T(\hat{T}^{i})(0, \tilde{u}, \vartheta)=0, i=1, 2.
\]
By Gronwall's inequality, we have 
\[
\|T(\hat{T}^{i})\|_{L^{\infty}(\Sigma_{t}^{\ut_{*}, \tilde{u}^{*}})}\leq P(M)\varepsilon t.
\]
Since $|\hat{X}^{1}|\leq P(M)\varepsilon t$,
\[
|L(T(\hat{T}^{1}))|\leq P(M)\varepsilon^{2} t.
\]
Integrating from $0$ to $t$, we have 
\[
\|T(\hat{T}^{1})\|_{L^{\infty}(\Sigma_{t}^{\ut_{*}, \tilde{u}^{*}})}\leq P(M)\varepsilon^{2} t^{2}.
\] 

We commute $T$ with transport equation \eqref{eq: formulas to control the geometry} of $\kappa$,
\[
\begin{split}
L(T(\kappa))&=TL(\kappa)-(\zeta+\eta)\cdot\hat X(\kappa)\\
&=T\big(-T(c)-\hat{T}^{i}T(\psi_{i})\big)-\big(-\kappa \hat{T}^{j}\hat{X}(\psi_{j})-\hat{X}^{i}T(\psi_{i})-\kappa\hat{X}(c)+c\hat{X}(\kappa)\big)\cdot \hat{X}(\kappa).
\end{split}
\]
In view of estimate of $\hat{X}(\hat{T}^{i}), T(\hat{T}^{i})$ and $(\mathbf{B}_{\infty})$,
\[
|L(T(\kappa))|\leq P(M)\varepsilon t.
\]
Recall that
\[
T(\kappa)(0, \tilde{u}, \vartheta)=0.
\]
Integrating from $0$ to $t$, we have 
\[
|T(\kappa-\kappar)|\leq P(M)\varepsilon t^{2}.
\] 
\end{itemize}
\end{proof}

As a byproduct, we have 
\begin{proposition}
For $t\in [0, t^{*}]$, we have
\begin{equation}\label{eq: Control of g}
\begin{cases}
\|\chi\|_{L^{\infty}(\Sigma_{t}^{\ut_{*}, \tilde{u}^{*}})}\leq P(M)\varepsilon,\\
\|\zeta+\eta\|_{L^{\infty}(\Sigma_{t}^{\ut_{*}, \tilde{u}^{*}})}\leq P(M)\varepsilon t,\\
\|\slashed{g}-1\|_{L^{\infty}(\Sigma_{t}^{\ut_{*}, \tilde{u}^{*}})}\leq P(M)\varepsilon t,\\
\|\Xi\|_{L^{\infty}(\Sigma_{t}^{\ut_{*}, \tilde{u}^{*}})}\leq P(M)\varepsilon t^{2}.
\end{cases}
\end{equation} 
\end{proposition}

\begin{proof}
The bound for $\chi=-\hat{X}^{i}\hat{X}(\psi_{i})+c\hat{X}(\hat{X}^{2})\hat{X}^{1}-c\hat{X}(\hat{X}^{1})\hat{X}^{2}$ and $\zeta+\eta=-\kappa\hat{T}^{j}\hat{X}(\psi_{j})-\hat{X}^{i}T(\psi_{i})-\kappa\hat{X}(c)+c\hat{X}(\kappa)$ follows from the estimate \eqref{eq: 1-th control of the frame} from the previous proposition.

And we can use transport equation for $\slashed{g}$ and $\Xi$,
\[
\begin{cases}
L(\slashed{g})=2\chi\cdot \slashed{g},\\
L(\Xi)=\frac{1}{\sqrt{\slashed{g}}}(\zeta+\eta).
\end{cases}
\]
with initial data
\[
\slashed{g}(0, \tilde{u}, \vartheta)=1, \Xi(0, \tilde{u}, \vartheta)=0.
\]
By Gronwall inequality, we have
\[
\|\slashed{g}-1\|_{L^{\infty}(\Sigma_{t}^{\ut_{*}, \tilde{u}^{*}})}\leq P(M)\varepsilon t, \|\Xi\|_{L^{\infty}(\Sigma_{t}^{\ut_{*}, \tilde{u}^{*}})}\leq P(M)\varepsilon t^{2}.
\]
provided that $\varepsilon$ is sufficiently small.
\end{proof}

The Jacobian matrix of the coordinate transformation $(t, \tilde{u},\vartheta)\mapsto (t,\mathring{u},x_{2})$ is given by
\begin{align*}
\begin{pmatrix}
\frac{\partial t}{\partial t}&\frac{\partial t}{\partial \tilde{u}}&\frac{\partial t}{\partial \vartheta}\\
\frac{\partial \mathring{u}}{\partial t}&\frac{\partial \mathring{u}}{\partial \tilde{u}}&\frac{\partial \mathring{u}}{\partial \vartheta}\\
\frac{\partial x_{2}}{\partial t}&\frac{\partial x_{2}}{\partial \tilde{u}}&\frac{\partial x_{2}}{\partial \vartheta}
\end{pmatrix}=
\begin{pmatrix}
1&0&0\\
\frac{v^{1}-c\hat{T}^{1}-\mathring{u}}{\kappar}&\frac{\kappa\hat{T}^{1}+\Xi\sqrt{\slashed{g}}\hat{X}^{1}}{\kappar}&\sqrt{\slashed{g}}\frac{\hat{X}^{1}}{\kappar}\\
(v^{2}-c\hat{T}^{2})&\kappa\hat{T}^{2}+\Xi\sqrt{\slashed{g}}\hat{X}^{2}&\sqrt{\slashed{g}}\hat{X}^{2}\\
\end{pmatrix}
\end{align*}
Using \eqref{eq: 0-th control of the frame} and \eqref{eq: Control of g}, for $t\in [0,t^*]$ we obtain
\[
\|\begin{pmatrix}
\frac{\partial \mathring{u}}{\partial \tilde{u}}&\frac{\partial \mathring{u}}{\partial \vartheta}\\
\frac{\partial x_{2}}{\partial \tilde{u}}&\frac{\partial x_{2}}{\partial \vartheta}
\end{pmatrix}-\begin{pmatrix}
-1&0\\
0&1
\end{pmatrix}\|_{L^{\infty}(\Sigma_{t}^{\ut_*,\tilde{u}^{*}})}\leq P(M)\varepsilon.
\]
which implies 
\[
\|\begin{pmatrix}
\frac{\partial \tilde{u}}{\partial \mathring{u}}&\frac{\partial \tilde{u}}{\partial x_2}\\
\frac{\partial \vartheta}{\partial \mathring{u}}&\frac{\partial\vartheta }{\partial x_{2}}
\end{pmatrix}-\begin{pmatrix}
-1&0\\
0&1
\end{pmatrix}\|_{L^{\infty}(\Sigma_{t}^{\ut_*,\tilde{u}^{*}})}\leq P(M)\varepsilon.
\]
This improves the Jacobian bound for $(\mathring{u}, x_{2})\mapsto (\tilde{u}, \vartheta)$ and closes the diffeomorphism bootstrap $(\mathbf{B}_{\text{diffeomorphism}})$ in \eqref{eq: B diff}.

\subsection{Second- and third-order geometric estimates}
The next estimates control ordered geometric derivatives through
order three. Their proof gives the dependence on the bootstrap constant
and the powers of time needed in the cutoff limit.
These are the geometric bounds used in the product estimates and in
the positive-time construction.

\begin{proposition}
For any $t\in [0, t^{*}]$, $|\alpha|\leq 3$, and $Z\in \{T, \hat{X}\}$, we have
\begin{equation}\label{eq: higher order control of the frame}
\begin{cases}
\|Z^{\alpha}\big(\hat{T}^{1}+1\big)\|_{L^{2}(\Sigma_{t}^{\ut_{*}, \tilde{u}^{*}})}\leq P(M)\varepsilon^{2}t^{2},\\
\|Z^{\alpha}\big(\hat{T}^{2}\big)\|_{L^{2}(\Sigma_{t}^{\ut_{*}, \tilde{u}^{*}})}\leq P(M)\varepsilon t.\\
\|Z^{\alpha}\big(\kappa-\kappar\big)\|_{L^{2}(\Sigma_{t}^{\ut_{*}, \tilde{u}^{*}})}\leq P(M)\varepsilon t^{2}.
\end{cases}
\end{equation}
\end{proposition}

\begin{proof}
We work on a regular positive-time approximation, with
$\delta\leq t\leq1$ and the full label interval
$I_\alpha=[\tilde u_*,\tilde u^*]$.
All $L^p$ norms below use the slice measure
$\sqrt{\slashed g}\,d\tilde u\,d\vartheta$.
The constants $C$ and the coefficients of $P(M)$ depend only on
the fixed parameters and are independent of $\delta$, $\alpha$,
and the minimum sound speed. Each sum over $|\beta|=m$ includes
every ordered word of length $m$ in $\{T,\hat X\}$.
The preceding propositions give
\[
\begin{aligned}
\|\hat T^2\|_\infty+\sum_{Z\in\{T,\hat X\}}\|Z(\hat T^2)\|_\infty
 &\leq P(M)\varepsilon t,\\
\|\kappa-t\|_\infty+\sum_{Z\in\{T,\hat X\}}\|Z(\kappa-t)\|_\infty
 &\leq P(M)\varepsilon t^2.
\end{aligned}
\]
Since $\hat T^1=-\sqrt{1-(\hat T^2)^2}$, its differentiated
estimates follow from those of $\hat T^2$. In particular,
\[
\hat T^1+1=\frac{(\hat T^2)^2}{1+\sqrt{1-(\hat T^2)^2}},
\]
and, once the lower orders have been bounded,
\[
\|Z^\beta(\hat T^1+1)\|_2
\leq P(M)\varepsilon t
 \left(\sum_{|\beta'|=m}\|Z^{\beta'}(\hat T^2)\|_2
                    +P(M)\varepsilon t\right),\qquad |\beta|=m.
\]

\smallskip\noindent\textit{Transport identities and field derivatives.}
The exact identities used below are
\begin{equation}\label{eq: explicit geometric coefficients}
\begin{aligned}
L(\hat T^2)
 &=-\hat T^1\big(\hat X(v^1+c)+(\hat T^1+1)\hat X(\psi_1)
                              +\hat T^2\hat X(\psi_2)\big),\\
L(\kappa-t)
 &=-\big(T(v^1+c)+1\big)-(\hat T^1+1)T(\psi_1)-\hat T^2T(\psi_2),\\
\chi
 &=-\hat X^j\hat X(\psi_j)
   +c\hat X(\hat X^2)\hat X^1-c\hat X(\hat X^1)\hat X^2,\\
\zeta+\eta
 &=-\kappa\big(\hat X(v^1+c)+(\hat T^1+1)\hat X(\psi_1)
                              +\hat T^2\hat X(\psi_2)\big)\\
 &\quad-\hat T^2T(\psi_1)+\hat T^1T(\psi_2)+c\hat X(\kappa-t),\\
[L,\hat X]&=-\chi\hat X,\qquad [L,T]=-(\zeta+\eta)\hat X,\\
[\hat X,T]&=\kappa\theta\hat X.
\end{aligned}
\end{equation}
Here $\theta$ is the commutator coefficient already defined in
\eqref{eq: commutators t x_1 x_1}. By its definition, the displayed
formula for $\chi$ is equivalently
$\chi=-\hat X^j\hat X(\psi_j)-c\theta$.
Also, $T(\psi_1)=\frac{2}{\gamma+1}+O(M\varepsilon t)$,
$T(\psi_2)=O(M\varepsilon t)$, and
$T(v^1+c)+1=O(M\varepsilon t)$. Thus the constant part of the
normal field derivative in $L(\kappa-t)$ is multiplied by $\hat T^1+1$.

The field bounds in $(\mathbf B_\infty)$ put the leading
$\hat X$ or $T$ on the left. For $Z^\beta=Z_1\cdots Z_m$
and $Y\in\{\hat X,T\}$, we use
\[
[Z^\beta,Y]=\sum_{j=1}^m
 Z_1\cdots Z_{j-1}[Z_j,Y]Z_{j+1}\cdots Z_m.
\]
Each nonzero bracket is $\pm\kappa\theta\hat X$.
After Leibniz expansion, at most $m-1$ derivatives fall on
$\kappa\theta$ and at most $m$ on the remaining field.
The leading term is $YZ^\beta(\psi)$.
For $m\leq3$, the highest geometric term in the differentiated
coefficient is
\[
\kappa Z^{m-1}(\theta).
\]
To estimate it, use $\hat X^1=\hat T^2$ and
$\hat X^2=-\hat T^1$ in the definition of $\theta$ from
\eqref{eq: commutators t x_1 x_1}. This gives
\[
\theta=\hat X(\hat T^1)\hat T^2-\hat X(\hat T^2)\hat T^1.
\]
The term carrying the top derivatives of $\hat T^2$ is
$-\hat T^1 Z^{m-1}\hat X(\hat T^2)$, and
$-\hat T^1=1+O(P(M)\varepsilon^2t^2)$. The other top-order term
contains $\hat T^2$ and derivatives of $\hat T^1+1$, so the
previous frame estimates control it. Thus the $L^2$ norm of
$\kappa Z^{m-1}(\theta)$ is bounded by
$Ct\sum_{|\beta|=m}\|Z^\beta(\hat T^2)\|_2$.
The remaining terms use at most $m-1$ derivatives of $\kappa-t$
and at most $m$ derivatives of $\hat T^1+1$ and $\hat T^2$; these
are controlled by the frame estimates already established.
Their highest field partner is tangential and is bounded by the
small tangential field estimates.
Products with a third field derivative have a first geometric
derivative in $L^\infty$; those with a second geometric derivative
have at most a second field derivative in $L^\infty$.
Consequently the reordered field terms give
\begin{equation}\label{eq: geometric differentiated sources}
\begin{aligned}
\sum_{|\beta|=m}\|Z^\beta L(\hat T^2)\|_2
 &\leq P(M)\varepsilon
       \left(1+\sum_{|\beta|=m}\|Z^\beta(\hat T^2)\|_2\right),\\
\sum_{|\beta|=m}\|Z^\beta L(\kappa-t)\|_2
 &\leq P(M)\varepsilon t
       \left(1+\sum_{|\beta|=m}\|Z^\beta(\hat T^2)\|_2\right).
\end{aligned}
\end{equation}
For $m=2$ these estimates use the previously proved first-order
bounds; for $m=3$ we first prove the order-two bounds below.
The second line keeps the factor $t$ from the normal field bounds,
from $\kappa$ in the bracket coefficient, or from the differentiated
factor $\hat T^1+1$. The geometric derivatives have order at most three.

\smallskip\noindent\textit{Commutators and the order-two estimate.}
For every smooth scalar function $f$, the commutator identity is
\[
\begin{aligned}
[L,Z^\beta](f)
 &=-\sum_{\substack{1\leq j\leq m\\Z_j=\hat X}}
 Z_1\cdots Z_{j-1}\big(\chi\hat XZ_{j+1}\cdots Z_m(f)\big)\\
 &\quad-\sum_{\substack{1\leq j\leq m\\Z_j=T}}
 Z_1\cdots Z_{j-1}\big((\zeta+\eta)\hat XZ_{j+1}\cdots Z_m(f)\big).
\end{aligned}
\]
We apply the commutator identity to $\hat T^2$ and $\kappa-t$.
The coefficient formulas above give
\begin{equation}\label{eq: geometric commutator coefficient bounds}
\begin{aligned}
\|\chi\|_\infty&\leq P(M)\varepsilon,
&\|\zeta+\eta\|_\infty&\leq P(M)\varepsilon t,\\
\sum_{|\beta|=m-1}\|Z^\beta(\chi)\|_2
 &\leq C\sum_{|\beta|=m}\|Z^\beta(\hat T^2)\|_2+P(M)\varepsilon,\\
\sum_{|\beta|=m-1}\|Z^\beta(\zeta+\eta)\|_2
 &\leq C\sum_{|\beta|=m}\|Z^\beta(\kappa-t)\|_2+P(M)\varepsilon t.
\end{aligned}
\end{equation}
The top terms are $-cZ^{m-1}(\theta)$ and
$cZ^{m-1}\hat X(\kappa-t)$, respectively. The first is controlled
by the expression for $\theta$ just given and the estimates for
$\hat T^1+1$ and $\hat T^2$. The other terms have at most $m$
field derivatives and at most $m-1$ geometric derivatives.
In $\zeta+\eta$, derivatives of $\hat T^2T(\psi_1)$ use
\[
\sum_{|\beta|=m-1}\|Z^\beta(\hat T^2)\|_2
\leq P(M)\varepsilon t.
\]
For $m=2$, the first-order pointwise bounds imply this estimate
because the slice has uniformly bounded measure.

At order two the commutator has either an undifferentiated
coefficient times a second derivative of $f$, or a first coefficient
derivative times a first derivative of $f$.
Combining \eqref{eq: geometric differentiated sources} and
\eqref{eq: geometric commutator coefficient bounds} with
$L(\sqrt{\slashed g})=\chi\sqrt{\slashed g}$ gives
\begin{equation}\label{eq: coupled geometric norms}
\begin{aligned}
\frac{d}{dt}\sum_{|\beta|=m}\|Z^\beta(\hat T^2)\|_2
 &\leq P(M)\varepsilon
 \sum_{|\beta|=m}\big(\|Z^\beta(\hat T^2)\|_2
                    +\|Z^\beta(\kappa-t)\|_2\big)+P(M)\varepsilon,\\
\frac{d}{dt}\sum_{|\beta|=m}\|Z^\beta(\kappa-t)\|_2
 &\leq P(M)\varepsilon t
             \sum_{|\beta|=m}\|Z^\beta(\hat T^2)\|_2\\
 &\quad+P(M)\varepsilon
             \sum_{|\beta|=m}\|Z^\beta(\kappa-t)\|_2+P(M)\varepsilon t.
\end{aligned}
\end{equation}
Initially $m=2$; the same inequalities for $m=3$ are proved below.
Norm derivatives are understood as upper derivatives, obtained by
regularizing the squared norm. For either value of $m$, these
inequalities imply
\begin{equation}\label{eq: weighted geometric comparison}
\begin{aligned}
&\frac{d}{dt}\sum_{|\beta|=m}
 \left(\|Z^\beta(\hat T^2)\|_2
                 +\frac{\|Z^\beta(\kappa-t)\|_2}{t}\right)\\
&\quad\leq3P(M)\varepsilon\sum_{|\beta|=m}
 \left(\|Z^\beta(\hat T^2)\|_2
                 +\frac{\|Z^\beta(\kappa-t)\|_2}{t}\right)
       +2P(M)\varepsilon
       -\frac{1}{t^2}\sum_{|\beta|=m}\|Z^\beta(\kappa-t)\|_2\\
&\quad\leq3P(M)\varepsilon\sum_{|\beta|=m}
 \left(\|Z^\beta(\hat T^2)\|_2
                 +\frac{\|Z^\beta(\kappa-t)\|_2}{t}\right)
       +2P(M)\varepsilon.
\end{aligned}
\end{equation}
The initial bounds \eqref{eq: positive time initial weights} give
\[
\sum_{|\beta|=m}
 \left(\|Z^\beta(\hat T^2)(\delta)\|_2
        +\frac{\|Z^\beta(\kappa-t)(\delta)\|_2}{\delta}\right)
\leq C\varepsilon\delta+C\varepsilon^2\delta^2.
\]
Hence, whenever \eqref{eq: coupled geometric norms} holds,
\[
\begin{aligned}
\sum_{|\beta|=m}
 \left(\|Z^\beta(\hat T^2)\|_2
                 +\frac{\|Z^\beta(\kappa-t)\|_2}{t}\right)
&\leq e^{3P(M)\varepsilon(t-\delta)}
 \big(C\varepsilon\delta+C\varepsilon^2\delta^2
                         +2P(M)\varepsilon(t-\delta)\big)\\
&\leq P(M)\varepsilon t.
\end{aligned}
\]
This proves the order-two bounds, uniformly in $\delta$ and $\alpha$.

\smallskip\noindent\textit{Uniform interpolation at order three.}
Let $D=(\partial_{\mathring u},\partial_{x_2})$ and let
$\Omega_{t,\delta,\alpha}$ be the image of $I_\alpha\times\mathbb T$
in the spatial coordinates $(\urr,x_2)$.
The length $\tilde u^*-\tilde u_*$ is bounded above and below
by fixed positive constants. The spatial map and inverse are uniformly
close in first derivatives to
$(\tilde u,\vartheta)\mapsto(-\tilde u,\vartheta)$.
The two curves $\tilde u=\tilde u_*$ and $\tilde u=\tilde u^*$ on a time
slice are therefore graphs of $\urr$ over $x_2$ with uniformly bounded
slopes. At fixed $x_2$, $\partial\tilde u/\partial\mathring u$
is bounded away from zero, so the two graph values are separated
by a fixed positive constant. This supplies a Lipschitz atlas with a
fixed number and size of charts, independently of $t,\delta,\alpha$.
Periodic charts can be localized using a fixed partition of unity.
Stein's extension operator \cite[Chapter VI, Section 3]{Stein1970}
is consequently bounded simultaneously on $W^{1,\infty}$ and
$H^3$ with one uniform constant.

We write the frame conversion explicitly to keep track of its
powers of $t$:
\begin{equation}\label{eq: uniform third order frame matrices}
\begin{gathered}
\begin{pmatrix}T\\\hat X\end{pmatrix}=\mathsf A D,
\qquad
\mathsf A=\begin{pmatrix}
\frac{\kappa\hat T^1}{t}&\kappa\hat T^2\\
\frac{\hat T^2}{t}&-\hat T^1
\end{pmatrix},\\
\mathsf A^{-1}=\begin{pmatrix}
\frac{t\hat T^1}{\kappa}&t\hat T^2\\
\frac{\hat T^2}{\kappa}&-\hat T^1
\end{pmatrix}.
\end{gathered}
\end{equation}
Here $\det\mathsf A=-\frac{\kappa}{t}$ and
$\frac{\kappa}{t}=1+O(P(M)\varepsilon t)$.
The zeroth- and first-order bounds therefore give
\[
\begin{aligned}
\|\mathsf A\|_\infty+\|\mathsf A^{-1}\|_\infty&\leq C,\\
\|D(\mathsf A)\|_\infty+\|D(\mathsf A^{-1})\|_\infty
 &\leq P(M)\varepsilon.
\end{aligned}
\]
Converting two derivatives uses only these first coefficient
derivatives. The order-two estimates already proved then give
\[
\begin{aligned}
\|D^2(\hat T^2)\|_2&\leq P(M)\varepsilon t,
&\|D^2(\kappa-t)\|_2&\leq P(M)\varepsilon t^2,\\
\|D^2(\mathsf A)\|_2+\|D^2(\mathsf A^{-1})\|_2
 &\leq P(M)\varepsilon.
\end{aligned}
\]
In the last inequality, the coefficients containing $\frac1t$ are
$\frac{D^2(\hat T^2)}{t}$ and $\frac{D^2(\kappa-t)}{t}$ and their
lower-order products. Their bounds are $P(M)\varepsilon$ and
$P(M)\varepsilon t$. The ratio
\[
\frac{\hat T^2}{\kappa}
 =\frac{\frac{\hat T^2}{t}}{1+\frac{\kappa-t}{t}}
\]
has the same bounds. The measure conversion
\[
d\mathring u\,dx_2
 =\frac{\kappa}{t}\sqrt{\slashed g}\,d\tilde u\,d\vartheta
\]
has fixed upper and lower bounds.

For a scalar function $f$, conversion of a third derivative
contains $Z^3(f)$, a bounded first coefficient derivative times
$D^2(f)$, and a second coefficient derivative times $D(f)$.
Applying this identity to $\hat T^2$ and $\kappa-t$, whose first
derivatives are bounded, gives
\[
\begin{aligned}
\|\hat T^2\|_{H^3(\Omega_{t,\delta,\alpha})}
 &\leq C\sum_{|\beta|=3}\|Z^\beta(\hat T^2)\|_2+P(M)\varepsilon t,\\
\|\kappa-t\|_{H^3(\Omega_{t,\delta,\alpha})}
 &\leq C\sum_{|\beta|=3}\|Z^\beta(\kappa-t)\|_2+P(M)\varepsilon t^2.
\end{aligned}
\]
For the same extension $E$, the Gagliardo--Nirenberg inequality
\cite[Lecture II]{Nirenberg1959} gives
\[
\|D^2(f)\|_4^2
\leq C\|D(Ef)\|_\infty\|D^3(Ef)\|_2
\leq C\|f\|_{W^{1,\infty}}\|f\|_{H^3}.
\]
For example, its Euclidean form follows by integrating
$\int|D(g)|^4=-\int g\,\operatorname{div}(|D(g)|^2D(g))$ for each
$g=\partial_j(Ef)$, and then using H\"older's inequality; localization
and approximation justify the formula for the extension.
Using
$Z^2(f)=\mathsf A\mathsf A D^2(f)+\mathsf A(D(\mathsf A))D(f)$
componentwise yields
\begin{equation}\label{eq: quadratic second geometry}
\begin{aligned}
\sum_{|\beta|=2}\|Z^\beta(\hat T^2)\|_4^2
 &\leq P(M)\varepsilon t
 \left(\sum_{|\beta|=3}\|Z^\beta(\hat T^2)\|_2+P(M)\varepsilon t\right),\\
\sum_{|\beta|=2}\|Z^\beta(\kappa-t)\|_4^2
 &\leq P(M)\varepsilon t^2
 \left(\sum_{|\beta|=3}\|Z^\beta(\kappa-t)\|_2+P(M)\varepsilon t^2\right).
\end{aligned}
\end{equation}
H\"older's inequality and weighted Young inequalities also give,
for $|\beta|=|\beta'|=2$,
\begin{equation}\label{eq: mixed second geometric product}
\begin{aligned}
\|Z^\beta(\hat T^2)\,Z^{\beta'}(\kappa-t)\|_2
\leq P(M)\varepsilon\bigg(&t^2\sum_{|\beta''|=3}\|Z^{\beta''}(\hat T^2)\|_2\\
 &+t\sum_{|\beta''|=3}\|Z^{\beta''}(\kappa-t)\|_2+P(M)\varepsilon t^3\bigg).
\end{aligned}
\end{equation}

\smallskip\noindent\textit{Closing all third-order words together.}
At order three the Leibniz expansion of $[L,Z^\beta](f)$ has
only the derivative distributions $(0,3),(1,2),(2,1)$ between
the coefficient and $f$. The first and last use
\eqref{eq: geometric commutator coefficient bounds}.
For the middle distribution the exact coefficient formulas give
\[
\begin{aligned}
\left\|Z(\chi)+Z(c\theta)\right\|_\infty
 &\leq P(M)\varepsilon,\\
\left\|Z(\zeta+\eta)-Z\big(c\hat X(\kappa-t)\big)\right\|_\infty
 &\leq P(M)\varepsilon t.
\end{aligned}
\]
Indeed, after the reordering above, the differences inside these norms
contain only first geometric derivatives and field derivatives of order
at most two. In the second line, the terms $Z(\hat T^2)T(\psi_1)$ and
\[
\kappa Z\big(\hat X(v^1+c)+(\hat T^1+1)\hat X(\psi_1)
                              +\hat T^2\hat X(\psi_2)\big)
\]
both have a factor $t$. The three quadratic products that remain are
$(Z^2(\hat T^2))^2$, $Z^2(\hat T^2)Z^2(\kappa-t)$, and
$(Z^2(\kappa-t))^2$.
Equations \eqref{eq: quadratic second geometry} and
\eqref{eq: mixed second geometric product} estimate all three.
In particular, summing over all ordered words gives
\[
\begin{aligned}
\sum_{|\beta|=3}\|[L,Z^\beta](\hat T^2)\|_2
 &\leq P(M)\varepsilon\sum_{|\beta|=3}
       \big(\|Z^\beta(\hat T^2)\|_2+\|Z^\beta(\kappa-t)\|_2\big)
       +P(M)\varepsilon^2t,\\
\sum_{|\beta|=3}\|[L,Z^\beta](\kappa-t)\|_2
 &\leq P(M)\varepsilon\sum_{|\beta|=3}\|Z^\beta(\kappa-t)\|_2\\
 &\quad+P(M)\varepsilon t^2\sum_{|\beta|=3}\|Z^\beta(\hat T^2)\|_2
       +P(M)\varepsilon^2t^2.
\end{aligned}
\]
Combining these estimates with \eqref{eq: geometric differentiated sources}
and the slice-measure identity proves \eqref{eq: coupled geometric norms}
for $m=3$. Applying \eqref{eq: weighted geometric comparison} now gives
\[
\sum_{|\beta|=3}\|Z^\beta(\hat T^2)\|_2\leq P(M)\varepsilon t,
\qquad
\sum_{|\beta|=3}\|Z^\beta(\kappa-t)\|_2\leq P(M)\varepsilon t^2.
\]
The algebraic formula for $\hat T^1+1$ gives its
$P(M)\varepsilon^2t^2$ bound through order three.

The constants depend on the uniform width and Lipschitz bounds for
$\Omega_{t,\delta,\alpha}$ in $(\urr,x_2)$, the first-order frame bounds,
the order-two estimates, and the stated field bootstrap.
Choose $M$ first and then take $P(M)\varepsilon$ small.
The constants in \eqref{eq: weighted geometric comparison} are
independent of $\delta$ and $\alpha$. The zero-lower-time formula
follows by setting the initial terms to zero.
In the construction, these bounds hold for every approximation
before removing either cutoff. Their passage to the limiting acoustic
frame follows from Section~\ref{subsec: common limit}:
\eqref{eq: kappa limit from rescaled map} recovers the inverse density,
\eqref{eq: slice measure after label rescaling} identifies the slice
measure, and \eqref{eq: locally uniform weak geometry convergence}
identifies every ordered word through order three for the same
sequence at every positive time. The exhaustion in
\eqref{eq: cutoff limit geometric bounds} gives the full-slice bounds
up to vacuum.
\end{proof}

As a byproduct, we have
\begin{proposition}
For $t\in [0, t^{*}], Z\in \{T, \hat{X}\}$, we have
\begin{equation}\notag
\begin{cases}
\|Z(\chi)\|_{L^{2}(\Sigma_{t}^{\ut_{*}, \tilde{u}^{*}})}\leq P(M)\varepsilon,\\
\|Z(\eta+\zeta)\|_{L^{2}(\Sigma_{t}^{\ut_{*}, \tilde{u}^{*}})}\leq P(M)\varepsilon t,\\
\|Z\big(\slashed{g}-1\big)\|_{L^{2}(\Sigma_{t}^{\ut_{*}, \tilde{u}^{*}})}\leq P(M)\varepsilon t,\\
\|Z\big(\Xi\big)\|_{L^{2}(\Sigma_{t}^{\ut_{*}, \tilde{u}^{*}})}\leq P(M)\varepsilon t^{2}.
\end{cases}
\end{equation}
\end{proposition}

\begin{proof}
The bound for $Z(\chi)=Z\big(-\hat{X}^{i}\hat{X}(\psi_{i})+c\hat{X}(\hat{X}^{2})\hat{X}^{1}-c\hat{X}(\hat{X}^{1})\hat{X}^{2}\big)$ and $Z(\zeta+\eta)=Z\big(-\kappa\hat{T}^{j}\hat{X}(\psi_{j})-\hat{X}^{i}T(\psi_{i})-\kappa\hat{X}(c)+c\hat{X}(\kappa)\big)$ follows from the estimate \eqref{eq: higher order control of the frame} from previous proposition.

Commuting $Z$ with the transport equation \eqref{eq: formulas to control the geometry} of $\slashed{g}$, we have
\[
\begin{split}
L(Z(\slashed{g}))&=Z(L(\slashed{g}))+[L, Z](\slashed{g})\\
&=Z(2\chi\cdot \slashed{g})-(\textbf{$\chi$ or $\zeta+\eta$})\cdot\hat{X}(\slashed{g})\\
&=2Z(\chi)\cdot \slashed{g}+2\chi\cdot Z(\slashed{g})-(\textbf{$\chi$ or $\zeta+\eta$})\cdot \hat{X}(\slashed{g}).
\end{split}
\]
Therefore, we have
\[
L\big( \|Z(\slashed{g})\|_{L^{2}(\Sigma_{t}^{\ut_*,\tilde{u}^{*}})}\big)\leq P(M)\varepsilon \sum_{|\beta|=1}\|Z^{\beta}(\slashed{g})\|_{L^{2}(\Sigma_{t}^{\ut_*,\tilde{u}^{*}})}+P(M)\varepsilon.
\]
Recall that 
\[
Z(\slashed{g}-1)(0, \tilde{u}, \vartheta)=0.
\]
By Gronwall inequality, we have 
\[
\|Z(\slashed{g}-1)\|_{L^{2}(\Sigma_{t}^{\ut_*,\tilde{u}^{*}})}\leq P(M)\varepsilon t.
\]
provided that $\varepsilon$ is sufficiently small.

Commuting $Z$ with the transport equation \eqref{eq: formulas to control the geometry} of $\Xi$, we have
\[
\begin{split}
L(Z(\Xi))&=Z\big(L(\Xi)\big)+[L, Z](\Xi)\\
&=Z\big(\frac{1}{\sqrt{\slashed{g}}}(\zeta+\eta)\big)-(\textbf{$\chi$ or $\zeta+\eta$})\cdot\hat{X}(\Xi)\\
&=-\frac{1}{2}\frac{Z(\slashed{g})}{\slashed{g}^{\frac{3}{2}}}\cdot (\zeta+\eta)+\frac{1}{\sqrt{\slashed{g}}}\cdot Z(\zeta+\eta)-(\textbf{$\chi$ or $\zeta+\eta$})\cdot\hat{X}(\Xi).
\end{split}
\]
Therefore, we have
\[
L\big( \|Z(\Xi)\|_{L^{2}(\Sigma_{t}^{\ut_*,\tilde{u}^{*}})}\big)\leq P(M)\varepsilon \sum_{|\beta|=1}\|Z^{\beta}(\Xi)\|_{L^{2}(\Sigma_{t}^{\ut_*,\tilde{u}^{*}})}+P(M)\varepsilon t.
\]
Recall that 
\[
Z(\Xi)(0, \tilde{u}, \vartheta)=0.
\]
By Gronwall inequality, we have 
\[
\|Z(\Xi)\|_{L^{2}(\Sigma_{t}^{\ut_*,\tilde{u}^{*}})}\leq P(M)\varepsilon t^{2}.
\]
\end{proof}

\subsection{The first Sobolev inequality and second derivatives of the metric}
We first establish the embedding from one square-integrable spatial
derivative to fourth-power integrability on the slices. It controls
products of first derivatives and yields $L^2$ estimates for the second
ordered geometric derivatives of both the metric coefficient
$\slashed g$ and the shift coefficient $\Xi$, using the first-order pointwise bounds and the stated $L^2$ bounds.

\begin{lemma}
For all $t\in [0, t^{*}]$ and $f\in W^{1, 2}(\Sigma_{t}^{\ut_{*}, \tilde{u}^{*}})$, we have
\begin{equation}\label{eq: W12 to L4}
\begin{cases}
\|f\|_{L^{4}(\Sigma_{t}^{\ut_{*}, \tilde{u}^{*}})}\lesssim \Big( \|f\|_{L^{2}(\Sigma_{t}^{\ut_{*}, \tilde{u}^{*}})}+\|\Xr(f)\|_{L^{2}(\Sigma_{t}^{\ut_{*}, \tilde{u}^{*}})}+\|\Tr(f)\|_{L^{2}(\Sigma_{t}^{\ut_{*}, \tilde{u}^{*}})}\Big),\\
\|f\|_{L^{4}(\Sigma_{t}^{\ut_{*}, \tilde{u}^{*}})}\lesssim \Big( \|f\|_{L^{2}(\Sigma_{t}^{\ut_{*}, \tilde{u}^{*}})}+\|\hat{X}(f)\|_{L^{2}(\Sigma_{t}^{\ut_{*}, \tilde{u}^{*}})}+\|T(f)\|_{L^{2}(\Sigma_{t}^{\ut_{*}, \tilde{u}^{*}})}\Big).
\end{cases}
\end{equation}
where $C$ depends only on the fixed parameters of
Section~\ref{subsec: notation conventions}, is independent of the cutoff
$\alpha\in(0,1/100)$, and is valid for sufficiently small $\varepsilon$. 
\end{lemma}

\begin{proof}
We recall the classical Sobolev embedding $W^{1, 2}\rightarrow L^{4}$ for rectangular region $[\tilde{u}_{*}, \tilde{u}^{*}]\times \mathbb{T}$, 
\[
\|f\|^{2}_{L^{4}([\tilde{u}_{*}, \tilde{u}^{*}]\times \mathbb{T})}\leq C\Big( \|f\|^{2}_{L^{2}([\tilde{u}_{*}, \tilde{u}^{*}]\times \mathbb{T})}+\|\frac{\partial f}{\partial \tilde{u}}\|^{2}_{L^{2}([\tilde{u}_{*}, \tilde{u}^{*}]\times \mathbb{T})}+\|\frac{\partial f}{\partial \vartheta}\|^{2}_{L^{2}([\tilde{u}_{*}, \tilde{u}^{*}]\times \mathbb{T})} \Big).
\]
The interval length is
$\frac{(\frac34-\alpha)(\gamma+1)\mathbf c_*}{\gamma-1}$, so it is bounded
above and below by positive fixed constants for
$0<\alpha<1/100$. Rescaling to a fixed rectangle therefore gives
a constant $C$ independent of $\alpha$. Note that
\[
\begin{cases}
\frac{\partial f}{\partial \tilde{u}}=T(f)+\Xi\sqrt{\slashed{g}}\cdot \hat{X}(f),\\
\frac{\partial f}{\partial \vartheta}=\sqrt{\slashed{g}}\cdot \hat{X}(f),\\
\frac{\partial f}{\partial \tilde{u}}=-\frac{(\kappa\hat{T}^{1}+\Xi\sqrt{\slashed{g}}\hat{X}^{1})}{\kappar}\cdot \Tr(f)+(\kappa\hat{T}^{2}+\Xi\sqrt{\slashed{g}}\hat{X}^{2})\cdot \Xr(f),\\
\frac{\partial f}{\partial \vartheta}=-\frac{\sqrt{\slashed{g}}\hat{X}^{1}}{\kappar}\cdot \Tr(f)+\sqrt{\slashed{g}}\hat{X}^{2}\cdot \Xr(f).
\end{cases}
\]
Note that all the coefficients are bounded by $O(1)$ provided that $\varepsilon$ is sufficiently small. Therefore, we have
\[
\begin{cases}
|\frac{\partial f}{\partial \tilde{u}}|^{2}+|\frac{\partial f}{\partial \vartheta}|^{2}\lesssim |T(f)|^{2}+|\hat{X}(f)|^{2},\\
|\frac{\partial f}{\partial \tilde{u}}|^{2}+|\frac{\partial f}{\partial \vartheta}|^{2}\lesssim |\Tr(f)|^{2}+|\Xr(f)|^{2}.
\end{cases}
\]
This completes the proof of the lemma.
\end{proof}

\begin{proposition}
For $t\in [0, t^{*}], Z\in \{T, \hat{X}\}$ and $|\beta|=2$, we have
\begin{equation}\notag
\begin{cases}
\|Z^{\beta}(\slashed{g})\|_{L^{2}(\Sigma_{t}^{\ut_{*}, \tilde{u}^{*}})}\leq P(M)\varepsilon t,\\
\|Z^{\beta}(\Xi)\|_{L^{2}(\Sigma_{t}^{\ut_{*}, \tilde{u}^{*}})}\leq P(M)\varepsilon t^{2}.
\end{cases}
\end{equation}
\end{proposition}

\begin{proof}
Let $Z_1,Z_2\in\{T,\hat X\}$ be arbitrary, so that $Z_1Z_2$
includes both mixed orders as well as $T^2$ and $(\hat X)^2$.
The commutator identity
\[
[L,Z_1Z_2](f)=[L,Z_1]Z_2(f)+Z_1([L,Z_2](f))
\]
and $L(\slashed g)=2\chi\slashed g$ give the exact formula
\[
\begin{aligned}
LZ_1Z_2(\slashed g)
&=2Z_1Z_2(\chi)\slashed g+2Z_2(\chi)Z_1(\slashed g)
  +2Z_1(\chi)Z_2(\slashed g)+2\chi Z_1Z_2(\slashed g)\\
&\quad+[L,Z_1]Z_2(\slashed g)+Z_1([L,Z_2](\slashed g)).
\end{aligned}
\]
Here $[L,Z_j]=-\chi\hat X$ when $Z_j=\hat X$, and
$[L,Z_j]=-(\zeta+\eta)\hat X$ when $Z_j=T$.
The terms with one derivative on each factor are bounded by
H\"older's inequality and \eqref{eq: W12 to L4}; for example,
\[
\begin{aligned}
\|Z_1(\chi)Z_2(\slashed g)\|_{L^2(\Sigma_t^{\ut_*,\ut^*})}
&\leq \|Z_1(\chi)\|_{L^4(\Sigma_t^{\ut_*,\ut^*})}
       \|Z_2(\slashed g)\|_{L^4(\Sigma_t^{\ut_*,\ut^*})}\\
&\leq C\left(\|Z_1(\chi)\|_{L^2(\Sigma_t^{\ut_*,\ut^*})}
 +\sum_{Z\in\{T,\hat X\}}\|ZZ_1(\chi)\|_{L^2(\Sigma_t^{\ut_*,\ut^*})}\right)\\
&\quad\cdot\left(\|Z_2(\slashed g)\|_{L^2(\Sigma_t^{\ut_*,\ut^*})}
 +\sum_{Z\in\{T,\hat X\}}\|ZZ_2(\slashed g)\|_{L^2(\Sigma_t^{\ut_*,\ut^*})}\right).
\end{aligned}
\]
The same estimate applies with $\zeta+\eta$ in place of $\chi$.
Their derivatives through order two are controlled by
\eqref{eq: higher order control of the frame}, their defining formulas,
and $(\mathbf B_\infty)$. Using also the first-order estimates just
proved and $L(\sqrt{\slashed g})=\chi\sqrt{\slashed g}$, we obtain
\[
\frac{d}{dt}\|Z_1Z_2(\slashed g)\|_{L^2(\Sigma_t^{\ut_*,\ut^*})}
\leq P(M)\varepsilon
 \sum_{|\beta|=2}\|Z^\beta(\slashed g)\|_{L^2(\Sigma_t^{\ut_*,\ut^*})}
 +P(M)\varepsilon.
\]
Summing over all four ordered pairs and using
\eqref{eq: positive time initial weights}, we obtain
\[
\sum_{|\beta|=2}\|Z^\beta(\slashed g)(\delta)\|_{L^2(\Sigma_\delta^{\ut_*,\ut^*})}
\leq C\varepsilon^2\delta^2.
\]
Gronwall's inequality therefore gives, for $\delta\leq t\leq1$,
\[
\begin{aligned}
\sum_{|\beta|=2}\|Z^\beta(\slashed g)\|_{L^2(\Sigma_t^{\ut_*,\ut^*})}
&\leq e^{P(M)\varepsilon(t-\delta)}
 \big(C\varepsilon^2\delta^2+P(M)\varepsilon(t-\delta)\big)\\
&\leq P(M)\varepsilon t.
\end{aligned}
\]

For $\Xi$, differentiating $L(\Xi)=\slashed g^{-1/2}(\zeta+\eta)$
gives
\[
\begin{aligned}
LZ_1Z_2(\Xi)
&=Z_1Z_2(\slashed g^{-1/2})(\zeta+\eta)
 +Z_2(\slashed g^{-1/2})Z_1(\zeta+\eta)\\
&\quad+Z_1(\slashed g^{-1/2})Z_2(\zeta+\eta)
 +\slashed g^{-1/2}Z_1Z_2(\zeta+\eta)\\
&\quad+[L,Z_1]Z_2(\Xi)+Z_1([L,Z_2](\Xi)).
\end{aligned}
\]
In particular,
\[
Z_1Z_2(\slashed g^{-1/2})
=\frac34\slashed g^{-5/2}Z_1(\slashed g)Z_2(\slashed g)
 -\frac12\slashed g^{-3/2}Z_1Z_2(\slashed g).
\]
Use $\frac14\leq\slashed g\leq4$, the preceding bounds, and the same
$L^4\cdot L^4\to L^2$ estimate for the products of first derivatives.
The resulting inequality is
\[
\frac{d}{dt}\|Z_1Z_2(\Xi)\|_{L^2(\Sigma_t^{\ut_*,\ut^*})}
\leq P(M)\varepsilon
 \sum_{|\beta|=2}\|Z^\beta(\Xi)\|_{L^2(\Sigma_t^{\ut_*,\ut^*})}
 +P(M)\varepsilon t.
\]
Summing over the four ordered pairs, we use the initial bound
\[
\sum_{|\beta|=2}\|Z^\beta(\Xi)(\delta)\|_{L^2(\Sigma_\delta^{\ut_*,\ut^*})}
\leq C\varepsilon\delta^2
\]
from \eqref{eq: positive time initial weights}. Thus
\[
\begin{aligned}
\sum_{|\beta|=2}\|Z^\beta(\Xi)\|_{L^2(\Sigma_t^{\ut_*,\ut^*})}
&\leq e^{P(M)\varepsilon(t-\delta)}
 \left(C\varepsilon\delta^2+P(M)\varepsilon\int_\delta^t s\,ds\right)\\
&\leq P(M)\varepsilon t^2.
\end{aligned}
\]

Both estimates are uniform in $\delta,\alpha$; the zero-lower-time
versions follow by setting the initial terms to zero.
This proves both assertions with the ordered-word convention of
Section~\ref{subsec: notation conventions}.
\end{proof}

In summary, we have
\begin{proposition}
For $t\in [0, t^{*}], Z\in \{T, \hat{X}\}$ and $|\alpha|\leq 1$, we have $L^{\infty}$ estimate on geometry,
\begin{equation}\label{eq: L infty control of the geometry}
\begin{cases}
\|Z^{\alpha}(\hat{T}^{1}+1)\|_{L^{\infty}(\Sigma_{t}^{\ut_*,\ut^*})}\leq P(M)\varepsilon^{2}t^{2},\\
\|Z^{\alpha}(\hat{T}^{2})\|_{L^{\infty}(\Sigma_{t}^{\ut_*,\ut^*})}\leq P(M)\varepsilon t,\\
\|Z^{\alpha}(\kappa-t)\|_{L^{\infty}(\Sigma_{t}^{\ut_*,\ut^*})}\leq P(M)\varepsilon t^{2},\\
\|\chi\|_{L^{\infty}(\Sigma_{t}^{\ut_*,\ut^*})}\leq P(M)\varepsilon,\\
\|(\zeta+\eta)\|_{L^{\infty}(\Sigma_{t}^{\ut_*,\ut^*})}\leq P(M)\varepsilon t,\\
\|\slashed{g}-1\|_{L^{\infty}(\Sigma_{t}^{\ut_*,\ut^*})}\leq P(M)\varepsilon t,\\
\|\Xi\|_{L^{\infty}(\Sigma_{t}^{\ut_*,\ut^*})}\leq P(M)\varepsilon t^{2}.
\end{cases}
\end{equation}
For $t\in [0, t^{*}], Z\in \{T, \hat{X}\}$ and $|\alpha|\leq 3, |\beta|\leq 2$, we have $L^{2}$ estimate on geometry,
\begin{equation}\label{eq: L 2 control of the geometry}
\begin{cases}
\|Z^{\alpha}(\hat{T}^{1}+1)\|_{L^{2}(\Sigma_{t}^{\ut_*,\ut^*})}\leq P(M)\varepsilon^{2}t^{2},\\
\|Z^{\alpha}(\hat{T}^{2})\|_{L^{2}(\Sigma_{t}^{\ut_*,\ut^*})}\leq P(M)\varepsilon t,\\
\|Z^{\alpha}(\kappa-t)\|_{L^{2}(\Sigma_{t}^{\ut_*,\ut^*})}\leq P(M)\varepsilon t^{2},\\
\|Z^{\beta}(\chi)\|_{L^{2}(\Sigma_{t}^{\ut_*,\ut^*})}\leq P(M)\varepsilon,\\
\|Z^{\beta}(\zeta+\eta)\|_{L^{2}(\Sigma_{t}^{\ut_*,\ut^*})}\leq P(M)\varepsilon t,\\
\|Z^{\beta}(\slashed{g}-1)\|_{L^{2}(\Sigma_{t}^{\ut_*,\ut^*})}\leq P(M)\varepsilon t,\\
\|Z^{\beta}(\Xi)\|_{L^{2}(\Sigma_{t}^{\ut_*,\ut^*})}\leq P(M)\varepsilon t^{2}.
\end{cases}
\end{equation}
\end{proposition}

\subsection{The Sobolev inequality for pointwise control}
The second embedding bounds a function pointwise by two
square-integrable spatial derivatives. Applied to the profile remainder,
it converts the four-derivative energy estimate into the lower-order
pointwise bounds used in the final bootstrap step.

\begin{lemma}
For all $t\in [0, t^{*}]$ and $f\in W^{2, 2}(\Sigma_{t}^{\ut_{*}, \tilde{u}^{*}})$, we have
\begin{equation}\label{eq: W22 to L infty}
\|f\|_{L^{\infty}(\Sigma_{t}^{\ut_{*}, \tilde{u}^{*}})}\leq C \sum_{a+b\leq 2}\|\Tr^{a}\Xr^{b}(f)\|_{L^{2}(\Sigma_{t}^{\ut_{*}, \tilde{u}^{*}})}.
\end{equation}
where $C$ depends only on the fixed parameters of
Section~\ref{subsec: notation conventions}, is independent of the cutoff
$\alpha\in(0,1/100)$, and is valid for sufficiently small $\varepsilon$.
\end{lemma}

\begin{proof}
We recall the classical Sobolev embedding $W^{2, 2}\rightarrow L^{\infty}$ for rectangle $[\tilde{u}_{*}, \tilde{u}^{*}]\times \mathbb{T}$,
\[
\begin{aligned}
\|f\|^2_{L^\infty([\tilde u_*,\tilde u^*]\times\mathbb T)}
&\leq C\Big(\|f\|^2_{L^2([\tilde u_*,\tilde u^*]\times\mathbb T)}
+\|\partial_{\tilde u}(f)\|^2_{L^2([\tilde u_*,\tilde u^*]\times\mathbb T)}\\
&\quad+\|\partial_{\tilde u}^2(f)\|^2_{L^2([\tilde u_*,\tilde u^*]\times\mathbb T)}
+\|\partial_\vartheta(f)\|^2_{L^2([\tilde u_*,\tilde u^*]\times\mathbb T)}\\
&\quad+\|\partial_\vartheta^2(f)\|^2_{L^2([\tilde u_*,\tilde u^*]\times\mathbb T)}\Big).
\end{aligned}
\]
In view of the Jacobian matrix of $(t, \tilde u, \vartheta)\mapsto (t, x_{1}, x_{2})$, 
\[
\begin{pmatrix}
\frac{\partial t}{\partial t}&\frac{\partial t}{\partial \tilde u}&\frac{\partial t}{\partial \vartheta}\\
\frac{\partial x_{1}}{\partial t}&\frac{\partial x_{1}}{\partial \tilde u}&\frac{\partial x_{1}}{\partial \vartheta}\\
\frac{\partial x_{2}}{\partial t}&\frac{\partial x_{2}}{\partial \tilde u}&\frac{\partial x_{2}}{\partial \vartheta}
\end{pmatrix}=
\begin{pmatrix}
1&0&0\\
L^{1}&\kappa\hat{T}^{1}+\Xi\sqrt{\slashed{g}}\hat{X}^{1}&\sqrt{\slashed{g}}\hat{X}^{1}\\
L^{2}&\kappa\hat{T}^{2}+\Xi\sqrt{\slashed{g}}\hat{X}^{2}&\sqrt{\slashed{g}}\hat{X}^{2}\\
\end{pmatrix}
\]
Note that
\[
\begin{cases}
\frac{\partial f}{\partial \tilde{u}}=-\frac{(\kappa\hat{T}^{1}+\Xi\sqrt{\slashed{g}}\hat{X}^{1})}{\kappar}\cdot \Tr(f)+(\kappa\hat{T}^{2}+\Xi\sqrt{\slashed{g}}\hat{X}^{2})\cdot \Xr(f),\\
\frac{\partial f}{\partial \vartheta}=-\frac{\sqrt{\slashed{g}}\hat{X}^{1}}{\kappar}\cdot \Tr(f)+\sqrt{\slashed{g}}\hat{X}^{2}\cdot \Xr(f).
\end{cases}
\]
Note that all the coefficients are bounded by a fixed constant provided that $\varepsilon$ is sufficiently small. Therefore, we have
\[
|\frac{\partial f}{\partial \tilde{u}}|^{2}+|\frac{\partial f}{\partial \vartheta}|^{2}\lesssim |\Tr(f)|^{2}+|\Xr(f)|^{2}.
\]
and we also have (applying the previous observation to $\frac{\partial f}{\partial \tilde{u}}$ and $\frac{\partial f}{\partial \vartheta}$ separately),
\[
|\frac{\partial^2 f}{\partial \tilde{u}^{2}}|^{2}+|\frac{\partial^2 f}{\partial \vartheta^{2}}|^{2}\lesssim |\Tr(\frac{\partial f}{\partial \tilde{u}})|^{2}+ |\Xr(\frac{\partial f}{\partial \tilde{u}})|^{2}+|\Tr(\frac{\partial f}{\partial \vartheta})|^{2}+ |\Xr(\frac{\partial f}{\partial \vartheta})|^{2}
\]
Note that
\[
\begin{cases}
\Zr \Big( \kappa\hat{T}^{i}+\Xi\sqrt{\slashed{g}}\hat{X}^{i} \Big)=\Zr(\kappa)\hat{T}^{i}+\kappa\Zr(\hat{T}^{i})+\Zr(\Xi)\sqrt{\slashed{g}}\hat{X}^{i}+\Xi\frac{1}{2}\slashed{g}^{-\frac{1}{2}}\Zr(\slashed{g})\hat{X}^{i}+\Xi\sqrt{\slashed{g}}\Zr(\hat{X}^i),\\
\Zr \Big( \sqrt{\slashed{g}}\hat{X}^{i} \Big)=\frac{1}{2}\slashed{g}^{-\frac{1}{2}}\hat{X}^{i}\Zr(\slashed{g })+\sqrt{\slashed{g}}\Zr(\hat{X}^{i}).
\end{cases}
\]
In view of \eqref{eq: L infty control of the geometry} and \eqref{eq: L 2 control of the geometry}, we have
\[
\begin{cases}
|\Zr \Big( \frac{\kappa\hat{T}^{1}+\Xi\sqrt{\slashed{g}}\hat{X}^{1}}{\kappar} \Big)|\leq P(M)\varepsilon + P(M)\varepsilon \cdot |\Zr(\slashed{g})|+ P(M)\varepsilon \cdot |\Zr(\Xi)|,\\
|\Zr \Big( \kappa\hat{T}^{2}+\Xi\sqrt{\slashed{g}}\hat{X}^{2} \Big)|\leq P(M)\varepsilon + P(M)\varepsilon \cdot  |\Zr(\slashed{g})|+ 100\cdot |\Zr(\Xi)|,\\
|\Zr \Big( \frac{\sqrt{\slashed{g}}\hat{X}^{1}}{\kappar} \Big)|\leq P(M)\varepsilon \cdot |\Zr(\slashed{g})| + P(M)\varepsilon \\
|\Zr \Big( \sqrt{\slashed{g}}\hat{X}^{2} \Big)|\leq 100\cdot |\Zr(\slashed{g})|+ P(M)\varepsilon.
\end{cases}
\]
Therefore, we have
\[
\begin{split}
&\|\frac{\partial^2 f}{\partial \tilde{u}^{2}}\|_{L^{2}(\Sigma_{t}^{\ut_{*}, \tilde{u}^{*}})}+\|\frac{\partial^2 f}{\partial \vartheta^{2}}\|_{L^{2}(\Sigma_{t}^{\ut_{*}, \tilde{u}^{*}})}\\ &\leq C\Big(\sum_{|\alpha|\leq 2}\|\Zr^{\alpha}(f)\|_{L^{2}(\Sigma_{t}^{\ut_*,\tilde{u}^{*}})}\\ &\qquad+ \|\Zr(\slashed{g})\cdot \Zr(f)\|_{L^{2}(\Sigma_{t}^{\ut_{*}, \tilde{u}^{*}})}\\ &\qquad+\|\Zr(\Xi)\cdot \Zr(f)\|_{L^{2}(\Sigma_{t}^{\ut_{*}, \tilde{u}^{*}})}\Big)\\
&\leq C\Big(\sum_{|\alpha|\leq 2}\|\Zr^{\alpha}(f)\|_{L^{2}(\Sigma_{t}^{\ut_{*}, \tilde{u}^{*}})}\\ &\qquad+ \|\Zr(\slashed{g})\|_{L^{4}(\Sigma_{t}^{\ut_{*}, \tilde{u}^{*}})}\cdot \|\Zr(f)\|_{L^{4}(\Sigma_{t}^{\ut_{*}, \tilde{u}^{*}})}\\
&+\|\Zr(\Xi)\|_{L^{4}(\Sigma_{t}^{\ut_{*}, \tilde{u}^{*}})}\cdot \|\Zr(f)\|_{L^{4}(\Sigma_{t}^{\ut_{*}, \tilde{u}^{*}})}\Big)\\
&\leq C\Big(\sum_{|\alpha|\leq 2}\|\Zr^{\alpha}(f)\|_{L^{2}(\Sigma_{t}^{\ut_{*}, \tilde{u}^{*}})}\\ &\qquad+ \|Z(\slashed{g})\|_{L^{4}(\Sigma_{t}^{\ut_{*}, \tilde{u}^{*}})}\cdot \|\Zr(f)\|_{L^{4}(\Sigma_{t}^{\ut_*,\tilde{u}^{*}})}\\
&+\|Z(\Xi)\|_{L^{4}(\Sigma_{t}^{\ut_{*}, \tilde{u}^{*}})}\cdot \|\Zr(f)\|_{L^{4}(\Sigma_{t}^{\ut_{*}, \tilde{u}^{*}})}\Big)
\end{split}
\]
Applying $W^{1, 2}\rightarrow L^{4}$ estimate \eqref{eq: W12 to L4}, we can conclude that 
\[
\|\frac{\partial^2 f}{\partial \tilde{u}^{2}}\|_{L^{2}(\Sigma_{t}^{\ut_{*}, \tilde{u}^{*}})}+\|\frac{\partial^2 f}{\partial \vartheta^{2}}\|_{L^{2}(\Sigma_{t}^{\ut_{*}, \tilde{u}^{*}})}\leq C\sum_{|\alpha|\leq 2}\|\Zr^{\alpha}(f)\|_{L^{2}(\Sigma_{t}^{\ut_{*}, \tilde{u}^{*}})}
\]
provided that $\varepsilon$ is sufficiently small. This completes the proof of the lemma.
\end{proof}

The geometric estimates and the two Sobolev inequalities in this section provide the geometric bootstrap improvement and the pointwise conversion used in the proofs of Theorems~\ref{theorem: a prior estimate in W_1} and~\ref{theorem: a prior estimate in W_2}. This completes the geometric part of those proofs.

Combining $(\mathbf{B}_{2})$ with the $W^{2, 2}\rightarrow L^{\infty}$ embedding \eqref{eq: W22 to L infty} gives the following $L^{\infty}$ bounds for $\Zr^{\alpha}(V_{k})=\Zr^{\alpha}(U-\Ur_{k}), k=0, 1$: 
\begin{proposition}
In the near-vacuum chart take $(k,\gamma)$ as in
\eqref{eq: different treatment}. For $0<t\leq t^*$,
$\Zr\in\{\Tr,\Xr\}$ and $|\alpha|\leq2$,
\[
\|\Zr^{\alpha}(V_{k})\|_{L^{\infty}(\Sigma_{t}^{\ut_{*}, \tilde{u}^{*}})}\lesssim P(M)\varepsilon t^{k+1}.
\]
\end{proposition}

\section{Energy estimates}\label{sec: energy estimates}
We derive an energy inequality for the commuted profile remainders and
use it in the two regions. Away from vacuum, the characteristic flux
controls the leading null terms. Near vacuum, this flux degenerates;
we therefore keep the full quadratic form produced by the
damping--antidamping cancellation and estimate its largest eigenvalue.
A Gronwall inequality with a singular time coefficient then closes the
energy bootstrap in both regions.

This proves the energy part of Theorems~\ref{theorem: a prior estimate in W_1}
and~\ref{theorem: a prior estimate in W_2}. Subsection~6.1 derives
the commuted energy inequality, Subsection~6.2 proves the required
Gronwall inequality, and Subsections~6.3--6.4 close $(\mathbf B_2)$
away from vacuum and up to vacuum, respectively.

For solutions started at $\delta>0$, every energy identity below
has lower time $\delta$, with initial energy
$\mathcal E(V_{k;a,b})(\delta)$ included.
All commutation and bulk estimates are unchanged.
Equation~\eqref{eq: positive time energy inequality} and its
integrating-factor argument give the uniform version needed
before taking the limit.

We use the scaled residual $\mathring R_k$ defined in
Lemma~\ref{lem: bounds on Res}; hence the residual in physical
coordinates is $t^{-1}\mathring R_k$.
The remainder equations in this section are written in physical
coordinates: their $\partial_t$ keeps $(x_1,x_2)$ fixed, while
$\Tr=-t\partial_{x_1}$. The coordinate and measure conventions are
collected in Section~\ref{subsec: notation conventions}.
We write $E=\mathcal E$ and $F=\mathcal F$ for the energy and flux.
In the near-vacuum chart, $\mathcal D(t,\tilde u)=\mathcal D(t;\tilde u_*,\tilde u)$, whereas in the away-from-vacuum chart, $\mathcal D(t,u)=\mathcal D(t;0,u)$.
Unless explicitly written otherwise, spacetime integrals in this section use the measure
$(\kappa(t',u,\vartheta)/t')\sqrt{\slashed g(t',u,\vartheta)}\,dt'\,du\,d\vartheta$
(with $\tilde u$ in place of $u$ in the near-vacuum chart).
Slice integrals use $\sqrt{\slashed g}\,du\,d\vartheta$, and integrals over $C_u^t$ use $\sqrt{\slashed g}\,dt\,d\vartheta$.

Recall that $V_{k}=U-\Ur_{k}$ is the remainder after subtracting the approximate solution $\Ur_{k}$. This remainder satisfies
\[
\partial_{t}(V_{k})+A_{1}(U)\partial_{x_1}(V_{k})+A_{2}(U)\partial_{x_2}(V_{k})=-A_{1}(V_{k})\partial_{x_1}(\Ur_{k})-A_{2}(V_{k})\partial_{x_2}(\Ur_{k})-t^{-1}\mathring{R}_{k}.
\]
and $k$ will be $0$ or $1$ in this section.

Recall that 
\[
\begin{cases}
\Ur_{0}(t, \mathring{u}, x_{2})=\Vr_{0}(\mathring{u}, x_{2}),\\
\Ur_{1}(t, \mathring{u}, x_{2})=\Vr_{0}(\mathring{u}, x_{2})+t\cdot \Vr_{1}(\mathring{u}, x_{2}).
\end{cases}
\]
The properties of $\Vr_{k}$ and $\mathring{R}_{k}$ needed in this section are as follows:
\begin{itemize}
\item The exact form of $\Vr_{0}$,
\[
\Ur_{0}(\mathring{u}, x_{2})=\begin{pmatrix}
\frac{2}{\gamma-1}c_{r}(0, x_{2})+\frac{2}{\gamma+1}\big( \mathring{u}-(v_{r}^{1}+c_{r})(0, x_{2}) \big)\\
v^{1}_{r}(0, x_{2})+\frac{2}{\gamma+1}\big( \mathring{u}-(v_{r}^{1}+c_{r})(0, x_{2}) \big)\\
v_{r}^{2}(0, x_{2})
\end{pmatrix}
\]
\item When $k=1$ (for all $\gamma\in(1,3)$ away from vacuum, and for $\gamma\in\{7/5,5/3\}$ near vacuum), the rough bound for $\Vr_{1}(\mathring{u}, x_{2})$, for any $a+b\leq \Ntop+1$, 
\[
|\frac{\partial^{a}}{\partial \mathring{u}^{a}}\frac{\partial^{b}}{\partial x_{2}^{b}}(\Vr_{1})|\leq C\varepsilon.
\]
which is equivalent to
\[
|\Tr^{a}\Xr^{b}\big( \Vr_{1}(\frac{x_{1}}{t}, x_{2}) \big)|\leq C\varepsilon.
\]
\item The rough bound for $\mathring{R}_{k}(t, \mathring{u}, x_{2})$, for any $a+b\leq \Ntop$, 
\[
|\frac{\partial^{a}}{\partial \mathring{u}^{a}}\frac{\partial^{b}}{\partial x_{2}^{b}}(\Rr_{k})|\leq C\varepsilon t^{k+1}.
\]
which is equivalent to
\[
|\Tr^{a}\Xr^{b}\big( \Rr_{k}(t, \frac{x_{1}}{t}, x_{2}) \big)|\leq C\varepsilon t^{k+1}.
\]
\end{itemize}

In view of the basic energy estimate, we commute $\Tr^{a}\Xr^{b}$ with $a+b\leq\Ncomm$, where
$\Ncomm=\Ntop$ away from vacuum and $\Ncomm=4$ near vacuum, with 
\[
t\big(\partial_{t}(V_{k})+A_{1}(U)\partial_{x_1}(V_{k})+A_{2}(U)\partial_{x_2}(V_{k})\big)=t\big(-A_{1}(V_{k})\partial_{x_1}(\Ur_{k})-A_{2}(V_{k})\partial_{x_2}(\Ur_{k})-t^{-1}\mathring{R}_{k}\big),
\]
or equivalently,
\begin{equation}\label{eq: eq of V_k}
t\partial_{t}(V_{k})-A_{1}(U)\Tr(V_{k})+t\cdot A_{2}(U)\partial_{x_2}(V_{k})=A_{1}(V_{k})\Tr(\Ur_{k})-t\cdot A_{2}(V_{k})\partial_{x_2}(\Ur_{k})-\mathring{R}_{k}.
\end{equation}
The source $A_{1}(V_{k})\Tr(\Ur_{k})$ has no additional factor of $t$. To isolate its damping contribution, write
\[
A_{1}(V_{k})\Tr(\Ur_{k})=\begin{cases}
A_{1}(V_{0})\Tr(\Vr_{0}),\ \ k=0,\\
A_{1}(V_{1})\Tr(\Vr_{0})+t\cdot A_{1}(V_{1})\Tr(\Vr_{1}), \ \ k=1.
\end{cases}
\]
Note that
\[
A_{1}(V_{k})\Tr(\Vr_{0})=-\frac{2}{\gamma+1}\begin{pmatrix}
\frac{\gamma-1}{2}&1&0\\
\frac{\gamma-1}{2}&1&0\\
0&0&0
\end{pmatrix}V_{k}.
\]
This decomposition puts the cases $k=0, 1$ in the common form
\begin{equation}\label{eq: rough eq of V_k}
t\partial_{t}(V_{k})-A_{1}(U)\Tr(V_{k})+t\cdot A_{2}(U)\partial_{x_2}(V_{k})=-\frac{2}{\gamma+1}\begin{pmatrix}
\frac{\gamma-1}{2}&1&0\\
\frac{\gamma-1}{2}&1&0\\
0&0&0
\end{pmatrix}V_{k}-t\cdot F(V_{k}, \Wr_{O(\varepsilon)})-\Rr_{k}.
\end{equation}
Here $F$ denotes a finite sum of bilinear contractions with fixed
coefficients, and $\Wr_{O(\varepsilon)}$ denotes the corresponding
finite family of profile coefficients. The particular contraction
may change from one schematic formula to the next. In the remainder
equation it is vector-valued; in a scalar bulk integral it is
scalar-valued after contraction. For any $a+b\leq\Ntop$, each member
of the coefficient family satisfies
\begin{equation}\notag
|\Tr^{a}\Xr^{b}(\Wr_{O(\varepsilon)})|\leq C\varepsilon.
\end{equation}
The notation $\Wr_{O(1)}$ has the same meaning with the bound $C$
in place of $C\varepsilon$. Both coefficient families are unrelated
to the region $\Wr$ introduced in Section~\ref{subsec: opened fan}.

\subsection{Commuted remainder equations and the master energy inequality}
For simplicity, we use $V_{k; a, b}$ to denote $\Tr^{a}\Xr^{b}(V_{k})=\Tr^{a}\Xr^{b}(U-\Ur_{k})$ with $a+b\leq \Ncomm$, where 
\[
\Ncomm=\begin{cases}
\Ntop, \ \ \textbf{away from vacuum case},\\
4, \ \ \textbf{up to vacuum case}.
\end{cases}
\]

First commute $\Xr^{b}$ with \eqref{eq: rough eq of V_k} and obtain
\begin{equation}\label{eq: eq of V_k 0 b}
\begin{split}
&t\partial_{t}(V_{k; 0, b})-A_{1}(U)\Tr(V_{k; 0, b})+t\cdot A_{2}(U)\partial_{x_2}(V_{k; 0, b})\\
&=-\frac{2}{\gamma+1}\begin{pmatrix}
\frac{\gamma-1}{2}&1&0\\
\frac{\gamma-1}{2}&1&0\\
0&0&0
\end{pmatrix}V_{k; 0, b}\\
&-t\cdot \sum_{b_{1}+b_{2}=b} \binom{b}{b_{1}} F\big(V_{k; 0, b_{1}}, \Xr^{b_2}(\Wr_{O(\varepsilon)})\big)-\Xr^{b}(\Rr_{k})+I_{k; 0, b}.
\end{split}
\end{equation}
where
\[
I_{k; 0, b}:=\sum_{b_{1}+b_{2}=b, b_{1}\geq 1}\binom{b}{b_{1}}\Xr^{b_1}(U^{m})\frac{\partial A_{1}(U)}{\partial U^{m}}V_{k; 1, b_{2}}-\kappar\sum_{b_{1}+b_{2}=b, b_{1}\geq 1}\binom{b}{b_{1}}\Xr^{b_1}(U^{m})\frac{\partial A_{2}(U)}{\partial U^{m}}V_{k; 0, b_{2}+1}.
\]

Next commute $\Tr^{a}$ with \eqref{eq: eq of V_k 0 b} and obtain
\begin{equation}\label{eq: eq of V_k a b}
t\partial_{t}(V_{k; a, b})-A_{1}(U)\Tr(V_{k; a, b})+t\cdot A_{2}(U)\partial_{x_2}(V_{k; a, b})=\textbf{Source terms}_{k; a, b}
\end{equation}
where terms containing the index $a-1$ are omitted when $a=0$, and
\[
\begin{split}
&\textbf{Source terms}_{k; a, b}\\
&:=\underbrace{-\frac{2}{\gamma+1}\begin{pmatrix}
\frac{\gamma-1}{2}&1&0\\
\frac{\gamma-1}{2}&1&0\\
0&0&0
\end{pmatrix}V_{k; a, b}+aV_{k; a, b}+a\Tr(U^{m})\frac{\partial A_{1}(U)}{\partial U^{m}}V_{k; a, b}-a\kappar\Tr(U^{m})\frac{\partial A_{2}(U)}{\partial U^{m}}V_{k; a-1, b+1}}_{\mathbf{S}_{1; a, b}}\\
&\underbrace{-t\cdot \sum_{a_{1}+a_{2}=a}\sum_{b_{1}+b_{2}=b} \binom{b}{b_{1}}\binom{a}{a_{1}} F\big(V_{k; a_{1}, b_{1}}, \Tr^{a_{2}}\Xr^{b_2}(\Wr_{O(\varepsilon)})\big)}_{\mathbf{S}_{2; a, b}}\underbrace{-\Tr^{a}\Xr^{b}(\Rr_{k})}_{\mathbf{S}_{3; a, b}}+\underbrace{\Tr^{a}(I_{k; 0, b})+I_{k; a, b}}_{\mathbf{S}_{4; a, b}}.
\end{split}
\]
and
\[
I_{k; a, b}:=\sum_{a_{1}+a_{2}=a, a_{1}\geq 2}\binom{a}{a_1}\Tr^{a_{1}}(U^{m})\frac{\partial A_{1}(U)}{\partial U^{m}}V_{k; a_{2}+1, b}-\sum_{a_{1}+a_{2}=a, a_{1}\geq 2}\binom{a}{a_1}\kappar\Tr^{a_{1}}(U^{m})\frac{\partial A_{2}(U)}{\partial U^{m}}V_{k; a_{2}, b+1}.
\]
Applying the basic energy identity to $V_{k; a, b}$ in \eqref{eq: eq of V_k a b} gives
\begin{equation}\label{eq: higher order energy estimate}
\begin{split}
&E(V_{k; a, b})(t, \tilde{u})+F(V_{k; a, b})(t, \tilde{u})\leq E(V_{k; a, b})(0, \tilde{u})+F(V_{k; a, b})(t, \tilde{u}_{*})\\
&+\sum_{j=1}^{4}\underbrace{\int_{\mathcal{D}(t, \tilde{u})} 2\big( V_{k; a, b}, \mathbf{S}_{j; a, b} \big)  }_{\textbf{bulk terms from source terms}}\\
&+\underbrace{\int_{\mathcal{D}(t, \tilde{u})} \Big( -(V_{k; a, b}, \frac{\partial A_{1}(U)}{\partial U^{m}}V_{k; a, b})\Tr(U^{m})+\kappar (V_{k; a, b}, \frac{\partial A_{2}(U)}{\partial U^{m}}V_{k; a,b})\Xr(U^{m}) \Big)}_{\textbf{bulk terms from integration by part}}.
\end{split}
\end{equation}

\paragraph{Bulk terms and cancellation.}
We now separate the contributions of the commuted sources and the
integration-by-parts terms. Keeping their leading quadratic parts
explicit reveals the cancellation that will be used in both bootstrap
arguments.
\begin{itemize}
\item The bulk term associated with $\mathbf{S}_{1; a, b}$. We decompose $a\Tr({U^{m}})\frac{\partial A_{1}(U)}{\partial U^{m}}V_{k ; a, b}$ as follows,
\[
\begin{split}
a\Tr({U^{m}})\frac{\partial A_{1}(U)}{\partial U^{m}}V_{k ; a, b}&=a\Tr(\Ur_{k}^{m})\frac{\partial A_{1}(U)}{\partial U^{m}}V_{k ; a, b}+aV_{k; 1, 0}^{m}\frac{\partial A_{1}(U)}{\partial U^{m}}V_{k ; a, b}\\
&=a\Tr(\Vr_{0}^{m})\frac{\partial A_{1}(U)}{\partial U^{m}}V_{k ; a, b}+\underbrace{akt\cdot \Tr(\Vr_{1}^{m})\frac{\partial A_{1}(U)}{\partial U^{m}}V_{k ; a, b}}_{t\cdot F(V_{k; a, b}, \Wr_{O(\varepsilon)})}+\underbrace{aV_{k; 1, 0}^{m}\frac{\partial A_{1}(U)}{\partial U^{m}}V_{k ; a, b}}_{F(V_{k; 1, 0}, V_{k; a, b})}\\
&=a\cdot \begin{pmatrix}
-\frac{2}{\gamma+1} & -\frac{\gamma-1}{\gamma+1} & 0\\
-\frac{\gamma-1}{\gamma+1} & -\frac{2}{\gamma+1} & 0\\
0 & 0 & -\frac{2}{\gamma+1}
\end{pmatrix}V_{k; a, b}+t\cdot F(V_{k; a, b}, \Wr_{O(\varepsilon)})+F(V_{k; 1, 0}, V_{k; a, b}).
\end{split}
\]
and note that
\[
aV_{k; a, b}+a\cdot \begin{pmatrix}
-\frac{2}{\gamma+1} & -\frac{\gamma-1}{\gamma+1} & 0\\
-\frac{\gamma-1}{\gamma+1} & -\frac{2}{\gamma+1} & 0\\
0 & 0 & -\frac{2}{\gamma+1}
\end{pmatrix}V_{k; a, b}=\underbrace{a\cdot \begin{pmatrix}
\frac{\gamma-1}{\gamma+1} & -\frac{\gamma-1}{\gamma+1} & 0\\
-\frac{\gamma-1}{\gamma+1} & \frac{\gamma-1}{\gamma+1} & 0\\
0 & 0 & \frac{\gamma-1}{\gamma+1}
\end{pmatrix}V_{k; a, b}}_{\textbf{Null structure in commutators}}.
\]

We also decompose $-a\kappar\Tr(U^{m})\frac{\partial A_{2}(U)}{\partial U^{m}}V_{k; a-1, b+1}$  as follows,
\[
\begin{split}
-a\kappar\Tr(U^{m})\frac{\partial A_{2}(U)}{\partial U^{m}}V_{k; a-1, b+1}&=-a\kappar\Tr(\Ur_{k}^{m})\frac{\partial A_{2}(U)}{\partial U^{m}}V_{k; a-1, b+1}-a\kappar\cdot V_{k; 1, 0}^{m}\frac{\partial A_{2}(U)}{\partial U^{m}}V_{k; a-1, b+1}\\
&=t\cdot F(V_{k; a-1, b+1}, \Wr_{O(1)})+t\cdot F(V_{k; 1, 0}, V_{k; a-1, b+1}).
\end{split}
\]
and note that
\begin{equation}\label{eq: damping mechanism}
\begin{split}
2\big( V_{k; a, b}, -\frac{2}{\gamma+1}\begin{pmatrix}
\frac{\gamma-1}{2} & 1 & 0\\
\frac{\gamma-1}{2} & 1 & 0\\
0 & 0 & 0
\end{pmatrix}V_{k; a, b}\big)&=\big( V_{k; a, b}, \begin{pmatrix}
-\frac{2(\gamma-1)}{\gamma+1} & -1 & 0\\
-1 & -\frac{4}{\gamma+1} & 0 \\
0 & 0 & 0
\end{pmatrix}V_{k; a, b} \big)\\
&=\underbrace{\big( V_{k; a, b}, \begin{pmatrix}
-\frac{2(\gamma-1)}{\gamma+1}-1 & 0 & 0\\
0 & -\frac{4}{\gamma+1}-1 & 0 \\
0 & 0 & 0
\end{pmatrix}V_{k; a, b} \big)}_{\textbf{Damping by subtracting approximate solution}}\\ &\quad+\underbrace{\big( V_{k; a, b}, \begin{pmatrix}
1 & -1 & 0\\
-1 & 1 & 0 \\
0 & 0 & 0
\end{pmatrix}V_{k; a, b} \big)}_{\textbf{Null structure}}.
\end{split}
\end{equation}

Therefore, we can rewrite the bulk term associated with $\mathbf{S}_{1; a, b}$ as follows,
\begin{equation}\label{eq: null structure in commutators}
\begin{split}
&\int_{\mathcal{D}(t, \tilde{u})}\Big( \underbrace{\big(V_{k; a, b}, \begin{pmatrix}
-\frac{2(\gamma-1)}{\gamma+1}&-1&0\\
-1&-\frac{4}{\gamma+1}&0\\
0&0&0
\end{pmatrix}V_{k; a, b} \big)}_{\textbf{Damping by subtracting approximate solution}}+\underbrace{2a\big(V_{k; a, b}, \begin{pmatrix}
\frac{\gamma-1}{\gamma+1}&-\frac{\gamma-1}{\gamma+1}&0\\
-\frac{\gamma-1}{\gamma+1}&\frac{\gamma-1}{\gamma+1}&0\\
0&0&\frac{\gamma-1}{\gamma+1}
\end{pmatrix}V_{k; a, b} \big)}_{\textbf{Null structure in commutators}} \Big)\\
&+\int_{\mathcal{D}(t, \tilde{u})}\Big( t'\cdot \big(V_{k; a, b}, F(V_{k; a, b}, \Wr_{O(\varepsilon)})\big)+t'\cdot \big(V_{k; a, b}, F(V_{k; a-1, b+1}, \Wr_{O(1)})\big)\Big)\\
&+\int_{\mathcal{D}(t, \tilde{u})}\Big( \big(V_{k; a, b}, F(V_{k; 1, 0}, V_{k; a, b})\big)+t'\cdot \big(V_{k; a, b}, F(V_{k; 1, 0}, V_{k; a-1, b+1})\big)\Big).
\end{split}
\end{equation}
\item \textbf{Bulk terms from integration by parts}. For the normal-derivative contribution, we have
\begin{equation}\label{eq: anti damping by integration by parts}
\begin{split}
&-\big( V_{k; a, b}, \frac{\partial A_{1}(U)}{\partial U^{m}}V_{k; a, b} \big)\Tr(U^{m})\\ &=-\big( V_{k; a, b}, \frac{\partial A_{1}(U)}{\partial U^{m}}V_{k; a, b} \big)\Tr(\Ur_{k}^{m})\\ &\quad-\big( V_{k; a, b}, \frac{\partial A_{1}(U)}{\partial U^{m}}V_{k; a, b} \big)V_{k; 1, 0}^{m}\\
&=-\big( V_{k; a, b}, \frac{\partial A_{1}(U)}{\partial U^{m}}V_{k; a, b} \big)\Tr(\Vr_{0}^{m})\\
&-kt\cdot \big( V_{k; a, b}, \frac{\partial A_{1}(U)}{\partial U^{m}}V_{k; a, b} \big)\Tr(\Vr_{1}^{m})-\big( V_{k; a, b}, \frac{\partial A_{1}(U)}{\partial U^{m}}V_{k; a, b} \big)V_{k; 1, 0}^{m}\\
&=\big(V_{k; a, b}, \begin{pmatrix}
\frac{2}{\gamma+1}&\frac{\gamma-1}{\gamma+1}&0\\
\frac{\gamma-1}{\gamma+1}&\frac{2}{\gamma+1}&0\\
0&0&\frac{2}{\gamma+1}
\end{pmatrix} V_{k; a, b}\big)+t\Wr_{O(\varepsilon)}\cdot F(V_{k; a, b}, V_{k; a, b})+\big(V_{k; a, b}, F(V_{k; a, b}, V_{k; 1, 0}) \big)\\
&=\underbrace{\big(V_{k; a, b}, \begin{pmatrix}
-\frac{\gamma-1}{\gamma+1}&\frac{\gamma-1}{\gamma+1}&0\\
\frac{\gamma-1}{\gamma+1}&-\frac{\gamma-1}{\gamma+1}&0\\
0&0&\frac{2}{\gamma+1}
\end{pmatrix} V_{k; a, b}\big)}_{\textbf{Null structure}}\\ &\quad+\underbrace{\big(V_{k; a, b}, \begin{pmatrix}
1&0&0\\
0&1&0\\
0&0&0
\end{pmatrix} V_{k; a, b}\big)}_{\textbf{Anti damping by integration by parts}}\\
&+t\Wr_{O(\varepsilon)}\cdot F(V_{k; a, b}, V_{k; a, b})+\big(V_{k; a, b}, F(V_{k; a, b}, V_{k; 1, 0}) \big)
\end{split}
\end{equation}
We can also decompose $\kappar\cdot \big( V_{k; a, b}, \frac{\partial A_{2}(U)}{\partial U^{m}}V_{k; a, b} \big)\Xr(U^{m})$ as follows,
\[
\begin{split}
\kappar\cdot \big( V_{k; a, b}, \frac{\partial A_{2}(U)}{\partial U^{m}}V_{k; a, b} \big)\Xr(U^{m})&=\kappar\cdot \big( V_{k; a, b}, \frac{\partial A_{2}(U)}{\partial U^{m}}V_{k; a, b} \big)\Xr(\Ur_{k}^{m})+\kappar\cdot \big( V_{k; a, b}, \frac{\partial A_{2}(U)}{\partial U^{m}}V_{k; a, b} \big)V_{k; 0, 1}^{m}\\
&=t\Wr_{O(\varepsilon)}\cdot F(V_{k; a, b}, V_{k; a, b})+t\cdot \big( V_{k; a, b}, F(V_{k; a, b}, V_{k; 0, 1}) \big).
\end{split}
\]
Combining these contributions, the \textbf{bulk terms from integration by parts} are
\[
\begin{split}
&\int_{\mathcal{D}(t, \tilde{u})}\Big(\underbrace{\big(V_{k; a, b}, \begin{pmatrix}
-\frac{\gamma-1}{\gamma+1}&\frac{\gamma-1}{\gamma+1}&0\\
\frac{\gamma-1}{\gamma+1}&-\frac{\gamma-1}{\gamma+1}&0\\
0&0&\frac{2}{\gamma+1}
\end{pmatrix} V_{k; a, b}\big)}_{\textbf{Null structure}}+\underbrace{\big(V_{k; a, b}, \begin{pmatrix}
1&0&0\\
0&1&0\\
0&0&0
\end{pmatrix} V_{k; a, b}\big)}_{\textbf{Anti damping by integration by parts}}\Big)\\
&+\int_{\mathcal{D}(t, \tilde{u})}\Big(t'\Wr_{O(\varepsilon)}\cdot F(V_{k; a, b}, V_{k; a, b})+\big(V_{k; a, b}, F(V_{k; a, b}, V_{k; 1, 0}) \big)+t'\cdot \big( V_{k; a, b}, F(V_{k; a, b}, V_{k; 0, 1}) \big)\Big).
\end{split}
\]
\item The bulk term associated with $\mathbf{S}_{2; a, b}$ can be rewritten as
\[
\int_{\mathcal{D}(t,\tilde u)}2(V_{k;a,b},\mathbf{S}_{2;a,b})=\int_{\mathcal{D}(t, \tilde{u})} t'\cdot \big(V_{k; a, b}, F(V_{k; \leq a, \leq b}, \Wr_{O(\varepsilon)})\big).
\]
\item The bulk term associated with $\mathbf{S}_{3; a, b}$ can be rewritten as
\[
\int_{\mathcal{D}(t,\tilde u)}2(V_{k;a,b},\mathbf{S}_{3;a,b})=\int_{\mathcal{D}(t, \tilde{u})} t'^{1+k} \cdot \big(V_{k; a, b}, \Wr_{O(\varepsilon)}\big).
\]
\item The bulk term associated with $\mathbf{S}_{4; a, b}$. Decomposing $U$ as $\Ur_{k}+V_{k}$, a typical term in $\mathbf{S}_{4; a, b}$ takes the form
\[
F\big(\Zr^{\leq \Ncomm-2}(V_{k}), \Zr^{\leq \Ncomm}(V_{k})\big)+F\big(\Wr_{O(\varepsilon)}, \Zr^{\leq \Ncomm}(V_{k})\big)
\]
The bulk term associated with $\mathbf{S}_{4; a, b}$ can be rewritten as
\[
\int_{\mathcal{D}(t, \tilde u)}\Big( \Big( V_{k; a, b},F\big(\Zr^{\leq \Ncomm-2}(V_{k}), \Zr^{\leq \Ncomm}(V_{k})\big) \Big)+\Big( V_{k; a, b}, F\big(\Wr_{O(\varepsilon)}, \Zr^{\leq \Ncomm}(V_{k})\big)\Big)\Big).
\]
\end{itemize}

\paragraph{The master energy inequality.}
Combining the preceding decompositions gives the following form of
\eqref{eq: higher order energy estimate} for $a+b\leq \Ncomm$.
In terms estimated by generic constants, we use the uniform bounds
for $\kappa/t$ and $\slashed g$; we still use the stated spacetime measure:
\begin{equation}\label{eq: simple higher order energy estimate}
\begin{split}
&E(V_{k; a, b})(t, \tilde{u})+F(V_{k; a, b})(t, \tilde{u})\leq E(V_{k; a, b})(0, \tilde{u})+F(V_{k; a, b})(t, \tilde{u}_{*})\\
+&\underbrace{\int_{\mathcal{D}(t, \tilde{u})}\Big(\big(V_{k; a, b}, \begin{pmatrix}
-\frac{2(\gamma-1)}{\gamma+1}&0&0\\
0&-\frac{4}{\gamma+1}&0\\
0&0&0
\end{pmatrix}V_{k; a, b} \big)+(2a\frac{\gamma-1}{\gamma+1}+\frac{2}{\gamma+1})\cdot \big(V_{k; a, b}, \begin{pmatrix}
1&-1&0\\
-1&1&0\\
0&0&1
\end{pmatrix}V_{k; a, b} \big)\Big)}_{\mathbf{J}_{1}}\\
&+\underbrace{C\int_{\mathcal{D}(t, \tilde{u})}t'\cdot F\big(V_{k; a, b}, V_{k; a-1, b+1}, \Wr_{O(1)}\big)}_{\mathbf{J}_{2}}+\underbrace{C\int_{\mathcal{D}(t, \tilde{u})} t'^{k+1}\cdot F\big( V_{k; a, b}, \Wr_{O(\varepsilon)} \big)}_{\mathbf{J}_{3}}\\
&+\underbrace{C\int_{\mathcal{D}(t, \tilde{u})} F\Big( V_{k; a, b}, \Zr^{\leq \Ncomm}(V_{k}), \Zr^{\leq \Ncomm-2}(V_{k}) \Big)}_{\mathbf{J}_{4}}+\underbrace{C\int_{\mathcal{D}(t, \tilde{u})}F\Big( V_{k; a, b}, \Wr_{O(\varepsilon)}, \Zr^{\leq \Ncomm}(V_{k})\Big)}_{\mathbf{J}_{5}}.
\end{split}
\end{equation}
Terms with additional powers of $t$ are included in the corresponding upper bounds. Here $F(\cdot, \cdot, \cdot): (\mathbb{R}^{3})^{3}\rightarrow \mathbb{R}$ denotes a generic trilinear contraction.

The last two error terms are controlled together. Suppress the fixed
upper characteristic label and write
$\mathcal E_{\leq\Ncomm}(s)=\sum_{a+b\leq\Ncomm}\mathcal E(V_{k;a,b})(s)$.
The bootstrap energy bound and the Sobolev estimates give
$\mathcal E_{\leq\Ncomm}(s)\leq P(M)\varepsilon^2s^{2k+3}$ and
$\|\Zr^{\leq\Ncomm-2}(V_k)(s)\|_\infty\leq P(M)\varepsilon s^{k+1}$.
The coefficient family $\Wr_{O(\varepsilon)}$ is bounded by
$C\varepsilon$. Cauchy--Schwarz on each slice, with the stated
spacetime measure, therefore yields
\[
\begin{aligned}
|\mathbf J_4|+|\mathbf J_5|
&\leq C\varepsilon\int_\delta^t
 (1+P(M)s^{k+1})\frac{\mathcal E_{\leq\Ncomm}(s)}{s}\,ds\\
&\leq P(M)\varepsilon^3t^{2k+3}.
\end{aligned}
\]
Here $\delta>0$ is the starting time; one may also take $\delta=0$
when the initial-slice formulation applies. The slice and Sobolev
constants are uniform in the vacuum cutoff, and fixed powers of
$t^*$ are absorbed into $P(M)$. In each of the two closing arguments
below, the symbol $\mathbf J_4$ denotes this combined contribution
of $\mathbf J_4$ and $\mathbf J_5$ from
\eqref{eq: simple higher order energy estimate}.

The time singularity in this inequality is handled by the following Gronwall lemma, which will be used to close $(\mathbf{B}_{2})$ in both regions.
\subsection{A Gronwall inequality with singular coefficients}

\begin{lemma}\label{lem: Gronwall}
Let $0<t^*\leq1$, $k\in\{0,1\}$, and let
$f\in C^0([0,t^*])$ be nonnegative and satisfy
\[
\begin{cases}
f(t)\leq C\varepsilon^{2}t^{2k+3}+C\int_{0}^{t} f(t')dt'+B\int_{0}^{t} \frac{f(t')}{t'} dt',\\
\overline{\lim}_{\delta\rightarrow 0} \frac{f(\delta)}{\delta^{3+2k}}<+\infty.
\end{cases}
\]
where the two constants on the right may be enlarged to one fixed
$C>0$ and $B\in(0,2k+3)$. In this lemma, keep that $C$ fixed and set
\[
F(t)=\int_0^t\frac{f(t')}{t'}\,dt'
     +\frac CB\int_0^t f(t')\,dt'.
\]
Then we have
\[
f(t)\leq C\varepsilon^{2}t^{2k+3}+B\cdot F(t)\leq \Big(C+\frac{(CB+C^{2})e^{Ct}}{2k+3-B}\Big)\varepsilon^{2} t^{2k+3}.
\]
\end{lemma}

\begin{proof}
Set
\[
F(t)=\int_{0}^{t} \frac{f(t')}{t'} dt' +\frac{C}{B}\int_{0}^{t} f(t') dt'.
\]
Then we have
\[
f(t)\leq C\varepsilon^{2}t^{2k+3}+B\cdot F(t),
\]
and
\[
\begin{split}
F'(t)&=\frac{f(t)}{t}+\frac{C}{B}f(t)\\
&\leq \frac{C\varepsilon^{2}t^{2k+3}+B\cdot F(t)}{t}+\frac{C}{B}\cdot \big( C\varepsilon^{2}t^{2k+3}+B\cdot F(t) \big)\\
&\leq (C+\frac{C^{2}}{B})\cdot \varepsilon^{2}t^{2k+2}+B\cdot \frac{F(t)}{t}+C\cdot F(t).
\end{split}
\]
since $t^{*}=1$ are fixed in this paper. Hence, we have
\[
\frac{d}{dt}\big( F(t)t^{-B}e^{-Ct} \big)\leq  (C+\frac{C^{2}}{B})\cdot \varepsilon^{2}t^{2k+2-B}
\]
since $\lim_{\delta\rightarrow 0} F(\delta)\delta^{-B}=0$ and $t^{2k+2-B}$ is integrable in [0, t], integrating from $0$ to $t$,
\[
F(t)\leq \frac{(C+C^{2}/B)e^{Ct}}{2k+3-B}\cdot \varepsilon^{2}t^{2k+3}.
\]
Therefore
\[
f(t)\leq C\varepsilon^{2}t^{2k+3}+B\cdot F(t)\leq \Big(C+\frac{(CB+C^{2})e^{Ct}}{2k+3-B}\Big)\varepsilon^{2} t^{2k+3}.
\]
This completes the proof of the lemma.
\end{proof}

\subsection{Closing the energy bootstrap away from vacuum} 

Away from vacuum, take $\Ncomm=\Ntop$. The argument applies to the full range $1<\gamma<3$. The positive lower bound for the sound speed allows the characteristic flux to control the leading null terms.

The higher energy estimate \eqref{eq: simple higher order energy estimate} holds by formally replacing $\tilde{u}$ with $u$ and replacing $\tilde{u}\in [\tilde{u}_{*}, \tilde{u}^{*}]$ by $u\in [0,u^{*}]$,
\[
\begin{split}
&E(V_{k; a, b})(t, u)+F(V_{k; a, b})(t, u)\leq E(V_{k; a, b})(0, u)+F(V_{k; a, b})(t, 0)\\
+&\underbrace{\int_{\mathcal{D}(t, u)}\Big(\big(V_{k; a, b}, \begin{pmatrix}
-\frac{2(\gamma-1)}{\gamma+1}&0&0\\
0&-\frac{4}{\gamma+1}&0\\
0&0&0
\end{pmatrix}V_{k; a, b} \big)+(2a\frac{\gamma-1}{\gamma+1}+\frac{2}{\gamma+1})\cdot \big(V_{k; a, b}, \begin{pmatrix}
1&-1&0\\
-1&1&0\\
0&0&1
\end{pmatrix}V_{k; a, b} \big)\Big)}_{\mathbf{J}_{1}}\\
&+\underbrace{C\int_{\mathcal{D}(t, u)}t'\cdot F\big(V_{k; a, b}, V_{k; a-1, b+1}, \Wr_{O(1)}\big)}_{\mathbf{J}_{2}}+\underbrace{C\int_{\mathcal{D}(t, u)} t'^{k+1}\cdot F\big( V_{k; a, b}, \Wr_{O(\varepsilon)} \big)}_{\mathbf{J}_{3}}\\
&+\underbrace{C\int_{\mathcal{D}(t, u)} F\Big( V_{k; a, b}, \Zr^{\leq \Ncomm}(V_{k}), \Zr^{\leq \Ncomm-2}(V_{k}) \Big)+C\int_{\mathcal{D}(t, u)}F\Big( V_{k; a, b}, \Wr_{O(\varepsilon)}, \Zr^{\leq \Ncomm}(V_{k})\Big)}_{\mathbf{J}_{4}}.
\end{split}
\]
where
\[
\begin{cases}
\mathcal{E}(V_{k; a, b})(t, u)=\int_{\Sigma_{t}^{0,u}} \kappa\cdot (V_{k; a, b}, V_{k; a, b}),\\
\begin{aligned}[t]\mathcal{F}(V_{k; a, b})(t, u)&=\int_{C_{u}^{t}}\Big(c\cdot \big(V_{k; a, b}, \begin{pmatrix}
1&\hat{T}^{1}&\hat{T}^{2}\\
\hat{T}^{1}&(\hat{T}^{1})^{2}&\hat{T}^{1}\hat{T}^{2}\\
\hat{T}^{2}&\hat{T}^{1}\hat{T}^{2}&(\hat{T}^{2})^{2}
\end{pmatrix}V_{k; a, b}\big)\\ &\quad+c\cdot \big(V_{k; a, b}, \begin{pmatrix}
0&0&0\\
0&(\hat{X}^{1})^{2}&\hat{X}^{1}\hat{X}^{2}\\
0&\hat{X}^{1}\hat{X}^{2}&(\hat{X}^{2})^{2}
\end{pmatrix}V_{k; a, b}\big)\Big).\end{aligned}
\end{cases}
\]

\begin{remark}\label{re: how to use null structure}
Though the first linear term in $\mathbf{J}_{1}$ has favorable sign, the second linear term in $\mathbf{J}_{1}$ is nonnegative and hence unfavorable, which can be bounded by $E(V_{k;a,b})(t,u)$ as follows (if we try to use $E(V_{k;a,b})(t,u)$ to bound it),
\[
\int_{\mathcal{D}(t, u)} (2a\frac{\gamma-1}{\gamma+1}+\frac{2}{\gamma+1})\cdot \big(V_{k; a, b}, \begin{pmatrix}
1&-1&0\\
-1&1&0\\
0&0&1
\end{pmatrix}V_{k; a, b} \big)\leq C_{1}\cdot \int_{0}^{t} \frac{E(V_{k; a, b})(t', u)}{t'} dt'.
\]
This gives an inequality of the form
\[
E(V_{k;a,b})(t,u)\leq C_1\int_0^t\frac{E(V_{k;a,b})(t',u)}{t'}\,dt'+C\varepsilon^2t^{2k+3}.
\]
For sufficiently large $a$, the coefficient $C_1$ exceeds $2k+3$.
Such an integral inequality alone cannot give the required uniform bound.
Indeed, for any $K>0$, the nonnegative function
\[
f(t)=\varepsilon^2t^{2k+3}+K\varepsilon^2t^{C_1}
\]
satisfies $\lim_{t\to0}f(t)/t^{2k+3}=\varepsilon^2$ and
\[
C_1\int_0^t\frac{f(t')}{t'}\,dt'
=\frac{C_1}{2k+3}\varepsilon^2t^{2k+3}+K\varepsilon^2t^{C_1}\geq f(t).
\]
Nevertheless, $f(1)=(1+K)\varepsilon^2$ is arbitrarily large relative to
$\varepsilon^2$. Thus this estimate does not improve $(\mathbf B_2)$.
The flux $\mathcal{F}(V_{k; a, b})$ controls the leading positive terms in $J_{1}$ because they have the same null structure as the flux density, up to the small change of frame: 
\[
\begin{gathered}\begin{pmatrix}
1&\hat{T}^{1}&\hat{T}^{2}\\
\hat{T}^{1}&(\hat{T}^{1})^{2}&\hat{T}^{1}\hat{T}^{2}\\
\hat{T}^{2}&\hat{T}^{1}\hat{T}^{2}&(\hat{T}^{2})^{2}
\end{pmatrix}-\begin{pmatrix}
1&-1&0\\
-1&1&0\\
0&0&0
\end{pmatrix}=O(P(M)\varepsilon t),\\ \begin{pmatrix}
0&0&0\\
0&(\hat{X}^{1})^{2}&\hat{X}^{1}\hat{X}^{2}\\
0&\hat{X}^{1}\hat{X}^{2}&(\hat{X}^{2})^{2}
\end{pmatrix}-\begin{pmatrix}
0&0&0\\
0&0&0\\
0&0&1
\end{pmatrix}=O(P(M)\varepsilon t).\end{gathered}
\]
and
\[
\frac{1}{4}\mathbf{c}_{*}\leq c\leq 4\mathbf{c}_{*}, \ \ \textbf{in away from vacuum case}.
\]
Therefore, the dangerous term can be bounded by
\[
\int_{\mathcal{D}(t, u)} (2a\frac{\gamma-1}{\gamma+1}+\frac{2}{\gamma+1})\cdot \big(V_{k; a, b}, \begin{pmatrix}
1&-1&0\\
-1&1&0\\
0&0&1
\end{pmatrix}V_{k; a, b} \big)\leq C_{a}\int_{0}^{u} F(V_{k;a,b})(t,u')du'+P(M)\varepsilon^{3} t^{2k+3}.
\]
This leads to a Gronwall inequality in $u$ with a uniformly bounded coefficient, since the sound speed is bounded below away from vacuum.
\end{remark}

We estimate $\mathbf{J}_{1}, \mathbf{J}_{2}, \mathbf{J}_{3}, \mathbf{J}_{4}$ in turn.
\begin{itemize}
\item The term $\mathbf{J}_{1}$ takes the form
\[
\int_{\mathcal{D}(t, u)}\Big(\big(V_{k; a, b}, \begin{pmatrix}
-\frac{2(\gamma-1)}{\gamma+1}&0&0\\
0&-\frac{4}{\gamma+1}&0\\
0&0&0
\end{pmatrix}V_{k; a, b} \big)+(2a\frac{\gamma-1}{\gamma+1}+\frac{2}{\gamma+1})\cdot \big(V_{k; a, b}, \begin{pmatrix}
1&-1&0\\
-1&1&0\\
0&0&1
\end{pmatrix}V_{k; a, b} \big)\Big)
\]
which can be bounded by Cauchy--Schwarz as follows,
\[
C_{a}\cdot \int_{0}^{u} F(V_{k; a, b})(t, u')du'+P(M)\varepsilon^{3}\cdot t^{2k+3}.
\]
\item The term $\mathbf{J}_{2}$ can be bounded by Cauchy--Schwarz,
\[
\begin{split}
C\int_{\mathcal{D}(t, u)}t'\cdot F\big(V_{k; a, b}, V_{k; a-1, b+1}, \Wr_{O(1)}\big)&\leq C\int_{0}^{t} t' \|V_{k; a, b}\|_{L^{2}(\Sigma_{t'}^{0,u})}\|V_{k; a-1, b+1}\|_{L^{2}(\Sigma_{t'}^{0,u})} dt'\\
&\leq C\int_{0}^{t} \mathcal{E}(V_{k; a, b})(t', u) dt'+C\int_{0}^{t} \mathcal{E}(V_{k; a-1, b+1})(t', u) dt'
\end{split}
\]
\item The term $\mathbf{J}_{3}$ can be bounded by
\[
\begin{split}
&C\int_{\mathcal{D}(t, u)} t'^{k+1}\cdot F\big( V_{k; a, b}, \Wr_{O(\varepsilon)} \big)\\
&\leq C\int_{0}^{t} t'^{k+1}\varepsilon\cdot \|V_{k; a, b}\|_{L^{1}(\Sigma_{t'}^{0,u})} dt'\\
&\leq C\int_{0}^{t} t'^{k+\frac{1}{2}}\varepsilon \mathcal{E}(V_{k; a, b})(t', u)^{\frac{1}{2}} dt'\\
&\leq \mu\cdot \int_{0}^{t} \frac{\mathcal{E}(V_{k; a, b})(t', u)}{t'}dt'+ \frac{C}{\mu}\cdot \varepsilon^{2}t^{2k+3}.
\end{split}
\]
where $\mu\in (0, 1)$ is a parameter to choose later (see \eqref{eq: choice of mu away from vacuum}).
\item The term $\mathbf{J}_{4}$ can be bounded by $P(M)\varepsilon^{3}t^{2k+3}$ via Sobolev inequality \eqref{eq: W22 to L infty}.
\end{itemize}

In summary, we have for $a+b\leq \Ntop$,
\[
\begin{split}
&\sum_{a+b\leq \Ntop}\big(\mathcal{E}(V_{k; a, b})(t, u)+\mathcal{F}(V_{k; a, b})(t, u)\big)\\
&\leq \big( C+P(M)\varepsilon \big)\cdot\varepsilon^{2}t^{2k+3}\\
&+C \int_{0}^{u}\sum_{a+b\leq \Ntop}\mathcal{F}(V_{k; a, b})(t, u') du'\\
&+C \int_{0}^{t}\sum_{a+b\leq \Ntop}\mathcal{E}(V_{k; a, b})(t', u) dt'\\
&+\mu\cdot \int_{0}^{t} \frac{\sum_{a+b\leq \Ntop}\mathcal{E}(V_{k; a, b})(t', u)}{t'} dt'.
\end{split}
\]
Applying Gronwall in $u\in [0, u^{*}]$ gives
\[
\begin{split}
&\sum_{a+b\leq \Ntop} \mathcal{F}(V_{k; a, b})(t, u)\\
&\leq \big( C+P(M)\varepsilon \big)e^{Cu^{*}}\cdot\varepsilon^{2}t^{2k+3}\\
&+Ce^{Cu^{*}} \int_{0}^{t}\sum_{a+b\leq \Ntop}\mathcal{E}(V_{k; a, b})(t', u) dt'\\
&+\mu e^{Cu^{*}}\cdot \int_{0}^{t} \frac{\sum_{a+b\leq \Ntop}\mathcal{E}(V_{k; a, b})(t', u)}{t'} dt'.
\end{split}
\]
Then we have
\[
\begin{split}
&\sum_{a+b\leq \Ntop}\mathcal{E}(V_{k; a, b})(t, u)\\
&\leq 2\big( C+P(M)\varepsilon \big)e^{Cu^{*}}\cdot\varepsilon^{2}t^{2k+3}\\
&+2Ce^{Cu^{*}} \int_{0}^{t}\sum_{a+b\leq \Ntop}\mathcal{E}(V_{k; a, b})(t', u) dt'\\
&+2e^{Cu^{*}}\mu\cdot \int_{0}^{t} \frac{\sum_{a+b\leq \Ntop}\mathcal{E}(V_{k; a, b})(t', u)}{t'} dt'.
\end{split}
\]
and we choose $\mu$ such that 
\begin{equation}\label{eq: choice of mu away from vacuum}
 2e^{Cu^{*}}\mu=1<2k+3.
\end{equation}
Lemma~\ref{lem: Gronwall}, with $B=1$, then gives
\[
\sum_{a+b\leq \Ntop}\Big(\mathcal{E}(V_{k; a, b})(t, u)+\mathcal{F}(V_{k; a, b})(t, u)\Big) \leq \big(C+P(M)\varepsilon\big)\cdot \varepsilon^{2}t^{2k+3}. 
\]
provided that $\varepsilon$ is sufficiently small.

Thus $(\mathbf{B}_{2})$ improves away from vacuum: for all $t\in [0, t^{*}]$ and $a+b\leq \Ntop$,
\begin{equation}\notag
\mathcal{E}(V_{k; a, b})(t, u^{*})+\mathcal{F}(V_{k; a, b})(t, u^{*})\leq C\varepsilon^{2}t^{2k+3}.
\end{equation}
In view of Sobolev inequality, we also have improved $L^{\infty}$ estimates for $V_{k; a, b}, a+b\leq \Ntop-2$,
\begin{equation}\notag
\|V_{k; a, b}\|_{L^{\infty}(\Sigma_{t}^{0,u^{*}})}\leq C \varepsilon t^{k+1}.
\end{equation}

\subsection{Closing the energy bootstrap up to vacuum}
Near vacuum, take $\Ncomm=4$ and use the following choices of $k$ and $\gamma$. The argument must be uniform as the vacuum cutoff tends to zero. 
\[
\begin{cases}
k=0, \ \ \gamma\in (1, 1.3],\\
k=1, \ \ \gamma\in \{\frac{7}{5}, \frac{5}{3}\}.
\end{cases}
\]

We apply the master energy inequality \eqref{eq: simple higher order energy estimate} on $\mathcal{D}(t, \tilde{u}^{*}): (t', \tilde{u}, \vartheta)\in[0, t]\times [\tilde{u}_{*}, \tilde{u}^{*}]\times \mathbb{T}$,
\[
\begin{split}
&E(V_{k; a, b})(t, \tilde{u})+F(V_{k; a, b})(t, \tilde{u})\leq E(V_{k; a, b})(0, \tilde{u})+F(V_{k; a, b})(t, \tilde u_*)\\
+&\underbrace{\int_{\mathcal{D}(t, \tilde{u})}\Big(\big(V_{k; a, b}, \begin{pmatrix}
-\frac{2(\gamma-1)}{\gamma+1}&0&0\\
0&-\frac{4}{\gamma+1}&0\\
0&0&0
\end{pmatrix}V_{k; a, b} \big)+(2a\frac{\gamma-1}{\gamma+1}+\frac{2}{\gamma+1})\cdot \big(V_{k; a, b}, \begin{pmatrix}
1&-1&0\\
-1&1&0\\
0&0&1
\end{pmatrix}V_{k; a, b} \big)\Big)}_{\mathbf{J}_{1}}\\
&+\underbrace{C\int_{\mathcal{D}(t, \tilde{u})}t'\cdot F\big(V_{k; a, b}, V_{k; a-1, b+1}, \Wr_{O(1)}\big)}_{\mathbf{J}_{2}}+\underbrace{C\int_{\mathcal{D}(t, \tilde{u})} t'^{k+1}\cdot F\big( V_{k; a, b}, \Wr_{O(\varepsilon)} \big)}_{\mathbf{J}_{3}}\\
&+\underbrace{C\int_{\mathcal{D}(t, \tilde{u})} F\Big( V_{k; a, b}, \Zr^{\leq \Ncomm}(V_{k}), \Zr^{\leq \Ncomm-2}(V_{k}) \Big)+C\int_{\mathcal{D}(t, \tilde{u})}F\Big( V_{k; a, b}, \Wr_{O(\varepsilon)}, \Zr^{\leq \Ncomm}(V_{k})\Big)}_{\mathbf{J}_{4}}.
\end{split}
\]
where
\[
\begin{cases}
\mathcal{E}(V_{k; a, b})(t, \tilde{u})=\int_{\Sigma_{t}^{\ut_*,\tilde{u}}} \kappa\cdot (V_{k; a, b}, V_{k; a, b}),\\
\begin{aligned}[t]\mathcal{F}(V_{k; a, b})(t, \tilde{u})&=\int_{C_{\tilde{u}}^{t}}\Big(c\cdot \big(V_{k; a, b}, \begin{pmatrix}
1&\hat{T}^{1}&\hat{T}^{2}\\
\hat{T}^{1}&(\hat{T}^{1})^{2}&\hat{T}^{1}\hat{T}^{2}\\
\hat{T}^{2}&\hat{T}^{1}\hat{T}^{2}&(\hat{T}^{2})^{2}
\end{pmatrix}V_{k; a, b}\big)\\ &\quad+c\cdot \big(V_{k; a, b}, \begin{pmatrix}
0&0&0\\
0&(\hat{X}^{1})^{2}&\hat{X}^{1}\hat{X}^{2}\\
0&\hat{X}^{1}\hat{X}^{2}&(\hat{X}^{2})^{2}
\end{pmatrix}V_{k; a, b}\big)\Big).\end{aligned}
\end{cases}
\]

\begin{remark}
The flux argument of Remark~\ref{re: how to use null structure} does not control the null terms in $\mathbf{J}_{1}$ uniformly in $\alpha\in (0, \frac{1}{100})$. The weight in $\mathcal{F}(V_{k; a, b})(t, \tilde{u})$ degenerates because $c\approx -\frac{\gamma-1}{\gamma+1}\tilde{u}$ for $\tilde{u}\in [\tilde{u}_{*}, \tilde{u}^{*}]$ in view of \eqref{eq: c almost constant on C_u}, 
\[
\begin{gathered}\int_{\mathcal{D}(t, \tilde{u})} (2a\frac{\gamma-1}{\gamma+1}+\frac{2}{\gamma+1})\cdot \big(V_{k; a, b}, \begin{pmatrix}
1&-1&0\\
-1&1&0\\
0&0&1
\end{pmatrix}V_{k; a, b} \big)\\ \leq (8a+\frac{8}{\gamma-1})\int_{\tilde{u}_{*}}^{\tilde{u}} \frac{1}{-\tilde{u}'}\cdot \mathcal{F}(V_{k; a, b})(t, \tilde{u}') d\tilde{u}'+P(M)\varepsilon^{3}t^{2k+3}.\end{gathered}
\]
Here the displayed coefficient is one convenient upper bound; its exact value is immaterial for this remark. A Gronwall argument using this positive kernel introduces the factor
\[
\exp\left((8a+\frac{8}{\gamma-1})\int_{\tilde u_*}^{\tilde u}\frac{d\tilde u'}{-\tilde u'}\right)
=\left(\frac{-\tilde u_*}{-\tilde u}\right)^{8a+8/(\gamma-1)}.
\]
At $\tilde u=\tilde u^*=-\alpha\frac{\gamma+1}{\gamma-1}\mathbf c_*$ this equals
$(3/(4\alpha))^{8a+8/(\gamma-1)}$, which diverges as $\alpha\to0$.
Thus this flux-based estimate does not give a bound uniform in $\alpha$.
We instead use the energy and the precise maximal eigenvalue of the full quadratic form below.
\end{remark}

We estimate $\mathbf{J}_{1}, \mathbf{J}_{2}, \mathbf{J}_{3}, \mathbf{J}_{4}$ in turn.
\begin{itemize}
\item The term $\mathbf{J}_{1}$ takes the form
\[
\int_{\mathcal{D}(t, \tilde u)}\Big(\big(V_{k; a, b}, \begin{pmatrix}
-\frac{2(\gamma-1)}{\gamma+1}&0&0\\
0&-\frac{4}{\gamma+1}&0\\
0&0&0
\end{pmatrix}V_{k; a, b} \big)+(2a\frac{\gamma-1}{\gamma+1}+\frac{2}{\gamma+1})\cdot \big(V_{k; a, b}, \begin{pmatrix}
1&-1&0\\
-1&1&0\\
0&0&1
\end{pmatrix}V_{k; a, b} \big)\Big)
\]
The maximal eigenvalue of
\[
\begin{pmatrix}
-\frac{2(\gamma-1)}{\gamma+1}&0&0\\
0&-\frac{4}{\gamma+1}&0\\
0&0&0
\end{pmatrix}+(2a\frac{\gamma-1}{\gamma+1}+\frac{2}{\gamma+1})\cdot\begin{pmatrix}
1&-1&0\\
-1&1&0\\
0&0&1
\end{pmatrix}
\]
is 
\[
\lambda_{\textbf{max}}=\text{max}\{\frac{(2a-1)(\gamma-1)}{\gamma+1}+\frac{\sqrt{(\gamma-3)^{2}+4+4a^{2}(\gamma-1)^{2}+8a(\gamma-1)}}{\gamma+1}, \frac{\gamma-1}{\gamma+1}(2a-1)+1\}.
\]
and we have uniformly for $0\leq a\leq 4$,
\begin{equation}\notag
\begin{cases}
\lambda_{\textbf{max}}\leq\frac{21+5\sqrt{89}}{23}\leq 2.97<3, \ \ \gamma\in (1, 1.3],\\
\lambda_{\textbf{max}}\leq\frac{7+5\sqrt{5}}{4}\leq 4.6<5,\ \ \gamma\in \{\frac{7}{5}, \frac{5}{3}\}.
\end{cases}
\end{equation}
For completeness, put $\theta=(\gamma-1)/(\gamma+1)$.
Increasing $a$ by one adds $2\theta Q$ to the matrix, where
$Q$ is the positive semidefinite matrix in the second summand.
Thus the maximum occurs at $a=4$. Its largest eigenvalue is
$7\theta+\sqrt{2+10\theta+53\theta^2}$, which is increasing for
$\theta>0$. This proves the endpoint bounds above.
In the remaining estimates, $\lambda_{\max}$ denotes this maximum.
The precise measure convention gives
\[
\int_{\mathcal D(t,\tilde u)}|V|^2
=\int_0^t\frac{\mathcal E(V)(s,\tilde u)}{s}\,ds,
\]
with no multiplicative comparison constant. Consequently
$\mathbf J_1$ is bounded by
\[
\lambda_{\max}\cdot\int_{0}^{t} \frac{\mathcal{E}(V_{k; a, b})(t', \tilde{u})}{t'}\,dt'
\]
where \textbf{$\lambda_{\max}<2k+3$ for both the case $k=0, \gamma\in (1, 1.3]$ and the case $k=1, \gamma\in \{\frac{7}{5}, \frac{5}{3}\}$}.

\item The term $\mathbf{J}_{2}$ can be bounded by Cauchy--Schwarz,
\[
\begin{split}
C\int_{\mathcal{D}(t, \tilde{u})}t'\cdot F\big(V_{k; a, b}, V_{k; a-1, b+1}, \Wr_{O(1)}\big)&\leq C\int_{0}^{t} t' \|V_{k; a, b}\|_{L^{2}(\Sigma_{t'}^{\ut_*,\tilde{u}})}\|V_{k; a-1, b+1}\|_{L^{2}(\Sigma_{t'}^{\ut_*,\tilde{u}})} dt'\\
&\leq C\int_{0}^{t} \mathcal{E}(V_{k; a, b})(t', \tilde{u}) dt'+C\int_{0}^{t} \mathcal{E}(V_{k; a-1, b+1})(t', \tilde{u}) dt'
\end{split}
\]
\item The term $\mathbf{J}_{3}$ can be bounded by
\[
\begin{split}
&C\int_{\mathcal{D}(t, \tilde{u})} t'^{k+1}\cdot F\big( V_{k; a, b}, \Wr_{O(\varepsilon)} \big)\\
&\leq C\int_{0}^{t} t'^{k+1}\varepsilon\cdot \|V_{k; a, b}\|_{L^{1}(\Sigma_{t'}^{\ut_*,\tilde{u}})} dt'\\
&\leq C\int_{0}^{t} t'^{k+\frac{1}{2}}\varepsilon \mathcal{E}(V_{k; a, b})(t', \tilde{u})^{\frac{1}{2}} dt'\\
&\leq \mu\cdot \int_{0}^{t} \frac{\mathcal{E}(V_{k; a, b})(t', \tilde{u})}{t'}dt'+ \frac{C}{\mu}\cdot \varepsilon^{2}t^{2k+3}.
\end{split}
\]
where $\mu\in (0, 1)$ is a parameter to be chosen later (see \eqref{eq: choice of mu up to vacuum}).
\item The term $\mathbf{J}_{4}$ can be bounded by $P(M)\varepsilon^{3}t^{2k+3}$ via Sobolev inequality \eqref{eq: W22 to L infty}.
\end{itemize}

In summary, we have for $a+b\leq 4$,
\[
\begin{split}
&\sum_{a+b\leq \Ntop}\big(\mathcal{E}(V_{k; a, b})(t, \tilde{u})+\mathcal{F}(V_{k; a, b})(t, \tilde{u})\big)\\
&\leq \big( C+P(M)\varepsilon \big)\cdot\varepsilon^{2}t^{2k+3}\\
&+C \int_{0}^{t}\sum_{a+b\leq \Ntop}\mathcal{E}(V_{k; a, b})(t', \tilde u) dt'\\
&+\big(\mu+\lambda_{\textbf{max}}\big)\cdot \int_{0}^{t} \frac{\sum_{a+b\leq \Ntop}\mathcal{E}(V_{k; a, b})(t', \tilde u)}{t'} dt'.
\end{split}
\]
and we choose $\mu$ such that 
\begin{equation}\label{eq: choice of mu up to vacuum}
 \mu=\frac{1}{2} (2k+3-\lambda_{\textbf{max}}),
\end{equation}
Applying Lemma~\ref{lem: Gronwall} with $B=\lambda_{\textbf{max}}+\frac{1}{2} (2k+3-\lambda_{\textbf{max}})<2k+3$ gives
\[
\sum_{a+b\leq \Ntop}\Big(\mathcal{E}(V_{k; a, b})(t, \tilde u)+\mathcal{F}(V_{k; a, b})(t, \tilde u)\Big) \leq \big(C+P(M)\varepsilon\big)\cdot \varepsilon^{2}t^{2k+3}. 
\]
provided that $\varepsilon$ is sufficiently small.

\paragraph{The estimate with initial time $\delta>0$.}
Put $p=2k+3$ and
$E_\delta(t)=\sum_{a+b\leq4}\mathcal E(V_{k;a,b})(t)$.
Keeping the initial energy in the preceding identity gives
\begin{equation}\label{eq: positive time energy inequality}
E_\delta(t)\leq A\varepsilon^2(\delta^p+t^p)
+C\int_\delta^t E_\delta(s)\,ds
+B_*\int_\delta^t\frac{E_\delta(s)}s\,ds,
\qquad B_*<p,
\end{equation}
where $A=C+P(M)\varepsilon$ and
$B_*=\lambda_{\max}+(p-\lambda_{\max})/2$, with
$\lambda_{\max}$ the maximum over all commutation orders at most four.
The exact measure identity above also holds with lower time $\delta$;
no comparison constant multiplies $\lambda_{\max}$.
Denote the right side of
\eqref{eq: positive time energy inequality} by $H(t)$.
Then $E_\delta\leq H$,
$H(\delta)=2A\varepsilon^2\delta^p$, and
\[
H'(t)\leq Ap\varepsilon^2t^{p-1}+(C+B_*/t)H(t).
\]
The integrating factor $e^{-Ct}t^{-B_*}$ therefore gives
\begin{equation}\label{eq: uniform energy with initial slice}
\begin{aligned}
E_\delta(t)
&\leq e^C A\varepsilon^2
 \left(2\delta^{p-B_*}t^{B_*}
       +\frac{p}{p-B_*}t^p\right)\\
&\leq C_\gamma\big(C+P(M)\varepsilon\big)\varepsilon^2t^p.
\end{aligned}
\end{equation}
Here $p-B_*>0$ depends only on the fixed $\gamma$.
The factor $\delta^{p-B_*}$ has a positive exponent, so
$\delta^{p-B_*}t^{B_*}\leq t^p$ for $\delta\leq t$.
For any truncated label interval the same initial and incoming
terms have the same bounds. Its bulk terms are bounded using
the full-slice energy just estimated. The energy identity thus
also bounds each outgoing flux by $C\varepsilon^2t^p$,
uniformly in its endpoint.
Neither estimate imposes a cutoff-dependent condition on
$\varepsilon$.

Thus $(\mathbf{B}_{2})$ improves uniformly up to vacuum: for all $t\in [0, t^{*}]$ and $a+b\leq 4$,
\begin{equation}\label{eq: close of B_2 up to vacuum}
\mathcal{E}(V_{k; a, b})(t, \tilde{u}^{*})+\mathcal{F}(V_{k; a, b})(t, \tilde{u}^{*})\leq C\varepsilon^{2}t^{2k+3}.
\end{equation}
and improved $L^{\infty}$ estimates for $V_{k; a, b}, a+b\leq 2$ via Sobolev inequality \eqref{eq: W22 to L infty}.
\begin{equation}\label{eq: improved L infty bound up to vacuum}
\|V_{k; a, b}\|_{L^{\infty}(\Sigma_{t}^{\ut_*,\tilde{u}^{*}})}\leq C \varepsilon t^{k+1}.
\end{equation}

The preceding estimates prove the energy improvement $(\mathbf B_2)$ in both characteristic charts, including the positive-time initial terms. Together with the Sobolev inequalities of Section~\ref{sec: Preparations for the energy estimate}, they also give the stated $L^\infty$ consequence in Corollary~2.10. This completes the energy part of the proofs of Theorems~\ref{theorem: a prior estimate in W_1} and~\ref{theorem: a prior estimate in W_2}.

\section{Closing the pointwise bootstrap}\label{sec: closing B infinity}
We use the energy bounds to close the pointwise bootstrap. The change
of frame and the remainder estimates give the required bounds for $w$
and $\psi_2$. For the outgoing invariant, we instead use a transport
argument for $v^1+c$, because its leading normal derivative does not
vanish. The fixed-label measure is needed to retain the sharp coefficient
in the critical integral estimate. This proves the pointwise part of
Theorems~\ref{theorem: a prior estimate in W_1} and
\ref{theorem: a prior estimate in W_2} by closing $(\mathbf B_\infty)$.

For the positive-time construction, use lower time $\delta$ in
the transport integrations below. We include the three initial terms
and prove the critical-integral estimate
\eqref{eq: positive time critical inequality} in this section.
Section~\ref{subsec: common limit} passes the resulting spatial
bounds to the common limit at every positive time.

We give the argument up to vacuum; the same estimates apply in the region away from vacuum. For a derivative of total order $n$, we use $p=\infty$ when $n\leq2$ and $p=2$ when $3\leq n\leq4$; the corresponding lower-order $L^2$ estimates follow from the bounded slice volume. For products in which all but at most one factor are controlled pointwise, we use
\[
\|\Pi_{j=1}^{m} f_{j}\|_{L^{p}(\Sigma_t^{\tilde u_*,\tilde u^*})}\leq \Pi_{j=1}^{m-1} \|f_{j}\|_{L^{\infty}(\Sigma_t^{\tilde u_*,\tilde u^*})}\cdot \|f_{m}\|_{L^{p}(\Sigma_t^{\tilde u_*,\tilde u^*})}.
\]
At total order four, the change-of-frame expansion can also contain products of two second-order geometric derivatives. We estimate those terms by $L^4\cdot L^4\to L^2$, using \eqref{eq: quadratic second geometry} and \eqref{eq: mixed second geometric product}. For $|\beta|=|\beta'|=2$, these estimates give
\[
\begin{aligned}
\|Z^\beta(\hat T^2)Z^{\beta'}(\hat T^2)\|_2
&\leq P(M)\varepsilon t\left(\sum_{|\nu|=3}\|Z^\nu(\hat T^2)\|_2+P(M)\varepsilon t\right),\\
\|Z^\beta(\hat T^2)Z^{\beta'}(\kappa-t)\|_2
&\leq P(M)\varepsilon\left(t^2\sum_{|\nu|=3}\|Z^\nu(\hat T^2)\|_2+t\sum_{|\nu|=3}\|Z^\nu(\kappa-t)\|_2+P(M)\varepsilon t^3\right),\\
\|Z^\beta(\kappa-t)Z^{\beta'}(\kappa-t)\|_2
&\leq P(M)\varepsilon t^2\left(\sum_{|\nu|=3}\|Z^\nu(\kappa-t)\|_2+P(M)\varepsilon t^2\right).
\end{aligned}
\]
The third-order bounds of Section~\ref{sec: Preparations for the energy estimate} make these terms $O(P(M)\varepsilon^2t^2)$ or smaller. Thus they fit the error bounds used below.

In view of \eqref{eq: improved L infty bound up to vacuum}, we have for $t\in [0, t^{*}], \Zr\in \{\Tr, \Xr\}$,
\begin{equation}\notag
\begin{cases}
\|\Zr^{\alpha}(U-\Ur_{0})\|_{L^{\infty}(\Sigma_{t}^{\ut_*,\tilde{u}^{*}})}\leq C\varepsilon t, \ \ |\alpha|\leq 2,\\
\|\Zr^{\beta}(U-\Ur_{0})\|_{L^{2}(\Sigma_{t}^{\ut_*,\tilde{u}^{*}})}\leq C \varepsilon t, \ \ 3\leq |\beta|\leq 4.
\end{cases}
\end{equation}
for \textbf{both the case $k=0, \gamma\in (1, 1.3]$ and the case $k=1, \gamma\in \{\frac{7}{5}, \frac{5}{3}\}$}\footnote{Write $U-\Ur_{0}=U-\Ur_{1}+t\cdot \Vr_{1}.$}. 

Recall that 
\[
\begin{cases}
\mathring{c}=c_{r}(0, x_{2})+\frac{\gamma-1}{\gamma+1}\big( \frac{x_1}{t}-(v_{r}^{1}+c_{r})(0, x_{2}) \big),\\
\mathring{v}^{1}=v^{1}_{r}(0, x_{2})+\frac{2}{\gamma+1}\big( \frac{x_1}{t}-(v_{r}^{1}+c_{r})(0, x_{2}) \big),\\
\mathring{v}^{2}=v_{r}^{2}(0, x_{2}).
\end{cases}
\]
which implies that
\[
\begin{cases}
\lambda_{3}(\Ur_{0}, -\hat{\Tr})=\mathring{v}^{1}+\mathring{c}=\frac{x_{1}}{t},\\
\mathring{w}=w_{r}(0, x_{2}),\\
\mathring{v}^{2}=v_{r}^{2}(0, x_{2}).
\end{cases}
\]
Therefore, we have for $t\in [0, t^{*}]$ and $\Zr\in \{\Tr, \Xr\}$,
\begin{equation}\label{eq: improved estimates for Zr(psi)}
\begin{cases}
\|\Zr^{\alpha}(v^{1}+c-\frac{x_1}{t})\|_{L^{\infty}(\Sigma_{t}^{\ut_*,\tilde{u}^{*}})}\leq C \varepsilon t, \ \ |\alpha|\leq 2,\\
\|\Zr^{\beta}(v^{1}+c-\frac{x_1}{t})\|_{L^{2}(\Sigma_{t}^{\ut_*,\tilde{u}^{*}})}\leq C\varepsilon t, \ \ 3\leq |\beta|\leq 4,\\
\|\Tr\Zr^{\alpha}(w)\|_{L^{\infty}(\Sigma_{t}^{\ut_*,\tilde{u}^{*}})}+\|\Tr\Zr^{\alpha}(\psi_2)\|_{L^{\infty}(\Sigma_{t}^{\ut_*,\tilde{u}^{*}})}\leq C\varepsilon t, \ \ |\alpha|\leq 1,\\
\|\Tr\Zr^{\beta}(w)\|_{L^{2}(\Sigma_{t}^{\ut_*,\tilde{u}^{*}})}+\|\Tr\Zr^{\beta}(\psi_2)\|_{L^{2}(\Sigma_{t}^{\ut_*,\tilde{u}^{*}})}\leq C \varepsilon t, \ \ 2\leq |\beta|\leq 3,\\
\|\Xr^{n}(w)\|_{L^{\infty}(\Sigma_{t}^{\ut_*,\tilde{u}^{*}})}+\|\Xr^{n}(\psi_2)\|_{L^{\infty}(\Sigma_{t}^{\ut_*,\tilde{u}^{*}})}\leq C\varepsilon , 1\leq n\leqslant 2,\\
\|\Xr^{n}(w)\|_{L^{2}(\Sigma_{t}^{\ut_*,\tilde{u}^{*}})}+\|\Xr^{n}(\psi_2)\|_{L^{2}(\Sigma_{t}^{\ut_*,\tilde{u}^{*}})}\leq C\varepsilon , 3\leq n\leqslant 4.
\end{cases}
\end{equation}

\begin{remark}
Note that 
\[
\lambda_{3}(U, -\hat{\Tr})-\lambda_{3}(\Ur_{0}, -\hat{\Tr})=O(\varepsilon t). 
\]
is the \textbf{additional vanishing} used in Luo and Yu's work \cite{Luo-YuRare1, Luo-YuRare2}.
\end{remark}

In view of \eqref{eq: L infty control of the geometry}, \eqref{eq: L 2 control of the geometry} and
\[
\begin{cases}
\Tr=-(\frac{\kappar}{\kappa}\hat{T}^{1})\cdot T-(\kappar\hat{X}^{1})\cdot \hat{X},\\
\Xr=(\frac{\hat{T}^{2}}{\kappa})\cdot T+\hat{X}^{2}\cdot \hat{X},\\
T=-(\frac{\kappa}{\kappar}\hat{T}^{1})\cdot \Tr+(\kappa\hat{T}^{2})\cdot \Xr,\\
\hat{X}=-(\frac{\hat{X}^{1}}{\kappar})\cdot \Tr+\hat{X}^{2}\cdot \Xr.
\end{cases}
\]
we also obtain, for $t\in [0, t^{*}], \Zr\in \{\Tr, \Xr\}$ and $|\alpha|\leq 1, 2\leq |\beta|\leq 3$, 
\begin{equation}\label{eq: Zr control of the geometry}
\begin{cases}
\|\Zr^{\alpha}(\hat{T}^{1}+1)\|_{L^{\infty}(\Sigma_{t}^{\ut_*,\ut^*})}\leq P(M)\varepsilon^{2}t^{2},\\
\|\Zr^{\alpha}(\hat{T}^{2})\|_{L^{\infty}(\Sigma_{t}^{\ut_*,\ut^*})}\leq P(M)\varepsilon t,\\
\|\Zr^{\alpha}(\kappa-t)\|_{L^{\infty}(\Sigma_{t}^{\ut_*,\ut^*})}\leq P(M)\varepsilon t^{2},\\
\|\Zr^{\beta}(\hat{T}^{1}+1)\|_{L^{2}(\Sigma_{t}^{\ut_*,\ut^*})}\leq P(M)\varepsilon^{2}t^{2},\\
\|\Zr^{\beta}(\hat{T}^{2})\|_{L^{2}(\Sigma_{t}^{\ut_*,\ut^*})}\leq P(M)\varepsilon t,\\
\|\Zr^{\beta}(\kappa-t)\|_{L^{2}(\Sigma_{t}^{\ut_*,\ut^*})}\leq P(M)\varepsilon t^{2}.
\end{cases}
\end{equation}

\subsection{The incoming Riemann invariant and transverse velocity}
For a general $\psi\in \{w, \psi_{2}\}$ and $Z\in \{T, \hat{X}\}, |\alpha|=n, n\in \{1, 2, 3, 4\}$, we have schematically,
\begin{equation}\label{eq: fuzhu}
\begin{split}
Z_{1}\cdots Z_{n}(\psi)&=\sum_{\Zr_1,\ldots,\Zr_n\in\{\Tr,\Xr\}}\big(Z_{1}^{(\Zr_{1})}\Zr_{1}\big)\cdots \big(Z_{n}^{(\Zr_{n})}\Zr_{n}\big)(\psi)\\
&=\sum_{Z_j^{(\Zr_j)}\in\{T^{(\Tr)},\hat X^{(\Xr)}\}\ \forall j}Z_{1}^{(\Zr_{1})}\cdots Z_{n}^{(\Zr_{n})} \Zr_{1}\cdots \Zr_{n}(\psi)\\
&+\sum_{\textbf{$Z_{k}^{(\Zr_{k})}\notin \{T^{(\Tr)}, \hat{X}^{(\Xr)}\}$ for some k}} Z_{1}^{(\Zr_{1})}\cdots Z_{n}^{(\Zr_{n})} \Zr_{1}\cdots \Zr_{n}(\psi)\\
&+\sum \Zr^{ \alpha_{1}}(\textbf{coefficient})\cdots \Zr^{ \alpha_{m}}(\textbf{coefficient})\cdot \Zr^{\beta_{m+1}}(\psi)
\end{split}
\end{equation}
where the summation is for $\sum_{k=1}^{m}|\alpha_{k}|+|\beta_{m+1}|\leq n, |\alpha_{k}|\leq n-1, |\beta_{m+1}|\leq n-1$ and 
\[
\textbf{coefficient}\in \{ T^{(\Tr)}, T^{(\Xr)}, \hat X^{(\Tr)}, \hat X^{(\Xr)} \}.
\]

\begin{itemize}
\item If $Z_{1}\cdots Z_{n}(\psi)=(\hat X)^{n}(\psi)$, then in view of \eqref{eq: improved estimates for Zr(psi)}, \eqref{eq: Zr control of the geometry}, and \eqref{eq: fuzhu}, we have 
\[
\|(\hat X)^{n}(\psi)\|_{L^{p}(\Sigma_{t}^{\ut_*,\tilde{u}^{*}})}\leq C\varepsilon+P(M)\varepsilon^{2}\leq C\varepsilon, 
\]
provided that $\varepsilon$ is sufficiently small.
\item If $Z_{1}\cdots Z_{n}(\psi)\not=(\hat X)^{n}(\psi)$, we can replace $Z_{1}\cdots Z_{n}(\psi)$ by $Z_{1}\cdots Z_{n-1}T(\psi)$ (by relabeling the index) and
\[
\|Z_{1}\cdots Z_{n}(\psi)-Z_{1}\cdots Z_{n-1}T(\psi)\|_{L^{p}(\Sigma_{t}^{\ut_*,\tilde{u}^{*}})}\leq P(M)\varepsilon^{2} t.
\]
Then in view of \eqref{eq: improved estimates for Zr(psi)}, \eqref{eq: Zr control of the geometry}, and \eqref{eq: fuzhu}, we have
\[
\|Z_{1}\cdots Z_{n-1}T(\psi)\|_{L^{p}(\Sigma_{t}^{\ut_*,\tilde{u}^{*}})}\leq C\varepsilon t+P(M)\varepsilon^{2}t\leq C\varepsilon t,
\]
provided that $\varepsilon$ is sufficiently small.
\end{itemize}

Therefore, we have for $t\in [0, t^{*}]$ and $\psi\in \{w, \psi_{2}\}$,
\begin{equation}\label{eq: improved estimates for w and psi_2}
\begin{cases}
\|(\hat X)^n(\psi)\|_{L^{\infty}(\Sigma_{t}^{\ut_*,\tilde{u}^{*}})}\leq C\varepsilon, 1\leq n\leq 2\\
\|(\hat X)^n(\psi)\|_{L^{2}(\Sigma_{t}^{\ut_*,\tilde{u}^{*}})}\leq C\varepsilon, 3\leq n\leq 4\\
\|TZ^{\alpha}(\psi)\|_{L^{\infty}(\Sigma_{t}^{\ut_*,\tilde{u}^{*}})}\leq C\varepsilon t, 0\leq |\alpha|\leq 1,\\
\|TZ^{\beta}(\psi)\|_{L^{2}(\Sigma_{t}^{\ut_*,\tilde{u}^{*}})}\leq C\varepsilon t, 2\leq |\beta|\leq 3.
\end{cases}
\end{equation}

\subsection{The outgoing Riemann invariant and characteristic speed}
In view of \eqref{eq: improved estimates for w and psi_2}, it suffices to bound the maximal characteristic speed $v^{1}+c$ in place of $\underline{w}$ since $v^{1}+c=\frac{\gamma+1}{2}\wb+\frac{\gamma-3}{2}w$.

For the transport integrations in this subsection use the fixed
characteristic-label measure:
\[
\|h\|_{p,0;t}:=
\left(\int_{\ut_*}^{\ut^*}\int_{\mathbb T}
 |h(t,\tilde u,\vartheta)|^p\,d\vartheta\,d\tilde u\right)^{1/p},
\]
with the usual essential-supremum definition for $p=\infty$.
Here $p=\infty$ at orders at most two and $p=2$ at orders three
and four. These norms are uniformly equivalent to the slice norms,
since $\slashed g$ is bounded above and below independently of the
cutoff. Use this equivalence for the already established source
bounds and for the final conclusion, while keeping the fixed
measure throughout each transport integration. Since $L=\partial_t$
in these labels, Minkowski's inequality then has coefficient exactly
one.

We first prove the positive-time critical-integral estimate. If $f\geq0$ is
continuous on $[\delta,t^*]$, $g_\delta,A\geq0$, and
\[
f(t)\leq \frac{g_\delta+\int_\delta^t f(s)\,ds}{t}+At,
\qquad t\geq\delta>0,
\]
then $F(t)=(g_\delta+\int_\delta^t f(s)\,ds)/t$ satisfies
$F'(t)=(f(t)-F(t))/t\leq A$ and
$F(\delta)=g_\delta/\delta$. Consequently
\begin{equation}\label{eq: positive time critical inequality}
f(t)\leq\frac{g_\delta}{\delta}+2At.
\end{equation}
We keep $g_\delta$ in all three applications below.

\begin{itemize}
\item If $Z^{\alpha}(v^{1}+c)=(\hat X)^{n}(v^{1}+c), n=1, 2, 3, 4$. Since $\hat{X}=-\hat{T}^{1}\Xr-\frac{\hat{T}^{2}}{\kappar}\Tr$, we can rewrite $(\hat X)^{n}(v^{1}+c)$ as
\[
(\hat{X}^{(\Tr)}\Tr+\hat{X}^{(\Xr)}\Xr)^{n}(v^{1}+c)
\]
The sum consists of terms of the form
\[
\Zr_{1}^{\beta_{1}}(f_{1})\cdots \Zr^{\beta_{m}}_{m}(f_{m})\Zr^{\alpha_{1}}(v^{1}+c), \sum_{k=1}^{m}|\beta_{k}|+|\alpha_{1}|=n, |\alpha_{1}|\geq 1,
\]
where $f_{i}\in \{\hat{X}^{(\Xr)}, \hat{X}^{(\Tr)}\}$. The terms can be divided into three cases:
\begin{itemize}
\item main term: $\beta_{1}=...=\beta_{m}=0, f_{1}=\cdots=f_{m}=\hat{X}^{(\Xr)}$ 
\[
\|(\hat{X}^{(\Xr)})^{n}\Xr^n(v^{1}+c)\|_{p,0;t}\leq C\varepsilon t.
\]
\item \textbf{main error term}: $(\hat X)^{n-1}(\hat{X}^{(\Tr)})\Tr(v^{1}+c)$. We only have
\[
\|(\hat X)^{n-1}(\hat{X}^{(\Tr)})\Tr(v^{1}+c)\|_{p,0;t}\leq P(M)\varepsilon.
\]
which is not enough to close $(\mathbf{B}_{\infty})$.
\item The remaining terms satisfy $\Zr^{\alpha_{1}}(v^{1}+c)\not=\Tr(v^{1}+c)$, and hence
\[
\|\Zr_{1}^{\beta_{1}}(f_{1})\cdots \Zr^{\beta_{m}}_{m}(f_{m})\Zr^{\alpha_{1}}(v^{1}+c)\|_{p,0;t}\leq P(M)\varepsilon^{2} t.
\]
\end{itemize}

In view of $\|\Tr(v^{1}+c)+1\|_{L^{\infty}(\Sigma_{t}^{\ut_*,\tilde{u}^{*}})}\leq C\varepsilon t$, we can rewrite the \textbf{main error term} as 
\[
\|(\hat X)^{n-1}(\hat X^{(\Tr)})\Tr(v^1+c)+(\hat X)^{n-1}(\hat X^{(\Tr)})\|_{p,0;t}\leq P(M)\varepsilon^2t.
\]
and we have
\[
\|(\hat X)^{n}(v^{1}+c)+(\hat X)^{n-1}(\hat{X}^{(\Tr)})\|_{p,0;t}\leq C\varepsilon t+P(M)\varepsilon^{2}t.
\]
In view of the transport equation for $\hat{T}^{2}$, we also have
\[
\|L(\hat X)^{n-1}(\hat{T}^{2})\|_{p,0;t}\leq \|(\hat X)^{n}(v^{1}+c)\|_{p,0;t}+ P(M)\varepsilon^{2}t,
\]
Integrating from $\delta$ to $t$, set
$g_\delta=\|(\hat X)^{n-1}(\hat T^2)(\delta)\|_{p,0;\delta}$.
The initial geometry \eqref{eq: positive time initial weights}
gives $g_\delta\leq C\varepsilon\delta$, and
\[
\|(\hat X)^{n-1}(\hat{T}^{2})\|_{p,0;t}
\leq g_\delta+\int_{\delta}^{t}\|(\hat X)^{n}(v^{1}+c)\|_{p,0;\tau}d\tau
 +P(M)\varepsilon^{2}t^{2}.
\]
Since $\hat{X}^{(\Tr)}=-\frac{\hat{T}^{2}}{\kappar}$, we have
\[
\|(\hat X)^{n}(v^{1}+c)\|_{p,0;t}
\leq\frac{g_\delta+\int_{\delta}^{t}\|(\hat X)^{n}(v^{1}+c)\|_{p,0;\tau}d\tau}{t}
 +P(M)\varepsilon^{2}t+C\varepsilon t.
\]
Apply the positive-time inequality with
$f(t)=\|(\hat X)^n(v^1+c)\|_{p,0;t}$,
$A=C\varepsilon+P(M)\varepsilon^2$, and
$g_\delta/\delta\leq C\varepsilon$.
It follows that
\[
\|(\hat X)^n(v^1+c)\|_{p,0;t}
\leq C\varepsilon+C\varepsilon t+P(M)\varepsilon^2t
\leq C\varepsilon,
\]
provided that $\varepsilon$ is sufficiently small.

\item The case $Z^{\alpha}(v^{1}+c), Z^{\alpha}\not=(\hat X)^{|\alpha|}, 1\leq |\alpha|\leq n, n=1, 2, 3, 4$. 

We first deal with the case $T(v^{1}+c)+1$,
\[
\begin{split}
T(v^{1}+c)+1&=-\frac{\kappa}{\kappar}\hat{T}^{1}\Tr(v^{1}+c)+\kappa\hat{T}^{2}\Xr(v^{1}+c)+1\\
&=-\frac{\kappa}{\kappar}+1+O(C\varepsilon t+P(M)\varepsilon^{2}t)\\
\end{split}
\]
In view of the transport equation for $\kappa$,
\[
\big|L(\kappa-\kappar)+T(v^1+c)+1\big|\leq P(M)\varepsilon^2t.
\]
Integrating from $\delta$ to $t$ along a fixed characteristic label,
\[
\left|\kappar-\kappa-(\delta-\kappa(\delta))
 -\int_\delta^t\big(T(v^1+c)+1\big)\,d\tau\right|
\leq P(M)\varepsilon^2t^2.
\]
Set $g_\delta=|\delta-\kappa(\delta)|$ at that label.
Then $g_\delta\leq C\varepsilon^2\delta^3$ by
\eqref{eq: positive time initial weights}, and
\[
\big|T(v^1+c)+1\big|
\leq\frac{g_\delta+\int_\delta^t\big|T(v^1+c)+1\big|\,d\tau}{t}
 +C\varepsilon t+P(M)\varepsilon^2t.
\]
The positive-time inequality applied pointwise gives
\[
\big|T(v^1+c)+1\big|
\leq C\varepsilon^2\delta^2+C\varepsilon t+P(M)\varepsilon^2t
\leq C\varepsilon t.
\]
Here $\delta^2\leq t$ on the normalized interval
$\delta\leq t\leq1$, and $\varepsilon$ is sufficiently small.

It remains to consider the case $Z^{\beta}\big(T(v^{1}+c)\big)$ with $|\beta|=m, m=1, 2, 3$. We can rewrite $Z^{\beta}T(v^{1}+c)$ as
\[
(Z_{1}^{(\Tr)}\Tr+Z_{1}^{(\Xr)}\Xr)\cdots (Z_{m}^{(\Tr)}\Tr+Z_{m}^{(\Xr)}\Xr)\Big((T^{(\Xr)}\Xr+T^{(\Tr)}\Tr)(v^{1}+c)\Big),
\]
where the sum consists of terms of the form
\[
\Zr_{1}^{\beta_{1}}(f_{1})\cdots \Zr^{\beta_{s}}_{s}(f_{s})\Zr^{\alpha_{1}}\Big((T^{(\Xr)}\Xr+T^{(\Tr)}\Tr)(v^{1}+c)\Big), \sum_{k=1}^{s}|\beta_{k}|+|\alpha_{1}|=m, |\alpha_{1}|\geq 1,
\]
where $f_{i}\in \{\hat{X}^{(\Xr)}, \hat{X}^{(\Tr)}, T^{(\Xr)}, T^{(\Tr)}\}$. The terms can be divided into three cases:
\begin{itemize}
\item main term: $\beta_{1}=...=\beta_{s}=0, f_{k}\in \{T^{(\Tr)}, \hat{X}^{(\Xr)}\}$, 
\[
\|Z_{1}^{(\Zr_{1})}\cdots Z_{m}^{(\Zr_{m})}T^{(\Tr)}\Zr_{1}\cdots \Zr_{m}\Tr(v^{1}+c)\|_{p,0;t}\leq C\varepsilon t.
\]
\item \textbf{main error term}: $Z^{\beta}(T^{(\Tr)})\Tr(v^{1}+c)$. The bound $\|Z^{\beta}(T^{(\Tr)})\Tr(v^{1}+c)\|_{p,0;t}\leq P(M)\varepsilon t$ for \textbf{main error term} is insufficient to close the $(\mathbf{B}_{\infty})$
\item For the remaining terms, $\Zr^{\alpha_{1}}(v^{1}+c)\not=\Tr(v^{1}+c)$, so
\[
\|\Zr_{1}^{\beta_{1}}(f_{1})\cdots \Zr^{\beta_{s}}_{s}(f_{s})\Zr^{\alpha_{1}}(v^{1}+c)\|_{p,0;t}\leq P(M)\varepsilon^{2} t.
\]
\end{itemize}
In view of $\|\Tr(v^{1}+c)+1\|_{L^{\infty}(\Sigma_{t}^{\ut_*,\tilde{u}^{*}})}\leq C\varepsilon t$, we can write the main error term as
\[
\|Z^{\beta}(T^{(\Tr)})\Tr(v^{1}+c)+Z^{\beta}(T^{(\Tr)})\|_{p,0;t}\leq P(M)\varepsilon^{2}t,
\]
which implies that
\[
\|Z^{\beta}T(v^{1}+c)+Z^{\beta}(T^{(\Tr)})\|_{p,0;t}\leq C\varepsilon t+ P(M)\varepsilon^{2}t.
\]

We also have
\[
\begin{cases}
\|L(Z^{\beta}(\kappa\hat{T}^{1}))-Z^{\beta}T(v^{1}+c)\|_{p,0;t}\leq P(M)\varepsilon^{2}t,\\
\|Z^{\beta}(\kappa\hat{T}^{1})+Z^{\beta}(\kappa)\|_{p,0;t}\leq P(M)\varepsilon^{2}t^{2}.
\end{cases}
\]
Integrating $LZ^{\beta}(\kappa\hat{T}^{1})$ from $\delta$ to $t$, set
$g_\delta=\|Z^\beta(\kappa\hat T^1)(\delta)\|_{p,0;\delta}$.
The explicit initial geometry gives
$g_\delta\leq C\varepsilon^2\delta^3$ for $1\leq|\beta|\leq3$;
see Proposition~\ref{prop: positive time construction}. Hence
\[
\|Z^{\beta}(\kappa\hat{T}^{1})\|_{p,0;t}
\leq g_\delta+\int_{\delta}^{t}\|Z^{\beta}T(v^{1}+c)\|_{p,0;\tau}d\tau
 +P(M)\varepsilon^{2}t^{2}.
\]
Since $T^{(\Tr)}=-\frac{\kappa}{\kappar}\hat{T}^{1}$, we obtain
\[
\|Z^{\beta}T(v^{1}+c)\|_{p,0;t}
\leq\frac{g_\delta+\int_{\delta}^{t}\|Z^{\beta}T(v^{1}+c)\|_{p,0;\tau}d\tau}{t}
 +P(M)\varepsilon^{2}t+C\varepsilon t.
\]
The positive-time inequality therefore gives
\[
\|Z^{\beta}T(v^{1}+c)\|_{p,0;t}
\leq C\varepsilon^2\delta^2+P(M)\varepsilon^{2}t+C\varepsilon t
\leq C\varepsilon t,
\]
provided that $\varepsilon$ is sufficiently small, using
$\delta^2\leq t$ as above. This includes every ordered word
$Z^\beta$. 
\end{itemize}

The same estimates include the case with lower time zero. Fix $0<\sigma<t$, apply \eqref{eq: positive time critical inequality} on $[\sigma,t]$, and let $\sigma\rightarrow0$. We check the three initial terms separately. In the first application, the initial label satisfies
\[
\mathring u(0,\tilde u,\vartheta)=h(\vartheta)-\tilde u,\qquad h(x_2)=\left(v_r^1-\frac{2c_r}{\gamma-1}\right)(0,x_2).
\]
The initial Jacobian identities give $\hat X|_{t=0}=\partial_\vartheta$ and $\hat T^2/t\rightarrow\partial_\vartheta(\mathring u(0,\tilde u,\vartheta))$. Hence, for $1\leq n\leq4$ and the corresponding choice of $p$,
\[
\limsup_{\sigma\rightarrow0}\frac{\| (\hat X)^{n-1}(\hat T^2)(\sigma)\|_{p,0;\sigma}}{\sigma}
\leq C\|\partial_{x_2}^{n}h\|_{L^p(\mathbb T)}\leq C\varepsilon.
\]
This is the initial contribution allowed in \eqref{eq: positive time critical inequality}. In the second application, $\kappa/t\rightarrow1$ gives
\[
\lim_{\sigma\rightarrow0}\frac{|\sigma-\kappa(\sigma)|}{\sigma}=0.
\]
In the third application, $\kappa/t\rightarrow1$ and $\hat T^1\rightarrow-1$ on the initial slice, with all positive-order spatial derivatives of their product equal to zero there. Thus, for $1\leq|\beta|\leq3$,
\[
\lim_{\sigma\rightarrow0}\frac{\|Z^\beta(\kappa\hat T^1)(\sigma)\|_{p,0;\sigma}}{\sigma}=0.
\]
The three limits therefore give the required initial contributions for the zero-lower-time estimates.

Converting the resulting fixed-label norms back to the slice
norms preserves the $C\varepsilon$ and $C\varepsilon t$ bounds,
with constants independent of both $\delta$ and $\alpha$.
These are estimates for the approximations on $[\delta,1]$.
Section~\ref{subsec: common limit} passes these estimates to the
limit along the same sequence: the local strong convergence of the
fields and frame identifies ordered derivatives through order four,
and weighted lower semicontinuity followed by exhaustion gives the
full-slice bounds at every positive time, up to $\tilde u=0$.

Finally, in view of Euler equation \eqref{Euler equations:form 2}
\[
\begin{cases}
L (\wb) &= -c \widehat{T}(\wb)(\widehat{T}^1+1)+\frac{1}{2}c \widehat{T}(\psi_2)\widehat{T}^2 +\frac{1}{2}c \Xh(\psi_2)\Xh^2-c\Xh(\wb)\Xh^1,\\
L (w) &= 	c \widehat{T}(w)(\widehat{T}^1-1)	+\frac{1}{2}c \widehat{T}(\psi_2)\widehat{T}^2 +c\Xh(w)\Xh^1+\frac{1}{2}c\Xh(\psi_2)\Xh^2,\\
L (\psi_2) &= -c \widehat{T}(\psi_2)+c\widehat{T}( w+\wb)\widehat{T}^2+c\Xh( w+\wb)
\Xh^2.
\end{cases}
\]
we have 
\[
\|c^{-1}L(\wb)\|_{L^{\infty}(\Sigma_{t}^{\ut_*,\tilde{u}^{*}})}+\|c^{-1}L(w)\|_{L^{\infty}(\Sigma_{t}^{\ut_*,\tilde{u}^{*}})}+\|c^{-1}L(\psi_2)\|_{L^{\infty}(\Sigma_{t}^{\ut_*,\tilde{u}^{*}})}\leq C\varepsilon.
\]

These estimates improve $(\mathbf{B}_{\infty})$ in both regions.
Choose the bootstrap constant $M$ first, larger than four times
all the fixed constants in the improved estimates and in the
initial and incoming bounds. Then choose $\varepsilon_0$ so small
that every contribution of the form $P(M)\varepsilon$ relative to
those bounds is at most $M/4$ (after combining the finitely many
terms in each inequality). The bounds are then those of
$(\mathbf B_2)$ and $(\mathbf B_\infty)$ with $M/2$. The Jacobian
bound and the estimate of $\hat T^1+1$ in
Section~\ref{sec: Preparations for the energy estimate} are bounded
by $P(M)\varepsilon$, which can also be made at most $1/200$.
This proves the stated improvement of all three inequalities. The constants are independent of $\alpha$,
so this order of choices is uniform in the vacuum cutoff.
This closes the a priori bootstrap on the positive-time intervals
used in Proposition~\ref{prop: positive time construction}, including
their initial terms. The continuation and common-limit
argument there supplies the solution up to vacuum.
This completes the proofs of Theorems~\ref{theorem: a prior estimate in W_1} and~\ref{theorem: a prior estimate in W_2}: Section~5 supplies the geometric estimates, Section~6 proves the energy improvement, and this section proves the pointwise improvement. The remaining conclusions of the main theorem follow from the boundary, comparison, regularity, and entropy arguments in Section~\ref{subsec: proof main theorem}.

\section{Completion of the solution and proof of Theorem~\ref{theorem: stability of RVR}}\label{subsec: proof main theorem}
This section completes the proof of Theorem~\ref{theorem: stability of RVR}.
We first prove Proposition~\ref{prop: positive time construction}: the
first two subsections construct the positive-time approximations and pass
to a common limit. We then determine the boundary graphs, compare the
left and right profiles, verify the conservation laws and entropy
inequality, prove the additional spacetime bounds, and return to the
coordinates of the main theorem.

Throughout this section, $\delta>0$ denotes the positive starting time
and $\alpha>0$ the vacuum cutoff. In the limiting argument, we choose
a common sequence $(\delta_n,\alpha_n)\to(0,0)$ satisfying
$0<\delta_n\leq d_0\alpha_n$.

\subsection{Construction from positive time}
\label{subsec: positive time construction}\label{subsubsec: positive time construction}
\begin{proof}[Proof of Proposition~\ref{prop: positive time construction}(i)]
We prove part~(i) in Steps~1--2. The common limit required for part~(ii) is treated in Section~\ref{subsec: common limit}.

The positive starting time $\delta$ is used only to construct regular
approximating solutions. The limiting solution is constructed from the
initial rarefaction fan at $t=0$. The reason for starting at $t=\delta>0$
is that, in the opened coordinates, the Euler equations take the form
\[
t\partial_t(U)+(A_1(U)-\urr I)\partial_{\urr}(U)
+tA_2(U)\partial_{x_2}(U)=0.
\]
The coefficient of the time derivative vanishes at $t=0$. Thus $t=0$
is a singular initial hypersurface, and the usual local Cauchy theorem
for a nondegenerate symmetric-hyperbolic system does not give a direct
construction there. We instead prescribe data on the regular slice
$t=\delta$, construct the solution for $t\geq\delta$, and use estimates
independent of $\delta$ to pass to the limit $\delta\to0$.

This is an adaptation of the last-slice argument introduced by
Christodoulou and Klainerman in their proof of the global nonlinear
stability of Minkowski space \cite[Chapter~14]{CK}. Luo and Yu used this
method to construct irrotational rarefaction waves away from vacuum
\cite[Theorem~1 and Section~2.3.1]{Luo-YuRare2}. In our earlier work
\cite[Sections~3 and~7]{JiaLuo}, we used the same method for
rarefaction waves with vorticity away from vacuum. Here we adapt it
further and remove the positive starting time and the vacuum cutoff
together, subject to $0<\delta\leq d_0\alpha$.

The positive-time approximation uses the construction of
Luo and Yu \cite[Corollary~1 and Sections~4.1--4.3]{Luo-YuRare2}
together with the one-sided construction from our earlier work
\cite[Section~7.1]{JiaLuo}. We describe the initialization, fixed
overlap and uniform continuation here; the common limit is treated in
Section~\ref{subsec: common limit}. The estimates uniform in both
cutoffs are proved in Sections
\ref{sec: Preparations for the energy estimate}--\ref{sec: closing B infinity}.

\smallskip\noindent\textit{1. Initial data and the two cutoffs.}
Choose a fixed smooth $\chi_{\rm cut}$, zero for $s\leq3/5$
and one for $s\geq2/3$, and set
\begin{equation}\label{eq: positive time data}
\begin{aligned}
U_\delta(\urr,x_2)={}&\Ur_k(\delta,\urr,x_2)\\
&+\chi_{\rm cut}\!\left(\frac{c_0(\urr,x_2)}{\mathbf c_*}\right)
\big(U^a(\delta,\urr,x_2)-\Ur_k(\delta,\urr,x_2)\big).
\end{aligned}
\end{equation}
The product with $\chi_{\rm cut}$ is set to zero where $c_0\leq3\mathbf c_*/5$.
The bounds in (B) give
\begin{equation}\label{eq: positive time data bounds}
\|U_\delta-\Ur_k(\delta)\|_{H^4}
\leq C\varepsilon\delta^{k+1},\qquad
\|U_\delta-\Vr_0\|_{H^4}\leq C\varepsilon\delta.
\end{equation}
On the fixed overlap these bounds hold through the available data order.
Near $c_0=\alpha\mathbf c_*$, the data are exactly $\Ur_k(\delta)$:
the affine profile for $k=0$, and a polynomial in $c_0$ for
$k=1$, $\gamma=7/5,5/3$.
Initialize the characteristic coordinates at $t=\delta$ by
\[
\tilde u(\delta,\urr,x_2)=
-\urr+\left(v_r^1-\frac{2c_r}{\gamma-1}\right)(0,x_2),
\qquad \vartheta(\delta,\urr,x_2)=x_2.
\]
Then, with every quantity below evaluated at $t=\delta$,
\begin{equation}\label{eq: positive time initial geometry}
\begin{aligned}
\slashed g&=1+\delta^2(\partial_{x_2}(\tilde u))^2,
&\hat T&=\frac{(-1,\delta\partial_{x_2}(\tilde u))}{\sqrt{\slashed g}},\\
\kappa&=\frac{\delta}{\sqrt{\slashed g}},
&\Xi&=-\frac{\delta^2\partial_{x_2}(\tilde u)}{\slashed g}.
\end{aligned}
\end{equation}
Differentiation gives the weighted initial bounds
\eqref{eq: positive time initial weights}; together with
\eqref{eq: positive time data bounds}, these initialize the pointwise
bootstrap and give energy at most $C\varepsilon^2\delta^{2k+3}$.
Also $|c_\delta-c_0|\leq C\varepsilon\delta$.
Thus $\delta\leq d_0\alpha$, with $d_0$ fixed sufficiently small,
gives $3c_0/4\leq c_\delta\leq5c_0/4$ on $c_0\geq\alpha\mathbf c_*$.
This is the only coupling of the two cutoffs; it does not reduce
the allowed size of $\varepsilon$ as $\alpha\to0$.

\smallskip\noindent\textit{2. Local construction and uniform continuation.}
To keep an overlap with $U^a$, choose the graphs $x_1=G(t,x_2)$
and $x_1=B(t,x_2)$, where
\[
\begin{aligned}
G(t,x_2)&=t\left[\left(v_r^1-\frac{2c_r}{\gamma-1}\right)(0,x_2)
+\left(\frac75+\frac{t}{50}\right)u^*\right],\\
B(t,x_2)&=t\left[\left(v_r^1-\frac{2c_r}{\gamma-1}\right)(0,x_2)
+\left(\frac95+\frac{t}{50}\right)u^*\right].
\end{aligned}
\]
Both graphs lie in the known region. For $f=G,B$, the matrix
$N_f=\partial_t(f) I-A_1(U^a)+\partial_{x_2}(f) A_2(U^a)$ has least eigenvalue
$\partial_t(f)-v^1+\partial_{x_2}(f)\,v^2-c\sqrt{1+(\partial_{x_2}(f))^2}$.
Since $v^1+c-\urr=O(\varepsilon t)$ there,
\begin{equation}\label{eq: overlap spacelike matrix}
N_G,\ N_B\geq\frac{u^*t}{100}I.
\end{equation}
Prescribe $U^a$ on $x_1=B$ and \eqref{eq: positive time data} on
$t=\delta$. The data agree near their intersection; rounding these
surfaces within that region preserves the positive time matrix.
The symmetric-hyperbolic Cauchy theorem
\cite[Theorem~2.4.2]{Tice2017} therefore applies.
Extend the initial data slightly below $c_0=\alpha\mathbf c_*$, keeping
$c>0$, and restrict the solution to the region to the right of the fastest outgoing
characteristic from that curve. Its nonnegative boundary matrix
gives uniqueness in that region by the difference-energy
identity, independently of the auxiliary extension.
The same identity and \eqref{eq: overlap spacelike matrix} give
$U^{\delta,\alpha}=U^a$ for $G<x_1<B$.

The characteristic flow in the fixed nonvacuum region has opened
displacement $O(\varepsilon(t-\delta))$. Hence the characteristic
issued from $\tilde u=\tilde u_*$ at $t=\delta$ stays between $G$ and $B$,
and the open region $18\mathbf c_*/25<c_0<73\mathbf c_*/100$
is a common overlap.
Figure~\ref{fig: double cutoffs} shows fixed-$x_2$ sections of
the two cutoffs and this chosen overlap. The hatched region is contained
in the domain of $U^a$ and includes the green strip. Since $c_0$ is
independent of $t$ in the opened coordinates, the two sides of the green
strip are vertical there and become rays under $x_1=t\urr$.

\begin{figure}[htbp]
\centering
\begin{tikzpicture}[x=1cm,y=1cm,font=\small,line join=round,line cap=round]
\definecolor{DCgas}{RGB}{54,106,156}
\definecolor{DCvac}{RGB}{170,64,63}
\definecolor{DCtime}{RGB}{176,118,24}
\definecolor{DCoverlap}{RGB}{22,132,91}
\def\DCdelta{0.28}
\def\DCcut{0.15}
\def\DCknownleft{0.48}
\def\DCknownright{1.03}
\def\DCoverlapleft{0.62}
\def\DCoverlapright{0.74}
\def\DCouterstart{0.82}
\def\DCx#1#2{7.9+7.8*(1-\DCphysical+\DCphysical*(#1))*(-0.85+(#2))}
\def\DCy#1{\DCbase+3.8*(#1)}
\def\DCvacuum#1{0.012*(#1)+0.006*(#1)^2}
\def\DCcutoff#1{\DCcut+0.018*((#1)-\DCdelta)+0.005*((#1)-\DCdelta)^2}
\def\DCouter#1{\DCouterstart+0.008*((#1)-\DCdelta)-0.004*((#1)-\DCdelta)^2}
\foreach \DCphysical/\DCbase in {1/6.0,0/0.25}{
  \path[fill=DCvac!12]
    ({\DCx{\DCdelta}{\DCvacuum{\DCdelta}}},{\DCy{\DCdelta}})
    -- plot[domain=\DCdelta:1,samples=35,variable=\DCt]
       ({\DCx{\DCt}{\DCvacuum{\DCt}}},{\DCy{\DCt}})
    -- ({\DCx{1}{\DCcutoff{1}}},{\DCy{1}})
    -- plot[domain=1:\DCdelta,samples=35,variable=\DCt]
       ({\DCx{\DCt}{\DCcutoff{\DCt}}},{\DCy{\DCt}}) -- cycle;
  \path[fill=black!7]
    ({\DCx{0}{\DCknownleft}},{\DCy{0}})
    -- plot[domain=0:1,samples=30,variable=\DCt]
       ({\DCx{\DCt}{\DCknownleft}},{\DCy{\DCt}})
    -- ({\DCx{1}{\DCknownright}},{\DCy{1}})
    -- plot[domain=1:0,samples=30,variable=\DCt]
       ({\DCx{\DCt}{\DCknownright}},{\DCy{\DCt}}) -- cycle;
  \path[fill=DCgas!17]
    ({\DCx{\DCdelta}{\DCcut}},{\DCy{\DCdelta}})
    -- plot[domain=\DCdelta:1,samples=35,variable=\DCt]
       ({\DCx{\DCt}{\DCcutoff{\DCt}}},{\DCy{\DCt}})
    -- ({\DCx{1}{\DCouter{1}}},{\DCy{1}})
    -- plot[domain=1:\DCdelta,samples=35,variable=\DCt]
       ({\DCx{\DCt}{\DCouter{\DCt}}},{\DCy{\DCt}}) -- cycle;
  \path[fill=DCoverlap!34]
    ({\DCx{\DCdelta}{\DCoverlapleft}},{\DCy{\DCdelta}})
    -- ({\DCx{1}{\DCoverlapleft}},{\DCy{1}})
    -- ({\DCx{1}{\DCoverlapright}},{\DCy{1}})
    -- ({\DCx{\DCdelta}{\DCoverlapright}},{\DCy{\DCdelta}}) -- cycle;
  \path[pattern=north east lines,pattern color=black!28]
    ({\DCx{0}{\DCknownleft}},{\DCy{0}})
    -- plot[domain=0:1,samples=30,variable=\DCt]
       ({\DCx{\DCt}{\DCknownleft}},{\DCy{\DCt}})
    -- ({\DCx{1}{\DCknownright}},{\DCy{1}})
    -- plot[domain=1:0,samples=30,variable=\DCt]
       ({\DCx{\DCt}{\DCknownright}},{\DCy{\DCt}}) -- cycle;
  \foreach \DClevel in {\DCknownleft,\DCknownright}{
    \draw[black!40,densely dotted,line width=0.6pt]
      ({\DCx{0}{\DClevel}},{\DCy{0}})
      -- ({\DCx{1}{\DClevel}},{\DCy{1}});
  }
  \foreach \DClevel in {\DCoverlapleft,\DCoverlapright}{
    \draw[DCoverlap,line width=0.65pt]
      ({\DCx{\DCdelta}{\DClevel}},{\DCy{\DCdelta}})
      -- ({\DCx{1}{\DClevel}},{\DCy{1}});
  }
  \draw[DCvac!85!black,densely dashed,line width=0.8pt]
    plot[domain=0:1,samples=35,variable=\DCt]
      ({\DCx{\DCt}{\DCvacuum{\DCt}}},{\DCy{\DCt}});
  \draw[DCvac,line width=1.1pt]
    plot[domain=\DCdelta:1,samples=35,variable=\DCt]
      ({\DCx{\DCt}{\DCcutoff{\DCt}}},{\DCy{\DCt}});
  \draw[DCgas,dash pattern=on 3pt off 1.6pt,line width=1.1pt]
    plot[domain=\DCdelta:1,samples=35,variable=\DCt]
      ({\DCx{\DCt}{\DCouter{\DCt}}},{\DCy{\DCt}});
  \draw[black!25,line width=0.4pt]
    ({\DCx{1}{\DCvacuum{1}}},{\DCy{1}})
    -- ({\DCx{1}{\DCknownright}},{\DCy{1}});
  \draw[->,>=stealth,black!70,line width=0.55pt]
    (0.65,\DCbase) -- (0.65,\DCbase+4.10) node[above] {$t$};
  \draw[->,>=stealth,black!70,line width=0.55pt]
    (0.65,\DCbase) -- (9.65,\DCbase);
  \node[anchor=east,font=\footnotesize] at (0.55,\DCbase+3.8) {$1$};
  \draw[DCtime!70,densely dotted,line width=0.5pt]
    (0.65,{\DCy{\DCdelta}})
    -- ({\DCx{\DCdelta}{\DCcut}},{\DCy{\DCdelta}});
  \draw[DCtime,line width=1.1pt]
    ({\DCx{\DCdelta}{\DCcut}},{\DCy{\DCdelta}})
    -- ({\DCx{\DCdelta}{\DCouterstart}},{\DCy{\DCdelta}});
  \node[anchor=east,text=DCtime!90!black,font=\footnotesize]
    at (0.55,{\DCy{\DCdelta}}) {$\delta$};
  \fill[DCvac] ({\DCx{\DCdelta}{\DCcut}},{\DCy{\DCdelta}}) circle (1.25pt);
  \fill[DCgas] ({\DCx{\DCdelta}{\DCouterstart}},{\DCy{\DCdelta}}) circle (1.25pt);
  \ifnum\DCphysical=1
    \node[anchor=west,font=\small\bfseries] at (0.25,\DCbase+4.75)
      {(a) Physical coordinates $(t,x_1)$ at fixed $x_2$};
    \node[anchor=west] at (9.72,\DCbase) {$x_1$};
    \fill[black!75] (7.9,\DCbase) circle (1.1pt);
    \node[anchor=north,font=\footnotesize] at (7.9,\DCbase-0.13)
      {$t=0,\ x_1=0$};
    \node[text=DCgas!85!black,fill=white,inner sep=2pt]
      at (5.15,\DCbase+2.9) {$U^{\delta,\alpha}$};
    \node[text=DCvac,anchor=east,font=\footnotesize]
      at (6.3,\DCbase+0.40) {$c_0=\alpha\mathbf c_*$};
    \draw[DCvac,line width=0.45pt] (6.35,\DCbase+0.42)
      -- ({\DCx{\DCdelta}{\DCcut}},{\DCy{\DCdelta}});
  \else
    \node[anchor=west,font=\small\bfseries] at (0.25,\DCbase+4.75)
      {(b) Opened coordinates $(t,\urr)$ at the same fixed $x_2$};
    \node[anchor=west] at (9.72,\DCbase) {$\urr$};
    \node[anchor=north,font=\footnotesize] at (4.8,\DCbase-0.13)
      {$t=0$: the opened initial interval};
    \node[text=DCgas!85!black,fill=white,inner sep=2pt]
      at (4.1,\DCbase+2.1) {$U^{\delta,\alpha}$};
    \node[text=DCvac,anchor=west,font=\footnotesize]
      at (2.45,\DCbase+0.40) {$c_0=\alpha\mathbf c_*$};
    \draw[DCvac,line width=0.45pt] (2.4,\DCbase+0.42)
      -- ({\DCx{\DCdelta}{\DCcut}},{\DCy{\DCdelta}});
  \fi
  \node[anchor=west,align=left,font=\footnotesize,text=DCvac!85!black]
    at (1.25,\DCbase+4.22) {vacuum limit $\tilde u=0$};
  \draw[DCvac!85!black,line width=0.45pt] (2.0,\DCbase+4.00)
    -- ({\DCx{0.97}{\DCvacuum{0.97}}},{\DCy{0.97}});
  \node[anchor=west,font=\footnotesize,text=DCvac]
    at (4.7,\DCbase+4.22) {cutoff $\tilde u=\tilde u^*$};
  \draw[DCvac,line width=0.45pt] (4.9,\DCbase+4.00)
    -- ({\DCx{0.94}{\DCcutoff{0.94}}},{\DCy{0.94}});
  \node[anchor=west,align=left,text=DCgas,font=\footnotesize]
    at (10.5,\DCbase+3.65)
      {$\tilde u=\tilde u_*$\\[2pt]from $c_0=\frac34\mathbf c_*$ on $t=\delta$};
  \draw[DCgas,line width=0.45pt] (10.25,\DCbase+3.65)
    -- ({\DCx{0.90}{\DCouter{0.90}}},{\DCy{0.90}});
  \node[anchor=west,align=left,text=DCoverlap!85!black,font=\footnotesize]
    at (10.5,\DCbase+2.65)
      {chosen common overlap\\[3pt]
       $\frac{18}{25}\mathbf c_*<c_0<\frac{73}{100}\mathbf c_*$\\[3pt]
       $U^{\delta,\alpha}=U^a$};
  \draw[DCoverlap!85!black,line width=0.55pt] (10.25,\DCbase+2.65)
    -- ({\DCx{0.68}{0.68}},{\DCy{0.68}});
  \node[anchor=west,align=left,text=black!70,font=\footnotesize]
    at (10.5,\DCbase+1.35)
      {known $U^a$ (positive density)\\[3pt]
       $\frac35\mathbf c_*\leq c_0\leq\frac{19}{20}\mathbf c_*$};
  \draw[black!60,line width=0.45pt] (10.25,\DCbase+1.35)
    -- ({\DCx{0.43}{0.86}},{\DCy{0.43}});
  \draw[->,>=stealth,DCtime!90!black,line width=0.6pt]
    (10.7,\DCbase+0.68) -- (10.7,\DCbase+0.15);
  \node[anchor=west,font=\footnotesize,text=DCtime!90!black]
    at (10.85,\DCbase+0.42) {$\delta\rightarrow0$};
  \draw[->,>=stealth,DCvac,line width=0.6pt]
    (13.9,\DCbase+0.42) -- (13.15,\DCbase+0.42);
  \node[anchor=west,font=\footnotesize,text=DCvac]
    at (14.0,\DCbase+0.42) {$\alpha\rightarrow0$};
}
\node[anchor=west,font=\small] at (12.0,5.65) {$x_1=t\urr\quad(t>0)$};
\node[anchor=north,align=center,font=\footnotesize,text=black!75] at (8.1,-0.45)
  {The common limit uses $0<\delta\leq d_0\alpha$.};
\end{tikzpicture}
\caption{Fixed-$x_2$ sections of the near-vacuum region
$\delta\leq t\leq1$, $\tilde u_*\leq\tilde u\leq\tilde u^*$ (blue).
Hatching marks a region where $U^a$ is known. The chosen overlap
$18\mathbf c_*/25<c_0<73\mathbf c_*/100$ (green) lies in both regions,
on the vacuum side of $\tilde u=\tilde u_*$. Its sides are vertical
in opened coordinates and become rays under $x_1=t\urr$.
Red marks the cutoff, pale red the omitted gas, and dashed red the
limiting vacuum boundary. Each curve is the intersection of a surface
with this section.}
\label{fig: double cutoffs}
\end{figure}

The fourth-order bounds for $U^a-\Ur_k$ there supply incoming flux
at most $C\varepsilon^2\int_\delta^t s^{2k+2}\,ds$ before the
near-vacuum estimates are applied.
Theorem~\ref{theorem: a prior estimate in W_2}, with the initial
terms from $t=\delta$ treated in Sections
\ref{sec: Preparations for the energy estimate}--\ref{sec: closing B infinity},
now improves all three bootstrap bounds. Choose $M$ first and then
$\varepsilon_0$, independently of $\delta,\alpha$.
For each fixed pair, the improved bounds give $c\geq c_\alpha>0$,
$\kappa\geq t/2$, an invertible spatial map and bounded physical
$H^4$ norms on $[\delta,T]$. Sobolev extension on the bounded
Lipschitz domains, preserving the known overlap
\cite[Chapter~VI, Section~3]{Stein1970}, and the same local Cauchy
theorem give continuation to $t=1$. Higher finite regularity persists
by the commuted symmetric energy estimates.
The local lifespan may depend on the fixed cutoffs; the estimates
and final time do not. The case $\varepsilon=0$ is the planar fan.

This proves Proposition~\ref{prop: positive time construction}(i).
\end{proof}

\subsection{Passage to the common limit}
\label{subsec: common limit}
\par\noindent\textit{Proof of Proposition~\ref{prop: positive time construction}(ii).}\par
\paragraph{The common limit including vacuum.}
Choose $\alpha_n\rightarrow0$ and $0<\delta_n\leq d_0\alpha_n$.
Identify the label domains with $Q$ by
\begin{equation}\label{eq: fixed strip identification}
b_n=1-\frac43\alpha_n,\qquad
\phi_n(q)=b_nq-2\alpha_n u^*,\qquad \tilde u_*\leq q\leq0.
\end{equation}
If $Y_n$ is the map to $(\urr,x_2)$, set
$\widetilde Y_n(t,q,\vartheta)=Y_n(t,\phi_n(q),\vartheta)$ and
$W_n=U^{\delta_n,\alpha_n}\circ\widetilde Y_n$.
Map norms use the periodic displacement $\widetilde Y_n^2-\vartheta$.
The third-order frame and second-order metric and shift bounds,
together with the field estimates, give, for each $\tau>0$ and all
sufficiently large $n$ with $\delta_n\leq\tau$,
\begin{equation}\label{eq: closed strip compactness bounds}
\sup_{\tau\leq t\leq1}
\left(\|\widetilde Y_n\|_{H^3(Q)}+\|W_n\|_{H^3(Q)}
+\|\partial_t(\widetilde Y_n)\|_{H^2(Q)}
+\|\partial_t(W_n)\|_{H^2(Q)}\right)\leq C_\tau.
\end{equation}
Indeed, $\partial_\vartheta=\sqrt{\slashed g}\hat X$ and
$\partial_{\tilde u}=T+\Xi\sqrt{\slashed g}\hat X$ convert the
spatial bounds. For time derivatives, writing $A_j=v^jI+cS_j$ gives
\begin{equation}\label{eq: material compactness equation}
L(U)=-\frac ct(S_1+\hat T^1I)\partial_{\urr}(U)
     -c(S_2+\hat T^2I)\partial_{x_2}(U).
\end{equation}
Together with $L(\urr)=(v^1-\urr-c\hat T^1)/t$ and
$L(x_2)=v^2-c\hat T^2$, this gives the required time bounds on
$[\tau,1]$.
Compact embedding and interpolation give one diagonal subsequence
such that, for every $\tau>0$ and $2<s<3$,
\begin{equation}\label{eq: closed strip strong convergence}
(\widetilde Y_n,W_n)\longrightarrow(Y,W)
\quad\hbox{in }C([\tau,1];H^s(Q))
\hookrightarrow C([\tau,1];C^1(Q)).
\end{equation}
On the periodic lift the spatial Jacobians are uniformly close to
$\operatorname{diag}(-1,1)$, and the transverse maps have degree one.
The resulting lower Lipschitz bound preserves global injectivity and
an invertible Jacobian. The frame and slice measure are identified below;
we then pass the ordered geometric and field derivatives to this limit.
The fixed overlap also passes to this limit.

\smallskip\noindent\textit{The frame and the slice measure.}\par
We use the sequence and the maps in
\eqref{eq: fixed strip identification}--\eqref{eq: closed strip strong convergence}.
All estimates below hold for the approximations before either cutoff is removed.
The orientation $\hat X_n^2>0$ gives
\[
\begin{aligned}
\sqrt{\slashed g_n}
&=|\partial_\vartheta(t\widetilde Y_n^1,\widetilde Y_n^2)|,\\
\hat X_n
&=\frac{\partial_\vartheta(t\widetilde Y_n^1,\widetilde Y_n^2)}
 {|\partial_\vartheta(t\widetilde Y_n^1,\widetilde Y_n^2)|},\\
\hat T_n
&=\frac{(-\partial_\vartheta(\widetilde Y_n^2),
 t\partial_\vartheta(\widetilde Y_n^1))}{\sqrt{\slashed g_n}}.
\end{aligned}
\]
Since $\phi_n'=b_n$, the coordinate identity gives
\[
\partial_q(t\widetilde Y_n^1,\widetilde Y_n^2)
=b_n(\kappa_n\hat T_n+\Xi_n\sqrt{\slashed g_n}\hat X_n),
\qquad
\Xi_n=\frac{\hat X_n\cdot
\partial_q(t\widetilde Y_n^1,\widetilde Y_n^2)}
 {b_n\sqrt{\slashed g_n}}.
\]
Consequently,
\begin{equation}\label{eq: kappa limit from rescaled map}
\kappa_n=b_n^{-1}\hat T_n\cdot
\partial_q(t\widetilde Y_n^1,\widetilde Y_n^2)
\longrightarrow
\hat T\cdot\partial_q(tY^1,Y^2)=\kappa
\end{equation}
uniformly on $[\tau,1]\times Q$. The denominators are bounded away
from zero, and the approximation bound gives $\kappa\geq t/2$.
The phase defined by $\tilde u(t,Y(t,q,\vartheta))=q$ satisfies
$\nabla_x\tilde u=\hat T/\kappa$. Together with the limiting coordinate
flow this gives $L(\tilde u)=0$, so the recovered frame has the prescribed
acoustic normalization. The factor $b_n^{-1}$ in
\eqref{eq: kappa limit from rescaled map} is necessary for the rescaled label.

The Jacobian identity is
\begin{equation}\label{eq: slice measure after label rescaling}
|\det D\widetilde Y_n|
=\frac{b_n\kappa_n\sqrt{\slashed g_n}}{t},
\qquad
b_n\sqrt{\slashed g_n}\,dq\,d\vartheta
=\frac{t}{\kappa_n}\,d\urr\,dx_2.
\end{equation}
Thus $t/\kappa_n$ converges uniformly to $t/\kappa$ on common compact
subdomains, with fixed positive upper and lower bounds. Uniform convergence
of the maps, inverse Jacobians and first spatial field derivatives passes
the coordinate flow and \eqref{eq: material compactness equation} to the
limit. In particular, $U=W\circ Y^{-1}$ satisfies Euler in the gas region.

\smallskip\noindent\textit{Weak derivatives on every positive-time slice.}\par
Fix $t_0\in(0,1]$ and compact sets
$K\Subset K'\Subset\Omega(t_0)=Y(t_0,Q^\circ)$. There is a compact
interval $I\subset(0,1]$ containing $t_0$, one-sided at $1$, on which
$K'$ lies in the limiting and all sufficiently late approximating images.
Express all quantities in the same opened coordinates on $I\times K'$.
For each of $\hat T_n^2$, $\kappa_n-t$, and $\hat T_n^1+1$, the preceding
recovery gives uniform convergence, while the fixed-coordinate estimates
in the proof of \eqref{eq: higher order control of the frame} give uniform
$H^3(K')$ bounds. Localization and interpolation therefore give strong
convergence in $C(I;H^\sigma(K))$ for $2<\sigma<3$. In particular, the
coefficients of each spatial geometric field, written locally as
$Z_{i,n}=a_{i,n}\cdot D_z$ with $z=(\urr,x_2)$, converge in
$C(I;C^1(\overline K))$.

For each of the three frame components just listed, write $f_n$ for its
approximation and $f$ for its limit. For an ordered word
$Z_{1,n}\cdots Z_{m,n}$ of length $1\leq m\leq3$, integration by
parts gives, for $\psi\in C_c^\infty(K)$,
\[
\begin{aligned}
&\int_K Z_{1,n}\cdots Z_{m,n}(f_n)\psi\,dz\\
&\qquad=-\int_K Z_{2,n}\cdots Z_{m,n}(f_n)
\operatorname{div}_z(\psi a_{1,n})\,dz,
\end{aligned}
\]
where the product after $Z_{1,n}$ is the identity when $m=1$.
The lower-order factor on the right converges strongly in $C(I;L^2(K))$;
the coefficient $a_{1,n}$ converges in $C(I;C^1(\overline K))$.
The right side therefore converges uniformly in time and identifies the
limit with the distribution $Z_1\cdots Z_m(f)$. The uniform $L^2$ bound
gives an $L^2$ representative at every time. Density of smooth tests then
gives
\begin{equation}\label{eq: locally uniform weak geometry convergence}
\sup_{t\in I}\left|\int_K
\bigl(Z_{1,n}\cdots Z_{m,n}(f_n)-Z_1\cdots Z_m(f)\bigr)\psi\,dz\right|
\longrightarrow0,\qquad \psi\in L^2(K).
\end{equation}
This representative is weakly continuous in time, and representatives
on overlapping domains agree distributionally.

At each fixed time, the uniform convergence of $t/\kappa_n$ and weak
convergence of the ordered derivatives allow weighted lower
semicontinuity. Exhausting $\Omega(t_0)$ gives
\begin{equation}\label{eq: cutoff limit geometric bounds}
\begin{aligned}
\|Z^\beta(\hat T^1+1)(t_0)\|_{L^2(\Omega(t_0),(t_0/\kappa)\,dz)}
&\leq P(M)\varepsilon^2t_0^2,\\
\|Z^\beta(\hat T^2)(t_0)\|_{L^2(\Omega(t_0),(t_0/\kappa)\,dz)}
&\leq P(M)\varepsilon t_0,\\
\|Z^\beta(\kappa-t)(t_0)\|_{L^2(\Omega(t_0),(t_0/\kappa)\,dz)}
&\leq P(M)\varepsilon t_0^2,\qquad |\beta|\leq3.
\end{aligned}
\end{equation}
Apply lower semicontinuity to the finite sum before exhausting the domain
to keep the original summed bounds and constants. For $\slashed g_n$ and
$\Xi_n$, the recovered uniform convergence and spatial $H^2$ bounds give
local $C_tH^\sigma$ convergence for $1<\sigma<2$. The same argument passes
their ordered derivative bounds through order two.

For the fields, interpolation of the spatial $H^4$ bounds with the already
established convergence gives local $C_tH^s$ convergence for $3<s<4$ on
common compact subdomains of the gas region. Ordered field derivatives of
length at most three converge strongly in $C_tL^2$. The terms with two
derivatives of a frame coefficient combine its strong $L^2$ convergence
with uniform convergence of the first field derivatives; the remaining
terms converge by the stated Sobolev convergence. The same integration by
parts identifies the weak derivatives of order four. Fixed-coordinate
derivatives are the special case of constant coefficients. Their $L^2$
bounds and the $W^{2,\infty}$ bounds pass on each time slice by lower
semicontinuity and exhaustion; the latter also follow directly from local
convergence of the second derivatives. The spatial second derivatives are
continuous, and the Euler equation gives continuous mixed and second time
derivatives. Thus the limit is $C^2$ in the positive-time gas interior.
These arguments pass the geometric and field estimates in part~(ii) to
every positive-time slice.

Finally, the first-order remainder and frame estimates give
$\|L(\urr)\|_{W^{1,\infty}_{\urr,x_2}}+
\|L(x_2)\|_{W^{1,\infty}_{\urr,x_2}}\leq C\varepsilon$.
Integrating the flow and its spatial linearization gives
\[
\left\|\widetilde Y_n(t)-
\left(\left(v_r^1-\frac{2c_r}{\gamma-1}\right)(0,\vartheta)-\phi_n(q),
\vartheta\right)\right\|_{C^1(Q)}
\leq C\varepsilon(t-\delta_n).
\]
This proves the asserted extension at $t=0$; the remainder bound
gives the initial affine profile. At the initial time,
$c_n(\delta_n,\phi_n(q),\vartheta)
=-\frac{\gamma-1}{\gamma+1}\phi_n(q)+O(\varepsilon\delta_n)$.
Integrating $|L(c_n)|\leq C\varepsilon c_n$ and passing to the same
limit yields
\begin{equation}\label{eq: limit sound speed through vacuum}
e^{-C\varepsilon t}\frac{\gamma-1}{\gamma+1}(-q)
\leq c(t,q,\vartheta)
\leq e^{C\varepsilon t}\frac{\gamma-1}{\gamma+1}(-q).
\end{equation}
Thus $c>0$ for $q<0$ and $c=0$ exactly at $q=0$.
The first-order Euler formulas also pass to the limit and give
$|L(c)|+|L(v)|\leq C\varepsilon c$.
The vacuum velocity trace and graph are identified in the next subsection. This proves Proposition~\ref{prop: positive time construction}(ii).
\hfill$\square$\par

\subsection{\texorpdfstring{The boundary graphs in Theorem~\ref{theorem: stability of RVR}}{The boundary graphs in the main theorem}}
\label{subsubsec: boundary graphs theorem}
We record explicitly how the preceding construction gives the graph
description in the theorem. Write
$(c_r,v_r)(t,x_1,x_2)$ and $(c_l,v_l)(t,x_1,x_2)$ for the classical
right and left exterior solutions. Their initial spatial traces are the
two-argument functions $(c_r,v_r)(0,x_2)$ and $(c_l,v_l)(0,x_2)$ defined
in Section~\ref{subsec: main theorem preparation}. The outer graphs solve
\[
\begin{aligned}
\partial_t(f_0)(t,x_2)
&=\left[v_r^1-v_r^2\partial_{x_2}(f_0)
+c_r\sqrt{1+(\partial_{x_2}(f_0))^2}\right](t,f_0(t,x_2),x_2),\\
\partial_t(\overline f_0)(t,x_2)
&=\left[v_l^1-v_l^2\partial_{x_2}(\overline f_0)
-c_l\sqrt{1+(\partial_{x_2}(\overline f_0))^2}\right]
(t,\overline f_0(t,x_2),x_2),
\end{aligned}
\]
with $f_0(0,x_2)=\overline f_0(0,x_2)=0$. In these formulas only the
exterior fields are evaluated at the indicated three arguments; the graph
derivatives are evaluated at $(t,x_2)$. For example, on the right graph,
\[
\hat T=\frac{(-1,\partial_{x_2}(f_0))}{\sqrt{1+(\partial_{x_2}(f_0))^2}},\qquad
L(x_1-f_0)=0.
\]
Substitution into $L=\partial_t+(v-c\hat T)\cdot\nabla_x$ gives
\[
0=v_r^1-\partial_t(f_0)-v_r^2\partial_{x_2}(f_0)
+c_r\sqrt{1+(\partial_{x_2}(f_0))^2},
\]
which proves the first graph equation, with the prescribed convention
$\hat T^1\approx-1$. The second follows by reflection. Setting $t=0$
gives the two initial graph speeds stated in
Section~\ref{subsec: main theorem preparation}. Thus the function
$f$ used in Section~\ref{subsubsec: C_0 in t ur x_2} is exactly $f_0$.
The exterior classical estimates and these graph equations give
$\|f_0-t(\mathbf v_r^1+\mathbf c_r)\|_{W^{2,\infty}(\mathbb T)}
+\|\overline f_0-t(\mathbf v_l^1-\mathbf c_l)\|_{W^{2,\infty}(\mathbb T)}
\leq C\varepsilon t$.

Proposition~\ref{prop: positive time construction} gives convergence
on $Q=[\tilde u_*,0]\times\mathbb T$ in $(q,\vartheta)$ and identifies
the vacuum boundary $q=0$ by \eqref{eq: limit sound speed through vacuum}.
We record the corresponding trajectory estimate on any positive
label to verify the graph formula. Set
\[
b(y)=\left(v_r^1-\frac{2c_r}{\gamma-1},v_r^2\right)(0,y),
\]
and let $X_\sigma(t,y)$ be the physical trajectory with
$\tilde u=-\sigma<0$ and initial tangential label $y$. Thus
$X_\sigma(0,y)=(0,y)$ and
$\partial_t(X_\sigma)=v-c\hat T$. By
\eqref{eq: c almost constant on C_u}, $c(t,-\sigma,y)\leq C\sigma$.
The initial profile gives
$|v(0,-\sigma,y)-b(y)|\leq C\sigma$. Integrating
$|L(v)|\leq C\varepsilon c$ along this trajectory therefore yields
\[
|v(t,-\sigma,y)-b(y)|\leq C\sigma+C\varepsilon\sigma t,\qquad
|X_\sigma(t,y)-(tb^1(y),y+tb^2(y))|\leq C\sigma t.
\]
The constants are independent of $\sigma$. In particular, convergence
also holds after dividing the first physical coordinate by $t$.
The limiting trajectories satisfy
\[
x_1=t\left(v_r^1-\frac{2c_r}{\gamma-1}\right)(0,y),\qquad
x_2=y+tv_r^2(0,y).
\]
The same limit gives the gas-side velocity trace $v(t,0,y)=b(y)$
in characteristic coordinates. Hence the trajectory at $\tilde u=0$
satisfies $\partial_t(X)=v$ and $c=0$, and its graph obeys
\[
\partial_t(f_{\mathrm{vac}})-v^1+v^2\partial_{x_2}(f_{\mathrm{vac}})=0.
\]
The derivative of the second map with respect to $y$ is
$1+t\partial_y(v_r^2)(0,y)\geq1/2$. It is therefore an invertible
degree-one map of $\mathbb T$, and substitution of its inverse $y_r$
gives $f_{\mathrm{vac}}$. The reflected calculation gives
$\overline f_{\mathrm{vac}}$. For clarity, differentiating the inverse gives
\[
\partial_{x_2}(y_r)=\frac{1}{1+t\partial_y(v_r^2)(0,y_r)},\qquad
\partial_{x_2}^2(y_r)=
-\frac{t\partial_y^2(v_r^2)(0,y_r)}{(1+t\partial_y(v_r^2)(0,y_r))^3}.
\]
Also
\[
\partial_t(y_r)=-\frac{v_r^2(0,y_r)}
 {1+t\partial_y(v_r^2)(0,y_r)}.
\]
The same identities hold on the left. The implicit function theorem
with the denominator bounded below by $1/2$ gives $C^2$ regularity
jointly in $(t,x_2)$, including $t=0$. The chain rule gives the stated
$W^{2,\infty}(\mathbb T)$ estimates for the vacuum graphs. The exterior
classical fields and their characteristic graph equations likewise
give $C^2$ regularity of the outer graphs on the closed time interval. The four
unperturbed speeds are strictly ordered; their three consecutive gaps
are
\[
\frac{\gamma+1}{\gamma-1}\mathbf c_l,\qquad
\mathbf v_r^1-\mathbf v_l^1-\frac{2(\mathbf c_r+\mathbf c_l)}{\gamma-1},\qquad
\frac{\gamma+1}{\gamma-1}\mathbf c_r.
\]
All are positive. The $C\varepsilon t$ graph estimates preserve this
ordering when $\varepsilon_0$ is sufficiently small.

\subsection{The left profiles and the comparison estimates}
\label{subsubsec: left profile theorem}
The right zeroth-order profile follows from
$v^1+c=\urr$, $v^1-2c/(\gamma-1)=(v_r^1-2c_r/(\gamma-1))(0,x_2)$,
and $v^2=v_r^2(0,x_2)$. On the left these identities become
\[
v^1-c=\urr,\qquad
v^1+\frac{2c}{\gamma-1}
=\left(v_l^1+\frac{2c_l}{\gamma-1}\right)(0,x_2),\qquad
v^2=v_l^2(0,x_2).
\]
Subtracting their values at $\urr=(v_l^1-c_l)(0,x_2)$ yields
\[
c-c_l(0,x_2)=-\frac{\gamma-1}{\gamma+1}
\bigl(\urr-(v_l^1-c_l)(0,x_2)\bigr),\qquad
v^1-v_l^1(0,x_2)=\frac{2}{\gamma+1}
\bigl(\urr-(v_l^1-c_l)(0,x_2)\bigr).
\]
These are precisely the components of $\underline{\Vr}_0$ in
\eqref{eq: leading profiles main theorem}.

For the first-order profile and all remainder estimates, use the exact
Euler symmetry
\[
(x_1,c,v^1,v^2)\longmapsto(-x_1,c,-v^1,v^2).
\]
More precisely, set $S=\operatorname{diag}(1,-1,1)$ and apply the right-hand
construction to the normalized data
\[
\left(\frac{2c_l(-x_1,x_2)}{\gamma-1},-v_l^1(-x_1,x_2),v_l^2(-x_1,x_2)\right)^T,
\qquad x_1\geq0.
\]
If $\Vr_j^{\mathrm{ref}}$ denotes the resulting right profile, then
\[
\underline{\Vr}_j(\urr,x_2)=S\Vr_j^{\mathrm{ref}}(-\urr,x_2),\qquad j=0,1.
\]
For $j=1$, $\Vr_1^{\mathrm{ref}}$ is given by
\eqref{eq: exact form of Vr_1} with these reflected data; in particular,
all boundary derivatives occurring there are computed for the reflected
exterior solution. Reflection preserves all
the spatial norms in the theorem.

The estimates in the two fans follow from those for $V_k$ proved
above and in \eqref{eq: improved L infty bound up to vacuum}, since
$\Tr=-\partial_{\urr}$ and $\Xr=\partial_{x_2}$.
The exterior estimates for $U$ follow from the classical solution
estimates on $\mathcal D_l(t)$ and $\mathcal D_r(t)$.

\subsection{Conservation laws, entropy, and the initial trace}
\label{subsubsec: entropy gluing}
In each gas region recover the density and momentum by
\[
\rho=\left(\frac{c}{\sqrt{k\gamma}}\right)^{\frac{2}{\gamma-1}},
\qquad m=\rho v,
\]
and set $(\rho,m)=(0,0)$ in vacuum. The bounds for $U$ imply that
$(\rho,m)$ is bounded. Since $c\to0$ and $v$ remains bounded at each
vacuum boundary, the density, momentum, stress
$\rho v\otimes v+p\,\mathrm{Id}$, mechanical energy and energy flux
all have zero gas-side traces there. At the two outer characteristic
fronts the gas states agree, so their Euler and energy fluxes agree
as well.

Fix $\delta>0$ and integrate the classical conservation and energy
identities over the gas regions in $\{t>\delta\}$ against compactly
supported smooth tests. For an interface with spacetime conormal
$(n_t,n_x)$, the boundary terms are the jumps of
$n_t\rho+n_x\cdot m$,
$n_tm+(\rho v\otimes v+p\,\mathrm{Id})n_x$, and
$n_t\eta+\sum_jn_{x_j}q_j$, respectively.
The preceding trace identities make all these jumps zero.
Thus the glued pair satisfies Euler and the mechanical energy
equality in distributions for $t>0$, and hence the required entropy
inequality.

The region $F_1(t,x_2)<x_1<F_4(t,x_2)$ has width $O(t)$.
The boundedness just proved makes its contribution to any fixed local
spatial integral tend to zero as $t\rightarrow0$. Outside this region,
the classical exterior solutions converge to their prescribed initial
data. Therefore $(\rho,m)$ and its mechanical energy attain their
initial values in $L^1_{\mathrm{loc}}$.
Letting $\delta\rightarrow0$ in the integrated identities supplies
the initial terms in both weak formulations, tested against compactly
supported functions.

\subsection{\texorpdfstring{Time derivatives for $\gamma=7/5$ and $\gamma=5/3$}{Time derivatives for gamma = 7/5 and gamma = 5/3}}
\label{subsubsec: time regularity theorem}
In this paragraph $U$ is expressed in $(t,\urr,x_2)$; thus it is
$U^\sharp$ in the notation of the main theorem. Every $\partial_t$
below holds $(\urr,x_2)$ fixed.
Write $\Ur_1=\Vr_0+t\Vr_1$ and $V_1=U-\Ur_1$. On the right fan the
proved estimate is
\[
\sum_{a+b\leq2}\|\partial_{\urr}^a\partial_{x_2}^b(V_1)(t,\cdot)
\|_{L^\infty(\Sigma_t^r)}\leq C\varepsilon t^2.
\]
The powers in \eqref{eq: exact form of Vr_1} are integers for the two
exponents under consideration. Consequently $\Vr_1$ and its spatial
derivatives through order two are bounded by $C\varepsilon$ on these
slices. The explicit residual in Lemma~\ref{lem: bounds on Res} is
\[
\Rr_1=t^2\left(A_1(\Vr_1)\partial_{\urr}(\Vr_1)
+A_2(\Vr_0)\partial_{x_2}(\Vr_1)
+A_2(\Vr_1)\partial_{x_2}(\Vr_0)\right)
+t^3A_2(\Vr_1)\partial_{x_2}(\Vr_1).
\]
Its coefficients are independent of $t$. In particular,
\[
\|\Rr_1(t,\cdot)\|_{W^{1,\infty}(\Sigma_t^r)}\leq C\varepsilon t^2,
\qquad
\|\partial_t(\Rr_1)(t,\cdot)\|_{L^\infty(\Sigma_t^r)}\leq C\varepsilon t.
\]
Here $\partial_t(\Rr_1)$ differentiates the displayed polynomial at fixed
$(\urr,x_2)$.

Subtracting the approximate equation from the Euler equations, and using
the linearity of $A_1,A_2$ in their arguments, gives
\[
\begin{aligned}
t\partial_t(V_1)={}&-(A_1(U)-\urr I)\partial_{\urr}(V_1)
-tA_2(U)\partial_{x_2}(V_1)
-A_1(V_1)\partial_{\urr}(\Ur_1)
-tA_2(V_1)\partial_{x_2}(\Ur_1)-\Rr_1.
\end{aligned}
\]
All coefficients and spatial derivatives of $U,\Ur_1$ through order two
are uniformly bounded. Each term on the right is $O(\varepsilon t^2)$,
with a constant allowed to depend on $t^{*}$. Division by $t$ proves
\[
\|\partial_t(V_1)(t,\cdot)\|_{L^\infty(\Sigma_t^r)}\leq C\varepsilon t.
\]
For $D\in\{\partial_{\urr},\partial_{x_2}\}$, differentiation gives
\[
\begin{aligned}
t\partial_tD(V_1)={}&-[A_1(D(U))-(D(\urr))I]\partial_{\urr}(V_1)
-(A_1(U)-\urr I)\partial_{\urr}D(V_1)\\
&-tA_2(D(U))\partial_{x_2}(V_1)-tA_2(U)\partial_{x_2}D(V_1)\\
&-A_1(D(V_1))\partial_{\urr}(\Ur_1)-A_1(V_1)\partial_{\urr}D(\Ur_1)\\
&-tA_2(D(V_1))\partial_{x_2}(\Ur_1)-tA_2(V_1)\partial_{x_2}D(\Ur_1)-D(\Rr_1).
\end{aligned}
\]
Here $D(\urr)=1$ for $D=\partial_{\urr}$ and $D(\urr)=0$ for
$D=\partial_{x_2}$. Again every term on the right is
$O(\varepsilon t^2)$. Thus
\[
\|\partial_tD(V_1)(t,\cdot)\|_{L^\infty(\Sigma_t^r)}\leq C\varepsilon t,
\qquad
\|\partial_t(U)(t,\cdot)\|_{L^\infty(\Sigma_t^r)}
=\|\Vr_1+\partial_t(V_1)\|_{L^\infty(\Sigma_t^r)}\leq C\varepsilon.
\]
Finally, differentiation of the remainder equation with respect to $t$
gives the full identity
\[
\begin{aligned}
t\partial_t^2(V_1)={}&-\partial_t(V_1)
-A_1(\partial_t(U))\partial_{\urr}(V_1)
-(A_1(U)-\urr I)\partial_{\urr}\partial_t(V_1)\\
&-A_2(U)\partial_{x_2}(V_1)
-tA_2(\partial_t(U))\partial_{x_2}(V_1)
-tA_2(U)\partial_{x_2}\partial_t(V_1)\\
&-A_1(\partial_t(V_1))\partial_{\urr}(\Ur_1)
-A_1(V_1)\partial_{\urr}(\Vr_1)\\
&-A_2(V_1)\partial_{x_2}(\Ur_1)
-tA_2(\partial_t(V_1))\partial_{x_2}(\Ur_1)
-tA_2(V_1)\partial_{x_2}(\Vr_1)-\partial_t(\Rr_1).
\end{aligned}
\]
The bounds already obtained show that every term on the right is bounded
by $C\varepsilon t$. Hence
\[
\|\partial_t^2(V_1)(t,\cdot)\|_{L^\infty(\Sigma_t^r)}\leq C\varepsilon,
\quad
\|\partial_tD(U)(t,\cdot)\|_{L^\infty(\Sigma_t^r)}\leq C\varepsilon,
\quad
\|\partial_t^2(U)(t,\cdot)\|_{L^\infty(\Sigma_t^r)}\leq C\varepsilon.
\]
The identities hold in the $C^2$ gas interior established in
Proposition~\ref{prop: positive time construction}; equivalently,
they follow by weak differentiation of the Euler equation. The estimates are
uniform up to the vacuum and as $t\rightarrow0$, and therefore give the
claimed weak derivatives on $\mathcal W_r$. Reflection proves the same
statements on $\mathcal W_l$. This proves the spacetime bounds for
$U^\sharp$ in both fans.

For the other exponents treated in this paper, the available estimate
for $V_0=U-\Vr_0$ is $O(\varepsilon t)$ in the spatial
$W^{2,\infty}(\Sigma_t^r)$ norm. The first two calculations then give only
$\partial_t(V_0)=O(\varepsilon)$ and
$D\partial_t(V_0)=O(\varepsilon)$; a further time differentiation can give
$O(\varepsilon/t)$. Thus these estimates alone do not establish uniform
second time derivatives. This is why
Theorem~\ref{theorem: stability of RVR}(iii) is stated on time slices,
whereas part~(iv) of that theorem gives the additional spacetime bounds.

\subsection{Returning to the coordinates of the main theorem}
\label{subsubsec: undo normalization theorem}
We check the fixed-coordinate profile comparisons when undoing the
normalization of Section~\ref{subsec: opened fan}. This matters in
particular when the original reference transverse velocity is nonzero.
For one fan, keep the original constants $c_*,v_*,T_*$ and write
\[
r'=\frac{\urr-v_*^1}{c_*},\qquad
y'_t=\frac{x_2-v_*^2t}{c_*T_*},\qquad
\mathcal A z=c_*z+(0,v_*^1,v_*^2)^T.
\]
Let $P'_0,P'_1$ be the profiles in the normalized variables. The
corresponding profiles in the original opened coordinates are
\begin{equation}\label{eq: profiles under normalization}
\begin{aligned}
P_0(\urr,x_2)&=\mathcal A P'_0(r',y'_0),\\
P_1(\urr,x_2)&=\frac{c_*}{T_*}P'_1(r',y'_0)
                         -v_*^2\partial_{x_2}(P_0)(\urr,x_2).
\end{aligned}
\end{equation}
The second identity is the chain rule at fixed $(\urr,x_2)$.
It also follows by transforming the first profile equation and
its exterior boundary data. Those determine the profile uniquely
on every positive-sound-speed interval, so it is exactly the
profile prescribed in \eqref{eq: exact form of Vr_1}, and, after
reflection, the left profile of the theorem.

The initial data order supplies the spatial derivatives of $P'_0$
through order four and of $P'_1$ through order three needed here.
Taylor's formula in the transverse translation, applied also to
spatial derivatives through order two, gives
\[
\left\|\mathcal A P'_0(r',y'_t)
 +\frac{c_*t}{T_*}P'_1(r',y'_t)-P_0-tP_1\right\|_{W^{2,\infty}}
\leq C\varepsilon t^2.
\]
All differentiated coefficients of this translation error are
$O(\varepsilon)$, since the planar profiles are independent of
the transverse coordinate. This proves that the $k=1$ remainder
estimate has precisely the first-order profile in the main theorem.
For $k=0$ the first-order translation estimate gives the required
$C\varepsilon t$ comparison with $P_0$.
The opened-coordinate change $(t,\urr,x_2)\mapsto(t/T_*,r',y'_t)$
is affine with fixed coefficients. It preserves the asserted
spatial and spacetime $W^{2,\infty}$ bounds up to constants depending
only on the fixed parameters. Recover $(\rho,m)$ after returning to
these coordinates, as in Section~\ref{subsubsec: entropy gluing}.

Both one-sided constructions therefore have all the asserted
estimates in the same original coordinates. Take the smaller of
their fixed smallness thresholds and the threshold that preserves
the three positive reference gaps. The preceding boundary,
comparison, entropy, and time-regularity arguments then prove
every conclusion of Theorem~\ref{theorem: stability of RVR} on
the prescribed interval $[0,t^*]$. This completes the proof of Theorem~\ref{theorem: stability of RVR}.

\section*{Acknowledgment}

The authors thank Ruotong Zhang for helpful discussions on the
construction of the initial profile.

\thanks{JJ is supported by Starting Research Fund from Hunan University under Grant No. 531118010918. TL is supported by  a grant from the Research Grants Council of the Hong Kong Special
Administrative Region, China (Project No. 11313025).}

\section*{AI Use Declaration}
The authors used OpenAI's GPT-6 model to polish the language,
prepare TikZ figures, identify and correct typographical errors,
and review the exposition for clarity and logical completeness.
In this review, GPT-6 was used to identify possible omissions
in the exposition and to suggest editorial improvements.
No mathematical ideas or proof steps proposed by GPT-6 were
adopted in the mathematical arguments, and GPT-6 was not used
to carry out mathematical derivations or establish any of
the results. The authors reviewed all AI-assisted changes and
take full responsibility for the contents of the manuscript,
including the correctness of the mathematical arguments and
the completeness and accuracy of the references.

\end{document}